\documentclass[10pt,a4paper]{article}

\pdfoutput=1

\usepackage{a4}
\usepackage{amsmath}
\usepackage{amssymb}
\usepackage{amsthm}
\usepackage{mathrsfs}
\usepackage{eucal}
\usepackage{graphicx}
\usepackage{epsfig}
\usepackage[curve]{xypic}
\usepackage{stmaryrd}
\usepackage[intoc]{nomencl}
\usepackage{dsfont}
\usepackage{fancyhdr}
\usepackage{titlesec}
\usepackage{subfigure}
\usepackage[german,english]{babel}
\usepackage{enumerate}
\usepackage{color}
\usepackage{import}
\usepackage{pdflscape}
\usepackage{ifpdf}

\usepackage{hyperref}
\hypersetup {
  pdfpagemode = {None}, 
  pdftitle = {Geometric Transversality in Higher Genus Gromov-Witten Theory},
  pdfauthor = {Andreas Gerstenberger}
}

\mathchardef\mhyphen="2D

\DeclareFontFamily{U}{pxsyc}{}
\DeclareFontShape{U}{pxsyc}{m}{n}{<-> pxsyc}{}
\DeclareFontShape{U}{pxsyc}{bx}{n}{<-> pxbsyc}{}
\DeclareFontShape{U}{pxsyc}{l}{n}{<->ssub * pxsyc/m/n}{}
\DeclareFontShape{U}{pxsyc}{b}{n}{<->ssub * pxsyc/bx/n}{}
\DeclareSymbolFont{symbolsC}{U}{pxsyc}{m}{n}
\DeclareMathSymbol{\Perp}{\mathbin}{symbolsC}{121}

\DeclareFontFamily{U}{mathb}{\hyphenchar\font45}
\DeclareFontShape{U}{mathb}{m}{n}{
      <5> <6> <7> <8> <9> <10> gen * mathb
      <10.95> mathb10 <12> <14.4> <17.28> <20.74> <24.88> mathb12
      }{}
\DeclareSymbolFont{mathb}{U}{mathb}{m}{n}
\DeclareFontSubstitution{U}{mathb}{m}{n}
\DeclareMathSymbol{\corresponds}{\mathbin}{mathb}{'034}
\DeclareMathSymbol{\operateson}{\mathbin}{mathb}{'374}
\DeclareMathSymbol{\operatedonby}{\mathbin}{mathb}{'375}

\DeclareFontFamily{U}{mathx}{\hyphenchar\font45}
\DeclareFontShape{U}{mathx}{m}{n}{
      <5> <6> <7> <8> <9> <10>
      <10.95> <12> <14.4> <17.28> <20.74> <24.88>
      mathx10
      }{}
\DeclareSymbolFont{mathx}{U}{mathx}{m}{n}
\DeclareFontSubstitution{U}{mathx}{m}{n}
\DeclareMathSymbol{\bigtimes}{1}{mathx}{'221}

\title{Geometric Transversality in Higher Genus Gromov-Witten Theory}

\author{Andreas Gerstenberger}

\date{\today}

\begin{document}

\renewcommand{\H}{\mathbb{H}}
\newcommand{\D}{\mathbb{D}}
\newcommand{\field}{\mathds{k}}
\newcommand{\R}{\mathds{R}}
\newcommand{\C}{\mathds{C}}
\newcommand{\Z}{\mathds{Z}}
\newcommand{\N}{\mathds{N}}
\newcommand{\Q}{\mathds{Q}}
\newcommand{\Ztwo}{\Z_2}
\def\set#1{\{\s@t#1&\}}
\def\s@t#1&#2&{#1\;|\;#2}
\def\infset#1{{\inf\{\s@t#1&\}}}
\def\s@t#1&#2&{#1\;|\;#2}
\def\supset#1{{\sup\{\s@t#1&\}}}
\def\s@t#1&#2&{#1\;|\;#2}
\def\minset#1{{\min\{\s@t#1&\}}}
\def\s@t#1&#2&{#1\;|\;#2}
\def\maxset#1{{\max\{\s@t#1&\}}}
\def\s@t#1&#2&{#1\;|\;#2}
\newcommand{\definedas}{\mathrel{\mathop:}=}
\newcommand{\defines}{=\mathrel{\mathop:}}
\newcommand{\definedequiv}{\mathrel{\mathop:}\Leftrightarrow}

\newcommand{\tensoralg}{\mathrm{T}}
\newcommand{\symmalg}{\mathrm{S}}
\newcommand{\extalg}{\Lambda}
\newcommand{\Pol}{\operatorname{Pol}}
\newcommand{\extmult}{\mathrm{e}}
\newcommand{\insertion}{\iota}
\newcommand{\Sym}{\operatorname{Sym}}

\newcommand{\category}[1]{\text{\textnormal{\textbf{#1}}}}
\newcommand{\catC}{\category{C}}
\newcommand{\kAlg}{\text{\textnormal{$\field$-\textbf{Alg}}}}
\newcommand{\abGrp}{\text{\textnormal{\textbf{ab-Grp}}}}
\newcommand{\Top}{\text{\textnormal{\textbf{Top}}}}
\newcommand{\pchTop}{\text{\textnormal{\textbf{pch-Top}}}}
\newcommand{\mor}[3]{\operatorname{Mor}_{#1}({#2},{#3})}
\newcommand{\Mor}{\operatorname{Mor}}
\renewcommand{\hom}{\operatorname{Hom}}
\newcommand{\Hom}{\operatorname{Hom}}
\newcommand{\Iso}{\operatorname{Iso}}
\newcommand{\Aut}{\operatorname{Aut}}
\newcommand{\End}{\operatorname{End}}
\newcommand{\equ}{\operatorname{Equ}}

\newcommand{\open}{\text{\textnormal{\textbf{Open}}}}
\newcommand{\sheaf}{\mathscr}
\renewcommand{\O}{\mathcal{O}}
\newcommand{\Prshf}[2]{\text{\textnormal{\textbf{Prshf}}}({#1},{#2})}
\newcommand{\Shf}[2]{\text{\textnormal{\textbf{Shf}}}({#1},{#2})}
\newcommand{\homol}[3]{H^{#1}({#2};{#3})}
\newcommand{\stack}{\mathfrak}

\newcommand{\ie}{i.\,e.~}
\newcommand{\st}{s.\,t.~}
\newcommand{\wrt}{w.\,r.\,t.~}
\newcommand{\Wlog}{W.\,l.\,o.\,g.~}
\newcommand{\smallwlog}{w.\,l.\,o.\,g.~}
\newcommand{\eg}{e.\,g.~}
\newcommand{\cf}{cf.~}
\renewcommand{\iff}{iff~}
\newcommand{\etc}{etc.~}
\newcommand{\iid}{i.\,i.\,d.~}
\newcommand{\fs}{f.\,s.~}
\newcommand{\fa}{f.\,a.~}
\newcommand{\zB}{z.\,B.~}
\renewcommand{\dh}{d.\,h.~}
\renewcommand{\ae}{a.\,e.~}

\newcommand{\matrices}{\mathrm{Mat}}
\newcommand{\trace}{\operatorname{Tr}}
\newcommand{\supertrace}{\operatorname{Str}}
\newcommand{\superdet}{\operatorname{Sdet}}
\newcommand{\transpose}[1]{^\mathrm{t}{#1}}
\newcommand{\supertranspose}[1]{^\mathrm{st}{#1}}
\newcommand{\parity}{\mathrm{p}}
\newcommand{\commutator}[2]{[#1,#2]}
\newcommand{\supercommutator}[2]{\commutator{#1}{#2}_\mathrm{s}}
\newcommand{\im}{\operatorname{im}}
\newcommand{\envelopingalg}[1]{\mathcal{U}(#1)}

\newcommand{\Cf}{\mathcal{C}}
\newcommand{\Cffinord}{\mathcal{C}_{\mathrm{fo}}}
\newcommand{\Gammafinord}{\Gamma_{\mathrm{fo}}}
\newcommand{\Omegafinord}{\Omega_{\mathrm{fo}}}
\newcommand{\Xfinord}{\mathfrak{X}_{\mathrm{fo}}}
\newcommand{\DiffOpfinord}{\mathrm{DO}_{\mathrm{fo}}}
\newcommand{\DiffOp}{\mathrm{DO}}
\newcommand{\TDO}{\mathrm{TDO}}
\newcommand{\TDOfinord}{\mathrm{TDO}_{\mathrm{fo}}}
\newcommand{\functionalforms}{\mathcal{F}}
\newcommand{\euler}{\mathfrak{E}}
\newcommand{\eulerlagrange}{\mathfrak{E}}
\newcommand{\interioreuler}{\mathfrak{I}}
\newcommand{\functional}[1]{\mathfrak{#1}}
\newcommand{\functionals}{\mathfrak{Func}}
\newcommand{\contactideal}{\mathcal{C}}
\newcommand{\pdr}{\mathcal{R}}
\renewcommand{\d}{\mathrm{d}}
\newcommand{\Lie}{\mathcal{L}}
\newcommand{\DiffeoLoc}{\mathrm{Diff}_{\mathrm{loc}}}
\newcommand{\Diff}{\mathrm{Diff}}
\newcommand{\laplace}{\triangle}
\newcommand{\dvol}{\mathrm{dvol}}
\newcommand{\vol}{\mathrm{vol}}
\newcommand{\Lagrangian}{\mathcal{L}}
\newcommand{\dbar}{\overline{\partial}}
\newcommand{\ind}{\operatorname{ind}}
\newcommand{\inj}{\mathrm{inj}}
\newcommand{\dist}{\mathrm{dist}}
\newcommand{\coker}{\operatorname{coker}}

\newcommand{\permutations}{\mathcal{S}}
\newcommand{\sign}{\operatorname{sign}}
\newcommand{\GL}{\mathrm{GL}}
\newcommand{\U}{\mathrm{U}}
\newcommand{\SU}{\mathrm{SU}}
\newcommand{\SO}{\mathrm{SO}}

\newcommand{\projRinfty}{\mathrm{p}}
\newcommand{\hor}{\mathrm{Hor}}
\newcommand{\conn}{\mathrm{Conn}}
\newcommand{\der}{\mathrm{Der}}
\newcommand{\graph}{\mathrm{Graph}}
\newcommand{\prolongation}{\operatorname{pr}}
\newcommand{\totalvf}[1]{\operatorname{tot}({#1})}
\newcommand{\evolutionaryvf}[1]{{#1}_{\mathrm{ev}}}
\newcommand{\evol}{\mathrm{Evol}}
\newcommand{\diag}{\mathrm{diag}}
\newcommand{\framebundle}[1]{\mathcal{F}(#1)}
\newcommand{\orthogonalframebundle}[1]{\mathcal{O}(#1)}

\newcommand{\sigmaAlg}[1]{\mathscr{#1}}
\newcommand{\saF}{\sigmaAlg{F}}
\newcommand{\saB}{\sigmaAlg{B}}
\newcommand{\prob}[1]{\mathbb{#1}}
\newcommand{\varProb}[1]{\mathbf{#1}}
\newcommand{\wpo}{w.\,p.\,$1$}
\newcommand{\Fix}{\mathrm{Fix}}
\newcommand{\condexp}[2]{\prob{E}(#1\,|\,#2)}
\newcommand{\T}{\mathbb{T}}
\newcommand{\lstint}{\llbracket}
\newcommand{\rstint}{\rrbracket}
\newcommand{\cadlag}{c\`adl\`ag}
\newcommand{\quadrvar}[1]{\langle #1\rangle}
\newcommand{\pathspace}[1]{\Pi(#1)}
\newcommand{\martingales}{\mathcal{M}}

\newcommand{\inv}{^{-1}}
\newcommand{\id}{\mathrm{id}}
\newcommand{\eval}{\operatorname{ev}}
\newcommand{\ev}{\eval}
\newcommand{\supp}{\operatorname{supp}}
\newenvironment{notyetdone}{{\large\textbf{Not yet done: }}}{}
\newcommand{\interior}[1]{\operatorname{int}(#1)}
\newcommand{\compact}{\mathrm{c}}
\newcommand{\cpct}{\compact}
\newcommand{\bounded}{\mathrm{b}}
\newcommand{\bdd}{\bounded}
\newcommand{\loc}{\mathrm{loc}}
\newcommand{\double}{\mathrm{d}}
\newcommand{\Max}{\mathrm{max}}
\renewcommand{\i}{\mathbf{i}}
\newcommand{\connectedsum}{\operatorname{\#}}
\newcommand{\imaginaryPart}{\operatorname{Im}}
\newcommand{\realPart}{\operatorname{Re}}
\newcommand{\onto}{\twoheadrightarrow}
\newcommand{\into}{\hookrightarrow}
\newcommand{\pr}{\mathrm{pr}}
\newcommand{\Cpr}{\mathrm{Pr}}
\newcommand{\codim}{\operatorname{codim}}
\newcommand{\closure}{\operatorname{cl}}

\newcommand{\moduli}{\mathcal}
\newcommand{\Met}{\mathrm{Met}}
\newcommand{\J}{\mathcal{J}}
\newcommand{\JMet}{\J\text{-}\Met}
\newcommand{\Witt}{\mathsf{W}}
\newcommand{\Vir}{\mathsf{Vir}}
\newcommand{\bra}[1]{\langle #1|}
\newcommand{\ket}[1]{|#1\rangle}
\newcommand{\bracket}[2]{\langle #1|#2\rangle}
\newcommand{\correlator}[1]{\langle #1\rangle}
\newcommand{\Asa}{\mathrm{Asa}}
\newcommand{\Liouv}{\mathrm{Liouv}}
\newcommand{\LiouvilleBdle}{\mathfrak{Lb}}

\newcommand{\SLE}{\mathrm{SLE}}
\newcommand{\Schottky}{_{\mathrm{SD}}}
\newcommand{\origin}{\mathrm{o}}
\newcommand{\infinity}{\infty}
\newcommand{\poles}{\mathrm{p}}
\newcommand{\zeroes}{\mathrm{z}}
\newcommand{\horHeight}{h_{\mathrm{h}}}
\newcommand{\vertHeight}{h_{\mathrm{v}}}
\newcommand{\quadrDiff}{\mathcal{Q}}
\newcommand{\finiteQuadrDiff}{\mathcal{Q}^\mathrm{fin}}

\newcommand{\Sphere}[1]{\mathrm{S}^{#1}}
\newcommand{\RiemCurv}{\mathrm{Rm}}
\newcommand{\SecCurv}{\mathrm{Sec}}
\newcommand{\RicCurv}{\mathrm{Ric}}
\newcommand{\ScalCurv}{\mathrm{Scal}}

\newcommand{\norml}{\|}
\newcommand{\normr}{\|}

\newcommand{\cl}{\operatorname{cl}}


\theoremstyle{plain}
	\newtheorem{theorem}{Theorem}[section]
	\newtheorem{proposition}{Proposition}[section]
	\newtheorem{lemma}{Lemma}[section]
	\newtheorem{corollary}{Corollary}[section]
	\newtheorem{conjecture}{Conjecture}[section]
\theoremstyle{definition}
	\newtheorem{defn}{Definition}[section]
	\newtheorem{construction}{Construction}[section]
	\newtheorem{example}{Example}[section]
	\newtheorem{exercise}{Exercise}[section]
	\newtheorem{convention}{Convention}[section]
\theoremstyle{remark}
	\newtheorem{remark}{Remark}[section]
	\newtheorem{comment}{Commentary}[section]
	\newtheorem*{claim}{Claim}

\maketitle

\selectlanguage{english}

\begin{abstract}
The construction of manifold structures and fundamental classes on the (compactified) moduli spaces appearing in Gromov-Witten theory is a long-standing problem.
Up until recently, most successful approaches involved the imposition of topological constraints like semi-positivity on the underlying symplectic manifold to deal with this situation.
One conceptually very appealing approach that removed most of these restrictions is the approach by K. Cieliebak and K. Mohnke via complex hypersurfaces, \cite{MR2399678}.
In contrast to other approaches using abstract perturbation theory, it has the advantage that the objects to be studied still are spaces of holomorphic maps defined on Riemann surfaces.

This article aims to generalise this from the case of surfaces of genus $0$ dealt with in \cite{MR2399678} to the general case, also using some of the methods from \cite{MR1954264} and symplectic field theory, namely the compactness results from \cite{MR2026549}.
\end{abstract}

\pagenumbering{roman}
\tableofcontents

\pagenumbering{arabic}

\section{Introduction}

A much studied question in contemporary symplectic geometry that was started in \cite{MR809718} and taken further in, among many others, \cite{MR1366548} and \cite{MR1483992} concerns the existence of holomorphic curves.
In its simplest form, this means that given a closed symplectic manifold $(X, \omega)$ and an $\omega$-compatible (or tame) almost complex structure $J$ on $X$, as well as a Riemann surface $(S, j)$ and a homology class $A \in H_2(X)$, does there exist a holomorphic map $u : S \to X$, \ie $J\circ \d u = \d u\circ j$, that represents the homology class $A$ (if $A = 0$, then a trivial answer to this question is provided by the constant maps)?
The usual strategy to answer this question is the following: Find a way to ``count'' holomorphic curves (in homology class $A$) for a set of almost complex structures on $X$ that are dense (at least in a connected neighbourhood of the given $J$) in $\mathcal{J}_\omega(X)$ (the set of $\omega$-compatible almost complex structures on $X$) and in a way that is invariant under deformations of the almost complex structures.
Invariance here means that for a homotopy/deformation $(J_t)_{t\in [0,1]}$, the counts of $J_0$- and of $J_1$-holomorphic curves coincide.
Then by Gromov's compactness theorem, \cf \cite{MR1451624} and the references therein, one can conclude the existence of an, although broken, $J$-holomorphic curve.
The way this question is studied is usually the following: \\
Fix numbers $g, n \in \N_0$ with $2g - 2 + n > 0$. Then (see Definitions \ref{Definition_Riemann_surfaces_I} and \ref{Definition_Riemann_surfaces_II} for the notation used in the following)
\begin{align*}
\mathcal{M}_{g,n}(X, A, J) \definedas \{(S, j, r_\ast, u) \;|\; & (S, j, r_\ast) \text{ smooth marked Riemann surface} \\
& \text{of type $(g,n)$, } u : S \to X \text{ $j$-$J$-holomorphic,} \\
& [u] = A\}/_\sim\text{,}
\end{align*}
where $(S, j, r_\ast, u) \sim (S', j', r'_\ast, u')$ \iff there exists a diffeomorphism \\
$\phi \in \Diff((S, j, r_\ast), (S', j', r'_\ast))$ with $\phi^\ast u' = u$.
This comes with two maps
\begin{align*}
\ev : \mathcal{M}_{g,n}(X, A, J) &\to X^n \\
[(S, r_\ast, j, u)] &\mapsto [u(r_1), \dots, u(r_n)]\text{,}
\intertext{and}
\pi^{\mathcal{M}}_M : \mathcal{M}_{g,n}(X, A, J) &\to M_{g,n} \\
[(S, j, r_\ast, u)] &\mapsto [(S, j, r_\ast)]\text{,}
\end{align*}
where $M_{g,n}$ is the moduli space of smooth marked Riemann surfaces of type $(g,n)$, defined by
\begin{align*}
M_{g,n} \definedas \{(S, j, r_\ast) \;|\; & (S, j, r_\ast) \text{ smooth marked Riemann } \\
& \text{surface of type $(g,n)$}\}/_\sim\text{,}
\end{align*}
where $(S, j, r_\ast) \sim (S', j', r'_\ast)$ \iff $\Diff((S, j, r_\ast), (S', j', r'_\ast)) \neq \emptyset$. \\
``Counting invariant under deformations'' then usually refers to the question of whether, for a dense subset of $J$ in $\mathcal{J}_\omega(X)$, $\mathcal{M}_{g,n}(X, A, J)$ is an oriented manifold of a certain expected dimension that carries a fundamental class \st
\begin{equation}\label{Equation_Fundamental_cycle}
\pi^\mathcal{M}_M\times \ev : \mathcal{M}_{g,n}(X, A, J) \to M_{g,n}\times X^n
\end{equation}
defines a (singular or otherwise) chain and hence homology class in the image.
One asks that for any two such almost complex structures $J_0, J_1$ there exists a deformation $(J_t)_{t\in [0,1]}$ \st $\bigcup_{t\in [0,1]} \mathcal{M}_{g,n}(X, A, J_t)$ defines a cobordism between $\mathcal{M}_{g,n}(X, A, J_0)$ and $\mathcal{M}_{g,n}(X, A, J_1)$ that via $\pi^\mathcal{M}_M \times \ev$ induces a chain equivalence between the chains defined by these two spaces so that the corresponding homology classes coincide.
Assuming that one can construct a well-defined homology class in this way, one would then like to use Poincar{\'e}-duality in the image, in the form of intersection theory in homology, to define numerical invariants.

In this text a variant of the above construction will be carried out, roughly following the course of action from \cite{MR2399678}, where the motivation for the changes needed will be given below.
Before that, I will first state the main results in a simpler form, which can be found with all details added as Theorems \ref{Theorem_Main_Theorem_1} and \ref{Theorem_Main_Theorem_2} in Section \ref{Section_Definition_and_Outline}.
The third main theorem, Theorem \ref{Theorem_Main_Theorem_3}, deals with the dependence of the pseudocycle defined below on the choice of regular nodal family of marked Riemann surfaces (\cf below). \\
The definition of a Gromov-Witten pseudocycle given here a priori depends on a number of data, the first of which is a choice of regular nodal family of marked Riemann surfaces $(\Sigma \to M, R_\ast)$ of type $(g,n)$, which comes with associated regular nodal families of marked Riemann surfaces $(\Sigma^\ell \to M^\ell, R^\ell_\ast, T^\ell_\ast)$ of type $(g, n+\ell)$, \ie with $\ell \geq 0$ additional marked points, see Section \ref{Section_III.1}.
Each of the complex manifolds $M^\ell$ comes with a forgetful map $\pi^\ell : M^\ell \to M$ defined via forgetting the additional $\ell$ marked points and stabilising and contains an open subset $\overset{\circ}{M}{}^\ell$ with complement of codimension at least $2$ \st the fibres of $\Sigma^\ell$ over $\overset{\circ}{M}{}^\ell$ are smooth.
Given a symplectic manifold $(X,\omega)$ with integer symplectic form, \ie $[\omega] \in H^2(X;\Z)$, and an $\omega$-compatible almost complex structure, the second piece of data is that of a Donaldson pair $(Y,J_0)$, see \cite{MR2399678}, \ie an approximately $J_0$-holomorphic (in particular symplectic) hypersurface $Y \subseteq X$ with Poincar{\'e}-dual $\mathrm{PD}(Y) = D[\omega]$, where $D \in \N$ is called the degree of $Y$.
Because no restriction on the genus $g$ of the surfaces under consideration is put, to achieve the necessary transversality results, one has to use Hamiltonian perturbations, which are a special kind of connection on the symplectic fibration $\tilde{X}^\ell \definedas \Sigma^\ell\times X$, see Subsection \ref{Subsection_Hamiltonian_perturbations}.
Given a complex structure $J$ on $X$, each such Hamiltonian perturbation $H$ defines for $b \in M^\ell$ a perturbed (\wrt to the product structure) almost complex structure $J^H_b$ on $\tilde{X}^\ell_b \definedas \Sigma^\ell_b\times X$.
Setting $\tilde{Y}^\ell \definedas \Sigma^\ell \times Y$, there is a notion of $\tilde{Y}^\ell$-compatible Hamiltonian perturbation (\cf Section \ref{Section_Hypersurfaces}) and given such, one can for $b \in M^\ell$ look at $J^H_b$-holomorphic sections $u : \Sigma^\ell_b \to \tilde{X}^\ell_b$.
Given a homology class $A \in H_2(X;\Z)$, one can then define moduli spaces of holomorphic sections representing homology class $A$ and mapping the additional marked points $T^\ell_\ast$ to $\tilde{Y}^\ell$:
\begin{align*}
\hat{\mathcal{M}}(\tilde{X}^\ell, \tilde{Y}^\ell, A, J, H) \definedas \{ u_b : \Sigma^\ell_b \to \tilde{X}^\ell_b \;|\; & \text{$b\in \overset{\circ}{M}{}^\ell$, $u_b$ $J^H_b$-holomorphic}, \\
& u_b(T^\ell_i(b)) \in \tilde{Y}^\ell \;\forall\, i=1,\dots, \ell \\
& [\pr_2\circ u_b] = A \in H_2(X;\Z)\}\text{.}
\end{align*}
This moduli space allows the definition of a Gromov-Witten map (not yet shown to be a well-defined pseudocycle in any way)
\[
\mathrm{gw}^\ell_{\Sigma}(X, Y, A, J, H) : \hat{\mathcal{M}}(\tilde{X}^\ell, \tilde{Y}^\ell, A, J, H) \to M\times X^n\text{,}
\]
defined by mapping $u_b$ to $(\pi^\ell(b), \ev^{R^\ell_\ast}(u_b))$, where $\ev^{R^\ell_\ast}$ is given by evaluation at the first $n$ marked points.
Using this, one can formulate the main results:

\begin{theorem}\label{Theorem_Main_Theorems_Summary}
Let $(X, \omega)$ be a symplectic manifold with integer symplectic form, \ie $[\omega] \in H^2(X;\Z)$.
Let also a homology class $0 \neq A \in H_2(X;\Z)$ be given.
\begin{enumerate}[a)]
  \item Given an $\omega$-compatible almost complex structure $J_0$, there exist the following:
\begin{enumerate}[i)]
  \item an integer $D^\ast = D^\ast(X, \omega, J_0)$;
  \item for every $D\geq D^\ast$ a symplectic hypersurface $Y \subseteq X$, making $(Y, J_0)$ a Donaldson pair of degree $D$;
  \item a well-defined nonempty set of $\omega$-compatible almost complex structures $J$ making $Y$ a complex hypersurface;
  \item setting $\ell \definedas D\omega(A)$ a well-defined nonempty set of Hamiltonian perturbations $H$,
\end{enumerate}
such that $\mathrm{gw}^\ell_{\Sigma}(X, Y, A, J, H)$ is a well-defined pseudocycle in $M\times X^n$ of dimension ($\chi = 2-2g$ is the Euler characteristic, $c_1(A)$ the first Chern class of $X$ evaluated on $A$)
\[
\dim_\C(X)\chi + 2c_1(A) + \dim_\R(M)\text{.}
\]
  \item For choices as above, the rational pseudocycle
\[
\mathrm{gw}_{\Sigma}(X, Y, A, J, H) \definedas \frac{1}{\ell!}\mathrm{gw}^\ell_{\Sigma}(X, Y, A, J, H)
\]
is independent of the choices of $(Y, J_0)$, $J$ and $H$ up to rational cobordism and hence there is a well-defined equivalence class
\[
\mathrm{gw}_\Sigma(X,A)
\]
of rational pseudocycles in $M\times X^n$ of dimension
\[
\dim_\C(X)\chi + 2c_1(A) + \dim_\R(M)\text{.}
\]
\end{enumerate}
\end{theorem}

The above results were part of the author's PhD thesis at LMU Munich.
After the thesis was submitted (on $11^\text{th}$ February, 2013), the preprint \cite{1302.3472} was posted on the arXiv in which the authors independently prove similar results.

Two immediate reasons for the more complicated definition above are that neither is $M_{g,n}$ a manifold nor is it compact, so one cannot expect there to be Poincar{\'e}-duality in singular homology.
This also applies, \eg by taking $X$ to be a point, to $\mathcal{M}_{g,n}(X, A, J)$, since in general only closed oriented manifolds can be expected to carry a fundamental class in singular homology. \\
To fix the second problem, one has to compactify $M_{g,n}$ and $\mathcal{M}_{g,n}(X, A, J)$.
For $M_{g,n}$ this is done via the Deligne-Mumford compactification
\begin{align*}
\overline{M}_{g,n} \definedas \{(S, j, r_\ast, \nu) \;|\; & (S, j, r_\ast, \nu) \text{ stable marked nodal} \\
& \text{Riemann surface of type $(g,n)$}\}/_\sim\text{,}
\end{align*}
where $(S, j, r_\ast, \nu) \sim (S', j', r'_\ast, \nu')$ \iff $\Diff((S, j, r_\ast, \nu'), (S', j', r'_\ast, \nu')) \neq \emptyset$.
The compactification of $\mathcal{M}_{g,n}(X, A, J)$ by Gromov is a more difficult concept that requires some more preparation. But a first step is to define the moduli space of nodal holomorphic curves in $X$,
\begin{align*}
\overline{\mathcal{M}}_{g,n}(X, A, J) \definedas \{(S, j, r_\ast, \nu, u) \;|\; & (S, j, r_\ast, \nu) \text{ stable marked nodal Riemann surface} \\
& \text{of type $(g,n)$, } u : S \to X \text{ $j$-$J$-holomorphic,} \\
& u(n_1) = u(n_2) \;\forall\, \{n_1,n_2\} \in \nu,\; [u] = A\}/_\sim\text{,}
\end{align*}
where $(S, r_\ast, j, \nu, u) \sim (S', r'_\ast, j', \nu', u')$ \iff there exists a diffeomorphism \\
$\phi \in \Diff((S, j, r_\ast, \nu), (S', j', r'_\ast, \nu'))$ with $\phi^\ast u' = u$. \\
Analogously to before there are then also canonical extensions
\begin{align*}
\pi^\mathcal{M}_M : \overline{\mathcal{M}}_{g,n}(X, A, J) &\to \overline{M}_{g,n}
\intertext{and}
\ev : \overline{\mathcal{M}}_{g,n}(X, A, J) &\to X^n\text{.}
\end{align*}
This still leaves the first problem, namely that (for $g > 1$) $\overline{M}_{g,n}$ (as well as $M_{g,n}$) is not a manifold but only a complex orbifold, as is shown in \cite{MR2262197}.
So $\overline{M}_{g,n}$ (as a topological space) can be decomposed in two ways: By signature, \ie by homeomorphism type of the underlying surface, and via the stratification coming from the orbifold structure.
Since the morphisms in the groupoids (from \cite{MR2262197}) defining the orbifold structure are given by ismorphisms of nodal surfaces, which in particular preserve the signature, this stratification is compatible with the decomposition by orbit type.
More explicitely, if as in \cite{MR2262197}, esp.~Definitions 6.2 and 6.4, $(\pi : \Sigma \to M, R_\ast)$ is a universal marked nodal family of type $(g,n)$ and $(M, \Gamma, s, t, e, i, m)$ is the associated groupoid, then $M$ has a stratification by locally closed submanifolds. Here, two points $b, b' \in M$ lie on the same stratum \iff $\Sigma_b$ and $\Sigma_{b'}$ have the same signature (as marked nodal Riemann surfaces).
If an orbit of the groupoid structure on $M$ intersects a stratum of this stratification, then it is completely contained in that stratum.
Although this gives $\overline{M}_{g,n}$ a stratification with a connected top-dimensional stratum and all other strata of codimension at least two, this does not suffice to have Poincar{\'e}-duality in singular homology (examples for this can be found \eg in \cite{MacPherson_1990}).
The standard way, started in \cite{MR717614}, to remedy this is to regard, instead of $\overline{M}_{g,n}$, certain closed complex manifolds $\overline{M}^\lambda$ with maps $\pi^\lambda : \overline{M}^\lambda \to \overline{M}_{g,n}$ that are, in a certain sense, branched coverings (for existence results, see \eg \cite{MR1257324} or \cite{MR1739177} and the references therein).
Since one of the goals in this text is to keep to manifolds and smooth maps, esp.~to the description of $\overline{M}_{g,n}$ provided in \cite{MR2262197}, it is hard to make this precise.
But at least the part $M^\lambda$ of such a manifold $\overline{M}^\lambda$ that maps to $M_{g,n} \subseteq \overline{M}_{g,n}$ has an easy description: Remembering that if $\mathcal{T}_{g,n}$ denotes Teichm{\"u}ller space and $\Gamma_{g,n}$ denotes the mapping class group (both of smooth surfaces of type $(g,n)$), then $M_{g,n} \cong \mathcal{T}_{g,n}/_{\Gamma_{g,n}}$.
If $\Gamma^\lambda \subseteq \Gamma_{g,n}$ is a finite index normal subgroup that operates freely on $\mathcal{T}_{g,n}$, then $M^\lambda \definedas \mathcal{T}_{g,n}/_{\Gamma_{g,n}^\lambda}$ is a smooth manifold on which the finite group $G^\lambda \definedas \Gamma_{g,n}/_{\Gamma_{g,n}^\lambda}$ operates and the canonical projection $M^\lambda = \mathcal{T}_{g,n}/_{\Gamma_{g,n}^\lambda} \to (\mathcal{T}_{g,n}/_{\Gamma_{g,n}^\lambda})/_{G^\lambda} = \mathcal{T}_{g,n}/_{\Gamma_{g,n}} = M_{g,n}$ is an orbifold covering.
Now assume that such a manifold $M = \overline{M}^\lambda$ has been picked and let $\upsilon : M \to \overline{M}_{g,n}$ be the projection. \\
This requires to also modify the definition of $\overline{\mathcal{M}}_{g,n}(X, A, J)$, for there is no a priori reason for the map $\pi^\mathcal{M}_M : \overline{\mathcal{M}}_{g,n}(X, A, J) \to \overline{M}_{g,n}$ to factor through $M$.
Also, since the goal is to define a manifold of maps, it stands to reason to first of all fix the domains on which the maps that are the elements of this manifold are defined.
Since $\overline{M}_{g,n}$ contains equivalence classes of surfaces of different homeomorphism types, one first of all has to define a notion of smooth family of such nodal surfaces.
The notion used in this text is that of a (regular) marked nodal family of Riemann surfaces as in \cite{MR2262197}.
So the goal is not only to have a manifold $M$ as above together with a map $\upsilon : M \to \overline{M}_{g,n}$, but for this map to be defined via a regular marked nodal family of Riemann surfaces $(\pi : \Sigma \to M, R_\ast)$ \ie the map $\upsilon : M \to \overline{M}_{g,n}$ is supposed to map $b \in M$ to the equivalence class of the fibre $\Sigma_b$ of $\Sigma$ over $b$.
Or, in the reverse direction, $(\pi : \Sigma \to M, R_\ast)$ is a smooth choice of a marked nodal Riemann surface in the equivalence class $\upsilon(b)$ for each $b \in M$.
Collecting the basic definitions for and properties of such families is done at the beginning of this thesis in Section \ref{Section_III.1}.
Aside from this, that section also contains two results, Propositions \ref{Proposition_Reordering_marked_points} and \ref{Proposition_Sequence_of_branched_coverings}, that are not found in \cite{MR2262197}, but will be important in the later parts of this text, esp.~in the definition of the Gromov compactification in Section \ref{Section_Compactification}.
Namely first there is a natural operation on a stable marked nodal Riemann surface of type $(g,n+1)$, that forgets the last marked point and stabilises, \ie contracts every component that becomes unstable after removing the last marked point.
This provides a well-defined map
\[
f^{n+1}_{\mathrm{stab}} : \overline{M}_{g,n+1} \to \overline{M}_{g,n}\text{.}
\]
And second, there is an action
\[
\mathcal{S}_n\times\overline{M}_{g,n} \to \overline{M}_{g,n}
\]
of the permutation group $\mathcal{S}_n$ of $\{1, \dots, n\}$ on $\overline{M}_{g,n}$ by permuting the labels of the marked points of a marked nodal Riemann surface.
The question addressed in Propositions \ref{Proposition_Reordering_marked_points} and \ref{Proposition_Sequence_of_branched_coverings} then is, assuming that for every $n$ a marked nodal family $(\pi^n : \Sigma^n \to M^n, R^n_\ast)$ with induced map $\upsilon^n : M^n \to \overline{M}_{g,n}$ as above has been chosen, of whether or not one can lift these maps and actions to smooth ones on the manifolds $M^n$ which are covered by bundle morphisms on the $\Sigma^n$, \ie
\[
\xymatrix{
\Sigma^{n+1} \ar[d]_-{\pi^{n+1}} \ar[r] & \Sigma^n \ar[d]^-{\pi^n} \\
M^{n+1} \ar[d]_-{\upsilon^{n+1}} \ar[r] & M^n \ar[d]^-{\upsilon^{n}} \\
\overline{M}_{g,n+1} \ar[r]^-{f^{n+1}_{\mathrm{stab}}} & \overline{M}_{g,n}
}
\qquad\qquad\qquad
\xymatrix{
\mathcal{S}_n\times\Sigma^{n} \ar[d]_-{\id\times \pi^{n}} \ar[r] & \Sigma^n \ar[d]^-{\pi^n} \\
\mathcal{S}_n\times M^n \ar[d]_-{\id \times \upsilon^n} \ar[r] & M^n \ar[d]^-{\upsilon^n} \\
\mathcal{S}_n\times\overline{M}_{g,n} \ar[r] & \overline{M}_{g,n}
}
\]
This has the additional advantage that along the way the question of existence of the regular marked nodal family of Riemann surfaces $(\pi : \Sigma \to M, R_\ast)$ defining $\upsilon$ is reduced to the case $n=0$ and given such a choice, for all other values of $n$ there is then a natural one.
Also, it gives concrete differential-geometric meaning to the adages that ``the universal curve over $\overline{M}_{g,n}$ is isomorphic to $\overline{M}_{g,n+1}$'' and that adding marked points to a marked nodal Riemann surface kills automorphisms and doesn't add new ones.
Section \ref{Section_III.1} concludes with a remark about the construction of invariants, given the data that has been established so far. \\
The above results are well-known in the setting of algebraic geometry, usually phrased as saying that the above maps are representable morphisms between Deligne-Mumford stacks and that the universal curve over the Deligne-Mumford stack with $n$ marked points is equivalent to Deligne-Mumford stack with $n+1$ marked points.
While it is possible with some care to make this precise also in a (complex) differential-geometric setting, the author feels the language is not really that well-known by most differential geometers, hence the rather explicit formulations in Section \ref{Section_III.1}.

Now given a nodal family of marked Riemann surfaces $(\pi : \Sigma \to M, R_\ast)$, one can for $b\in M$ and a desingularisation $\hat{\iota} : S \to \Sigma_b \subseteq \Sigma$ make the definition
\begin{align*}
\mathcal{M}_b(\Sigma, X, A, J) &\definedas \{ u : \Sigma_b \to X \;|\; \hat{\iota}^\ast u : S \to \hat{X} \text{ is $j$-$J$-holomorphic, } [u] = A\}\text{,} \\
\mathcal{M}(\Sigma, X, A, J) &\definedas \coprod_{b\in M} \mathcal{M}_b(\Sigma, X, A, J)\text{.}
\end{align*}
The important difference to the definitions from before is that the elements of $\mathcal{M}(\Sigma, X, A, J)$ now are actual maps defined on the fibres of $\Sigma$ and not equivalence classes of maps any more (``all automorphisms have been fixed'').
By definition there are canonical maps
\[
\xymatrix{
\mathcal{M}(\Sigma, X, A, J) \ar[d] \ar[r] & \overline{\mathcal{M}}_{g,n}(X, A, J) \ar[d] \\
M \ar[r] & \overline{M}_{g,n}\text{.}
}
\]
Now that one has an actual set of maps to work with, there is a better chance to equip this set with a manifold structure using the usual methods from the Fredholm-theory of the Cauchy-Riemann operator.

In Section \ref{Section_Construction_smooth_structures} the construction of a smooth structure on $\mathcal{M}(\Sigma, X, A, J)$, or rather a generalisation of that space, is examined.
First of all, remember that on $M$ there is the stratification by signature, where a stratum is defined by the condition that the homeomorphism type of the fibres does not change.
Since general gluing results are quite difficult to prove and outside of the scope of the methods employed in this text, smooth structures will only be defined on the restrictions of the (universal) moduli spaces to these strata.
Over one of these strata the situation then basically can be reduced to the consideration of a smooth fibre bundle $\rho : S \to B$ with typical fibre a fixed smooth surface. Also, a smooth bundle endomorphism $j : VS \to VS$ ($VS$ is the vertical tangent bundle) with $j^2 = -\id$ is given, that turns every fibre $S_b$ into a Riemann surface $(S_b, j_b)$, together with sections $R_i : B \to S$, $i = 1, \dots, n$.
If this bundle is (topologically) trivial, then the construction follows the lines of the discussion in \cite{MR2045629} or \cite{MR2399678} rather closely: For a fixed Riemann surface (\ie the case where $B$ is a point), one constructs the universal Cauchy-Riemann operator \wrt an appropriately chosen Banach manifold of perturbations and hence the universal moduli space just as in these references.
At this point some familiarity with (universal) Cauchy-Riemann operators, and this line of argument via the Sard-Smale theorem is assumed.
Since we allow surfaces of arbitrary genus, this necessitates the use of Hamiltonian perturbations as in Chapter 8 in \cite{MR2045629}.
For the constant maps are always holomorphic, \wrt to any holomorphic structure on the target and it is easy to see that this also holds for domain dependent complex structures as used in \cite{MR2399678}.
But the Fredholm index of the Cauchy-Riemann operator at a constant map in the case of genus greater than 1 is negative, which contradicts transversality.
So instead of the space $\mathcal{M}(\Sigma, X, A, J)$ one considers spaces $\mathcal{M}(\tilde{X}, A, J, H)$, where $\tilde{X} \definedas \Sigma \times X$ is the trivial bundle and $H$ is a Hamiltonian perturbation on $\tilde{X}$ as defined in Subsection \ref{Subsection_Hamiltonian_perturbations} and the references therein. \\
If $B$ is not a point but the bundle $S$ over $B$ is topologically trivial, then the construction of the universal moduli space is essentially a parametrised version of the previous one. \\
In the case of varying complex structures that is not dealt with in \cite{MR2045629} (which only deals with a fixed complex structure and varying marked points and \cite{MR2399678} restricts to the genus $0$ case, where there is essentially only one complex structure) one has to consider the case of a topologically nontrivial family of surfaces.
The problem here is that there no longer is a globally defined Banach manifold on which to define a universal Cauchy-Riemann operator (see the explanation on page \pageref{Page_No_universal_CR_operator} and the references there) due to the failure of the diffeomorphism group of the base to act smoothly on the Sobolev spaces of sections of a fibre bundle over that base.
This requires one to patch together universal moduli spaces obtained via a trivialisation after restricting to an open subset of $B$ ``by hand''.
This is done in the discussion leading up to Corollary \ref{Corollary_Smooth_structure_on_M}.
Similar but slightly less difficult problems also arise for the smoothness of the evaluation maps at the varying marked points, which are dealt with in Subsection \ref{Subsection_Evaluation_maps}.

At that point, what one has achieved is the following:
A universal moduli space $\mathcal{M}(\tilde{X}, A, J, \mathcal{H}(\tilde{X}))$ has been defined that comes together with three maps
\begin{align*}
\pi^{\mathcal{M}}_M : \mathcal{M}(\tilde{X}, A, J, \mathcal{H}(\tilde{X})) &\to M\text{,} \\
\ev : \mathcal{M}(\tilde{X}, A, J, \mathcal{H}(\tilde{X})) &\to X^n
\intertext{and}
\pi^{\mathcal{M}}_{\mathcal{H}} : \mathcal{M}(\tilde{X}, A, J, \mathcal{H}(\tilde{X})) &\to \mathcal{H}(\tilde{X})
\end{align*}
\st if $B \subseteq M$ is a stratum of the stratification on $M$ by signature, then $(\pi^\mathcal{M}_M)\inv(B)$ is a smooth Banach manifold and the restriction of $\pi^\mathcal{M}_\mathcal{H}$ to $(\pi^\mathcal{M}_M)\inv(B)$ is a Fredholm map of the correct expected index $\dim_\C(X)\chi + 2c_1(A) + \dim_\R(B)$, where $\chi$ is the Euler characteristic of the surfaces in the family $\Sigma$ (which is $2(1-g)$). \\
Section \ref{Section_Compactification} then first of all equips this space with a topology that makes all of the above maps continuous, which is basically a variation of the classical Gromov topology. \\
Unfortunately, with this topology $\mathcal{M}(\tilde{X}, A, J, H)$ is not compact, due to the well-known bubbling phenomena.
Usually, these are dealt with by imposing topological conditions like semipositivity on $X$, see \eg \cite{MR2045629}, Section 6.4.
In \cite{MR2399678} a different approach was first introduced for the genus $0$ case and by now also applied succesfully in \cite{MR2563686} and \cite{1202.4685} to other situations, which in this text will be extended to the case of positive genus.
To do so first of all a description of the problem is given:
Remember that there were the operations of forgetting the last marked point and stabilising and permuting the marked points on the Deligne-Mumford moduli spaces $\overline{M}_{g,n}$.
These lift to maps and actions, for $\tilde{\ell} \geq \ell$,
\[
\xymatrix{
\Sigma^{\tilde{\ell}} \ar[d]_-{\pi^{\tilde{\ell}}} \ar[r]^-{\hat{\pi}^{\ell}_\ell} & \Sigma^{\ell} \ar[d]^-{\pi^{\ell}} \\
M^{\tilde{\ell}} \ar[r]^-{\pi^{\tilde{\ell}}_\ell} & M^{\ell}
}
\qquad\qquad\qquad
\xymatrix{
\mathcal{S}_\ell\times \Sigma^\ell \ar[d]_-{\id\times \pi^\ell} \ar[r]^-{\hat{\sigma}^\ell} & \Sigma^\ell \ar[d]^-{\pi^\ell} \\
\mathcal{S}_\ell\times M^\ell \ar[r]^-{\sigma^\ell} & M^\ell\text{,}
}
\]
where $\pi^\ell : \Sigma^\ell \to M^\ell$ is obtained from $\pi : \Sigma \to M$ by adding $\ell \geq 0$ additional marked points.
There are then induced maps
\[
(\hat{\pi}^{\tilde{\ell}}_\ell)^\ast : (\pi^{\tilde{\ell}}_\ell)^\ast \mathcal{M}((\hat{\pi}^\ell_0)^\ast \tilde{X}, A, J, (\hat{\pi}^\ell_0)^\ast \mathcal{H}(\tilde{X})) \to \mathcal{M}((\hat{\pi}^{\tilde{\ell}}_0)^\ast\tilde{X}, A, J, (\hat{\pi}^{\tilde{\ell}}_0)^\ast \mathcal{H}(\tilde{X}))
\]
and actions
\[
\tilde{\sigma}^\ell : \mathcal{S}_\ell \times \mathcal{M}((\hat{\pi}^\ell_0)^\ast\tilde{X}, A, J, (\hat{\pi}^\ell_0)^\ast\mathcal{H}(\tilde{X})) \to \mathcal{M}((\hat{\pi}^\ell_0)^\ast\tilde{X}, A, J, (\hat{\pi}^\ell_0)^\ast\mathcal{H}(\tilde{X}))\text{.}
\]
Using these structures one can define the Gromov compactification $\overline{\mathcal{M}}(\tilde{X}, A, J, \mathcal{H}(\tilde{X}))$ of $\mathcal{M}(\tilde{X}, A, J, \mathcal{H}(\tilde{X}))$ as the colimit of the spaces $\mathcal{M}((\hat{\pi}^\ell_0)^\ast\tilde{X}, A, J, (\hat{\pi}^\ell_0)^\ast\mathcal{H}(\tilde{X}))$ over the above maps and actions (\cf Definition \ref{Definition_Equivalence_Relation_Gromov_Compactness} and Remark \ref{Remark_Gromov_compactification_explicit}).
More explicitely, this compactification consists of the union over all the spaces \\
$\mathcal{M}((\hat{\pi}^\ell_0)^\ast \tilde{X}, A, J, (\hat{\pi}^\ell_0)^\ast \mathcal{H}(\tilde{X}))$ for $\ell \geq 0$, where two holomorphic sections $u'$ and $u''$ with domains $\Sigma^{\ell'}_{b'}$ and $\Sigma^{\ell''}_{b''}$ are identified if there exists the following: An $\tilde{\ell} \geq \ell', \ell''$ and a $b \in M^{\tilde{\ell}}$ as well as a holomorphic section $u$ with domain $\Sigma^{\tilde{\ell}}_b$ \st $\Sigma^{\ell'}_{b'}$ is obtained from $\Sigma^{\tilde{\ell}}_{b}$ by forgetting the last $\tilde{\ell} - \ell'$ marked points and the corresponding map $\Sigma^{\tilde{\ell}}_b \to \Sigma^{\ell'}_{b'}$ pulls $u'$ back to $u$.
Also, after possibly reordering the last $\tilde{\ell}$ marked points, $\Sigma^{\ell''}_{b''}$ is obtained from $\Sigma^{\tilde{\ell}}_{b}$ by forgetting the last $\tilde{\ell} - \ell''$ marked points and the corresponding map $\Sigma^{\tilde{\ell}}_b \to \Sigma^{\ell''}_{b''}$ pulls $u''$ back to $u$. \\
As before, $\overline{\mathcal{M}}(\tilde{X}, A, J, \mathcal{H}(\tilde{X}))$ comes with natural maps $\pi^{\overline{\mathcal{M}}}_M : \overline{\mathcal{M}}(\tilde{X}, A, J, \mathcal{H}(\tilde{X})) \to M$ and $\pi^{\overline{\mathcal{M}}}_\mathcal{H} : \overline{\mathcal{M}}(\tilde{X}, A, J, \mathcal{H}(\tilde{X})) \to \mathcal{H}(\tilde{X})$.
Roughly, the transversality problem then is that the Hamiltonian perturbations $(\hat{\pi}^\ell_0)^\ast H \in (\hat{\pi}^\ell_0)^\ast\mathcal{H}(\tilde{X})$ vanish on ghost components, \ie those components of $\Sigma^\ell$ that are mapped to a point under $\hat{\pi}^\ell_0$ or equivalently those that become unstable after forgetting the last $\ell$ marked points.
The solution to this problem, first applied in the genus $0$ case in \cite{MR2399678} and which will be extended to the present situation in the rest of this text, can now roughly be described as follows: \\
Construct subsets $\mathcal{K}^\ell \subseteq \mathcal{H}((\hat{\pi}^\ell_0)^\ast \tilde{X})$ of Hamiltonian perturbations, compatible under $\hat{\pi}^\ell_{\ell'}$ in the sense that $(\hat{\pi}^\ell_{\ell'})^\ast \mathcal{K}^{\ell'} \subseteq \mathcal{K}^\ell$ for all $\ell\geq \ell'$, and for every $\ell$ sufficiently large a subset $\mathcal{N}^\ell(\mathcal{K}^\ell) \subseteq \mathcal{M}((\hat{\pi}^\ell_0)^\ast \tilde{X}, A, J, \mathcal{K}^\ell)$ with $\pi^\mathcal{M}_M(\mathcal{N}^\ell(\mathcal{K}^\ell)) \subseteq \overset{\circ}{M}{}^\ell$ (the part corresponding to smooth curves, as in Section \ref{Section_III.1}) \st the closure of $\mathcal{N}^\ell(\mathcal{K}^\ell)$ in $\overline{\mathcal{M}}((\hat{\pi}^\ell_0)^\ast \tilde{X}, A, J, \mathcal{K}^\ell)$, lies in $\mathcal{M}((\hat{\pi}^\ell_0)^\ast \tilde{X}, A, J, \mathcal{K}^\ell)$.
Hence the restriction of $\pi^{\mathcal{M}}_{\mathcal{H}}$ to the closure of $\mathcal{N}^\ell(\mathcal{K}^\ell)$ is proper. \\
Since over $\overset{\circ}{M}{}^\ell$, $\hat{\pi}^\ell_0$ is an isomorphism on every fibre, for every $H \in \mathcal{K}^0$ there is a well-defined map $(\hat{\pi}^\ell_0)_\ast : \mathcal{N}^\ell((\hat{\pi}^\ell_0)^\ast H) \to \mathcal{M}(\tilde{X}, A, J, H)$ (the left-hand side is defined in the obvious way) given by $u \mapsto ((\hat{\pi}^\ell_{0,b})^{-1})^\ast u$, where $\pi^\mathcal{M}_M(u) = b$. \\
Then for generic $H \in \mathcal{K}^0$ the above will be \st $\mathcal{N}^\ell((\hat{\pi}^\ell_0)^\ast H)$ is a manifold of the correct dimension, invariant under the $\mathcal{S}_\ell$-action and the map $(\hat{\pi}^\ell_0)_\ast$ is an $\ell!$-sheeted covering on the complement of a subset of codimension at least $2$ (see Lemma \ref{Lemma_ell_factorial_sheeted_covering}).
For Hamiltonian perturbations of this form, apart from compactness, unfortunately not much can be said about the closure of $\mathcal{N}^\ell((\hat{\pi}^\ell_0)^\ast H)$.
But for generic $H \in \mathcal{K}^\ell$ it will be shown that the boundary of $\mathcal{N}^\ell(H)$ can be covered by manifolds of real dimension at least $2$ less than that of $\mathcal{N}^\ell(H)$, which suffices for the definition of a pseudocycle.

Roughly speaking, the $\mathcal{N}^\ell(\mathcal{K}^\ell)$ will be defined as follows: \\
Under the assumption that $[\omega] \in H^2(X;\Z)$, $\mathcal{N}^\ell(\mathcal{K}^\ell)$ and $\mathcal{K}^\ell$ depend on a choice of $J\in \mathcal{J}_\omega(X)$ and a closed $J$-complex submanifold $Y \subseteq X$ of real codimension $2$ with $\operatorname{PD}(Y) = D[\omega]$ for some integer $D\in \N$.
Then for $\ell \definedas D\omega(A)$, let $\tilde{X}^\ell \definedas \Sigma^\ell \times X$, $\tilde{Y}^\ell \definedas \Sigma^\ell \times Y$.
The $\mathcal{K}^\ell$ then are spaces of Hamiltonian perturbations on $\tilde{X}^\ell$ that are compatible with $\tilde{Y}^\ell$ in a certain way, see Definition \ref{Definition_Y_compatible_Hamiltonian_perturbation}.
If $\overset{\circ}{\Sigma}{}^\ell$ and $\overset{\circ}{M}{}^\ell$ denote the parts of $\Sigma^\ell$ and $M^\ell$, respectively, that correspond to the smooth curves, then the $\mathcal{N}^\ell(\mathcal{K}^\ell)$ are defined to be those holomorphic sections with domains in $\overset{\circ}{\Sigma}{}^\ell$ that map the last $\ell$ markings to $\tilde{Y}^\ell$. \\
One then has to show that the thus defined spaces $\mathcal{N}^\ell(\mathcal{K}^\ell)$ satisfy the properties above.
A major point in showing this is the positivity of intersection numbers of a holomorphic curve with a complex hypersurface.
Namely one can show that a (connected) holomorphic curve either has only a finite number of intersection points with a complex hypersurface or is completely contained in the hypersurface.
Furthermore, at each intersection point, the holomorphic curve is tangent to the hypersurface of some finite order $k$ and each such intersection point contributes by $k+1$ to the (homological) intersection number.
That all this still holds in a suitable sense in the presence of a Hamiltonian perturbation that satisfies suitable compatibility conditions is shown in Section \ref{Section_Hypersurfaces}. \\
Since for a holomorphic curve $u$ in the homology class $A$, $[Y]\cdot [u] = [Y]\cdot A = \operatorname{PD}(Y)(A) = D\omega(A) = \ell$, it follows that if there are $\ell$ disjoint intersection points, then these are unique up to reordering.
So for $H \in \mathcal{K}^0$, $\mathcal{N}^\ell((\hat{\pi}^\ell_0)^\ast H)$ defines an $\ell!$-sheeted covering of its image in $\mathcal{M}(X, A, J, H)$.
To show that, after a suitable perturbation, the complement of this image has codimension at least $2$, one has to consider spaces of holomorphic curves that intersect $Y$ in fewer than $\ell$ points. But, as was stated above, these then need to have a tangency of higher order at one of the intersection points.
It was shown in \cite{MR2399678} that these tangency conditions cut out, again after a suitable perturbation, submanifolds of the moduli space of holomorphic curves that have the correct (\ie high enough) codimensions.

Another major point is that, extending a result from the same reference, one can show that for suitably chosen $Y$, $J$ and $H$, $\mathcal{K}^\ell(H)$ has compact closure in $\mathcal{M}(\tilde{X}^\ell, A, J, H)$.
The boundary of $\mathcal{K}^\ell(H)$ in $\mathcal{M}(\tilde{X}^\ell, A, J, H)$ can then be described in terms of nodal holomorphic curves that have some components mapped into $Y$ and some components intersecting the complement of $Y$ in $X$.
Via a transversality argument, one then has to show that the spaces of such curves can be covered by manifolds of codimension at least $2$.
To do so, one first of all shows that, again for suitably chosen $H$, any component that lies in $Y$ needs to represent homology class $0$. \\ 
In the genus $0$ case this suffices, for a result in \cite{MR2399678} shows that one can choose $J$ \st any holomorphic sphere with image in $Y$ is constant (which is used in the proof of the compactness statement above).
This means that one can actually replace each such component with a point, \ie such a curve factors through a nodal curve with fewer components.
It is then shown in \cite{MR2399678} that this implies a tangency condition to $Y$ for this curve which suffices to give the necessary estimates on the dimension. \\
In the case of higher genus curves, this argument does not suffice for the following reason: Assume the domain $S$ of a curve in the boundary of $\mathcal{N}^\ell(H)$ has several components, some of which are mapped to $Y$, denoted by $S_i^Y$, say, and the others, denoted $S_j^X$, intersect $Y$ only in a finte number of points.
Then this curve lies in a moduli space that is the subset, cut out by the matching conditions at the nodes, of the product of the moduli spaces of curves defined on the $S^Y_i$ with target $Y$ and of the moduli spaces of curves defined on the $S^X_j$ with target $X$. \\
The reason one has to regard moduli spaces of curves in $Y$ (naively, a curve in $Y$ is in particular a curve in $X$) is that because of the compatibility condition of the Hamiltonian perturbations with $Y$, one otherwise can't achieve transversality. \\
If the genus of $S^Y_i$ is $g^Y_i$ then the contribution to the dimension formula of the moduli space of curve on $S^Y_i$ in $Y$ by the Riemann-Roch theorem is (for vanishing homology class) given by $\dim_\C(Y)(2-2g^Y_i) = \dim_\C(X)(2-2g^Y_i) + 2g^Y_i-2$, which is larger than that for curves in $X$.
Hence although these moduli spaces then cover the boundary of $\mathcal{N}^\ell(H)$, their dimensions are too large. \\
A further problem is that some of the additional $\ell$ marked point may lie on a component that is mapped to $Y$.
This means that the condition that these marked points lie on $Y$ does not provide for a nontrivial condition on these curves and does not serve to cut down the dimension of the moduli space any more. \\
The solution to this problem is to use an SFT-type compactness theorem, in this text from \cite{MR1954264}, for related results see also \cite{MR2026549}, esp.~the ``stretching of the neck'' construction.
This provides a more detailed description of the boundary of $\mathcal{N}^\ell(H)$.
The important consequence of this result here is that every component that is mapped to $Y$ comes together with a nonvanishing meromorphic section of the normal bundle of $Y$ in $X$ along the image of the curve.
First of all this provides an additional condition on the moduli spaces associated to the parts of a curve that are mapped to $Y$, which serves to cut down the dimension by exactly the factors $2(1-g^Y_i)$ above by which these were too large. \\
Additionally, these meromorphic sections are known to have zeroes only at the nodes and at the additional marked points and to have poles only at the nodes.
Also these satisfy the following matching conditions: If at a node, both components of the curve that border on the node are mapped into $Y$ and the meromorphic section over one has a zero of order $k$, then the other has a pole of order $k$ and vice versa.
If one component is mapped to $Y$ and the other intersects $Y$ only in a finite number of points, then the meromorphic section over the first has a pole of some order $k$ and the other has a tangency to $X$ at the node of order $k$.
Since every component in $Y$ represents homology class $0$, the first Chern number of the pullback of the normal bundle to $Y$ in $X$ under the holomorphic map vanishes.
Hence the total order of the poles equals the total order of the zeroes of a meromorphic section on every component.
The matching conditions above then imply that the total order of tangency to $Y$ of the part of the curve that is not mapped into $Y$ is still given by $\ell$.

\emph{Acknowledgements.}\quad The author wishes to thank K. Cieliebak for his support as advisor during the writing of the thesis this paper grew out of and K. Mohnke for comments that helped in improving the exposition.

The author also would like to thank A. Zinger for pointing out an error in Lemma \ref{Lemma_Existence_Transverse_Hypersurface} in the first version of this article. \\
The statement of that lemma had to be weakened and the proof of Theorem \ref{Theorem_Main_Theorem_3_Extension_of_1} has been modified accordingly.
\section{Families of complex curves}\label{Section_III.1}

When regarding moduli spaces of holomorphic curves in a symplectic manifold, where the complex structure on the domain is not fixed, as \eg in \cite{MR2045629}, Chapter 8, but is allowed to vary, before one can hope to define a smooth structure on such a moduli space, first of all one has to decide on a smooth space over which the complex structure on the domain is allowed to vary.
To a certain extent this is a matter of choice, the constructions later on certainly work for an arbitrary family $\rho : S \to B$, where $B$ is any manifold, $S \to B$ is a smooth fibre bundle and $j \in \Gamma(\End(VS))$ is a smooth family of (almost) complex structures on the vertical tangent bundle $VS = \ker \rho_\ast$ of $S$.
On the other hand, usually one would like to use the ``universal family'' of Riemann surfaces of a given genus $g$ and a given number of marked points $n$, the moduli space $M_{g,n}$ of Riemann surfaces of genus $g$ with $n$ marked points, or to get a compact moduli space, the Deligne-Mumford moduli space $\overline{M}_{g,n}$ of nodal Riemann surfaces.
But unless one is in the genus $0$ case, neither $M_{g,n}$, nor $\overline{M}_{g,n}$ is a smooth manifold (not even a set in certain interpretations), but depending on point of view an orbifold, Deligne-Mumford-stack, etc.
To make a definite choice in notation, without further qualification $\overline{M}_{g,n}$ will always denote the (compact Hausdorff) topological space underlying the Deligne-Mumford orbifold.
Then, at least locally, a function $B \to \overline{M}_{g,n}$ for a manifold $B$ should be given by a family of (nodal) Riemann surfaces of genus $g$ over $B$ together with $n$ sections defining the markings.
Regarding $\overline{M}_{g,n}$ simply as the quotient space of the groupoid with objects all nodal Riemann surfaces of genus $g$ with $n$ marked points and morphisms biholomorphic maps that respect the markings, the map corresponding to a family simply maps a point in $B$ to the equivalence class of the fibre over $b$.
The for the present purpose best way to make the above precise can be found in \cite{MR2262197} and hence all the notions of (proper {\'e}tale) Lie groupoid, (universal, marked) nodal family and related concepts used in this text are exactly the ones from \cite{MR2262197}, Sections 2--6. More explicitely, the following are the basic notions to be dealt with here, all taken from \cite{MR2262197}:
\begin{defn}\label{Definition_Riemann_surfaces_I}
{\ } \\
\vspace{-.7cm}
\begin{enumerate}
  \item A \emph{surface} is a closed oriented $2$-dimensional manifold $S$.
  \item A \emph{nodal surface} is a pair $(S, \nu)$, consisting of a surface $S$ together with a set of unordered pairs
	\[
	  \nu = \{\{n_1^1,n_1^2\},\dots, \{n_d^1,n_d^2\}\}
	\]
   of pairwise distinct points, called the \emph{nodal points}, $n^1_1, \dots, n^2_d \in S$. The points $n_i^1$ and $n_i^2$ defining one of the unordered pairs in $\nu$ will be said to correspond to the same node. Note that $S$ in this definition is still a smooth surface. \\
   A surface $S$ is considered as the nodal surface $(S, \emptyset)$.
  \item A \emph{marked nodal surface} is a triple $(S, r_\ast, \nu)$, where $(S, \nu)$ is a nodal surface and
	\[
	  r_\ast = (r_1, \dots, r_n)
	\]
 	is an ordered tuple of pairwise distinct points on $S$, called the \emph{marked points}, that are disjoint from all the nodal points. \\
	The marked and nodal points are also called \emph{special points}. \\
   A nodal surface $(S, \nu)$ is considered as the nodal surface $(S, \emptyset, \nu)$.
  \item The \emph{signature} of a marked nodal surface $(S, r_\ast, \nu)$ is the labelled graph with vertices $\{S_i\}_{i\in I}$ the connected components of $S$ and for every pair of nodal points $n_j^1, n_j^2$ corresponding to the same node an edge from $S_{i_1}$ to $S_{i_2}$, where $n_j^1 \in S_{i_1}$ and $n_j^2 \in S_{i_2}$. Each vertex $S_i$ is labelled by the genus $g_i$ of $S_i$ and the subset $\{r_j \in \{r_1, \dots, r_n\} \;|\; r_j \in S_i\}$.
  \item The \emph{Euler characteristic} $\chi(S, \nu)$ of a nodal surface $(S, \nu)$ is defined as the Euler characteristic of the smooth surface obtained by removing disk neighbourhoods of each pair of nodal points corresponding to the same node and gluing the resulting boundary components by an orientation reversing diffeomorphism. If that same smooth surface is connected, then $(S, \nu)$ is called connected.
  \item A marked nodal surface $(S, r_\ast, \nu)$ is said to be of type $(g,n)$, where $g,n\in\N_0$, if $(S, \nu)$ is connected, $\chi(S, \nu) = 2(1-g)$ and $r_\ast = (r_1, \dots, r_n)$. \\
Its signature is then also said to be of type $(g,n)$.
  \item An \emph{isomorphism of marked nodal surfaces} $(S, r_\ast, \nu)$ and $(S', r'_\ast, \nu')$ is an orientation preserving diffeomorphism $\phi : S \to S'$ \st $\phi(r_\ast) = r'_\ast$ and $\phi_\ast \nu = \nu'$ in the sense that if $r_\ast = (r_1, \dots, r_n)$, then $r'_\ast = (\phi(r_1), \dots, \phi(r_n))$ and $\phi$ maps each pair of nodal points on $S$ correponding to the same node to a pair of nodal points on $S'$ corresponding to the same node. \\
  An automorphism of $(S, r_\ast, \nu)$ is an isomorphism from this marked nodal surface to itself. \\
The sets consisting of these will be denoted by $\Diff((S, r_\ast, \nu), (S', r'_\ast, \nu'))$ and $\Aut(S, r_\ast, \nu)$ (which is a group), respectively.
\end{enumerate}
\end{defn}

\begin{remark}
\begin{enumerate}
  \item Two marked nodal surfaces are isomorphic \iff their signatures are isomorphic as labelled graphs.
  \item If the number of pairs of nodal points of a marked nodal surface $(S, r_\ast, \nu)$ is $d \in \N_0$ and $\{S_i\}_{i\in I}$ are the connected components of $S$, then $\chi(S, \nu) = \sum_{i\in I} \chi(S_i) - 2d = \sum_{i\in I} 2(1-g_i) - 2d$, where $g_i$ is the genus of $S_i$.
\end{enumerate}
\end{remark}

\begin{defn}\label{Definition_Riemann_surfaces_II}
\begin{enumerate}
  \item A \emph{marked nodal Riemann surface} is a tuple $(S, j, r_\ast, \nu)$ consisting of a marked nodal surface $(S, r_\ast, \nu)$ together with a complex structure $j \in \Gamma(\End(TS))$, $j^2 = -\id$, that induces the given orientation on $S$.
  \item An \emph{isomorphism of marked nodal Riemann surfaces} $(S, j, r_\ast, \nu)$ and \\ $(S', j', r'_\ast, \nu')$ is an isomorphism $\phi$ of the marked nodal surfaces $(S, r_\ast, \nu)$ and $(S', r'_\ast, \nu')$ \st $\phi_\ast j = j'$. The set of these will be denoted
\[
\Diff((S, j, r_\ast, \nu), (S', j', r'_\ast, \nu))\text{.}
\]
  An automorphism of $(S, j, r_\ast, \nu)$ is an isomorphism of this marked nodal Riemann surface to itself. The group of automorphisms of $(S, j, r_\ast, \nu)$ will be denoted by $\Aut(S, j, r_\ast, \nu)$.
  \item A marked nodal Riemann surface is called \emph{stable}, if $\Aut(S, j, r_\ast, \nu)$ is finite. This is the case \iff every component of $S$ of genus zero contains at least three special points and every component of $S$ of genus one contains at least one special point. \\
The signature of a stable marked nodal Riemann surface is called a \emph{stable signature}.
  \item For $g,n \in \N_0$ with $n > 2(1-g)$, as a set, the \emph{Deligne-Mumford moduli space (of type $(g,n)$)} $\overline{M}_{g,n}$ is the set of isomorphism classes of stable marked nodal Riemann surfaces of type $(g,n)$. 
\end{enumerate}
\end{defn}

\begin{remark}
That $\overline{M}_{g,n}$ indeed is a set is shown by picking, for every isomorphism class of stable signature of type $(g,n)$, a marked nodal surface of this signature. There are only finitely many choices of ismorphism classes of stable signatures of fixed type. For each such choice one then considers ismorphism classes of complex structures on a fixed surface, which, as sections of a bundle, form a set.
\end{remark}

The above only defines $\overline{M}_{g,n}$ as a set, so next a description of the smooth (or holomorphic) structure is required.
One way to define such a structure is by describing holomorphic functions from complex manifolds into $\overline{M}_{g,n}$.
Because $\overline{M}_{g,n}$ is supposed to serve as a kind of moduli space for marked nodal Riemann surfaces, a holomorphic map into $\overline{M}_{g,n}$ should correspond to holomorphic families of marked nodal Riemann surfaces, where by family of marked nodal Riemann surfaces, the following is meant:
\begin{defn}\label{Definition_Marked_nodal_family}
\begin{enumerate}
  \item \label{Definition_Marked_nodal_family_1} A \emph{marked nodal family of Riemann surfaces} is a pair $(\pi : \Sigma \to B, R_\ast)$, where $\Sigma$ and $B$ are complex manifolds with $\dim_\C(\Sigma) = \dim_\C(B) + 1$, $\pi : \Sigma \to B$ is a proper holomorphic map and $R_\ast = (R_1, \dots, R_n)$ is a sequence of pairwise disjoint complex submanifolds of $\Sigma$ \st the following hold: \\
For every $z\in \Sigma$, there exist holomorphic coordinates $(t_0, \dots, t_s)$, $s\definedas \dim_\C(B) = \dim_\C(\Sigma-1)$, around $z$ in $\Sigma$ and $(v_1, \dots, v_s)$ around $\pi(z)$ in $B$, mapping $z$ to $0 \in \C^{s+1}$ and $\pi(z)$ to $0 \in \C^{s}$, respectively, \st in these coordinates, $\pi$ is given by either
\begin{align}
	(t_0, \dots, t_s) &\mapsto (t_1, \dots, t_s)
\intertext{or}
	(t_0, \dots, t_s) &\mapsto (t_0t_1, t_2, \dots, t_s)\text{.}\label{Equation_nodal_coordinates}
\end{align}
In the first case, $p$ is called a \emph{regular point}, in the second case, $p$ is called a \emph{node} of \emph{nodal point}. \\
Furthermore, for each $i = 1, \dots, n$, $\pi|_{R_i} : R_i \to B$ is assumed to be a diffeomorphism.
Each $R_i$ hence defines a section of $\pi : \Sigma \to B$, with which it will usually be identified.
  \item A \emph{desingularisation} of a fibre $(\Sigma_b, R_{\ast, b})$, for $b\in B$ and $R_{\ast, b} \definedas R_\ast \cap \Sigma_b$, of a marked nodal family of Riemann surfaces $(\pi : \Sigma \to B, R_\ast)$ is a marked nodal Riemann surface $(S, j, r_\ast, \nu)$ together with a surjective holomorphic immersion $\hat{\iota} : S \to \Sigma_b \subseteq \Sigma$, that is an embedding from the complement of the nodal points on $S$ onto the complement of the nodes on $\Sigma_b$ and maps every pair of nodal points on $S$ corresponding to the same node to a node on $\Sigma_b$. Furthermore, if $R_\ast = (R_1, \dots, R_n)$, then $r_\ast = (r_1, \dots, r_n)$ and for each $i = 1, \dots, n$, $\hat{\iota}(r_i) = \Sigma_b\cap R_i$.
  \item A \emph{morphism} between marked nodal families of Riemann surfaces $(\pi : \Sigma \to B, R_\ast)$ and $(\pi' : \Sigma' \to B', R'_\ast)$ is a pair of holomorphic maps $\phi : B \to B'$ and $\Phi : \Sigma \to \Sigma'$ \st $\pi' \circ \Phi = \phi \circ \pi : \Sigma \to B'$. Furthermore, for every $b\in B$, if $(S, j, r_\ast, \nu)$ is a marked nodal Riemann surface and $\hat{\iota} : S \to \Sigma_b$ is a desingularisation of the fibre of $\Sigma$ over $b$, then $\Phi \circ \hat{\iota} : S \to \Sigma'_{\phi(b)}$ is a desingularisation of the fibre of $\Sigma'$ over $\phi(b)$.
  \item The \emph{signature} of a fibre $(\Sigma_b, R_{\ast, b})$, for $b\in B$, of a marked nodal family of Riemann surfaces $(\pi : \Sigma \to B, R_\ast)$ is the (isomorphism class of the) signature of a desingularisation of $(\Sigma_b, R_{\ast, b})$. \\
$(\Sigma_b, R_{\ast, b})$ is said to be \emph{stable} (\emph{of type $(g,n)$}), if a desingularisation of $(\Sigma_b, R_{\ast, b})$ is stable (of type $(g,n)$). \\
$(\pi : \Sigma \to b, R_\ast)$ is called stable (of type $(g,n)$), if every fibre is stable (of type $(g,n)$).
\end{enumerate}
\end{defn}

The above is well-defined by Lemma 4.3 in \cite{MR2262197}, \ie every fibre of a marked nodal family of Riemann surfaces has a desingularisation and for any two desingularisations of the same fibre, there is a unique isomorphism of the marked nodal Riemann surfaces that commutes with the maps to the fibre. \\
Hence every stable marked nodal family of Riemann surfaces of type $(g,n)$ comes with a well-defined map to $\overline{M}_{g,n}$, mapping a point in the base to the isomorphism class of a marked nodal Riemann surface of a desingularisation of the fibre over the point. The requirement that the maps obtained in this way are smooth then gives a criterion by which one can define a topology on $\overline{M}_{g,n}$, namely the finest one \st all the maps of this form are continuous. This abstract way of defining the topology does not provide a way to deal with the usual questions of topology like the verification of the Hausdorff property, $2^\text{nd}$-countability and compactness.
To deal with these, one singles out a special type of stable marked nodal family that serve as charts for an orbifold structure on $\overline{M}_{g,n}$ and define the topology as well:
\begin{defn}
Let $(S, j, r_\ast, \nu)$ be a stable marked nodal Riemann surface of type $(g,n)$. A \emph{(nodal) unfolding} of $(S, j, r_\ast, \nu)$ is a stable marked nodal family of Riemann surfaces of type $(g,n)$ $(\pi : \Sigma \to B, R_\ast)$ together with a point $b\in B$ and a desingularisation $\hat{\iota} : S \to \Sigma_b \subseteq \Sigma$ of the fibre over $b$. \\
The unfolding is called \emph{universal}, \iff for every other nodal unfolding $(\pi' : \Sigma' \to B', R'_\ast), b'\in B', \hat{\iota}' : S \to \Sigma'_{b'}$, there exists a unique germ of a morphism $(\Phi, \phi) : (\pi : \Sigma \to B, R_\ast) \to (\pi' : \Sigma' \to B', R'_\ast)$ \st $\phi(b) = b'$ and $\Phi \circ \hat{\iota} = \hat{\iota}'$.
\end{defn}

Some of the main theorems from \cite{MR2262197} can now be summed up as follows:
\begin{theorem}
\begin{enumerate}
  \item A marked nodal Riemann surface admits a universal unfolding \iff it is stable.
  \item If $(\pi : \Sigma \to B, R_\ast), b\in B, \hat{\iota} : S \to \Sigma_b$ is a universal nodal unfolding of the marked nodal Riemann surface $(S, j, r_\ast, \nu)$, then there exists a neighbourhood $U \subseteq B$ of $b$ \st it is a universal unfolding of every desingularisation of every fibre $\Sigma_{b'}$ for $b'\in U$.
\end{enumerate}
\end{theorem}

\begin{defn}
A \emph{local universal marked nodal family of Riemann surfaces of type $(g,n)$} is a stable marked nodal family of Riemann surfaces $(\pi : \Sigma \to B, R_\ast)$ of type $(g,n)$ with the property that for every $b \in B$ and every desingularisation $\hat{\iota} : S \to \Sigma_b$ of $\Sigma_b$ by a stable marked nodal Riemann surface $(S, j, r_\ast, \nu)$ of type $(g,n)$, $(\pi : \Sigma \to B, R_\ast), b, \hat{\iota} : S\to \Sigma_b$ is a universal unfolding of $(S, j, r_\ast, \nu)$. \\
If the canonical map $B \to \overline{M}_{g,n}$ is surjective, then it is called a \emph{universal marked nodal family of Riemann surfaces of type $(g,n)$}.
\end{defn}

A further important result about universal unfoldings, apart from the existence result above and uniqueness result built into the definition is that one can actually give a fairly explicit construction for them.
The relevant results can be found in the proof of Theorem 5.6 in \cite{MR2262197}, which comes in two parts, in Section 8 in the proof of Theorem 8.9 for the case of a marked Riemann surface without nodes and in Section 12 in the presence of nodes:
\begin{construction}\label{Construction_Universal_nodal_families}
\begin{enumerate}
  \item For a marked (nodal) Riemann surface $(S, j, r_\ast, \emptyset)$ of type $(g,n)$ with $S$ connected and $g \geq 2$, one can choose $(\pi : \Sigma \to B, R_\ast), b \in B, \hat{\iota} : S \to \Sigma_b$ in the following way:
\begin{itemize}
  \item $B = \D^{3(g - 1)}\times \D^n\cong \D^{3(g-1)+n}$;
  \item $b = \{0,0\}$;
  \item $\Sigma = B\times S$;
  \item The complex structure on $\Sigma$ is of the form $T_{(b,z)} \Sigma = T_b B \times T_zS \ni (X, \xi) \mapsto (\mathbf{i}X, \hat{j}(b_0) \xi)$, for $b = (b_0, (b_1, \dots, b_n)) \in B = \D^{3(g-1)}\times \D^n$, where $\mathbf{i}$ is the standard complex structure on $D^{3(g-1)}\times \D^n$ and $\hat{j} : \D^{3(g-1)} \to \mathcal{J}(S)$ is a holomorphic map to the set of complex structures on $S$ with $\hat{j}(0) = j$.
  \item The markings are of the form $R_i(b) = (b, \iota_i(b_0, b_i))$, for $b = (b_0, (b_1, \dots, b_n)) \in B = \D^{3(g-1)}\times \D^n$, where $\iota_i(b_0, 0) = r_i$ and for every $b_0 \in \D^{3(g-1)}$, the $\iota_i(b_0, \cdot) : \D \to S$ are $\hat{j}(b_0)$-holomorphic embeddings with pairwise disjoint images.
\end{itemize}
  \item For a marked (nodal) Riemann surface $(S, j, r_\ast, \emptyset)$ of type $(1,n)$ with $S$ connected and $n \geq 1$, one can choose $(\pi : \Sigma \to B, R_\ast), b \in B, \hat{\iota} : S \to \Sigma_b$ in the following way:
\begin{itemize}
  \item $B = \D\times \D^{n-1} \cong \D^{3(g-1)+n}$;
  \item $b = \{0,0\}$;
  \item $\Sigma = B\times S$;
  \item The complex structure on $\Sigma$ is of the form $T_{(b,z)} \Sigma = T_b B \times T_zS \ni (X, \xi) \mapsto (\mathbf{i}X, \hat{j}(b_0) \xi)$, for $b = (b_0, (b_1, \dots, b_{n-1})) \in B = \D\times \D^{n-1}$, where $\mathbf{i}$ is the standard complex structure on $D\times \D^{n-1}$ and $\hat{j} : \D \to \mathcal{J}(S)$ is a holomorphic map to the set of complex structures on $S$ with $\hat{j}(0) = j$.
  \item The markings are of the form $R_1(b) = (b, r_1)$ and for $i = 2, \dots, n$, $R_i(b) = (b, \iota_i(b_0, b_i))$, for $b = (b_0, (b_1, \dots, b_n)) \in B = \D\times \D^{n-1}$, where $\iota_i(b_0, 0) = r_i$ and for every $b_0 \in \D$, the $\iota_i(b_0, \cdot) : \D \to S$ are $\hat{j}(b_0)$-holomorphic embeddings with pairwise disjoint images that do not contain $r_1$ in their closures.
\end{itemize}
  \item For a marked (nodal) Riemann surface $(S, j, r_\ast, \emptyset)$ of type $(0,n)$ with $S$ connected and $n \geq 3$, one can choose $(\pi : \Sigma \to B, R_\ast), b \in B, \hat{\iota} : S \to \Sigma_b$ in the following way:
\begin{itemize}
  \item $B = \D^{n-3} \cong \D^{3(g-1)+n}$;
  \item $b = \{0\}$;
  \item $\Sigma = B\times S$;
  \item The complex structure on $\Sigma$ is the product of the standard complex structure on $\D^{n-3}$ and $j$.
  \item The markings are of the form $R_i(b) = (b, r_i)$ for $i = 1, 2, 3$ and for $i = 4, \dots, n$, $R_i(b) = (b, \iota_i(b_i))$, for $b = (b_1, \dots, b_n) \in B = \D^{n-3}$, where $\iota_i(0) = r_i$ and the $\iota_i : \D \to S$ are $j$-holomorphic embeddings with pairwise disjoint images that do not contain $r_1, r_2, r_3$ in their closures.
\end{itemize}
  \item \label{Construction_Universal_nodal_families_4} In the general case $(S, j, r_\ast, \nu)$, choose a numbering $\nu = \{\{n^1_1,n^2_1\}, \dots, \{n^1_d, n^2_d\}\}$ and consider the marked Riemann surface (without nodes) $(S, j, (r_\ast, n^1_\ast, n^2_\ast), \emptyset)$ where all the nodes have been replaced by marked points.
Denote by $\{S_i\}_{i\in I}$ the connected components of $S$ and by $g_i$ their genera.
Then for every $i\in I$, $(S_i, j|_{S_i}, (r^i_\ast, n^{1,i}_\ast, n^{2,i}_\ast))$ is a marked Riemann surface of one of the types above, where $r^i_\ast$ consists of those $r_j$ with $r_j \in S_i$ and analogously for $n^{1,i}_\ast$ and $n^{2,i}_\ast$.
Let $(\pi_i : \Sigma_i \to B_i, (R^i_\ast, N^{1,i}_\ast, N^{2,i}_\ast)), 0 \in B_i, \hat{\iota}_i : S_i \to \Sigma_{i,0}$ be the corresponding universal unfolding from above.
If $n_i \definedas |r^i_\ast|$, $d^{1,i} \definedas |n^{1,i}|$, $d^{2,i} \definedas |n^{2,i}|$, then $\dim_\C(B_i) = 3(g_i-1) + n_i + d^{1,i} + d^{2,i}$.
Define $B \definedas \left(\bigtimes_{i\in I} B_i\right) \times \D^d$ and $\hat{\Sigma} \definedas \bigsqcup_{i\in I} \pr_i^\ast \Sigma_i$, where $\pr_i : B \to B_i$ is the projection.
$B$ has dimension $\dim_\C(B) = \sum_{i\in I} \dim_\C(B_i) + d = \sum_{i\in I}(3(1-g_i) + n_i + d^{1,i} + d^{2,i}) + d = 3(\sum_{i\in I}(g_i-1) + d) + n = 3(g-1) + n$.
Denote by $\hat{\pi} : \hat{\Sigma} \to B$ the obvious projection.
This comes with markings $\hat{R}_\ast, \hat{N}^1_\ast, \hat{N}^2_\ast$, which are the pullbacks of the markings of the $R^i_\ast, N^{1,i}_\ast, N^{2,i}_\ast$ above.
Also, one can choose disjoint open sets $U_i, V_i \subseteq \hat{\Sigma}$, $i = 1, \dots, d$ that are tubular neighbourhoods of the $\hat{N}^1_\ast, \hat{N}^2_\ast$ that do not meet the $\hat{R}_\ast$ and come with holomorphic functions $x_i : U_i \to \D$ and $y_i : V_i \to \D$ \st $x_i(\hat{N}^1_i) = 0$, $y_i(\hat{N}^2_i) = 0$ and $(\hat{\pi}, x_i)$ and $(\hat{\pi}, y_i)$ are coordinates on $\hat{\Sigma}$.
For each $i = 1, \dots, d$, let $K_i \definedas \{\xi \in U_i \;|\; x_i(\xi) \leq |z_i|, \hat{\pi}(\xi) = (b, z_1, \dots, z_d), z_i \neq 0\}$ and $L_i \definedas \{\xi' \in V_i \;|\; y_i(\xi') \leq |z_i|, \hat{\pi}(\xi') = (b, z_1, \dots, z_d), z_i \neq 0\}$.
Also let $\hat{\Sigma}' \definedas \hat{\Sigma} \setminus \bigcup_{i=1}^d K_i\cup L_i$.
Now define $\Sigma \definedas \hat{\Sigma}'/_\sim$, where the equivalence relation on $\hat{\Sigma}'$ is generated by the following identification, for $\xi \in U_i, \xi' \in V_i$:
\begin{align*}
\xi \sim \xi' \quad \definedequiv \quad & \hat{\pi}(\xi) = \hat{\pi}(\xi') = (b, z_1, \dots, z_d) \\
& \text{and either } x_i(\xi)y_i(\xi') = z_i \neq 0 \\
& \text{or } x_i(\xi) = y_i(\xi') = z_i = 0\text{.}
\end{align*}
The projection $\pi : \Sigma \to B$ is given by $\pi([\xi]) \definedas \hat{\pi}(\xi)$ and the markings are given by the images of the $\hat{R}_\ast$ under the projection onto the quotient. \\
The above differs from the construction in the proof of Theorem 5.6 in \cite{MR2262197} by the removal of the subsets $K_i$ and $L_i$ from $\hat{\Sigma}$. But otherwise it seems to me the map $\pi : \Sigma \to B$ thus constructed does not have as fibres nodal surfaces.
\end{enumerate}
\end{construction}

The existence and explicit construction of the universal unfoldings above is useful for a number of reasons:
\begin{enumerate}
  \item Let $\pi : \Sigma \to B$ be the unfolding of a marked nodal Riemann surface $(S, j, r_\ast, \nu)$ with $d$ nodes from case \ref{Construction_Universal_nodal_families_4}.~above. Then $B$ is of the form $B = B_0\times \D^d$, so has coordinates $(b_0, z_1, \dots, z_d)$ and is stratified by the following locally closed submanifolds: Let $N \subseteq \{1, \dots, d\}$ be a subset.
Then one can look at the subset $B^N \definedas \{(b_0, z_1, \dots, z_d) \in B \;|\; z_i = 0 \text{ for } i\in N\}$.
These are precisely the subsets for which all $\Sigma_b$, $b\in B^N$, have the same signature.
Since the signatures of the fibres are preserved under morphisms of nodal families, these stratifications of the universal unfoldings of all stable marked nodal Riemann surfaces of type $(g,n)$ induces a stratification of $\overline{M}_{g,n}$, called the \emph{stratification by signature}. \\
Also, if $(\pi : \Sigma \to B, R_\ast)$ is any local universal marked nodal family of Riemann surfaces of type $(g,n)$, it also carries an induced stratification by signature.
  \item \label{Consequences_of_univ_unfoldings_2} If $(\pi : \Sigma \to B, R_\ast)$ is a local universal marked nodal family of Riemann surfaces of type $(g,n)$, then over every stratum of the stratification by signature one has the following parametrised version of a desingularisation.
Namely let $b\in B$ and let $(S, j, r_\ast, \nu), \hat{\iota} : S \to \Sigma_b$ be desingularisation of $\Sigma_b$.
Associated to this desingularisation is the universal unfolding from \ref{Construction_Universal_nodal_families_4}.~above, which defines a smooth (trivial) fibre bundle $\hat{\pi} : \hat{\Sigma} \to C$, where $C = C_0\times \D^d$, $d$ being the number of nodes on $\Sigma_b$.
Making $B$ small enough, this comes with a unique pair of maps $\phi : C \to B$ and $\Phi : \hat{\Sigma} \to \Sigma$. If $\hat{\Sigma}/_\sim$ is the quotient that defines the universal unfolding as in \ref{Construction_Universal_nodal_families_4}.~above, then there is a unique morphism $(\Phi', \phi)$ from $\hat{\Sigma}/_\sim$ to $\Sigma$ \st $\phi$ maps $(0,0)\in B_0\times \D^d$ to $b\in B$ and one can define $\Phi$ as the composition of $\Phi'$ with the projection from $\hat{\Sigma}$ to $\hat{\Sigma}/_\sim$.
Then $C' \definedas C_0 \times \{0\} \subseteq C$ is precisely the part of $C$ that gets mapped to the stratum $B'$ of the stratification by signature on $B$ that corresponds to the signature of $(S, j, r_\ast, \nu)$.
Also, the restriction $\hat{S} \definedas \hat{\Sigma}|_{C'}$ is a holomorphic family $\rho : \hat{S} \to C'$ of smooth Riemann surfaces, with a complex structure $\hat{j}$ on $\hat{S}$, that comes with $n$ sections $\hat{R}_\ast$ corresponding to the markings on $S$ and $d$ pairs of section $\hat{N}^1_\ast, \hat{N}^2_\ast$ corresponding to the nodes.
Furthermore, it comes with canonical maps $\iota : C' \to B$ and $\hat{\iota} : \hat{S} \to \Sigma$ that have the property that for every $c\in C'$, $(\hat{S}_c, \hat{j}_c, \hat{R}_{\ast,c}, \{\{\hat{N}^1_{i,c}, \hat{N}^2_{i,c}\}\}_{i=1}^d)$ together with $\hat{\iota}_c : \hat{S} \to \Sigma_{\iota(c)}$ is a desingularisation of $\Sigma_{\iota(c)}$.
By the universal properties of a universal unfolding and local universal marked nodal family, one can do this for every $b \in B'$, and the resulting (trivial) fibre bundles as above patch together to a fibre bundle over $\rho : \hat{S} \to B'$ with fibres smooth Riemann surfaces and that comes with $n$ sections $\hat{R}_\ast$. Furthermore, the $N^1_\ast, N^2_\ast$ define a discrete subbundle $\hat{N} \subseteq \hat{S}$ with structure group $\mathcal{S}(d,2)$ defined to be the subgroup of permutations of a set $(n^1_1, n^2_1, \dots, n^1_d, n^2_d)$, generated by the permutations in the lower indices, $(n^1_1, n^2_1, \dots, n^1_d, n^2_d) \mapsto (n^1_{\sigma(1)}, n^2_{\sigma(1)}, \dots, n^1_{\sigma(d)}, n^2_{\sigma(d)})$ for $\sigma \in \mathcal{S}(d)$ and switching a pair of upper indices, $(n^1_1, n^2_1, \dots, n^1_d, n^2_d) \mapsto (n^1_1, n^2_1, \dots, n^{\tau(1)}_j, n^{\tau(d)}_j, \dots, n^1_d, n^2_d)$ for $\tau \in \mathcal{S}(2)$.
So
\[
(\rho : \hat{S} \to C', \hat{R}_\ast, \hat{N})
\]
is a triple consisting of a smooth fibre bundle with fibre $S$ and structure group $\Aut(S, r_\ast, \nu)$, an $n$-tuple of sections of $\hat{S}$ and a discrete subbundle with fibre a $d$-tuple of pairs of points and structure group $\mathcal{S}(d,2)$.
\begin{defn}\label{Definition_Desingularisation}
A \emph{(parametrised) desingularisation} of a marked nodal family of Riemann surfaces is a tuple $(\rho : \hat{S} \to C', \hat{R}^\ast, \hat{N}, \iota, \hat{\iota})$, where $\rho : \hat{S} \to C'$ is a smooth fibre bundle equipped with a smooth family of complex structures $j$. $\hat{R}_\ast = (\hat{R}_1, \dots, \hat{R}_n)$ is an $n$-tuple of sections of $\rho : \hat{S} \to C'$, $N\subseteq \hat{S}$ is a $\mathcal{S}(d,2)$-subbundle and $\iota : C' \to B$ is an embedding of $C'$ as a locally closed submanifold of $B$. Furthermore, for every $b \in C'$, $(\hat{S}_b, j_b, \hat{R}_\ast(b), N_b), \iota(b), \hat{\iota}_b : \hat{S}_b \to \Sigma_{\iota(b)}$ is a desingularisation in the original sense.
\end{defn}
  \item It allows to single out ``especially nice'' maps to $\overline{M}_{g,n}$ that come from nodal families.
The most desirable case here would be the (local) universal marked nodal families.
Unfortunately, for the definition of invariants, one would like for the base of the (universal) family of marked nodal Riemann surfaces to be compact, which in general is not possible.
The next best kind of maps are the following:
Let $\pi : \Sigma \to B$ be a nodal family, $b\in B$ and let $(S, j, r_\ast, \nu), \kappa : S \to \Sigma_b$ be a desingularisation of $\Sigma_b$.
Associated to $(S, j, r_\ast, \nu)$ is a universal unfolding $(\tilde{\pi} : \tilde{\Sigma} \to \tilde{B}, \tilde{R}_\ast), \tilde{b} = (0,0) \in \tilde{B}, \tilde{\iota} : S \to \tilde{\Sigma}_b$, where $\tilde{B} = \D^{3(g-1) + n - d} \times \D^d$ and $d$ is the number of nodes on $\Sigma_b$.
By the universal property there then exists a neighbourhood $U \subseteq B$ of $b$ and a morphism
\[
\xymatrix{
\Sigma|_U \ar[r]^-{\Phi} \ar[d]_-{\pi} & \tilde{\Sigma} \ar[d]^-{\tilde{\pi}} \\
U \ar[r]^-{\phi} & \tilde{B}\text{.}
}
\]
Choosing $U$ to be a coordinate neighbourhood of $b$, holomorphically equivalent to $\D^{r}$, $r\definedas \dim_\C(B)$, with complex coordinates $(z_1, \dots, z_r)$, $\phi$ is equivalent to a map $\D^r \to \D^{3(g-1) + n - d} \times \D^d$.
A requirement one can then pose on the nodal family $\pi : \Sigma \to B$ is that $\dim_\C(B) = 3(g-1) + n$ and that around every point $b \in B$ one can choose the coordinate system as above \st in these coordinates $\phi$ is given by the map
\begin{align*}
\D^{3(g-1) + n} &\to \D^{3(g-1) + n - d} \times \D^d \\
(z_1, \dots, z_{3(g-1)+n}) &\mapsto ((z_1, \dots, z_{3(g-1)+n-d}), (z_{3(g-1)+n-d+1}^{l_1}, \dots, z_{3(g-1)+n}^{l_d}))
\end{align*}
for some constants $l_1, \dots, l_d \in \N^d$ (depending on $b \in B$), or in other words a branched covering that branches exactly over the strata of the stratification by signature.
\begin{defn}\label{Definition_Orbifold_branched_covering}
A marked nodal family of Riemann surfaces of type $(g,n)$ with the properties above is called an \emph{orbifold branched covering of $\overline{M}_{g,n}$ that branches over the Deligne-Mumford boundary}.
\end{defn}
This implies that on $B$ there also is a well-defined stratification by signature, where each stratum is a locally closed submanifold of complex codimension given by the number of nodes of a surface of that signature (\ie the number of edges of the graph).
If $\phi$, $U$ and $\tilde{B}$ are as above, then the restriction of $\phi$ to every stratum of the stratification by signature on $U$ is a (non-branched) covering of the corresponding stratum on $\tilde{B}$.
Also, one can pull back the parametrised desingularisations from \ref{Consequences_of_univ_unfoldings_2}.~above over the strata on each $\tilde{B}$ to the strata on $B$ to get over each such stratum $B_i$ a parametrised desingularisation $(\rho : \hat{S} \to B_i, \hat{R}_\ast, \hat{N})$.
  \item\label{Forgetting_last_marked_point} Last, one can examine the interactions between universal families of type $(g,n)$, where $g$ is fixed, but for different values of $n$, in these local models.
In the genus $g = 0$ case, it is well known that $\overline{M}_{0,n}$ is a closed complex manifold itself (follows from the results in \cite{MR2262197} because a stable sphere carries no nontrivial automorphisms) and there is a well-defined smooth map $\overline{M}_{0,n+1} \to \overline{M}_{0,n}$ that is defined by forgetting the $(n+1)^{\text{st}}$ marked point and stabilising. Furthermore, this map $\overline{M}_{0,n+1} \to \overline{M}_{0,n}$ defines a universal marked nodal family, see \cite{MR2262197}, Example 6.7. \\
In the higher genus case, the situation is built around the following model: \\
Let $(S, j, r_\ast, \nu)$ and $(\tilde{S}, \tilde{j}, \tilde{r}_\ast, \tilde{\nu})$ be stable marked nodal Riemann surfaces of types $(g,n)$ and $(g,n+1)$, respectively.
$(S, j, r_\ast, \nu)$ is said to be \emph{obtained from $(\tilde{S}, \tilde{j}, \tilde{r}_\ast, \tilde{\nu})$ by forgetting the last marked point and stabilising}, if the following holds: Let $\tilde{S}_i$ be the connected component of $\tilde{S}$ with $\tilde{r}_{n+1} \in \tilde{S}_i$. 
One has to distinguish three cases:
\begin{enumerate}
  \item If $\tilde{S}_i$ together with the special points on it other than $\tilde{r}_{n+1}$ is still stable, then define $\tilde{S}' \definedas \tilde{S}$, $\tilde{j}' \definedas \tilde{j}$, $\tilde{r}'_\ast \definedas (\tilde{r}_1, \dots, \tilde{r}_n)$ and $\tilde{\nu}' \definedas \tilde{\nu}$.
\end{enumerate}
Otherwise, define $\tilde{S}' \definedas \tilde{S} \setminus \tilde{S}_i$ and $\tilde{j}' \definedas \tilde{j}|_{\tilde{S}'}$.
$\tilde{S}_i$ then is a sphere with three special points, for if the genus of $\tilde{S}_i$ is $\geq 2$, then it is stable without any special points and if the genus is $1$, then because $(g,n)$ is also a stable type, \ie $n\geq 1$, and $\tilde{S}$ is connected, $\tilde{S}_i$ either contains a marked point other than $\tilde{r}_{n+1}$ (if $\tilde{S} = \tilde{S}_i$ is connected) or a nodal point.
The other two special points apart from $\tilde{r}_{n+1}$ then are either a nodal point and another marked point or two nodal points.
\begin{enumerate}\setcounter{enumii}{1}
  \item In the first case, let $\tilde{r}_l$ be the second marked point on $\tilde{S}_i$ and let $\tilde{n}^2_d$ be the nodal point on $\tilde{S}_i$.
Define $\tilde{r}'_\ast = (\tilde{r}_1, \dots, \tilde{n}^1_d, \dots, \tilde{r}_n)$, where $\tilde{n}^1_d$ replaces $\tilde{r}_l$, and $\tilde{\nu}' \definedas \{\{\tilde{n}^1_1, \tilde{n}^2_1\}, \dots, \{\tilde{n}^1_{d-1}, \tilde{n}^2_{d-1}\}\}$.
  \item In the second case, the two nodal points cannot correspond to the same node, for that would imply by connectedness of $\tilde{S}$ that $\tilde{S} = \tilde{S}_i$, so $g=1$ and there would be at least two marked points.
So assume $\tilde{\nu} = \{\{\tilde{n}^1_1, \tilde{n}^2_1\}, \dots, \{\tilde{n}^1_d, \tilde{n}^2_d\}\}$ and that the two nodal points on $\tilde{S}_i$ are $\tilde{n}^2_{d-1}$ and $\tilde{n}^1_d$.
Define $\tilde{r}'_\ast \definedas (\tilde{r}_1, \dots, \tilde{r}_n)$ and 
\[
\tilde{\nu}' \definedas \{\{\tilde{n}^1_1, \tilde{n}^2_1\}, \dots, \{\tilde{n}^1_{d-2}, \tilde{n}^2_{d-2}\}, \{\tilde{n}^1_{d-1}, \tilde{n}^2_d\}\}\text{.}
\]
\end{enumerate}
In all of these cases, $(\tilde{S}', \tilde{j}', \tilde{r}'_\ast, \tilde{\nu}')$ is a stable marked nodal surface of type $(g,n)$.
If $(\tilde{S}', \tilde{j}', \tilde{r}'_\ast, \tilde{\nu}')$ and $(S, j, r_\ast, \nu)$ are isomorphic, then the latter is said to be obtained from the former by forgetting the last marked point and stabilising. \\
Furthermore, the choice of such an isomorphism defines a (open and closed) holomorphic embedding of $S$ into $\tilde{S}$ that maps special points to special points (but may map a marked point to a nodal point).
Also, this inclusion defines an injection of $\Aut(\tilde{S}, \tilde{j}, \tilde{r}_\ast, \tilde{\nu})$ into $\Aut(S, j, r_\ast, \nu)$ (because the automorphism group of a sphere with three special points is trivial).
More precisely, there is a one-to-one correspondence between points on $S$ that are not nodal points or pairs of nodal points corresponding to the same node and stable marked nodal surfaces $(\tilde{S}, \tilde{j}, \tilde{r}_\ast, \tilde{\nu})$ of type $(g,n+1)$ up to unique equivalence as above: \\
If $z\in S$ is neither a marked point nor a nodal point, define $\tilde{S} \definedas S$, $\tilde{j} \definedas j$, $\tilde{r}_i \definedas r_i$ for $i = 1, \dots, n$, $r_{n+1} \definedas z$ and $\tilde{\nu} \definedas \nu$.
This corresponds to case (a) above, which conversely defines $z \definedas \tilde{r}_{n+1}$. \\
If $z = r_l \in S$ for some $l\in \{1, \dots, n\}$, define $\tilde{S} \definedas S \amalg S^2$, where $S^2 = \C \cup \{\infty\}$, $\tilde{j}|_S = j$ and $\tilde{j}|_{S^2}$ is the standard complex structure, $\tilde{r}_i = r_i$ for $i = 1, \dots, n$ with $i\neq l$, $\tilde{r}_l = \infty \in S^2$, $\tilde{r}_{n+1} \definedas 1 \in S^2$ and $\tilde{\nu} \definedas \nu \cup \{\{r_l, 0\}\}$ ($0\in S^2$). This corresponds to case (b) above, which conversely defines $z \definedas \tilde{r}_l$. \\
If w.\,l.\,o.\,g.~$\nu = \{\{n^1_1, n^2_1\}, \dots, \{n^1_{d-1}, n^2_{d-1}\}\}$ and $z = \{n^1_{d-1}, n^2_{d-1}\}$, define $\tilde{S} \definedas S\amalg S^2$, $\tilde{j}|_S = j$ and $\tilde{j}|_{S^2}$ the standard complex structure, $\tilde{r}_i \definedas r_i$ for $i = 1, \dots, n$, $r_{n+1} \definedas 1 \in S^2$ and
\[
\tilde{\nu} \definedas \{\{n^1_1, n^2_1\}, \dots, \{n^1_{d-2}, n^2_{d-2}\}, \{n^1_{d-1}, 0\}, \{\infty, n^2_{d-1}\}\}\text{.}
\]
This corresponds to case (c) above, which conversely defines $z \definedas \{\tilde{n}^1_{d-1}, \tilde{n}^2_d\}$.
\end{enumerate}

Marked nodal families of Riemann surfaces of type $(g,n)$ that define an orbifold branched covering of $\overline{M}_{g,n}$ that branches over the Deligne-Mumford boundary (hence in particular local universal marked nodal families) are a special case of a type of marked nodal family that is called regular in \cite{MR2262197} (Definition 12.1) and for which the above construction of forgetting the last marked point and stabilising has a global generalisation.
\begin{defn}\label{Definition_Regular_family}
Let $(\pi : \Sigma \to B, R_\ast)$ be a marked nodal family of Riemann surfaces.
Let $C \subseteq \Sigma$ be the submanifold of nodal points, which comes with the immersion $\pi|_C : C \to B$.
Given $b \in B$, $(\pi : \Sigma \to B, R_\ast)$ is called \emph{regular at $b$} if all self-intersections of $\pi(C)$ in $b$ are transverse in the following sense:
Either $b \not\in \pi(C)$ or if $b \in \pi(C)$, let $C_b \definedas C \cap \Sigma_b = \{n_1, \dots, n_d\}$, a finite set of points.
Then for all $1 \leq m \leq d$, $1 \leq i_1 < \cdots < i_m \leq d$
\[
\dim_\C(\im(\pi_{\ast, n_{i_1}}) \cap \dots \cap \im(\pi_{\ast, n_{i_m}})) = \dim_\C(B) - m\text{.}
\]
$(\pi : \Sigma \to B, R_\ast)$ is called \emph{regular} if it is regular at all points $b\in B$.
\end{defn}
By definition of a marked nodal family of Riemann surfaces, in the notation of the previous definition and if $b \in \pi(C)$, the following hold: For $i = 1, \dots, d$ there exist neighbourhoods $U_i \subseteq \Sigma$ of the $n_i$ not containing any of the marked points, neighbourhoods $V_i \subseteq B$ of $b$ and holomorphic maps $x_i, y_i : U_i \to \D$, $z_i : V_i \to \D$ obtained from a nodal coordinate system as in Equation \ref{Equation_nodal_coordinates} \st $(x_i, y_i) : U_i \to \D^2$ and $z_i : V_i \to \D$ are submersions and $z_i \circ \pi|_{U_i} = x_iy_i : U_i \to \D$.
Also, $C\cap U_i = (x_i, y_i)\inv(0,0)$, $\pi_\ast : \ker((x_i, y_i)_\ast) \to \ker(z_{i,\ast})$ is an isomorphism and $\im(\pi_{\ast, n_i}) = \ker(z_{i,\ast})$.
Making the $U_i$ and $V_i$ smaller, one can assume that $V_1 = \cdots = V_d \defines V$.
The transversality condition above then states that the $z_{i,\ast, b} : T_bB \to T_0\D$ are linearly independent.
By the implicit function theorem, after possibly making $V$ and the $U_i$ smaller, one hence can find holomorphic functions $t_1, \dots, t_k : V \to \D$, $k \definedas \dim_\C(B) - d$, \st $(z_1, \dots, z_d, t_1, \dots, t_k) : V \to \D^{\dim_\C(B)}$ is a holomorphic coordinate system on $B$ and \st $(z_1\circ \pi|_{U_i}, \dots, z_{i-1} \circ \pi|_{U_i}, x_i, y_i, z_{i+1} \circ \pi|_{U_i}, \dots, z_d\circ \pi|_{U_i}, t_1\circ \pi|_{U_i}, \dots, t_k\circ \pi|_{U_i}) : U_i \to \D^{\dim_\C(\Sigma)}$ is a holomorphic coordinate system on $\Sigma$.

\begin{lemma}\label{Lemma_Forget_marked_point}
Let $(\pi : \Sigma \to B, R_\ast)$ be a regular marked nodal family of Riemann surfaces of type $(g,n)$.
Then there exists a regular marked nodal family of Riemann surfaces $(\tilde{\pi} : \tilde{\Sigma} \to \Sigma, \tilde{R}_\ast)$ of type $(g,n+1)$ together with a holomorphic map $\hat{\pi} : \tilde{\Sigma} \to \Sigma$ with the following properties:
\[
\xymatrix{
\tilde{\Sigma} \ar[r]^-{\hat{\pi}} \ar[d]_-{\tilde{\pi}} & \Sigma \ar[d]^-{\pi} \\
\Sigma \ar[r]^-{\pi} & B\ar@{}[ul]|-{\circlearrowleft}
}
\]
commutes. Also, let $(S, j, r_\ast, \nu), b\in B, \hat{\iota} : S \to \Sigma$ be a desingularisation of $\Sigma_b$ and let $(\tilde{S}, \tilde{j}, \tilde{r}_\ast, \tilde{\nu})$ be a stable marked nodal Riemann surface of type $(g, n+1)$ \st $(S, j, r_\ast, \nu)$ is obtained from $(\tilde{S}, \tilde{j}, \tilde{r}_\ast, \tilde{\nu})$ by forgetting the last marked point and stabilising. Let $\kappa : S\to \tilde{S}$ be the resulting embedding.
Then there exists a unique $z \in \Sigma_b$ and a unique $\tilde{\iota} : \tilde{S} \to \tilde{\Sigma}_z \subseteq \tilde{\Sigma}$ \st $(\tilde{S}, \tilde{j}, \tilde{r}_\ast, \tilde{\nu}), z \in \Sigma, \tilde{\iota} : \tilde{S} \to \tilde{\Sigma}$ is a desingularisation and $\hat{\pi} \circ \tilde{\iota} \circ \kappa = \hat{\iota}$:
\[
\xymatrix{
\tilde{S} \ar[r]^-{\tilde{\iota}} & \tilde{\Sigma} \ar[d]^-{\hat{\pi}} \\
S \ar[r]^-{\hat{\iota}} \ar[u]^-{\kappa} & \Sigma\text{.}\ar@{}[ul]|-{\circlearrowleft}
}
\]
The stratification by signature on $\Sigma$ as base space of the marked nodal family $(\tilde{\pi} : \tilde{\Sigma} \to \Sigma, \tilde{R}_\ast)$ is given in the following way: For every stratum $C \subseteq B$ of the stratification by signature on $B$ consider the following subsets of $\pi\inv(C)$: The complement of the markings and nodes in $\pi\inv(C)$, for every marking $R_i$ the subset $R_i(C)$ and the connected components of the set of nodes in $\pi\inv(C)$. In particular, the restriction of $\pi$ to each of these is a submersion onto $C$. \\
If $(\pi : \Sigma \to B, R_\ast)$ is a local universal family or defines an orbifold branched covering of $\overline{M}_{g,n}$, then so does $(\tilde{\pi} : \tilde{\Sigma} \to \Sigma, \tilde{R}_\ast)$ (of $\overline{M}_{g,n+1}$).
\end{lemma}
\begin{proof}
Let $(\pi : \Sigma \to B, R_\ast)$ be as in the statement of the lemma.
The goal is to show that for every $z\in \Sigma$ there exists a neighbourhood $U'\subseteq \Sigma$ of $z$ that is the domain of a (nodal) coordinate system as in Definition \ref{Definition_Marked_nodal_family}, \ref{Definition_Marked_nodal_family_1}.~and is also the base of a marked nodal family of the type indicated in the statement of the lemma.
I will only indicate the definitions of $\tilde{\Sigma}$, $\tilde{\pi}$, $\hat{\pi}$ and the $\tilde{R}_i$, which are a variation of the constructions in the proof of Theorem 5.6 in \cite{MR2262197}. The smooth structure on $\tilde{\Sigma}$ is then also defined analogously to the smooth structures defined in that reference and the other properties of $\hat{\pi}$ follow from the remarks in \ref{Forgetting_last_marked_point}.~preceeding Definition \ref{Definition_Regular_family}. \\
The statements about local universal families and orbifold branched coverings then follow  because applying the construction below to the explicit local models from Construction \ref{Construction_Universal_nodal_families} produces again one of those local models.

If $z \in\Sigma_b$ is not one of the marked or nodal points, let $U' \subseteq \Sigma$ be a neighbourhood of $z$ disjoint from all the marked or nodal points and \st $\pi|_{U'} : U' \to B$ is a holomorphic submersion onto $B$.
Define $\tilde{\Sigma}|_{U'} \definedas (\pi|_{U'})^\ast \Sigma$, $\tilde{\pi}|_{\tilde{\Sigma}|_{U'}} : \tilde{\Sigma}|_{U'} \to U'$ is the canonical projection, $\tilde{R}_i \definedas (\pi|_{U'})^\ast R_i$ for $i = 1, \dots, n$ and $\tilde{R}_{n+1}(z') \definedas z' \in \Sigma_{\pi(z')} = \tilde{\Sigma}_{z'}$.
The restriction $\hat{\pi}|_{\tilde{\Sigma}|_{U'}} : \tilde{\Sigma}|_{U'} \to \Sigma$ is given by the canonical map $(\pi|_{U'})^\ast\Sigma \to \Sigma$.

If $z = R_l(b)$ for some $l \in \{1, \dots, n\}$, then there exists a neighbourhood  $U'\subseteq \Sigma$ of $z$ that does not contain any nodal points or marked points aside from those of the form $R_l(b')$ for $b' \in B$.
Also, as in Construction \ref{Construction_Universal_nodal_families}, \ref{Construction_Universal_nodal_families_4}., one can assume that there exists a holomorphic function $x' : U' \to \D$ \st $(\pi|_{U'}, x') : U' \to B\times \D$ is a holomorphic coordinate system on $U'$ and that $x'(R_l) = \{0\}$.
Define $U \definedas (\pi|_{U'})^\ast U' \subseteq (\pi|_{U'})^\ast\Sigma$ and $x \definedas x'\circ \Phi : U\to \D$, where $\Phi : U \to U'$ is the canonical bundle map covering $\pi|_{U'} : U' \to B$.
Consider $V \definedas U'\times\D \subseteq U'\times S^2$ and the function $y : V \to \D$ given by projection onto the second factor.
Let $q_1 : (\pi|_{U'})^\ast \Sigma \to U'$, $q_2 : U'\times S^2 \to U'$ be the projections and let $K \definedas \{\xi \in U \;|\; |x(\xi)| \leq |x'(q_1(\xi))| \neq 0 \}$, $L \definedas \{\xi' \in V \;|\; |y(\xi')| \leq |x'(q_2(\xi))| \neq 0\}$.
Denoting $\hat{\Sigma}_1 \definedas (\pi|_{U'})^\ast \Sigma \setminus K$, $\hat{\Sigma}_2 \definedas (U'\times S^2)\setminus L$ one can define $\tilde{\Sigma}|_{U'} \definedas \hat{\Sigma}_1 \amalg \hat{\Sigma}_2/_\sim$, where the equivalence relation is defined as in Construction \ref{Construction_Universal_nodal_families}, \ref{Construction_Universal_nodal_families_4}.
Namely $\xi \sim \xi'$ for $\xi\in U$, $\xi'\in V$ with $q_1(\xi) = q_2(\xi')$ and either $x(\xi)y(\xi') = x'(q_1(\xi)) \neq 0$ or $x(\xi) = y(\xi') = x'(q_1(\xi)) = 0$. \\
The projection $\tilde{\pi}|_{\tilde{\Sigma}|_{U'}} : \tilde{\Sigma}|_{U'} \to U'$ is then induced by the map $q_1\amalg q_2 : \hat{\Sigma}_1 \amalg \hat{\Sigma}_2 \to U'$. \\
The markings $\tilde{R}_i$ for $i \in \{1, \dots, n\}\setminus\{l\}$ are defined by $\tilde{R}_i \definedas (\pi|_{U'})^\ast R_i \subseteq (\pi|_{U'})^\ast \Sigma \setminus U \subseteq \tilde{\Sigma}|_{U'}$. $\tilde{R}_l \definedas U' \times \{\infty\} \subseteq U'\times S^2 \setminus V \subseteq \tilde{\Sigma}|_{U'}$ and $\tilde{R}_{n+1} \definedas U' \times \{1\} \subseteq U'\times S^2 \setminus V \subseteq \tilde{\Sigma}|_{U'}$. \\
The restriction $\hat{\pi}|_{\tilde{\Sigma}|_{U'}} : \tilde{\Sigma}|_{U'} \to \Sigma$ is given as follows: On $(\pi|_{U'})^\ast \Sigma \setminus U$, $\hat{\pi}|_{\tilde{\Sigma}|_{U'}}$ is given by the canonical morphism $(\pi|_{U'})^\ast \Sigma \to \Sigma$. To define $\hat{\pi}|_{\tilde{\Sigma}|_{U'}}$ on the remaining part of $\tilde{\Sigma}|_{U'}$, let $\zeta \in U'$. If $x'(\zeta) = 0$, $\tilde{\Sigma}_\zeta$ is the union of $\Sigma_{\pi(\zeta)}$ with $S^2$, with $R_l(\pi(\zeta)) \in \Sigma_{\pi(\zeta)}$ and $0 \in S^2$ identified. Let $\hat{\pi}|_{\tilde{\Sigma}_\zeta}$ be the identity on $\Sigma_{\pi(\zeta)}$ and on $S^2$ the constant map to $R_l(\pi(\zeta))$.
If $x'(\zeta) \neq 0$, $\tilde{\Sigma}_\zeta$ is given by the union of $\Sigma_{\pi(\zeta)} \setminus \{ z' \in \Sigma_{\pi(\zeta)} \;|\; |x'(z')| \leq |x'(\zeta)|\}$ with $S^2 \setminus \{ z' \in S^2 \;|\; |z'| \leq |x'(\zeta)| \}$, where $w \in \D \setminus \{ z' \in S^2 \;|\; |z'| \leq |x'(\zeta)| \} \subseteq S^2 \setminus \{ z' \in S^2 \;|\; |z'| \leq |x'(\zeta)| \}$ is identified with $(x'|_{U'_{\pi(\zeta)}})\inv\left( \frac{x'(\zeta)}{w}\right)$, where $U'_{\pi(\zeta)} \definedas U' \cap \Sigma|_{\pi(\zeta)}$.
Let $\hat{\pi}|_{\tilde{\Sigma}|_{U'}}$ be the identity on $\Sigma_{\pi(\zeta)} \setminus \{ z' \in \Sigma_{\pi(\zeta)} \;|\; |x'(z')| \leq |x'(\zeta)|\}$ and on $S^2 \setminus \{ z' \in S^2 \;|\; |z'| \leq |x'(\zeta)| \}$ be given by the map $w \mapsto (x'|_{U'_{\pi(\zeta)}})\inv\left( \frac{x'(\zeta)}{w}\right)$.
This is then a well-defined holomorphic diffeomorphism that maps $\infty \in S^2$ to $(x'|_{U'_{\pi(\zeta)}})\inv(0) = R_l(\pi(\zeta))$ and $1 \in S^2$ to $(x'|_{U'_{\pi(\zeta)}})\inv(x'(\zeta)) = \zeta$.

Finally, for $z$ one of the nodes, let the notation be as in the remark just before the statement of the lemma and assume w.\,l.\,o.\,g.~that $z = n_1$.
Denote $U' \definedas U_1$, $(x,y) \definedas (x_1,y_1) : U' \to \D^2$, $z' \definedas z_1 : V \to \D$.
Let $C_i \definedas C \cap U_i$, $C' \definedas C_1$. Note that $(\pi|_{U'})^\ast (\Sigma \setminus C')$ is a well-defined complex manifold and the projection onto $U'$ at every point is either a holomorphic submersion or has a neighbourhood that is the domain of nodal coordinates as in \ref{Equation_nodal_coordinates}.
I.\,e.~$(\pi|_{U'})^\ast (\Sigma \setminus C') \to U'$ satisfies the definition of a marked nodal family of Riemann surfaces, apart from the properness condition and the fibres are punctured marked nodal surfaces instead of marked nodal surfaces.
This is clear away from the subsets $(\pi|_{U'})^\ast C_i$, for the projection $\pi$ is a submersion away from the nodes.
In a neighbourhood of one of the $(\pi|_{U'})^\ast C_i$ for $i \geq 2$, $i = 2$, say, \wrt the coordinates described before the statement of the lemma, an explicit description of $(\pi|_{U'})^\ast (\Sigma \setminus C') \to U'$ is the following:
$\pi|_{U'} : U' \to V$ in coordinates is the map $f_1 : \D^{k+1} \to \D^k$, $(x,y,z_2, \dots, z_k) \mapsto (xy, z_2, \dots, z_k)$, $k \definedas \dim_\C(B)$, whereas $\pi|_{U_2} : U_2 \to V$ in coordinates is the map $f_2 : \D^{k+1} \to \D^k$, $(z_1, x_2, y_2, z_3, \dots, z_k) \mapsto (z_1, x_2y_2, z_3, \dots, z_k)$.
The pullback of the latter by the former hence explicitely is given by the map with domain
\begin{multline*}
\{(w_1, w_2) \in \D^{k+1} \times \D^{k+1} \;|\; f_1(w_1) = f_2(w_2)\} \\
= \{((x, y, x_2y_2, z_3, \dots, z_k), (xy, x_2, y_2, z_3, \dots, z_k)) \in \D^{k+1}\times \D^{k+1} \;|\; \\
(x, y, x_2, y_2, z_3, \dots, z_k) \in \D^{k+2} \} \cong \D^{k+2}
\end{multline*}
and projection given by $(x, y, x_2, y_2, z_3, \dots, z_k) \mapsto (x, y, x_2y_2, z_3, \dots, z_k)$.
Now to turn $(\pi|_{U'})^\ast(\Sigma\setminus C')$ into a marked nodal family, work in the local coordinates as before and consider the subset
\begin{multline*}
K \definedas \{(\zeta, z') \in U'\times S^2 \;|\; x(\zeta) \neq 0, |z'| \leq |x(\zeta)|\} \\
\cup \{(\zeta, z') \in U'\times S^2 \;|\; y(\zeta) \neq 0, |z'| \geq \frac{1}{|y(\zeta)|}\}
\end{multline*}
of $U'\times S^2$.
Then $\tilde{\Sigma}|_{U'} \definedas (\pi|_{U'})^\ast (\Sigma \setminus C') \amalg (U'\times S^2 \setminus K)/_\sim$, where the equivalence relation is defined as follows: \\
If $\zeta \in U'$ has coordinates $(x(\zeta), y(\zeta), z_1(\zeta), \dots, z_k(\zeta))$ with $x(\zeta) \neq 0$, $y(\zeta) \neq 0$, then $z' \in \{\zeta\}\times S^2 \setminus K$ is identified with the point on $((\pi|_{U'})^\ast (\Sigma \setminus C'))_\zeta = \Sigma_{\pi(\zeta)} \setminus C'$ with coordinates $\left(\frac{x(\zeta)}{z'}, z'y(\zeta), z_2(\zeta), \dots, z_k(\zeta)\right)$. \\
If $\zeta \in U'$ has coordinates $(x(\zeta), y(\zeta), z_1(\zeta), \dots, z_k(\zeta))$ with $x(\zeta) = 0$, $y(\zeta) \neq 0$, then $z' \in \{\zeta\}\times S^2 \setminus K$ with $z' \neq 0$ is identified with the point on $((\pi|_{U'})^\ast (\Sigma \setminus C'))_\zeta = \Sigma_{\pi(\zeta)} \setminus C'$ with coordinates $\left(0, z'y(\zeta), z_2(\zeta), \dots, z_k(\zeta)\right)$. \\
Analogously, if $\zeta \in U'$ has coordinates $(x(\zeta), y(\zeta), z_1(\zeta), \dots, z_k(\zeta))$ with $x(\zeta) \neq 0$, $y(\zeta) = 0$, then $z' \in \{\zeta\}\times S^2 \setminus K$ with $z' \neq \infty$ is identified with the point on $((\pi|_{U'})^\ast (\Sigma \setminus C'))_\zeta = \Sigma_{\pi(\zeta)} \setminus C'$ with coordinates $\left(\frac{x(\zeta)}{z'}, 0, z_2(\zeta), \dots, z_k(\zeta)\right)$. \\
Finally, if $\zeta \in U'$ has coordinates $(0, 0, z_1(\zeta), \dots, z_k(\zeta))$, then no identification takes place. \\
The projection $\tilde{\pi}|_{\tilde{\Sigma}|_{U'}} : \tilde{\Sigma}|_{U'} \to U'$ is induced by the projections $(\pi|_{U'})^\ast (\Sigma \setminus C') \to U'$ and $U'\times S^2 \to U'$. \\
The markings $\tilde{R}_i$ for $i = 1, \dots, n$ are given by the images of the pullbacks $(\pi|_{U'})^\ast R_i$ under the projection to the quotient $\tilde{\Sigma}|_{U'}$ and $\tilde{R}_{n+1}$ is the image of $U' \times \{1\}$ under the projection to $\tilde{\Sigma}|_{U'}$. \\
The restriction $\hat{\pi}|_{\tilde{\Sigma}|_{U'}} : \tilde{\Sigma}|_{U'} \to \Sigma$ to the image of $(\pi|_{U'})^\ast (\Sigma \setminus C')$ is given by the canonical map to $\Sigma$. This covers all of $\tilde{\Sigma}|_{U'}$ apart from the points $\{(\zeta, 0) \in U'\times S^2 \;|\; x(\zeta) = 0\}$, $\{(\zeta, \infty) \in U'\times S^2 \;|\; y(\zeta) = 0\}$ and $\{(\zeta, z') \in U'\times S^2 \;|\; x(\zeta) = y(\zeta) = 0\}$.
Each such point $(\zeta, z')$ is mapped to the point $\Sigma_{\pi(\zeta)} \cap C'$ in $\Sigma_{\pi(\zeta)}$. \\
Note that by construction, under $\hat{\pi}|_{\tilde{\Sigma}|_{U'}}$, the point corresponding to $(\zeta, 1) \in U'\times S^2$, \ie $\tilde{R}_{n+1}(\zeta)$ is mapped to $\zeta$. \\
The important thing here is the following: In the notation from before, locally the projection in a neighbourhood of the first node looks like the map $f_1 : \D^{k+1} \to \D^k$, $(x, y, z_2, \dots, z_k) \mapsto (xy, z_2, \dots, z_k)$, and analogously for the other nodes.
In these local coordinates, the pullback of $f_i$ for $i \geq 2$ by $f_1$ gave a well-defined nodal coordinate system.
But for $i = 1$ this is not the case, because both the subset $\{x = 0\}$ and the subset $\{y = 0 \}$ get mapped to $\{0\} \times \D^{k-1}$, the stratum along which the first node perseveres.
So the set of nodes in the naive pullback of $f_1$ by itself would have a set of nodes that looks like two hyperplanes intersecting transversely at the origin, which is not a submanifold, hence there can't exist a nodal coordinate system at this intersection.
The construction above then ``resolves'' this intersection by inserting a sphere, producing two different nodes at $(0, 0, z_2, \dots, z_k)$, one corresponding to the one which perseveres along $(0, y, z_2, \dots, z_k)$, the other to the one that perseveres along $(x, 0, z_2, \dots, z_k)$.
\end{proof}

As long as one does not impose any compactness condition, the existence of a local universal family \st the induced map to $\overline{M}_{g,n}$ is surjective is shown in \cite{MR2262197}, Proposition 6.3.
In the genus $g = 0$ case, one can also find such a family even with compact base space, for $\overline{M}_{0,n}$ itself is a complex manifold.
In the case of genus $g > 0$, such a result will not hold true.
But one can ask instead for the existence of a marked nodal family $(\pi : \Sigma \to B, R_\ast)$ that defines an orbifold branched covering of $\overline{M}_{g,n}$ that branches over the Deligne-Mumford boundary, maps $B$ surjectively onto $\overline{M}_{g,n}$ and has a compact base space $B$.
By the previous lemma, if one can show such a result for Riemann surfaces of type $(g,n)$, then the result also holds for all $(g,n')$ with $n' \geq n$.
First results in this direction were proved by Looijenga in \cite{MR1257324}, where it is shown that $\overline{M}_g$ has a finite branched covering by a smooth projective variety.
The difference to the result that I would like to use here is that this covering morphism does not come from a marked nodal family (which requires in particular the total space $\Sigma$ to be smooth), which is not the case for the branched covering constructed in \cite{MR1257324}.
But, although the construction in \cite{MR1257324} doesn't produce the desired result, see Proposition 1.4 in \cite{MR1739177}, there seems (to the author's limited understanding of algebraic geometry) to be a generalisation of that construction, see Theorem 3.9 in op.~cit. This shows, in conjunction with the previous lemma, \ie apply Theorem 3.9 in \cite{MR1739177} to get the marked nodal family $(\pi : \Sigma \to M, R_\ast)$ below and then apply the previous Lemma to get the families $(\pi^\ell : \Sigma^\ell \to M^\ell, R^\ell_\ast, T^\ell_\ast)$ for $\ell \geq 1$, the following conjecturally stated result.
\begin{proposition}\label{Proposition_Sequence_of_branched_coverings}
There exists a sequence of marked nodal families $(\pi^\ell : \Sigma^\ell \to M^\ell, R^\ell_\ast, T^\ell_\ast)$ for $\ell \geq 0$ of Riemann surfaces of type $(g, n+\ell)$, with markings $R^\ell_1, \dots, R^\ell_n$, $T^\ell_1, \dots, T^\ell_\ell$ \st $\Sigma^\ell = M^{\ell+1}$ for all $\ell \geq 0$, together with maps $\hat{\pi}^\ell : \Sigma^{\ell+1} \to \Sigma^{\ell}$ \st
\[
\xymatrix{
\cdots\ar[r]^-{\hat{\pi}^{\ell+1}} & \Sigma^{\ell+1} \ar[r]^-{\hat{\pi}^{\ell}} \ar[d]^-{\pi^{\ell+1}} & \Sigma^\ell \ar[r]^-{\hat{\pi}^{\ell-1}} \ar[d]^-{\pi^{\ell}} & \Sigma^{\ell-1} \ar[r]^-{\hat{\pi}^{\ell-2}} \ar[d]^-{\pi^{\ell-1}} & \cdots\ar[r]^-{\hat{\pi}^1} & \Sigma^1\ar[r]^-{\hat{\pi}^{0}} \ar[d]^-{\pi^{1}} & \Sigma^0 \ar[d]^-{\pi^0} \ar@{=}[r] & \Sigma \ar[d]^-{\pi} \\
\cdots\ar[r]^-{\pi^{\ell+1}} & M^{\ell+1} \ar[r]^-{\pi^{\ell}} & M^\ell \ar[r]^-{\pi^{\ell-1}} & M^{\ell-1} \ar[r]^-{\pi^{\ell-2}} & \cdots \ar[r]^-{\pi^1} & M^1 \ar[r]^-{\pi^{0}} & M^0 \ar@{=}[r] & M
}
\]
{\allowdisplaybreaks
\begin{align*}
&\xymatrix{
\Sigma^\ell \ar[r]^-{\hat{\pi}^{\ell-1}} \ar[d]^-{\pi^\ell} & \Sigma^{\ell-1} \ar[d]_-{\pi^{\ell-1}} \\
M^\ell \ar[r]^-{\pi^{\ell-1}} \ar@/^1pc/[u]^-{R^\ell_j} & M^{\ell-1} \ar@/_1pc/[u]_-{R^{\ell-1}_j\quad{\displaystyle\forall\, j=1,\dots, n}}
}
\\
&\xymatrix{
\Sigma^\ell \ar[r]^-{\hat{\pi}^{\ell-1}} \ar[d]^-{\pi^\ell} & \Sigma^{\ell-1} \ar[d]_-{\pi^{\ell-1}} \\
M^\ell \ar[r]^-{\pi^{\ell-1}} \ar@/^1pc/[u]^-{T^\ell_j} & M^{\ell-1} \ar@/_1pc/[u]_-{T^{\ell-1}_j\quad{\displaystyle\forall\, j=1,\dots, \ell-1}}
}
\end{align*}
}
all commute,
\[
\hat{\pi}^{\ell-1}\circ T^\ell_\ell = \id : \Sigma^{\ell-1} \to M^\ell
\]
and where $M$ is assumed to be closed, and hence so are the $M^\ell$ for all $\ell\geq 0$. Furthermore, for all $\ell \geq 0$, $(\pi^\ell : \Sigma^\ell \to M^\ell, R^\ell_\ast, T^\ell_\ast)$ defines an orbifold branched covering of $\overline{M}_{g,n+\ell}$ that branches over the Deligne-Mumford boundary and for every $z \in \Sigma^{\ell} = M^{\ell+1}$, putting $b \definedas \pi^{\ell}(z) \in M^{\ell}$, the map
\begin{multline*}
\hat{\pi}^{\ell-1}_z : (\Sigma^\ell_z, R^\ell_1(z), \dots, R^\ell_n(z), T^\ell_1(z), \dots, T^\ell_{\ell-1}(z)) \to  \\
\to (\Sigma^{\ell-1}_{b}, R^{\ell-1}_1(b), \dots, R^{\ell-1}_n(b), T^{\ell-1}_1(b), \dots, T^{\ell-1}_{\ell-1}(b))
\end{multline*}
is stabilising, \ie biholomorphic on every stable component of
\[
(\Sigma^\ell_z, R^\ell_1(z), \dots, R^\ell_n(z), T^\ell_1(z), \dots, T^\ell_{\ell-1}(z))
\]
and constant on every unstable component (of which there is at most one).
For $\ell > k$ denote the compositions
\begin{align*}
\hat{\pi}^\ell_k \definedas \hat{\pi}^{k}\circ \hat{\pi}^{k+1}\circ \cdots\circ \hat{\pi}^{\ell-1} : \Sigma^\ell &\to \Sigma^{k}
\intertext{and}
\pi^\ell_k \definedas \pi^k\circ \pi^{k+1}\circ \cdots\circ \pi^{\ell-1} : M^\ell &\to M^{k}\text{.}
\end{align*}
If $M^\ell_i$ is a stratum of the stratification of $M^\ell$ by signature, then for $k \leq \ell$ there exists a signature $j(i)$ \st $\pi^\ell_k|_{M^\ell_i} : M^\ell_i \to M^k_{j(i)}$ is a submersion.
\end{proposition}

\begin{defn}\label{Definition_ghost_component}
In the notation of the proposition above, a component of $\Sigma^\ell_b$, for any $b\in M^\ell$ and $\ell \in \N$, that is mapped to a point under $\hat{\pi}^\ell_0$ is called a \emph{ghost component}.
\end{defn}

On the spaces in Proposition \ref{Proposition_Sequence_of_branched_coverings} there are also canonical actions of permutation groups of the last $\ell$ markings, as follows directly from the construction of the spaces $\Sigma^\ell$, $M^\ell$ and maps $\pi^\ell$, $\hat{\pi}^\ell$.
\begin{proposition}\label{Proposition_Reordering_marked_points}
In the notation of the previous proposition: \\
For $\ell \geq 1$, let $\mathcal{S}_\ell$ be the group of permutations $\{1, \dots, \ell\}$.
Then there exist actions $\sigma^\ell$ and $\hat{\sigma}^\ell$ of $\mathcal{S}_\ell$ on $M^\ell$ and $\Sigma^\ell$ \st
\begin{equation}\label{Equation_Permutation_of_marked_points_1}
\xymatrix{
\mathcal{S}_\ell\times \Sigma^\ell \ar[r]^-{\hat{\sigma}^\ell} \ar[d]_-{\id\times \pi^\ell} & \Sigma^\ell \ar[d]^-{\pi^\ell} \\
\mathcal{S}_\ell \times M^\ell \ar[r]^-{\sigma^\ell} & M^\ell
}
\end{equation}
commutes. \\
Furthermore, for any $g\in \mathcal{S}_\ell$ and $k \in \{1, \dots, \ell\}$, 
\begin{equation}\label{Equation_Permutation_of_marked_points_2}
\hat{\sigma}^\ell_{g\inv}\circ T^\ell_k\circ \sigma^\ell_g = T^\ell_{g(k)}
\end{equation}
and under the inclusion $\mathcal{S}_\ell \subseteq \mathcal{S}_{\ell+1}$ as permutations of $\{1, \dots, \ell+1\}$ leaving $\ell+1$ fixed and the identification $\Sigma^\ell = M^{\ell+1}$,
\begin{equation}\label{Equation_Permutation_of_marked_points_3}
\xymatrix{
\Sigma^{\ell+1} \ar[r]^-{\hat{\sigma}^{\ell+1}_g} \ar[d]_-{\hat{\pi}^{\ell}} & \Sigma^{\ell+1} \ar[d]^-{\hat{\pi}^{\ell}} \\
\Sigma^{\ell} \ar[r]^-{\hat{\sigma}^{\ell}_g} & \Sigma^{\ell}
}\qquad
\xymatrix{
M^{\ell+1} \ar[r]^-{\sigma^{\ell+1}_g} \ar[d]_-{\pi^{\ell}} & M^{\ell+1} \ar[d]^-{\pi^{\ell}} \\
M^{\ell} \ar[r]^-{\sigma^{\ell}_g} & M^{\ell}
}
\end{equation}
commute and
\begin{equation}\label{Equation_Permutation_of_marked_points_4}
\hat{\sigma}^\ell = \sigma^{\ell+1}|_{\mathcal{S}_\ell \times M^{\ell+1}} : \mathcal{S}_\ell \times \Sigma^\ell \to \Sigma^\ell\text{.}
\end{equation}\label{Equation_Permutation_of_marked_points_5}
Denoting by $\tau_{\ell,\ell+1}\in \mathcal{S}_{\ell+1}$ the transposition exchanging $\ell$ and $\ell+1$,
\begin{equation}
\hat{\pi}^{\ell-1} = \pi^{\ell}\circ \sigma^{\ell+1}_{\tau_{\ell,\ell+1}} : M^{\ell+1} = \Sigma^\ell \to M^\ell = \Sigma^{\ell-1}\text{.}
\end{equation}
\end{proposition}
\begin{proof}
Again, not a complete proof, just a description of the construction of the actions. \\
Actually, the above characterisation serves at the same time as definition of these actions by induction: Because $\mathcal{S}_1 = \{\id\}$, the actions on $M^1$ and $\Sigma^1$ are automatically the identity.
Now assume that the actions of $\mathcal{S}_k$ for $k = 1, \dots, \ell$ have been defined.
Equation \ref{Equation_Permutation_of_marked_points_4} defines the restriction of $\sigma^{\ell+1}$ to $\mathcal{S}_\ell\times M^{\ell+1}$.
Equation \ref{Equation_Permutation_of_marked_points_2} requires that, for $g\in \mathcal{S}_{\ell}$, \ie $g(\ell+1) = \ell+1$, and $b \in M^{\ell+1}$, $\hat{\sigma}^{\ell+1}_{g}(T_{\ell+1}^{\ell+1}(b)) = T^{\ell+1}_{\ell+1}(\sigma^{\ell+1}_{g}(b))$.
Equation \ref{Equation_Permutation_of_marked_points_3} then gives $\hat{\sigma}^\ell_g\circ \hat{\pi}^\ell(T^{\ell+1}_{\ell+1}(b)) = \hat{\pi}^\ell(T^{\ell+1}_{\ell+1}(\sigma^{\ell+1}_g(b))) \in \Sigma^\ell_{\pi^{\ell}(b)}$.
Defining $z\definedas \hat{\pi}^\ell(T^{\ell+1}_{\ell+1}(b)) \in \Sigma^\ell_{\pi^{\ell}(b)}$, this shows that $\Sigma^\ell_{\pi^{\ell}(b)}$ is obtained from $\Sigma^{\ell+1}_b$ by forgetting the last marked point and stabilising associated to the point $z$ and $\Sigma^\ell_{\pi^\ell(\sigma^{\ell+1}_g(b))}$ is obtained from $\Sigma^{\ell+1}_{\sigma^{\ell+1}_g(b)}$ by forgetting the last marked point and stabilising associated to the point $\hat{\sigma}^\ell_g(z)$.
Lemma \ref{Lemma_Forget_marked_point} then gives a unique lift $\hat{\sigma}^{\ell+1}_g : \Sigma^{\ell+1}_b \to \Sigma^{\ell+1}_{\sigma^{\ell+1}_g(b)}$ of $\hat{\sigma}^\ell_g : \Sigma^\ell_{\pi^\ell(b)} \to \Sigma^\ell_{\sigma^\ell_g(\pi^\ell(b))}$ that maps $T^{\ell+1}_{\ell+1}(b)$ to $T^{\ell+1}_{\ell+1}(\sigma^{\ell+1}_g(b))$. \\
Because $\mathcal{S}_{\ell+1}$ is generated by $\mathcal{S}_\ell$ and $\tau_{\ell,\ell+1}$, it suffices to define $\sigma^{\ell+1}_{\tau_{\ell,\ell+1}}$ and $\hat{\sigma}^{\ell+1}_{\tau_{\ell,\ell+1}}$.
Now $M^{\ell+1}$ by definition in Lemma \ref{Lemma_Forget_marked_point} is the union of $(\Sigma^\ell \times_{\pi^\ell, M^\ell, \pi^\ell} \Sigma^\ell) \setminus C$, where $C$ is the diagonal in this fibre product over the nodes and markings in $\Sigma^\ell$, with a collection of spheres.
The action of $\sigma^{\ell+1}_{\tau_{\ell,\ell+1}}$ is then the one induced by the action on $\Sigma^\ell \times_{\pi^\ell, M^\ell, \pi^\ell} \Sigma^\ell$ exchanging the factors, and the identity on the spheres filling in $C$. \\
$\hat{\sigma}^{\ell+1}_{\tau_{\ell,\ell+1}}$ is then defined analogously to $\hat{\sigma}^{\ell+1}_g$ for $g \in \mathcal{S}_\ell$ before.
\end{proof}

The compactness statement in Proposition \ref{Proposition_Sequence_of_branched_coverings} is important for the following reason: In the genus $g = 0$ case, $\overline{M}_{0,n}$ is a compact complex manifold, hence in particular it is oriented and carries a fundamental class in its top homology group (with any coefficient group).
Hence any smooth (or continuous) map from $\overline{M}_{0,n}$ to another manifold defines a homology class in that manifold.
Now in the case of positive genus this holds no longer true for $\overline{M}_{g,n}$ itself.
But for any universal marked nodal family $(\pi : \Sigma \to M, R)$ of type $(g,n)$, $\overline{M}_{g,n}$ is the quotient space of the associated groupoid as in Definition 6.4 in \cite{MR2262197}.
As such both $\overline{M}_{g,n}$ as its quotient space and $M$ as the space of objects of this groupoid carry a stratification by orbit type, see \cite{1101.0180}, esp.~Section 5.
A stratum of $M$ in this stratification is a connected component of an equivalence class of the relation on $M$ given by abstract isomorphism of automorphism groups.
The stratification on $\overline{M}_{g,n}$ is then the one induced by the quotient map.
Since the morphisms of the associated groupoid are given by isomorphisms of nodal surfaces, this stratification respects the stratification by signature.
Now let $(\pi : \Sigma \to M, R_\ast)$ be a marked nodal family of Riemann surfaces of type $(g,n)$ with $M$ closed and \st the induced map $\upsilon : M \to \overline{M}_{g,n}$ defines an orbifold branched covering that branches over the Deligne-Mumford boundary.
Let $\overset{\circ}{M}$ be the top-dimensional part of the stratification by signature, \ie the set of those $b\in M$ \st $\Sigma_b$ is a smooth Riemann surface.
Let correspondingly $M_{g,n}$ be the part of $\overline{M}_{g,n}$ consisting of the equivalence classes of smooth Riemann surfaces.
Then $M_{g,n}$ is an orbifold, an orbifold structure (in the sense of Definition 2.4 in \cite{MR2262197}) being defined by the restriction of the orbifold structure for $\overline{M}_{g,n}$ constructed in \cite{MR2262197}.
By definition, $\overset{\circ}{\upsilon} \definedas \upsilon|_{\overset{\circ}{M}} : \overset{\circ}{M} \to M_{g,n}$ defines a (finite non-branched) orbifold covering.
Defining $\overset{\circ}{\Sigma} \definedas \Sigma|_{\overset{\circ}{M}}$, $\overset{\circ}{R}_\ast \definedas R_\ast|_{\overset{\circ}{M}}$, one can hence form the associated groupoid to the marked family of Riemann surfaces $(\overset{\circ}{\Sigma} \to \overset{\circ}{M}, \overset{\circ}{R}_\ast)$ as in Definition 6.4 in \cite{MR2262197}, which defines a groupoid structure on $M_{g,n}$.
$\upsilon$ as a branched orbifold covering is an open map and since $M$ is assumed compact, it is also a closed map.
Since $\overline{M}_{g,n}$ is connected, the restriction of $\upsilon$ to every connected component of $M$ is surjective and one can assume w.\,l.\,o.\,g.~that $M$ is connected as well.
Since the complement of $\overset{\circ}{M}$ in $M$ consists of submanifolds of real codimension at least two, $\overset{\circ}{M}$ is then connected as well.
Since the groupoid associated to $(\overset{\circ}{\Sigma} \to \overset{\circ}{M}, \overset{\circ}{R}_\ast)$, being complex is oriented, the stratification by orbit type on $\overset{\circ}{M}$ has a unique connected top-dimensional stratum and all other strata have codimension at least two in $\overset{\circ}{M}$.
Denote this top-dimensional stratum by $\overset{\circ\circ}{M}$.

Because $M$ is compact, one can assign two well-defined numbers, $|\mathcal{O}(\overset{\circ\circ}{M})|$, the length of the orbit $\mathcal{O}(b)$ of any point $b\in \overset{\circ\circ}{M}$ (by compactness of $M$ this is a finite number) and $|\Aut(\overset{\circ\circ}{M})|$, the order of the automorphism group $\Aut(b)$ of any point $b\in \overset{\circ\circ}{M}$ (which is a finite number by properness of the groupoid, irrespective of whether $M$ is compact or not).
With the help of these, to any map $f : \overline{M}_{g,n} \to X$, where $X$ is any manifold, \st $f\circ \pi^M_{M_{g,n}} : M \to X$ is smooth, $\pi^M_{M_{g,n}}$ being the quotient projection, one can assign a well-defined rational pseudocycle (as defined in Section 1 of \cite{MR2399678})
\[
\frac{1}{|\Aut(\overset{\circ\circ}{M})||\mathcal{O}(\overset{\circ\circ}{M})|} f\circ \pi^M_{M_{g,n}}|_{\overset{\circ\circ}{M}}\text{.}
\]

\section{Construction of smooth structures on moduli spaces}\label{Section_Construction_smooth_structures}

Throughout this section, fix a marked nodal family $(\pi : \Sigma \to M, R_\ast)$ of type $(g,n)$ and choose a metric $h$ on $\Sigma$ that is hermitian on every fibre of $\Sigma$.
Furthermore, let $(\kappa : X \to M, \omega)$ be a family of symplectic manifolds with fibres symplectomorphic to a closed symplectic manifold $(X_0, \omega_0)$ (in other words a fibre bundle with fibre $X_0$ and structure group $\mathrm{Symp}(X_0,\omega_0)$, the symplectomorphism group of $(X_0, \omega_0)$). Define $\tilde{\kappa} : \tilde{X} \to \Sigma$ as the pullback of $\kappa : X \to M$ to $\Sigma$ via $\pi$.
As before, $\mathcal{J}_\omega(X)$ is the set of $\omega$-compatible vertical almost complex structures on $X$, \ie bundle morphisms $J \in \End(VX)$ with $J^2 = -\id$ and \st $\omega(\cdot, J\cdot)$ defines a metric on $VX$.
In other words, for any $b\in M$, $J_b$ is a compatible almost complex structure on the symplectic manifold $(X_b, \omega_b)$.
Such a $J\in \mathcal{J}_\omega(X)$ is chosen and $\tilde{X}$ is equipped with the almost complex structure given by the pullback of $J$ to $\Sigma$ via the projection onto $M$ (and again denoted by $J$), and the metric $g^J$ on $V\tilde{X}$ defined by $\omega$ and $J$.
Finally, a locally trivial family $A$ of $2^\text{nd}$ homology classes $(A_b)_{b\in M}$, $A_b \in H_2(X_b;\Z)$, in the fibres of $X$ is fixed in the sense that there exists a covering $(U_i)_{i\in I}$ of $M$ and trivialisations $\phi_i : X|_{U_i} \cong U_i\times X_0$ \st $(\pr_2)_\ast \circ (\phi_i|_{X_b})_\ast A_b \in H_2(X_0; \Z)$ is independent of $b \in U_i$.

\subsection{Hamiltonian perturbations}\label{Subsection_Hamiltonian_perturbations}

For almost all of the notions and results on Hamiltonian perturbations,
see \cite{MR2045629}, Section 8.1.

The basic (separable) Banach space from which all perturbations will be chosen is defined in analogy with \cite{MR2399678}, Section 3.
\begin{defn}\label{Definition_Hamiltonian_perturbation}
Let $\varepsilon = (\varepsilon_i)_{i\in\N_0}$ be a fixed sequence of positive numbers.
Denote by $\tilde{\kappa} : \tilde{X} \to \Sigma$ the projection.
The space of Floer's $C^\varepsilon$-sections of $\tilde{\kappa}^\ast T^\ast\Sigma$ is
\[
\Gamma^\varepsilon(\tilde{\kappa}^\ast T^\ast\Sigma) \definedas 
\{ H\in \Gamma(\tilde{\kappa}^\ast T^\ast \Sigma) \;|\; \sum_{i=0}^\infty 
\varepsilon_i\|H\|_{C^i} < \infty \}\text{.}
\]
Let $C \subseteq \Sigma$ be the set of special points, \ie the union of all the markings and nodal points, and define $\tilde{C} \definedas \tilde{\kappa}\inv(C) \subseteq \tilde{X}$.
$C \subseteq \Sigma$ is a submanifold that intersects every fibre of $\Sigma$ in a finite number of points.
Define ($\operatorname{cl}$ denotes the closure)
\[
\Gamma^\varepsilon_0(\tilde{\kappa}^\ast T^\ast \Sigma) \definedas 
\operatorname{cl}\{ H \in \Gamma^\varepsilon(\tilde{\kappa}^\ast T^\ast \Sigma) \;|\; \supp(H) \subseteq \tilde{X} \setminus \tilde{C}\}\text{.}
\]
Let $0 < \delta < \frac{1}{4}$.
The space of \emph{Hamiltonian perturbations} is defined to be the open ball of radius $\delta$ in $\Gamma^\varepsilon_0(\tilde{\kappa}^\ast T^\ast \Sigma)$, \ie
\[
\mathcal{H}_{\varepsilon,\delta}(\tilde{X}) \definedas 
\{ H\in \Gamma^\varepsilon_0(\tilde{\kappa}^\ast T^\ast \Sigma) \;|\; \sum_{i=0}^\infty 
\varepsilon_i\|H\|_{C^i} < \delta \}\text{,}
\]
where $\varepsilon$ is chosen as in \cite{MR948771}, Lemma 5.1.
The subscripts $\varepsilon$ and $\delta$ will usually be dropped, \ie $\mathcal{H}(\tilde{X}) \definedas \mathcal{H}_{\varepsilon,\delta}(\tilde{X})$.
\end{defn}

The reason for the appearance of the constant $\delta$ in the above definition is so one can apply Exercise 8.1.3 from \cite{MR2045629}, and for any desingularisation $\hat{\iota}_b : S_b \to \Sigma_b \subseteq \Sigma$, for $b\in M$, of a fibre of $\Sigma$, equip the total space $\hat{X}_b$ of the pullback of the fibration $\tilde{X}$ to $S_b$, with a symplectic form, see (the proof of) Lemma \ref{Lemma_Dstar}.

\begin{construction}\label{Construction_Hamiltonian_connection}
Let $(S, j, r_\ast, \nu), b\in M, \hat{\iota} : S \to \Sigma_b$ be a desingularisation of $\Sigma_b$ and let $\hat{X} \definedas \hat{\iota}^\ast \tilde{X}$ with projection $\hat{\kappa} : \hat{X} \to S$.
Using $\hat{X} \cong S\times X_b$, an element $H \in \mathcal{H}(\tilde{X})$ defines a linear map $H_b : TS \to \coprod_{z\in S} C^\infty(\hat{X}_z, \R) \cong S\times C^\infty(X_b, \R)$ in the following way:
If $\zeta_z\in T_zS_b = V_zS\subseteq T_zS$, then $H_b(\zeta_z) : \hat{X}_z = X_b \to \R$, $x \mapsto H_{\tilde{\iota}(x)}(\hat{\iota}_\ast \zeta_z)$, where $\tilde{\iota} : \hat{X} \to \tilde{X}$ is the canonical map covering $\hat{\iota} : S \to \Sigma$. In this way, $H_b$ is considered as a $1$-form on $S$ with values in the smooth functions on the fibres of $\hat{X}$. \\
Furthermore, for $\zeta_z\in T_zS$, to the function $H_b(\zeta_z) \in C^\infty(\hat{X}_z,\R)$ corresponds a Hamiltonian vector field $X_{H_b(\zeta_z)}\in \mathfrak{X}(\hat{X}_z)$.
In this way one gets a fibrewise linear function $X_H : TS \to \coprod_{z\in S} \mathfrak{X}(\hat{X}_z)$, \ie a $1$-form with values in the space of Hamiltonian vector fields on the fibres of $\hat{X}$. \\
$H_b$ then defines a connection on $\hat{X}$ with projection onto the vertical tangent bundle given by
\begin{align*}
\pi^{T\hat{X}}_{V\hat{X}} : T\hat{X} \cong TS \times TX_b &\to
V\hat{X} \cong S\times TX_b \\
(\zeta_z,v_x) &\mapsto (z,v_x) + (z,X_{H_b(\zeta_z)}(x))\text{.}
\end{align*}
\end{construction}

\begin{defn}
For $H\in \mathcal{H}(\tilde{X})$ and $b\in M$ as above, $X_{H_b} : TS \to \mathfrak{X}(X_b)$ from the previous construction is called the \emph{Hamiltonian vector field} on $S$ associated to $H$ and
\[
X_{H_b}^{0,1} \definedas \frac{1}{2}(X_{H_b} + J_b\circ X_{H_b}\circ j_b)
\]
is its complex antilinear part.
\end{defn}

\begin{defn}\label{Definition_J_H}
Let $J \in \mathcal{J}_\omega(X)$ and let $H\in \mathcal{H}(\tilde{X})$, $b\in M$. Using the notation from the previous construction, the \emph{almost complex structure $\hat{J}^{H_b}$ on $\hat{X}$ defined by $J$ and $H$} is given by $\hat{J}^{H_b}|_{V\hat{X}} = J_b$, using the canonical identification $V\hat{X} \cong S\times V_b X$ and $\hat{J}^{H_b}|_{H\hat{X}} = \hat{\iota}^\ast j$ \wrt the decomposition $T\hat{X} \cong V\hat{X} \oplus H\hat{X}$ defined by the connection associated to $H$.
\end{defn}

\begin{remark}\label{Remark_tilde_J_H}
For $(w,v)\in T\hat{X} \cong TS\times TX_b$,
\[
\hat{J}^{H_b}(w,v) = (j w, J_bv + 2J_bX_{H_b(w)}^{0,1})\text{.}
\]
\end{remark}

The main existence result for Hamiltonian perturbations:

\begin{lemma}\label{Lemma_Existence_of_perturbations}
Let $(S, j, r_\ast, \nu), b\in M, \hat{\iota} : S \to \Sigma_b$ be a
desingularisation of $\Sigma_b \subseteq \Sigma$.
Given any $z \in S \setminus \left( \bigcup_{i=1}^n r_i \cup \bigcup_{i=1}^d \{n^1_i, n^2_i\} \right)$, where $\nu = \{\{n^1_1, n^2_1\}, \dots, \{n^1_d, n^2_d\}\}$, any $x \in \hat{\iota}^\ast \tilde{X}$ with $\hat{\iota}^\ast \tilde{\kappa}(x) = z$, any $\eta\in \Hom(T_zS, V_{x}\hat{\iota}^\ast\tilde{X})$ and any neighbourhood $U$ of $x$ in $\hat{\iota}^\ast \tilde{X}$, there exists an $H\in \mathcal{H}(\tilde{X})$ \st $\hat{\iota}^\ast H\in \mathcal{H}(\hat{\iota}^\ast \tilde{X})$ has support in $U$ and satisfies $(X_{\hat{\iota}^\ast H})_{x} = \eta$.
\end{lemma}
\begin{proof}
Choose a coordinate neighbourhood $V\subseteq M$ of $B$ and a symplectic trivialisation $X|_V \cong V\times X_0$ of $X$. Because $z$ does not coincide with any of the special points, there exists a coordinate neighbourhood $\tilde{V} \subseteq \Sigma$ that is mapped by $\pi$ onto $V$ and \st there exist coordinates $(t_0, \dots, t_k)\in \C\times \C^k$, $k \definedas \dim_\C(M)$ on $\tilde{V}$ and coordinates on $V$ \st $\pi|_{\tilde{V}} : \tilde{V} \to V$ in these coordinates is the map $(t^0, t^1, \dots, t^k) \mapsto (t^1, \dots, t^k)$ and $z$ corresponds to the point $(0, \dots, 0)$. Furthermore, let $x' \definedas \tilde{\iota}(x) \in \tilde{X}$, where $\tilde{\iota} : \hat{\iota}^\ast\tilde{X} \to \tilde{X}$ is the canonical map covering $\hat{\iota}$. Then there exists a neighbourhood of $x'$ in $\tilde{X}|_{\tilde{V}}$ of the form $\tilde{V}\times W$, where $W \subseteq X_0$ is a coordinate neighbourhood with coordinates $(x^1, \dots, x^l) \in \R^l$, $l \definedas \dim_\R(X_0)$, mapping $x'$ to zero. One can assume that $\tilde{V}\times W \cap \tilde{\kappa}\inv(\pi\inv(b)) \subseteq \hat{\iota}(U)$, so $\hat{\iota}\inv(\tilde{V}\times W) \subseteq U$.
Choose two smooth cutoff functions $\delta_X : \R^{\dim X} \to [0,1]$ and $\delta_\Sigma : \C\times \C^{k} \to [0,1]$ which are identically $1$ in
a neighbourhood of $0$ and have compact support inside the neighbourhoods of $0$ corresponding to $W$ and $\tilde{V}$, respectively. 
Let $\frac{\partial}{\partial t^i_j}$, $i=0, \dots, k$, $j=1,2$, be the coordinate vector fields for the real coordinates associated to the complex coordinates $t^i$.
Then $t^0$ defines a complex coordinate in a neighbourhood of $z$ in $S$ and one can evaluate $\omega(\eta(\frac{\partial}{\partial t^0_j}), \cdot) = \sum_m \lambda_{j,m} \d x^m$ for some $\lambda_{j,m}$ and define $H\in
\mathcal{H}(\tilde{X})$ as the $1$-form that vanishes identically outside $\tilde{V}\times W$ and in the above coordinates maps the $\frac{\partial}{\partial t^i_j}$ for $i > 0$ to zero and maps the $\frac{\partial}{\partial t^0_j}_{(t^0,t^1,\dots,t^k)}$ to the function that vanishes identically outside $W$ and maps
\[
(x^1,\dots,x^l) \mapsto \delta_\Sigma(t^0,t^1,\dots,t^k)
\delta_X(x^1,\dots,x^l) \sum_m \lambda_{j,m} x^m\text{.}
\]
\end{proof}

For such Hamiltonian perturbations $H \in \mathcal{H}(\tilde{X})$ and points $b \in M$, one can define the moduli spaces of holomorphic curves in the family $\Sigma$ with values in $X$. These are the main objects to be studied in this thesis:
\begin{defn}
Let $H \in \mathcal{H}(\tilde{X})$, let $b\in M$ and let $(S, j, r_\ast, \nu), b\in M, \hat{\iota} : S \to \Sigma_b$ be a desingularisation, $\hat{\iota}^\ast \tilde{X} \definedas \hat{X}$. Then
\begin{align*}
\mathcal{M}_b(\tilde{X}, A, J, H) \definedas \{ u : \Sigma_b \to \tilde{X} \;|\; & \tilde{\kappa}\circ u = \id_{\Sigma_b}, [\pr_2\circ u] = A_b \in H_2(X_b; \Z), \\
& \hat{\iota}^\ast u : S \to \hat{X} \text{ is $j$-$\hat{J}^{H_b}$-holomorphic}\}\text{,}
\end{align*}
which is independent of the choice of desingularisation and where $\pr_2 : \tilde{X}|_{\Sigma_b} \cong \Sigma_b\times X_b \to X_b$ is the projection. Hence
\begin{align*}
\mathcal{M}(\tilde{X}, A, J, \mathcal{H}(\tilde{X})) \definedas \coprod_{\substack{b\in M \\ H\in \mathcal{H}(\tilde{X})}} \mathcal{M}_b(\tilde{X}, A, J, H)
\end{align*}
is well-defined and comes with two projections
\begin{align*}
\pi^{\mathcal{M}}_M : \mathcal{M}(\tilde{X}, A, J, \mathcal{H}(\tilde{X})) &\to M
\intertext{and}
\pi^{\mathcal{M}}_{\mathcal{H}} : \mathcal{M}(\tilde{X}, A, J, \mathcal{H}(\tilde{X})) &\to \mathcal{H}(\tilde{X})\text{.}
\end{align*}
\end{defn}

The remainder of this chapter consists of defining (Banach) manifold structures over certain subsets of this space (although not on all of $\mathcal{M}(\tilde{X}, A, J, \mathcal{H}(\tilde{X}))$) in such a way that it reflects the stratified structure of $M$ by the stratification by signature (where well-defined) and to define a topology on $\mathcal{M}(\tilde{X}, A, J, \mathcal{H}(\tilde{X}))$ compatible with the manifold topologies on these parts.

\subsection{The case of a fixed Riemann surface}

In this first subsection, the case of a fixed Riemann surface and a fixed trivial symplectic fibre bundle over this surface, equipped with (fixed) almost complex and Hamiltonian structures, is treated.
First, the respective Fredholm problem is set up, \ie a Banach space bundle $\mathcal{E} \to \mathcal{B}$ and a Cauchy-Riemann operator $\dbar$ as a section of this bundle are defined.
Then the linearisation of this Cauchy-Riemann operator is calculated and using this, first, it is shown that this operator is Fredholm of the expected index (Corollary \ref{Corollary_Nonlinear_Riemann_Roch}).
Then a condition is derived for when the linearisation of $\dbar$ is complex linear (Corollary \ref{Corollary_Ddbar_complex_linear}), which will mainly be needed in the last third of this text. Finally, the well-known elliptic regularity result will be derived, namely that the elements in the solution set $\dbar\inv(0)$ of the Fredholm problem actually consist of smooth sections.
Most of these are rather well-known results, but first of all, they are all crucial for the later discussion, and second, using the results from the previous chapter and assuming the standard results for linear Cauchy-Riemann operators, the proofs are actually rather short.
Most of the proofs here actually follow the same scheme: By expressing everything in a chart for $\mathcal{B}$ and a trivialisation for $\mathcal{E}$, the problem is reduced to a known result about linear Cauchy-Riemann operators.

\begin{construction}
Let, for now, $(X, \omega)$ be a fixed closed symplectic manifold and let $A \in H_2(X)$.
Let $(S,j)$ be a smooth Riemann surface of Euler characteristic $\hat{\chi}$ equipped with a hermitian metric $h$, let $J\in \mathcal{J}_\omega(X)$ and let $H\in \mathcal{H}(\hat{X})$, where $\hat{X} \definedas S\times X$.
Then there are the connection defined by $H$ as in Construction \ref{Construction_Hamiltonian_connection}, together with $h$ and the metric on $\hat{X}$ defined via the connection by the metric $g^J \definedas \omega(\cdot, J\cdot)$ on the fibres of $\hat{X}$ and the pullback of $h$ via the projection on the horizontal tangent bundle. These turn $\pr_1 : \hat{X} \to S$ into a Riemannian submersion.
The covariant derivative on vertical vector fields will be denoted by $\nabla^H$.
Now over $\hat{X}$ there are the two vector bundles $\Hom(TS, V\hat{X})$ and its subbundle of complex antilinear morphisms
\[
\overline{\Hom}_{(j,J)}(TS, V\hat{X}) \definedas
\{\eta \in \Hom(TS, V\hat{X}) \;|\; \eta\circ j = -J\circ \eta\}\text{.}
\]
Both of these inherit a metric from the metrics $h$ and $g^J$ on $TS$ and $V\hat{X}$, respectively.
$\Hom(TS, V\hat{X})$ also inherits a connection from the connections on $TS$ and $V\hat{X}$.
But in general, this connection does not restrict to a well defined connection on $\overline{\Hom}_{(j,J)}$, since this subbundle is not invariant under parallel transport.
The problem here being that the Levi-Civita connection on $X$ (coming from $g^J$) is not hermitian (the metric $h$ on $S$ is automatically K{\"a}hler, $S$ being twodimensional). This can be solved by replacing the Levi-Civita connection on $X$ by the hermitian connection $\tilde{\nabla}$ defined by $g^J$ and $J$. It is shown in \cite{MR2045629}, Appendix C.7, that
\[
\tilde{\nabla}_XY = \nabla_XY - \frac{1}{2}J(\nabla_XJ)Y\text{.}
\]
$\tilde{\nabla}$ preserves $J$ and the metric $g^J$, but it is not torsion free, its torsion being given by
\[
T^{\tilde{\nabla}}(X,Y) = -\frac{1}{4}N_J(X,Y)\text{,}
\]
where $N_J$ denotes the Nijenhuis tensor of $J$.
Also, the map
\begin{align*}
\pi^{\Hom}_{\overline{\Hom}_{(j,J)}} : \Hom(TS, V\hat{X}) &\to
\overline{\Hom}_{(j,J)} (TS, V\hat{X}) \\
\eta &\mapsto \frac{1}{2}(\eta + J\circ \eta\circ j)
\end{align*}
defines a smooth bundle morphism. \\
Using these structures, one can make the following definitions: \\
Fix, once and for all, a real number $p > 2$. Furthermore, let $k\in \N$.
\[
\mathcal{B}^{k,p}(\hat{X},A,J,H) \definedas
\{u\in L^{k,p}(\hat{X},\pr_1,g^J) \;|\; [\pr_2\circ u] =
A\}\text{,}
\]
where $L^{k,p}(\hat{X},\pr_2, g^J)$ is the Sobolev space of sections of $\hat{X}$.
This is a Banach manifold, since it is a union of connected components, hence an open subset, of $L^{k,p}(\hat{X},\pr_1,g^J)$. For a continuous path in this space via the Sobolev embedding theorem defines a continuous path of continuous functions, hence two sections in the same connected component define the same homology class.
Proceeding,
\begin{align*}
\mathcal{E}^{k-1,p}(\hat{X},A,J,H) \definedas
\{&(\eta, u)\in L^{k-1,p}(\overline{\Hom}_{(j,J)}(TS,V\hat{X}), h^\ast\otimes
g^J, \nabla^S\otimes \nabla^H, \\
& \hat{X}, \pr_1, g^J) \;|\; u\in \mathcal{B}^{k,p}(\hat{X},
A,J,H)\}\text{,}
\end{align*}
which, as a restriction of the Banach space bundle $L^{k-1,p}(\overline{\Hom}_{(j,J)}(TS,V\hat{X}), h^\ast\otimes g^J, \nabla^S\otimes \nabla^H, \hat{X}, \pr_1, g^J)$ to an open subset, is a Banach space bundle over $\mathcal{B}^{k,p}(\hat{X}, A,J,H)$.
To define the Cauchy-Riemann operator, note that just the same way, there also is the Banach space bundle
\begin{align*}
\mathcal{F}^{k-1,p}(\hat{X},A,J,H) \definedas
\{&(\eta, u)\in L^{k-1,p}(\Hom(TS,V\hat{X}), h^\ast\otimes
g^J, \nabla^S\otimes \nabla^H, \\
& \hat{X}, \pr_1, g^J) \;|\; u\in \mathcal{B}^{k,p}(\hat{X},
A,J,H)\}
\end{align*}
over $\mathcal{B}^{k,p}(\hat{X},A,J,H)$ which comes with the section $D^{\mathrm{v}} : \mathcal{B}^{k,p}(\hat{X},A,J,H) \to \mathcal{F}^{k-1,p}(\hat{X},A,J,H)$.
$D^{\mathrm{v}}$ here denotes the vertical differential of a section, \ie for a section $u : S \to \hat{X}$, $D^{\mathrm{v}} u \definedas \pr^{T\hat{X}}_{V\hat{X}}\circ Du$, where $\pr^{T\hat{X}}_{V\hat{X}} : T\hat{X} \to V\hat{X}$ denotes the projection defined by the Hamiltonian connection associated to $H$.
Additionally, the bundle morphism $\pi^{\Hom}_{\overline{\Hom}_{(j,J)}}$ from above induces a morphism of Banach space bundles from $\mathcal{F}^{k-1,p}(\hat{X},A,J,H)$ to $\mathcal{E}^{k-1,p}(\hat{X}, A,J,H)$, hence the composition of the section $D^{\mathrm{v}}$ with this morphism defines a section of $\mathcal{E}^{k-1,p}(\hat{X},A,J,H)$. \\
Finally, note that $J$ induces almost complex structures on both $V\hat{X}$ and $\overline{\Hom}_{(j,J)}(TS, V\hat{X})$, which turn both $T\mathcal{B}^{k,p}(\hat{X},A,J,H)$ and $\mathcal{E}^{k-1,p}(\hat{X}, A,J,H)$ into complex Banach space bundles over $\mathcal{B}^{k,p}(\hat{X},A,J,H)$.
\end{construction}
\begin{defn}
The \emph{(nonlinear) Cauchy-Riemann operator} on $\hat{X}$ is the section
\begin{align*}
\dbar^{J,H}_S : \mathcal{B}^{k,p}(\hat{X},A,J,H) &\to
\mathcal{E}^{k-1,p}(\hat{X}, A,J,H) \\
u &\mapsto \left(\frac{1}{2}(D^{\mathrm{v}}u + J\circ D^{\mathrm{v}}u\circ
j), u\right)
\end{align*}
of the Banach space bundle
\[
\mathcal{E}^{k-1,p}(\hat{X}, A, J, H) \to \mathcal{B}^{k,p}(\hat{X}, A, J, H)\text{.}
\]
\end{defn}

\begin{lemma}\label{Lemma_Ddbar}
Let $u\in \Gamma^k(\hat{X},\pr_1,g^J)\cap \mathcal{B}^{k,p}(\hat{X}, A, J, H)$.
Then \wrt the chart for $\mathcal{B}^{k,p}(\hat{X},A,J,H)$ and the trivialisation of $\mathcal{E}^{k-1,p}(\hat{X}, A, J, H)$ around $u$, the linearisation of $\dbar^{J,H}_S$ at $u$ is given by (for $Z\in TS$)
\begin{align*}
(D\dbar^{J,H}_S)_u : T\mathcal{B}^{k,p}(\hat{X}, A, J, H)
&\to \mathcal{E}^{k-1,p}(\hat{X}, A, J, H)_u \\
((D\dbar^{J,H}_S)_u)\eta(Z) &= \nabla^{0,1}_Z\eta - \frac{1}{2}J\left(
(\nabla_\eta J) \partial u(Z) + \pi^{T\hat{X}}_{V\hat{X}}
(\hat{\nabla}^H_\eta \hat{J}^H) \tilde{Z}\right) \\
&= \nabla^{0,1}_Z\eta - \pi^{T\hat{X}}_{V\hat{X}}\left(
\frac{1}{2}\hat{J}^H \left( \hat{\nabla}^H_\eta \hat{J}^H\right)
(\partial u(Z) + \tilde{Z})\right) \\
&= \nabla^{0,1}_Z\eta - K_{\hat{J}^H}(\eta, \partial u(Z) + \tilde{Z}) \;- \\
&\qquad\quad\;\;\, -\; \frac{1}{8}\pi^{T\hat{X}}_{V\hat{X}} N_{\hat{J}^H}(\eta, \partial u(Z) + \tilde{Z})\text{,}
\intertext{where}
\nabla^{0,1}_Z\eta &\definedas \frac{1}{2}(\nabla_Z \eta +
J\nabla_{jZ} \eta)\text{,} \\
\partial u(Z) &\definedas \frac{1}{2}(D^{\mathrm{v}} u(Z) -
JD^{\mathrm{v}} u(jZ))\text{,}
\end{align*}
$K_{\hat{J}^H}$ is the symmetric part of the bundle morphism
\[
T\hat{X}\otimes T\hat{X} \to V\hat{X}, \quad (\eta, \xi) \mapsto \pi^{T\hat{X}}_{V\hat{X}}\left(\frac{1}{2}\hat{J}^H \left( \hat{\nabla}^H_\eta \hat{J}^H\right)\xi\right)
\]
and where $\tilde{Z}$ denotes the horizontal lift of $Z\in TS$ to
$\hat{X}$, $\hat{\nabla}^H$ denotes the Levi-Civita connection
on $\hat{X}$ and $\hat{J}^H$ denotes the almost complex
structure on $\hat{X}$ defined by $J$, $j$ and the connection
given by $H$ as in Definition \ref{Definition_J_H}.
\end{lemma}
\begin{proof}
See \cite{ediss15878}, Lemma II.17, p.~65.
\end{proof}

This result seems to differ by the term involving
$\tilde{Z}$ from the corresponding formula in \cite{MR2045629},
Section 8.3, p.~257\,f., esp.~Remark 8.3.8. Although that is not a
real argument for why the formula above is correct, Corollary
\ref{Corollary_Ddbar_complex_linear} and Lemma
\ref{Lemma_Holomorphic_section_is_holomorphic_map} at least show that
it produces the consequences one (or at least the author) would hope
for.

\begin{corollary}\label{Corollary_Nonlinear_Riemann_Roch}
\[
\dbar^{J,H}_S : \mathcal{B}^{k,p}(\hat{X}, A, J, H) \to
\mathcal{E}^{k-1,p}(\hat{X}, A, J, H)
\]
is a Fredholm operator of index
\[
\dim_\C(X)\hat{\chi} + 2c_1(A)\text{.}
\]
\end{corollary}

\begin{corollary}\label{Corollary_Ddbar_complex_linear}
Let $u\in \mathcal{B}^{k,p}(\hat{X}, A, J, H)$, $\eta \in
T_u\mathcal{B}^{k,p}(\hat{X}, A, J, H)$.  If $\dbar^{J,H}_S u = 0$, then
\[
(((D\dbar^{J,H}_S)_u)(J\eta) - J((D\dbar^{J,H}_S)_u)\eta)(Z) = \pi^{T\hat{X}}_{V\hat{X}}
N_{\hat{J}^H}(\eta, Du(Z))\text{,}
\]
where $Du : TS\to T\hat{X}$ is the usual differential. In particular,
if $N_{\hat{J}^H}(\eta, v) = 0$ for all $\eta\in
V\hat{X}|_{\im u}$ and $v\in \im Du$, then 
\[
(D\dbar^{J,H}_S)_u : T_u\mathcal{B}^{k,p}(\hat{X},A,J,H) \to
\mathcal{E}^{k-1,p}(\hat{X},A,J,H)_u
\]
is a complex linear operator.
\end{corollary}
\begin{proof}
First, assume that $\eta$ is of class $C^k$. Then by definition,
$\nabla_Z\eta = \pi^{T\hat{X}}_{V\hat{X}} \hat{\nabla}^H_{Du(Z)}\eta$
(in case $k=1$, the right hand side of this formula does not make any literal sense for sections of class $L^{k,p}$, whereas the left hand side does by definition of the $L^{k,p}$-spaces), where one considers $\eta$ as a vertical vector
field on $\hat{X}$ on the image of $u$. Furthermore, because $\eta$ is a
vertical vector field, $J\eta = \hat{J}^H\eta$ and
$\pi^{T\hat{X}}_{V\hat{X}}\hat{J}^H = J$. Also, by definition of
$\dbar^{J,H}_S u$ and $\partial u$, $Du(Z) = \dbar^{J,H}_S u(Z) + \partial u(Z) +
\tilde{Z}$, in particular $Du(Z) = \partial u(Z) + \tilde{Z}$ if $\dbar^{J,H}_S u =
0$. With this, by the second formula for $(D\dbar^{J,H}_S)_u$ from Lemma
\ref{Lemma_Ddbar},
\begin{align*}
(((D\dbar^{J,H}_S)_u)(J\eta) - J((D\dbar^{J,H}_S)_u)\eta)(Z) &= \pi^{T\hat{X}}_{V\hat{X}}
(((D\dbar^{J,H}_S)_u)(\hat{J}^H\eta) - \hat{J}^H((D\dbar^{J,H}_S)_u)\eta)(Z) \\
&= \pi^{T\hat{X}}_{V\hat{X}}\biggl(\hat{\nabla}^H_{Du(Z)} (\hat{J}^H
\eta) - \frac{1}{2} \hat{J}^H\left(\hat{\nabla}^H_{\hat{J}^H\eta}
\hat{J}^H\right)Du(Z) \;- \\
&\quad\; -\; \hat{J}^H\hat{\nabla}^H_{Du(Z)}\eta - \frac{1}{2}
\left(\hat{\nabla}^H_{\eta}\hat{J}^H\right)Du(Z)\biggr) \\
&= \pi^{T\hat{X}}_{V\hat{X}}\left( \left(\hat{\nabla}^H_{Du(Z)}
\hat{J}^H\right)\eta - \left(\hat{\nabla}^H_\eta \hat{J}^H\right)
Du(Z)\right)
\end{align*}
by the Leibniz rule and the left formula in line (C.7.5) of Lemma C.7.1,
p.~566, in \cite{MR2045629}. The claim for $\eta$ of class
$C^k$ now follows from the right formula in line (C.7.5) of
Lemma C.7.1, p.~566, in \cite{MR2045629}. \\
The general case ($\eta$ of class $L^{k,p}$) then follows by the
standard density argument.
\end{proof}

The following lemma should motivate the appearance of the almost
complex structure $\hat{J}^H$ in the lemma and corollary above.
\begin{lemma}\label{Lemma_Holomorphic_section_is_holomorphic_map}
In the notation of the above construction, for a section $u \in
\Gamma^k(\hat{X},h,g^J,\nabla^H)\cap \mathcal{B}^{k,p}(\hat{X},A,J,H)$,
$\pi^{T\hat{X}}_{H\hat{X}}\circ
\dbar^{\hat{J}^{H}}_S u = 0$ and $\pi^{T\hat{X}}_{V\hat{X}}\circ
\dbar^{\hat{J}^{H}}_S u = \dbar^{J,H}_S u$, where 
$\hat{J}^H$ is the almost complex structure on $\hat{X}$ as in
Construction \ref{Construction_Hamiltonian_connection} and
$\dbar^{\hat{J}^{H}}_S$ is the standard Cauchy-Riemann
operator on functions between the almost complex manifolds $S$ and
$\hat{X}$. In particular, $u$ satisfies $\dbar^{J,H}_S u = 0$ \iff $u
: S \to \hat{X}$ is a $(j,\hat{J}^{H})$-holomorphic map.
\end{lemma}
\begin{proof}
By definition of $\hat{J}^{H}$,
\begin{align*}
\dbar^{\hat{J}^{H}}_S u &= \frac{1}{2}(Du + \hat{J}^{H}\circ
Du\circ j) \\
&= \frac{1}{2}(\pi^{T\hat{X}}_{V\hat{X}}\circ Du + J\circ
\pi^{T\hat{X}}_{V\hat{X}} \circ Du\circ j \;+ \\
&\quad\; +\; \pi^{T\hat{X}}_{H\hat{X}}\circ Du +
(\pi_\ast|_{H\hat{X}})\inv \circ j\circ \pi_\ast \circ
(\pi^{T\hat{X}}_{H\hat{X}})_\ast\circ Du\circ j) \\
&= \dbar^{J,H}_S u + \frac{1}{2}((\pi_\ast|_{H\hat{X}})\inv +
(\pi_\ast|_{H\hat{X}})\inv \circ j\circ j)
\intertext{because by definition of a connection $\pi_\ast
  \circ(\pi^{T\hat{X}}_{H\hat{X}})_\ast = \pi_\ast$ and
  $\pi_\ast\circ Du = \id$ as well as
  $\pi^{T\hat{X}}_{H\hat{X}}\circ Du =
  (\pi_\ast|_{H\hat{X}})\inv$}
&= \dbar^{J,H}_S u + 0\text{.}
\end{align*}
\end{proof}

\begin{lemma}\label{Lemma_elliptic_regularity_I}
Let $v\in \mathcal{B}^{k,p}(\hat{X}, A, J, H)$ with $\dbar^{J,H}_S v = 0$. Then
$v$ is smooth, \ie $v\in \Gamma(\hat{X}, \pr_1, g^J)$.
\end{lemma}
\begin{proof}
See \cite{ediss15878}, Lemma II.19, p.~69.
\end{proof}

\subsection{The case of a smooth family of Riemann surfaces}

\begin{construction}
In the course of this construction, it will very soon be necessary to work with universal moduli spaces, in particular to fix some space of perturbations. Hence it is easier to start out with two families, namely a nodal family over which the perturbations are defined and a smooth desingularisation of this family over some locally closed submanifold.
So let $(\pi : \Sigma \to M, R)$ be a nodal family of Riemann surfaces of Euler characteristic $\chi$ with $n$ markings and let $(\rho : S \to B, \hat{R}, N, \iota, \hat{\iota})$ be a desingularisation of $\Sigma$ over $B$ as in Definition \ref{Definition_Desingularisation}.
Also, fix a metric $h$ on $S$ that induces a hermitian metric $h_b$ on every fibre $S_b \definedas \rho^{-1}(b)$ over a point $b\in B$.
As stated in the beginning of this section, let $(\kappa : X \to M, \omega)$ be a family of symplectic manifolds with fibres symplectomorphic to a closed symplectic manifold $(X_0, \omega_0)$.
Define $\tilde{\kappa} : \tilde{X} \to \Sigma$ as the pullback of $\kappa : X \to M$ to $\Sigma$ via $\pi$ and as before, let $A$ be a locally trivial family of $2^\text{nd}$ homology classes in the fibres of $X$.
Assume that $M$ is connected, hence there is a well-defined first Chern class $c_1(A) \definedas c_1^{TX_b}(A_b)$ for any $b\in M$. \\
Let finally $H\in \mathcal{H}(\tilde{X})$ be a Hamiltonian connection on $\tilde{X}$.
Using $\iota : B\to M$ and $\hat{\iota} : S\to \Sigma$, one can pull all these structures back to $B$ and $S$, \ie $\hat{\rho} : \hat{X} \definedas \hat{\iota}^\ast \tilde{X} = S\times X \to S$ is again a symplectic fibre bundle on which $\iota^\ast J$ and $\hat{\iota}^\ast H$ define almost complex and Hamiltonian structures, respectively.
For simplicity and by abuse of notation, $\iota^\ast J$ and $\hat{\iota}^\ast H$ will be denoted by $J$ and $H$, respectively.
For $b\in B$, denote by $J_b$, $g^J_b$ and $H_b$ the pullbacks of $J$, $g^J$ and $H$ to $\hat{X}_b \definedas \hat{X}|_{S_b}$, considered as a symplectic fibration over $S_b$ via the restriction $\hat{\rho}_b$ of $\hat{\rho}$.
Also denote by $j_b$ the complex structure on the Riemann surface $S_b$.
Denote
\begin{align*}
\mathcal{B}^{k,p}_b(\hat{X}, A, J, H) &\definedas \mathcal{B}^{k,p}(
\hat{X}_b, A, J_b, H_b)
\intertext{and}
\mathcal{E}^{k-1,p}_b(\hat{X}, A, J, H) &\definedas \mathcal{E}^{k-1,p}(
\hat{X}_b, A, J_b, H_b)\text{.}
\end{align*}
The reason for the additional superscript $H$ in comparison to the
previous notation will become clearer a little bit later. \\
With this define
\begin{align*}
\mathcal{B}^{k,p}(\hat{X}, A, J, H) &\definedas \coprod_{b\in B}
\mathcal{B}^{k,p}_b(\hat{X}, A, J, H), \\
\mathcal{E}^{k-1,p}(\hat{X}, A, J, H) &\definedas \coprod_{b\in B}
\mathcal{E}^{k-1,p}_b(\hat{X}, A, J, H)
\end{align*}
Furthermore, one can define a section (set-theoretically, at this point) by
\[
\dbar^{J,H} \definedas \coprod_{b\in B} \dbar^{J_b,H_b}_{S_b} :
\mathcal{B}^{k,p}(\hat{X}, A, J, H)\to \mathcal{E}^{k-1,p}(\hat{X}, A,
J, H)\text{.}
\]
\begin{defn}
The \emph{moduli space of $(J,H)$-holomorphic curves} in the family $S$ and representing the homology class $A$ is defined as the subset
\[
\mathcal{M}(\hat{X}, A, J, H) \definedas \left(\dbar^{J,H}\right)\inv(0)
\]
of $\mathcal{B}^{k,p}(\hat{X}, A, J, H)$ for any $k\in \N$, $p > 1$ with $kp > 2$, where $0$ denotes the image of the zero section in $\mathcal{E}^{k-1,p}(\hat{X}, A, J, H)$. This is well-defined by Lemma \ref{Lemma_elliptic_regularity_I}.
\end{defn}
The goal now is to equip this set with a manifold structure.
Following the usual course of action, to achieve this one wants to turn $\mathcal{E}^{k-1,p}(\hat{X}, A, J, H) \to \mathcal{B}^{k,p}(\hat{X}, A, J, H)$ into a Banach space bundle and $\dbar^{J,H}$ into a Fredholm section.
The straightforward way to attempt to define charts on $\mathcal{B}^{k,p}(\hat{X}, A, J, H)$ would be, for a point $a\in B$, to pick an open neighbourhood $U\subseteq B$ of $a$ and a smooth trivialisation
\[
\phi_a :  U\times S_a \overset{\cong}{\longrightarrow} S|_{U} \subseteq S\text{,}
\]
inducing maps
\[
\phi_{ab} : S_a \to S_b, \quad z\mapsto \phi_a(b, z)
\]
for $b\in U$. Defining
\[
\mathcal{B}^{k,p}_U(\hat{X}, A, J, H) \definedas \coprod_{b\in U}
\mathcal{B}^{k,p}_b(\hat{X}, A, J, H)\text{,}
\]
there is a bijection
\begin{align*}
\overline{\phi}_a : \mathcal{B}^{k,p}_U(\hat{X}, A, J, H) &\to U\times
\mathcal{B}^{k,p}_a(\hat{X}, A, J, H) \\
\mathcal{B}^{k,p}_b(\hat{X}, A, J, H)\ni u &\mapsto (b, \phi_{ab}^\ast u)
\text{,}
\end{align*}
where for $z\in S_b$, if $u(z) = (z, \overline{u}(z)) \in S_b\times X$, then for $w\in S_a$, $\phi_{ab}^\ast u(w) = (w, \overline{u}(\phi_{ab}(w)))$.
This map is well-defined, because first of all, it is clear that for $u\in \mathcal{B}^{k,p}(\hat{X}, A, J, H)$, $\phi_{ab}^\ast u\in \mathcal{B}^{k,p}(\hat{X}|_{(S_a, (\phi_{a}^\ast j)_b)}, A, (\phi_{a}^\ast J)_b, (\phi_{a}^\ast H)_b)$, where $\hat{X}|_{(S_a, (\phi_{a}^\ast j)_b)}$ denotes $\hat{X}_a$, but with the base space $S_a$ now equipped with the complex structure $(\phi_{a}^\ast j)_b$ instead of $j_a$.
But as Banach manifolds,
\[
\mathcal{B}^{k,p}(\hat{X}|_{(S_a, (\phi_{a}^\ast j)_b)},
A, (\phi_{a}^\ast J)_b, (\phi_{a}^\ast H)_b) = \mathcal{B}^{k,p}(
\hat{X}_a, A, J_a, H_a)\text{,}
\]
in the sense of a literal equality of sets as well as of equivalence classes of Banach manifold atlases. This raises the question why then to use the notation $\mathcal{B}^{k,p}(\hat{X}_a, A, J_a, H_a)$, and analogously $\mathcal{E}^{k-1,p}(\hat{X}_a, A, J_a, H_a)$, instead of the much shorter $\mathcal{B}^{k,p}(\hat{X}_a, A)$ and $\mathcal{E}^{k-1,p}(\hat{X}_a, A, J_a)$ (in the latter case $J_a$ is actually part of the definition). The reason is mainly due to the next construction where a copy of $\mathcal{B}^{k,p}(\hat{X}_a, A, J, H)$ appears for every $H\in \mathcal{H}(\tilde{X})$, which would then necessitate notation such as $\{H\}\times \mathcal{B}^{k,p}(\hat{X}_a, A)$. Also this notation serves as a reminder that every $\mathcal{B}^{k,p}(\hat{X}_a, A, J_a, H_a)$ comes with a distinguished atlas. \\
Now for another such chart given by an open subset $V\subseteq B$, trivialisation $\psi_c : V\times S_c \cong S|_V$ and corresponding trivialisation
\begin{align*}
\overline{\psi}_c : \mathcal{B}^{k,p}_V(\hat{X}, A, J, H) &\to V\times
\mathcal{B}^{k,p}_c(\hat{X}, A, J, H) \\
\mathcal{B}^{k,p}_b(\hat{X}, A, J, H)\ni u &\mapsto (b, \psi_{cb}^\ast u)
\text{,}
\end{align*}
the transition functions would be given by
\begin{align*}
(U\cap V) \times \mathcal{B}^{k,p}_c(\hat{X}, A, J, H) &\to
(U\cap V) \times \mathcal{B}^{k,p}_a(\hat{X}, A, J, H) \\
(b, u) &\mapsto (b, \phi_{ab}^\ast(\psi_{cb\ast}u))\text{,}
\end{align*}
where if $u(w) = (w, \overline{u}(w))\in \hat{X}_c$ for $w\in S_c$, then for $z\in S_a$, $\phi_{ab}^\ast(\psi_{cb\ast}u)(z) = (z, \overline{u}(\psi_{bc}\inv\phi_{ab}(z)))$.
In other words, there is a map $U\cap V \to \Diff(S_c, S_a)$, $b\mapsto \psi_{bc}\inv\circ \phi_{ab}$ and the transition functions are given in terms of the action of this map.\label{Page_No_universal_CR_operator}
But as is explained \eg in \cite{Wehrheim_Poly_Slides_1} or \cite{Wehrheim_Poly_Slides_2}, the induced action of the diffeomorphism group on Sobolev spaces simply is not smooth.
So the charts that have just been defined do \emph{not} patch together to give an atlas and it does not even make sense to ask whether or not $\dbar^{J,H}$ defines a smooth (Fredholm) section.
At this point one has to make a decision on how to proceed.
The more definitive way would be to use the sc-manifold/polyfold framework of Hofer, Wysocki and Zehnder, for an introduction see \eg the introduction by the inventors themselves \cite{MR2644764}, the slides cited above, or \cite{1210.6670}. \\
Here, I will take a slightly different route.
Namely remember that it is actually the spaces of holomorphic curves one is interested in, \ie the zero set of  $\dbar^{J,H}$ and one should actually look at the restriction of the transition functions above to this set.
By this what is meant is the following:
In analogy to $\mathcal{B}^{k,p}_U(\hat{X}, A, J, H)$, define
\[
\mathcal{E}^{k-1,p}_U(\hat{X}, A, J, H) \definedas \coprod_{b\in U}
\mathcal{E}^{k-1,p}_b(\hat{X}, A, J, H)
\]
and
\[
\dbar^{J,H}_U \definedas \dbar^{J,H}|_{\mathcal{B}^{k,p}_U(\hat{X},
A, J, H)} : \mathcal{B}^{k,p}_U(\hat{X}, A, J, H) \to
\mathcal{E}^{k-1,p}_U(\hat{X}, A, J, H)\text{.}
\]
Consequently,
\[
\mathcal{M}_U(\hat{X}, A, J, H) \definedas \left(\dbar^{J,H}_U\right)\inv(0) = \mathcal{M}(\hat{X}, A, J, H) \cap \mathcal{B}^{k,p}_U(\hat{X}, A, J, H)\text{.}
\]
$\mathcal{B}^{k,p}_U(\hat{X}, A, J, H)$ can be turned into a Banach manifold, giving it the product manifold structure of $U\times \mathcal{B}^{k,p}_a(\hat{X}, A, J, H)$ via the bijection above.
This structure then obviously depends on a choice of trivialisation $\phi_a$ of $S|_U$ and will be denoted by $\mathcal{B}^{k,p}_{U,\phi_a}(\hat{X}, A, J, H)$.
Analogously, $\mathcal{E}_U^{k-1,p}(\hat{X}, A, J, H)$ can be given a smooth Banach space bundle structure over $\mathcal{B}_{U,\phi_a}^{k,p}(\hat{X}, A, J, H)$ by identifying it with $U\times \mathcal{E}_a^{k-1,p}(\hat{X}, A, J, H)$ via the map
\begin{align*}
\hat{\phi}_a : \mathcal{E}_U^{k-1,p}(\hat{X}, A, J, H) &\to U\times \mathcal{E}_a^{k-1,p}(\hat{X}, A, J, H) \\
\mathcal{E}_b^{k-1,p}(\hat{X}, A, J, H)\ni (\eta, u) &\mapsto (b, (\pi^{\overline{\Hom}_{((\phi_{a}^\ast j)_b, (\phi_{a}^\ast J)_b)}}_{\overline{\Hom}_{(j_a, J_a)}} \phi_{ab}^\ast \eta, \phi_{ab}^\ast u))\text{.}
\end{align*}
For this to be well-defined it is assumed that $U$ is small enough \st
\[
\pi^{\overline{\Hom}_{((\phi_{a}^\ast j)_b, (\phi_{a}^\ast J)_b)}}_{\overline{\Hom}_{(j_a, J_a)}} : \overline{\Hom}_{((\phi_{a}^\ast j)_b, (\phi_{a}^\ast J)_b)}(TS_a, V\hat{X}_a) \to \overline{\Hom}_{(j_a, J_a)}(TS_a, V\hat{X}_a)
\]
is an isomorphism for all $b\in U$. Again, $\mathcal{E}_U^{k-1,p}(\hat{X}, A, J, H)$ equipped with this smooth structure will be denoted $\mathcal{E}_{U, \phi_a}^{k-1,p}(\hat{X}, A, J, H)$. \\
Finally, defining
\begin{align*}
\dbar^{J,H}_{U,\phi_a} : U\times \mathcal{B}_a^{k,p}(\hat{X}, A, J, H) &\to U\times \mathcal{E}_a^{k-1,p}(\hat{X}, A, J, H) \\
(b, u) &\mapsto (b, \pi^{\overline{\Hom}_{((\phi_{a}^\ast j)_b, (\phi_{a}^\ast J)_b)}}_{\overline{\Hom}_{(j_a, J_a)}} \dbar_{(S_a, (\phi_{a}^\ast j)_b)}^{(\phi_{a}^\ast J)_b, (\phi_{a}^\ast H)_b} u)\text{,}
\end{align*}
all of the above fit into a commutative diagram
\begin{equation}\label{Diagram_E_U_B_U_Dbar_U}
\xymatrix{
\mathcal{E}^{k-1,p}_{U,\phi_a}(\hat{X}, A, J, H) \ar[d] \ar[r]^-{\hat{\phi}_a} & U\times \mathcal{E}^{k-1,p}_a(\hat{X}, A, J, H) \ar[d] \\
\mathcal{B}^{k,p}_{U,\phi_a}(\hat{X}, A, J, H) \ar[r]^-{\overline{\phi}_a} \ar@/^1pc/[u]^-{\dbar^{J,H}_U} & U\times \mathcal{B}^{k,p}_a(\hat{X}, A, J, H) \ar@/_1pc/[u]_-{\dbar^{J,H}_{U,\phi_a}}
}
\end{equation}
\end{construction}

With this setup, $\dbar^{J,H}_U$ is a parametrised version of a Cauchy-Riemann operator, which hence is a Fredholm operator itself and the Fredholm index can be computed fairly easily.
$c_1(A)$ here is the first Chern number as in the beginning of this subsection.
\begin{lemma}
In the notation of the previous construction,
\[
\dbar^{J,H}_U :
\mathcal{B}^{k,p}_{U,\phi_a}(\hat{X}, A, J, H) \to \mathcal{E}^{k-1,p}_{U,\phi_a}(\hat{X}, A, J, H)
\]
is a Fredholm section of index
\[
\ind(\dbar^{J,H}_U) = \dim_\C(X)\hat{\chi} + 2c_1(A) + \dim_\R(U)\text{.}
\]
\end{lemma}
\begin{proof}
(Sketch only) The result will follow from the following functional
analytic claim:
\begin{claim}
Let $X, Y$ be Banach spaces, let $V\subseteq X$ and $U\subseteq \R^n$
be open subsets and let $F : V\times U \to Y$ be a continuously
differentiable map with the property that for every $b\in U$, the map
$F(\cdot, b) : V \to Y$ is a (nonlinear) Fredholm map of index
$d$. Then $F$ is Fredholm of index $d + n$.
\end{claim}
\begin{proof}
Let $(u,b) \in V\times U$. Denote by $D_1F_{(u,b)}$ and $D_2F_{(u,b)}$
the (partial) derivatives of $F$ at $(u,b)$ in the direction of $V$
and $U$, respectively. By assumption, $D_1F_{(u,b)} : X\to Y$ is a
Fredholm operator of index $d$. It follows that $D_1F_{(u,b)}\circ
\pr_1 : X\times \R^n \to Y$ is a Fredholm operator of index $d+n$
(it clearly has the same image as $D_1F_{(u,b)}$ and its kernel is
$\ker(D_1F_{(u,b)}) \times \R^n$. The operator $D_2F_{(u,b)}\circ
\pr_2 : X\times \R^n \to Y$ is compact, for the image of the unit ball
in $X\times \R^n$ is just the image of the (compact) unit ball in
$\R^n$, hence compact. Hence $DF_{(u,b)} = D_1F_{(u,b)}\circ \pr_1 +
D_2F_{(u,b)}\circ \pr_2$ is the sum of a Fredholm operator of index
$d+n$ and a compact operator, hence a Fredholm operator of index $d+n$
by a standard result about Fredholm operators.
\end{proof}
To apply this claim, around a point $a\in B$, consider diagram \ref{Diagram_E_U_B_U_Dbar_U} and the definition of $\dbar^{J,H}_{U,\phi_a}$ from the previous construction above.
In that definition, every
\begin{multline*}
\dbar^{(\phi_{a}^\ast J)_b, (\phi_{a}^\ast H)_b}_{(S_a, (\phi_{a}^\ast j)_b)} : \mathcal{B}^{k,p}_a(\hat{X}, A, J, H) = \mathcal{B}^{k,p}(\hat{X}_a, A, J_a, H_a) \to \\ \mathcal{E}^{k-1,p}(\hat{X}|_{(S_a, (\phi_{a}^\ast j)_b)}, A, (\phi_{a}^\ast J)_b, (\phi_{a}^\ast H)_b)
\end{multline*}
is a Fredholm operator of index $d = \dim_\C(X)\chi + 2c_1(A)$ by Corollary \ref{Corollary_Nonlinear_Riemann_Roch}.
Composing with the bundle isomorphism $\mathcal{E}^{k-1,p}(\hat{X}|_{(S_a, (\phi_{a}^\ast j)_b)}, A, (\phi_{a}^\ast J)_b, (\phi_{a}^\ast H)_b) \to \mathcal{E}^{k-1,p}_a(\hat{X}, A, J, H)$ defined by $\pi^{\overline{\Hom}_{((\phi_{a}^\ast j)_b, (\phi_{a}^\ast J)_b)}}_{\overline{\Hom}_{(j_a, J_a)}}$ does not change this.
Choosing a chart for $\mathcal{B}_a(\hat{X}, A, J, H)$ and a local trivialisation for $\mathcal{E}_a(\hat{X}, A, J, H)$ around a given point then brings one to the situation of the claim above.
\end{proof}

\begin{construction}\label{Construction_Universal_Dbar}
Using the same notation as in the previous construction, define
\begin{align*}
\mathcal{B}^{k,p}(\hat{X},A,J,\mathcal{H}(\tilde{X}))
&\definedas \coprod\limits_{H \in \mathcal{H}(\tilde{X})}
\mathcal{B}^{k,p}(\hat{X},A,J,H) \\
\mathcal{E}^{k-1,p}(\hat{X},A,J,\mathcal{H}(\tilde{X}))
&\definedas \coprod\limits_{H \in \mathcal{H}(\tilde{X})}
\mathcal{E}^{k-1,p}(\hat{X},A,J,H)
\intertext{and}
\dbar^{J,\mathcal{H}} \definedas \coprod\limits_{H \in \mathcal{H}(\tilde{X})}
\dbar^{J, H} & :
\mathcal{B}^{k,p}(\hat{X},A,J,\mathcal{H}(\tilde{X})) \to
\mathcal{E}^{k-1,p}(\hat{X},A,J,\mathcal{H}(\tilde{X}))\text{.}
\end{align*}
There are natural projections
\begin{align*}
\pi^{\mathcal{B}}_{\mathcal{H}} :
\mathcal{B}^{k,p}(\hat{X},A,J,\mathcal{H}(\tilde{X})) &\to
\mathcal{H}(\tilde{X})
\intertext{and}
\pi^{\mathcal{E}}_{\mathcal{H}} :
\mathcal{E}^{k-1,p}(\hat{X},A,J,\mathcal{H}(\tilde{X})) &\to
\mathcal{H}(\tilde{X})\text{.}
\end{align*}
\begin{defn}
\[
\mathcal{M}(\hat{X},A,J,\mathcal{H}(\tilde{X})) \definedas
\left(\dbar^{J,\mathcal{H}}\right)^{-1}(0)\text{.}
\]
\end{defn}

Again, given an open neighbourhood $U\subseteq B$ of $a\in B$ and a smooth trivialisation $\phi_a : U\times S_a \cong S|_U$, define
\begin{align*}
\mathcal{B}^{k,p}_U(\hat{X},A,J,\mathcal{H}(\tilde{X}))
&\definedas \coprod\limits_{H \in \mathcal{H}(\tilde{X})}
\mathcal{B}^{k,p}_U(\hat{X},A,J,H) \\
\mathcal{E}^{k-1,p}_U(\hat{X},A,J,\mathcal{H}(\tilde{X}))
&\definedas \coprod\limits_{H \in \mathcal{H}(\tilde{X})}
\mathcal{E}^{k-1,p}_U(\hat{X},A,J,H)
\intertext{and}
\dbar^{J,\mathcal{H}}_U \definedas \coprod\limits_{H \in \mathcal{H}(\tilde{X})}
\dbar^{J, H}_U & :
\mathcal{B}^{k,p}_U(\hat{X},A,J,\mathcal{H}(\tilde{X})) \to
\mathcal{E}^{k-1,p}_U(\hat{X},A,J,\mathcal{H}(\tilde{X})) \\
\mathcal{M}_U(\hat{X},A,J,\mathcal{H}(\tilde{X})) &\definedas \mathcal{M}(\hat{X},A,J,\mathcal{H}(\tilde{X})) \cap \mathcal{B}^{k,p}_U(\hat{X},A,J,\mathcal{H}(\tilde{X}))
\end{align*}
as sets. Denote by $\mathcal{B}^{k,p}_{U,\phi_a}(\hat{X}, A, J, \mathcal{H}(\tilde{X}))$ the set $\mathcal{B}^{k,p}_U(\hat{X}, A, J, \mathcal{H}(\tilde{X}))$ equipped with the product Banach manifold structure of $\mathcal{H}(\tilde{X}) \times \mathcal{B}^{k,p}_{U,\phi_a}(\hat{X}, A, J, H)$ for any fixed chosen $H\in \mathcal{H}(\tilde{X})$, again identifying all the $\mathcal{B}^{k,p}_U(\hat{X}, A, J, H)$ for different $H$ by the set theoretic identity.
$\mathcal{E}^{k-1,p}_{U,\phi_a}(\hat{X}, A, J, \mathcal{H}(\tilde{X}))$ is defined as a Banach manifold in the same way. In the trivialisations of these spaces defining their smooth structures, $\dbar^{J,\mathcal{H}}_U$ is given by
\begin{multline*}
\dbar^{J,\mathcal{H}}_{U,\phi_a} \definedas \coprod_{H\in \mathcal{H}(\tilde{X})} \dbar^{J,H}_{U,\phi_a} : \\
U\times \mathcal{B}^{k,p}_a(\hat{X},A,J,H) \times \mathcal{H}(\tilde{X}) \to U\times \mathcal{E}_a^{k-1,p}(\hat{X},A,J,H)\times \mathcal{H}(\tilde{X})\text{.}
\end{multline*}
\end{construction}

\begin{lemma}
In the notation of the above construction,
\[
\dbar^{J,\mathcal{H}}_U : \mathcal{B}^{k,p}_{U,\phi_a}(\hat{X}, A, J, \mathcal{H}(\tilde{X})) \to \mathcal{E}^{k-1,p}_{U,\phi_a}(\hat{X}, A, J, \mathcal{H}(\tilde{X}))
\]
is smooth. \\
Given $(b,u,H)\in U\times \mathcal{B}_a^{k,p}(\hat{X}, A, J, H)\times \mathcal{H}(\tilde{X})$ with $u\in \Gamma^k(\hat{X}|_{S_a})$, \wrt the charts on $\mathcal{B}^{k,p}_{U,\phi_a}(\hat{X}, A, J, \mathcal{H}(\tilde{X}))$ from the previous construction and the standard chart for $\mathcal{B}_a^{k,p}(\hat{X}, A, J, H)$ around $u$, the linearisation of $\dbar^{J,\mathcal{H}}_U$ at $\phi_{ba}^\ast u\in \mathcal{B}^{k,p}_b(\hat{X}, A, J, H)$ in the direction $(e, \xi, h)$, where $e\in T_b U$, $\xi\in T_u \mathcal{B}^{k,p}_a(\hat{X}, A, J, H)$ and $h$ is a $C^\varepsilon$-section of $\pr_1^\ast T^\ast \Sigma$, is given by
\begin{multline*}
\left(D\dbar^{J,\mathcal{H}}_{U,\phi_a}\right)_{(b,u,H)}(e, \xi, h) = \\
\pi^{\overline{\Hom}_{((\phi_{a}^\ast j)_b, (\phi_{a}^\ast J)_b)}}_{\overline{\Hom}_{(j_a, J_a)}}\left(\left(D\dbar^{(\phi_{a}^\ast J)_b, (\phi_{a}^\ast H)_b}_{(S_a, (\phi_{a}^\ast j)_b)}\right)_u \xi + K_{(b,u,H)}(e) + (\phi_{a}^\ast X^{0,1}_h)_b \right)\text{,}
\end{multline*}
with
\begin{align*}
K_{(b,u,H)}(e) &\definedas \frac{1}{2} D_b(\phi_{a}^\ast J)(e)\circ D^\mathrm{v}u \circ (\phi_{a}^\ast j)_b \;+ \\
&\quad\; +\;\frac{1}{2}\left(X_{D_b(\phi_{a}^\ast H)(e)} + (\phi_{a}^\ast J)_b\circ X_{D_b(\phi_{a}^\ast H)(e)}\circ (\phi_{a}^\ast j)_b\right) \;+ \\
&\quad\; +\; \frac{1}{2} (\phi_{a}^\ast J)_b\circ D^\mathrm{v}u \circ \left(D_b(\phi_{a}^\ast j)(e)\right)\text{,}
\end{align*}
where $D^\mathrm{v}u$ denotes the vertical derivative of $u$ \wrt the connection on $S_a\times X$ defined by $(\phi_{a}^\ast H)_b$ and $b\mapsto (\phi_{a}^\ast J)_b$, $b\mapsto (\phi_{a}^\ast H)_b$ and $b\mapsto (\phi_{a}^\ast j)_b$ are regarded as maps from $U$ to the space of $\omega$-compatible almost complex structures on $X$, the space of Hamiltonian structures on $S_a\times X$ and the space of complex structures on $S_a$, respectively.
\end{lemma}

\begin{remark}\label{Remark_Ddbar_universal}
For the moment the only two important things about the map $K_{(b,u,H)}$ above are that it defines a compact operator, for it factors through the finite dimensional space $T_bU$ and that its image consists of $C^{r-1}$-sections if $u$ is of class $C^r$.
\end{remark}

\begin{lemma}\label{Lemma_Main_transversality_result}
In the same situation as in the previous lemma, let $V\subseteq \phi_a^\ast \hat{X}$ be an open subset and let $W\subseteq u\inv(V)$ be an open subset that intersects every connected component of $\{b\}\times S_a$ nontrivially. Let $\mathcal{K} \subseteq T_H\mathcal{H}(\tilde{X})$ be the closure of the span of those Hamiltonian perturbations that have support in $\pr_1\inv(W)\cap V$ (as sections of $\pr_1^\ast T^\ast \Sigma$). Let furthermore $z_i \in S_a$, $i = 1,\dots, r$ be a collection of points on $S_a$. Then the following maps are surjective:
\begin{enumerate}[a)]
  \item\label{a} The restriction of $\left(D\dbar^{J,\mathcal{H}}_{U,\phi_a}\right)_{(b,u,H)}$ to $\{0\}\times \{\xi \in T_U\mathcal{B}^{k,p}_a(\hat{X}, A, J, H) \;|\; \xi(z_i) = 0 \;\forall\, i=1,\dots, r\}\times \mathcal{K}$.
  \item\label{b} The map 
    \begin{multline*}
      \left(D\dbar^{J,\mathcal{H}}_{U,\phi_a}\right)_{(b,u,H)} \times \ev_{1\ast}\times \cdots\times \ev_{r\ast} \times \left(\pi^\mathcal{B}_U\right)_\ast : \\
      T_bU\times T_u\mathcal{B}^{k,p}_a(\hat{X}, A, J, H)\times \mathcal{K} \to \\
      \mathcal{E}^{k-1,p}_a(\hat{X}, A, J, H)\times (u^\ast V\hat{X})_{z_1}\times \cdots\times (u^\ast V\hat{X})_{z_r}\times T_bU \\
      (e,\xi,h) \mapsto \left(\left(D\dbar^{J,\mathcal{H}}_{U,\phi_a}\right)_{(b,u,H)}(e,\xi,h), \xi(z_1), \dots, \xi(z_r), e\right)
    \end{multline*}
\end{enumerate}
\end{lemma}
\begin{proof}
\ref{b}) follows immediately from \ref{a}) and the proof of \ref{a}) follows exactly the same line of argument that appears several times in \cite{MR2045629}, \eg Proposition 3.2.1, Proposition 3.4.2, Proposition 6.2.7, or the most closely related Theorem 8.3.1, or in \cite{MR2399678} Lemma 4.1.
\end{proof}

\begin{defn}
For a closed affine subspace $\mathcal{K} \subseteq \mathcal{H}(\tilde{X})$, meaning the intersection of a closed affine subspace of the space of $C^\varepsilon$-sections of $\pr_1^\ast T^\ast \Sigma$ with $\mathcal{H}(\tilde{X})$, see Definition \ref{Definition_Hamiltonian_perturbation}, define
\begin{align*}
\mathcal{M}(\hat{X}, A, J, \mathcal{K})
&\definedas \left(\pi^{\mathcal{M}}_{\mathcal{H}}\right)\inv
(\mathcal{K}) \\
&\subseteq \mathcal{M}(\hat{X}, A, J, \mathcal{H}(\tilde{X}))\text{,}
\end{align*}
where $\pi^{\mathcal{M}}_{\mathcal{H}} \definedas \pi^{\mathcal{B}}_{\mathcal{H}}|_{\mathcal{M}(\hat{X}, A, J, \mathcal{H}(\tilde{X}))}$ and analogously $\mathcal{M}_b(\hat{X}, A, J, \mathcal{K})$ for $b\in B$ and $\mathcal{M}_U(\hat{X}, A, J, \mathcal{K})$ for $U\subseteq B$ open. \\
Furthermore, given any open subset $V\subseteq \tilde{X}$, define
\[
\mathcal{H}^V(\tilde{X})
\]
to be the closure of the set of those $H \in \mathcal{H}(\tilde{X})$ that have support in $V$ and
\begin{align*}
\mathcal{M}^V(\hat{X}, A, J, \mathcal{K}) \definedas \{ u\in \mathcal{M}_b(\hat{X}, A, J, \mathcal{K}) \;|\; & u(S_{b,i}) \cap \tilde{\iota}\inv(V) \neq \emptyset \text{ for every} \\
& \text{connected component $S_{b,i}$ of $S_b$} \}\text{,}
\end{align*}
where $\tilde{\iota} : \hat{X} \to \tilde{X}$ is the canonical map and analogously $\mathcal{M}^V_b(\hat{X}, A, J, \mathcal{K})$ and $\mathcal{M}^V_U(\hat{X}, A, J, \mathcal{K})$.
\end{defn}

\begin{lemma}\label{Lemma_Ddbar_universal}
In the notation of the above construction,
\[
\dbar^{J,\mathcal{H}}_U : \mathcal{B}^{k,p}_{U,\phi_a}(\hat{X}, A, J, \mathcal{H}(\tilde{X})) \to \mathcal{E}^{k-1,p}_{U,\phi_a}(\hat{X}, A, J, \mathcal{H}(\tilde{X}))
\]
is split transverse to the zero section and $\mathcal{M}_{U}(\hat{X},A,J,\mathcal{H}(\tilde{X}))$ is a split Banach submanifold of $\mathcal{B}^{k,p}_{U,\phi_a}(\hat{X},A,J,\mathcal{H}(\tilde{X}))$. \\
Furthermore, with respect to this Banach manifold structure, $\pi^{\mathcal{M}}_{\mathcal{H}} : \mathcal{M}_U(\hat{X},A,J,\mathcal{H}(\tilde{X})) \to \mathcal{H}(\tilde{X})$ is a Fredholm map of index
\[
\ind(\pi^{\mathcal{M}}_{\mathcal{H}}) = \dim_\C(X)\hat{\chi} + 2c_1(A)
+ \dim_\R(U)\text{.}
\]
Given an open subset $V \subset \tilde{X}$, for any $H\in\mathcal{H}(\tilde{X})$,
\[
\mathcal{M}^V_U(\hat{X}, A, J, H + \mathcal{H}^V(\tilde{X}))
\]
inherits a Banach manifold structure from $\mathcal{M}_U(\hat{X}, A, J, \mathcal{H}(\tilde{X}))$ \st the projection onto $H + \mathcal{H}^V(\tilde{X})$ is a Fredholm map of the same index as before.
\end{lemma}
\begin{proof}
Lemma A.3.6 in \cite{MR2045629}, and the previous Lemma together with the implicit function theorem.
\end{proof}

The set $\mathcal{M}_U(\hat{X},A,J,\mathcal{H}(\tilde{X}))$ equipped with the Banach manifold structure from the previous lemma, which a priori does depend on $k$, $p$ and $\phi_a$, will be denoted by
\[
\mathcal{M}^{k,p}_{U,\phi_a}(\hat{X},A,J,\mathcal{H}(\tilde{X}))\text{.}
\]
The goal now is to show that the Banach manifold structure on $\mathcal{M}^{k,p}_{U,\phi_a}(\hat{X}, A, J, \mathcal{H}(\tilde{X}))$ does not depend on the choice of $k\geq 1$ and $p > 1$ with $kp > 2$ nor on $\phi_a : U\times S_a \cong S|_U$. Hence writing $\mathcal{M}_U(\hat{X}, A, J, \mathcal{H}(\tilde{X}))$ makes sense, and consequently given any trivialisation $(U_i, \phi_{a_i})_{i\in I}$ of $S$, the Banach manifolds $\mathcal{M}_U(\hat{X}, A, J, \mathcal{H}(\tilde{X}))$ patch together to a Banach manifold structure on $\mathcal{M}(\hat{X}, A, J, \mathcal{H}(\tilde{X}))$. To sum the argument up in two words: Elliptic regularity.

\begin{lemma}\label{Lemma_elliptic_regularity_II}
Let $k,\ell\in \N$, $1 < p,q < \infty$ with $kp, \ell q > 2$ and assume that $k > \ell$ and $k - \frac{2}{p} > \ell - \frac{2}{q}$.
Then the inclusion $\mathcal{B}^{k,p}_{U,\phi_a}(\hat{X},A,J,\mathcal{H}(\tilde{X})) \to \mathcal{B}^{\ell,q}_{U,\phi_a}(\hat{X},A,J,\mathcal{H}(\tilde{X}))$ defined via the Sobolev embedding theorem induces a diffeomorphism $\mathcal{M}^{k,p}_{U,\phi_a}(\hat{X},A,J,\mathcal{H}(\tilde{X})) \cong \mathcal{M}^{\ell,q}_{U,\phi_a}(\hat{X},A,J,\mathcal{H}(\tilde{X}))$.
\end{lemma}
\begin{proof}
By Lemma \ref{Lemma_elliptic_regularity_I}, one has the set-theoretic identity $\mathcal{M}^{k,p}_{U,\phi_a}(\hat{X},A,J,\mathcal{H}(\tilde{X})) = \mathcal{M}^{\ell,q}_{U,\phi_a}(\hat{X},A,J,\mathcal{H}(\tilde{X}))$. \\
Now go through the proof of the implicit function theorem, by which the smooth structures on the above spaces are defined, to show that this identity mapping is also a diffeomorphism. \\
For more details on this see either the proof of Lemma II.24, p.~78, in \cite{ediss15878}, or also Section 3.1, esp.~Lemma 3.1.4, in \cite{1208.1340}.
\end{proof}

\begin{lemma}
Using the notation of Construction \ref{Construction_Universal_Dbar}, let $\phi_a, \psi_a : U\times S_a \overset{\cong}{\longrightarrow} S|_U$ be two smooth trivialisations and let $r\in\N$.
Then the set-theoretic inclusion
\[
\mathcal{B}^{k+r,p}_{U,\phi_a}(\hat{X}, A, J, \mathcal{H}(\tilde{X})) \hookrightarrow \mathcal{B}^{k,p}_{U,\psi_a}(\hat{X}, A, J, \mathcal{H}(\tilde{X}))
\]
is a map of class $C^{r-1}$.
\end{lemma}
\begin{proof}
Let $\rho \definedas \pr_2\circ \psi_a\inv\circ \phi_a : U\times S_a \to S_a$.
In other words, $\rho$ is a family $\rho_b : S_a \to S_a$, $b\in U$, of diffeomorphisms of $S_a$.
Fix any $H\in\mathcal{H}(\tilde{X})$.
Then in the trivialisations $\mathcal{B}^{k+r,p}_{U,\phi_a}(\hat{X}, A, J, \mathcal{H}(\tilde{X})) \cong U\times \mathcal{B}_a^{k+r,p}(\hat{X}, A, J, H)\times \mathcal{H}(\tilde{X})$ and $\mathcal{B}^{k,p}_{U,\psi_a}(\hat{X}, A, J, \mathcal{H}(\tilde{X})) \cong U\times \mathcal{B}_a^{k,p}(\hat{X}, A, J, H)\times \mathcal{H}(\tilde{X})$ defining their smooth structures, the coordinate expression for the inclusion is the map
\begin{align*}
U\times \mathcal{B}_a^{k+r,p}(\hat{X}, A, J, H)\times \mathcal{H}(\tilde{X}) &\to U\times \mathcal{B}_a^{k,p}(\hat{X}, A, J, H)\times \mathcal{H}(\tilde{X}) \\
(b, u, H) &\mapsto (b, u\circ \rho_b, H)\text{.}
\end{align*}
The only question about differentiability of this map arises from the middle component, the map
\begin{align*}
\Psi : U\times \mathcal{B}_a^{k+r,p}(\hat{X}, A, J, H) &\to \mathcal{B}_a^{k,p}(\hat{X}, A, J, H) \\
(b, u) &\mapsto u\circ \rho_b\text{.}
\end{align*}
Fix a point $(b, u) \in U\times \mathcal{B}_a^{k+r,p}(\hat{X}, A, J, H)$ with $u$ of class $C^{k+r}$.
We want to express $\Psi$ in coordinates around $(b,u)$ and $\Psi(b,u) = u\circ \rho_b$.
First, assume that $U$ is an open subset of some $\R^d$.
Then the coordinate expression $\tilde{\Psi} : U\times L^{k+r,p}(u^\ast V\hat{X}_a) \to L^{k,p}((u\circ \rho_b)^\ast V\hat{X}_a)$ of $\Psi$ is given by the string of maps
\[
(b', \xi) \mapsto (b',\exp_u^\perp(\xi)) \mapsto \exp_u^\perp(\xi)\circ \rho_{b'} \mapsto (\exp^\perp_{u\circ \rho_b})\inv(\exp_u^\perp(\xi)\circ \rho_{b'})\text{.}
\]
For simplicity from now on I will drop the subscript $a$ on $S_a$ and consequently $\hat{X}_a$ and denote by $S$ the Riemann surface $S_a$ and by $\hat{X}$ the trivial fibre bundle $S\times X$ over $S$ with fibres $\hat{X}_z \cong X$ at the points $z\in S$. Then the above formula can be evaluated at some point $z\in S$ and the definition of $\exp^\perp$ for the fibre bundle $\hat{X}$ can be inserted to give
\[
\tilde{\Psi}(b', \xi)(z) = \bigl(\exp^{\hat{X}_z}_{u\circ \rho_b(z)}\bigr)\inv \left(\exp^{\hat{X}_{\rho_{b'}(z)}}_{u\circ \rho_{b'}(z)}(\xi\circ \rho_{b'}(z))\right)\text{.}
\]
First, note that the right hand side is well-defined for $\|\xi\|_{L^{1,p}(u^\ast V\hat{X})}$ small enough, independent of $b'$, because by compactness, $\sup \{\inj(\hat{X}_z) \;|\; z\in S\}$ is finite and an $L^{1,p}$-bound on $\xi$ implies a pointwise bound by the Sobolev embedding theorem.
Second, this can be written as
\begin{multline*}
\bigl(\exp^{\hat{X}_z}_{u\circ \rho_b(z)}\bigr)\inv \left(\exp^{\hat{X}_{\rho_{b'}(z)}}_{u\circ \rho_{b'}(z)}(\xi\circ \rho_{b'}(z))\right) = \\
\underbrace{\left(\bigl(\exp^{\hat{X}_z}_{u\circ \rho_b(z)}\bigr)\inv\circ \exp^{\hat{X}_{\rho_{b'}(z)}}_{u\circ \rho_{b}(z)}\right)}_{\text{($\ast$)}}\circ \underbrace{\left(\bigl(\exp^{\hat{X}_{\rho_{b'}(z)}}_{u\circ \rho_{b}(z)}\bigr)\inv \circ \exp^{\hat{X}_{\rho_{b'}(z)}}_{u\circ \rho_{b'}(z)}\right)}_{\text{($\ast\ast$)}}(\xi\circ \rho_{b'}(z))\text{,}
\end{multline*}
which can be interpreted as follows: Over $U\times S$, consider the two fibre bundles $\rho^\ast \hat{X}$ and $\pr_2^\ast \hat{X}$, where $\pr_2 : U\times S \to S$ is the projection.
Both of these bundles are canonically identified with the trivial one, but carry two different structures of Riemannian submersion.
Furthermore $u$ is a section of $\hat{X}$, and so is $u\circ \rho_b$.
Hence $\rho^\ast u$ and $\pr_2^\ast (u\circ \rho_b)$ are sections of $\rho^\ast \hat{X}$ and $\pr_2^\ast \hat{X}$, respectively, and $\rho^\ast \xi$ is a section of $V\rho^\ast \hat{X} = \rho^\ast V\hat{X}$ (along $\rho^\ast u$).
Then the first term ($\ast$) above is the coordinate expression for the identification $L^{k+r,p}(\pr_2^\ast \hat{X}) \cong L^{k+r,p}(\rho^\ast \hat{X})$ induced by the canonical identification of $\pr_2^\ast \hat{X} \cong \rho^\ast \hat{X}$ in charts around the section $\pr_2^\ast (u\circ \rho_b)$, whereas the second one ($\ast\ast$) is the usual coordinate transformation on $L^{k,p}(\rho^\ast \hat{X})$ from the chart around $\rho^\ast u$ to the chart around $\pr_2^\ast (u\circ \rho_b)$.
So the above map $\tilde{\Psi}$ can be interpreted as mapping $\xi$ to $\rho^\ast \xi\in L^{k+r,p}((\rho^\ast u)^\ast V\rho^\ast \hat{X})$, then applying the two coordinate transformations above and finally restricting to the slice $\{b\}\times \hat{X} \subseteq \pr_2^\ast \hat{X}$.
A derivative of $\tilde{\Psi}(b',\xi)$ in the first variable $b'$ then corresponds to a covariant derivative of $\rho^\ast \xi$ in a direction tangent to the first factor of $U\times S$.
The maps ($\ast$) and ($\ast\ast$) have bounded derivatives of all orders after restricting to $V\times S$, where $V \subseteq U$ is a precompact open subset of $U$. \\
Now $\nabla^s \rho^\ast\xi$ can be expressed (by the Leibniz rule, basically) as a linear combination of $\xi, \dots, \nabla^s\xi$ with coefficients depending on the $s$-jet of $\rho$.
Again after restricting to a precompact subset $V\subseteq U$, these coefficients can be bounded.
Combining the above, at least for $\xi\in\Gamma^s(u^\ast V\hat{X})$ and $b'\in V$, via these pointwise estimates one can estimate $\|(D^s\tilde{\Psi})(b',\xi)\|_{L^{k,p}} \leq \sum_{j=0}^s c_j \|\nabla^j\xi\|_{L^{k,p}} \leq c \|\xi\|_{L^{k+s,p}}$ for some constants $c_j,c$.
Applying the usual density argument, which causes the loss of one derivative (hence it says $C^{r-1}$, not $C^r$, in the statement), shows the lemma.
\end{proof}

\begin{corollary}\label{Corollary_Smooth_structure_on_M}
The set-theoretic identity defines a diffeomorphism
\[
\mathcal{M}_{U,\phi_a}(\hat{X}, A, J, \mathcal{H}(\tilde{X})) \overset{\cong}{\to} \mathcal{M}_{U,\psi_a}(\hat{X}, A, J, \mathcal{H}(\tilde{X}))\text{.}
\]
In particular, any choice of covering $(U_i)_{i\in I}$ of the base $B$ of $S$ and trivialisations $(\phi_i : U_i\times S_{a_i} \overset{\cong}{\longrightarrow} S|_{U_i})_{i\in I}$ defines a cocycle for a Banach manifold structure on $\mathcal{M}(\hat{X}, A, J, \mathcal{H}(\tilde{X}))$ independent of these choices. \\
If $\mathcal{C}$ is any other Banach manifold and $f : \mathcal{B}^{k_0,p}(\hat{X}, A, J, \mathcal{H}(\tilde{X})) \to \mathcal{C}$, for some $k_0\in\N$ and $p > 1$ with $k_0p > 2$, a map with the property that there exists an $r\in\Z$, $r \leq k_0$ \st $f|_{\mathcal{B}^{k,p}(\hat{X}, A, J, \mathcal{H}(\tilde{X}))} : \mathcal{B}^{k,p}(\hat{X}, A, J, \mathcal{H}(\tilde{X})) \to \mathcal{C}$ is of class $C^{k-r}$ for every $k\geq k_0$, then $f|_{\mathcal{M}(\hat{X}, A, J, \mathcal{H}(\tilde{X}))} : \mathcal{M}(\hat{X}, A, J, \mathcal{H}(\tilde{X})) \to \mathcal{C}$ is smooth. \\
With respect to this Banach manifold structure,
\[
\pi^{\mathcal{M}}_{\mathcal{H}} : \mathcal{M}(\hat{X}, A, J, \mathcal{H}(\tilde{X})) \to \mathcal{H}(\tilde{X})
\]
is a Fredholm map of index
\[
\ind(\pi^{\mathcal{M}}_{\mathcal{H}}) = \dim_\C(X)\hat{\chi} + 2c_1(A) + \dim_\R(B)\text{.}
\]
Given an open subset $V\subseteq \tilde{X}$ and any $H\in \mathcal{H}(\tilde{X})$, the same holds for $\mathcal{M}^V(\hat{X}, A, J, H + \mathcal{H}^V(\tilde{X}))$ and the projection onto $H + \mathcal{H}^V(\tilde{X})$.
\end{corollary}
\begin{proof}
Immediate from the preceding three lemmas.
\end{proof}

\subsection{Evaluation maps and nodal families}\label{Subsection_Evaluation_maps}

Of interest are two kinds of evaluation maps: Evaluation at the marked points $\hat{R}^i$ and at the points corresponding to the nodes of $\Sigma|_B$.
While the former can be defined as maps on $\mathcal{M}(\hat{X}, A, J, \mathcal{H}(\tilde{X}))$, the latter can not. For the nodes only form a discrete subbundle of $\Sigma|_B$ or their desingularisations one of $S$.
The evaluations at these points are of importance since in the desingularisation $S$ of $\Sigma|_B$ all the nodes are resolved to pairs of points and hence the space $\mathcal{M}(\hat{X}, A, J, \mathcal{H}(\tilde{X}))$ contains ``too many'' curves in the sense that one is only interested in those which map each pair of points corresponding to a node to a single point. For only on this set does there exist an inclusion into $\mathcal{M}(\tilde{X}, A, J, \mathcal{H}(\tilde{X}))$.
But one can still choose a covering $(U_i)_{i\in I}$ and trivialisations $(\phi_i : U_i\times S_{a_i} \to S|_{U_i})_{i\in I}$ with the property that $\phi_i$ trivialises $N|_{U_i}$ as well, \ie after choosing some numbering $N^{i,1}_j(a_i), N^{i,2}_j(a_i)$, $j=1, \dots, d$, of $N_{a_i}$ \st $N^{i,1}_j(a_i)$ and $N^{i,2}_j(a_i)$ correspond to the same node and defining $N^{i,1}_j(b) \definedas \phi_i(b,N^{i,1}_1(a_i))$, $N^{i,2}_j(b) \definedas \phi_i(b,N^{i,2}_j(a_i))$, for $b\in U_i$, $j = 1, \dots, d$, one has $N_b = \{ N^{i,1}_j(b), N^{i,2}_j(b) \;|\; j = 1, \dots, d\}$. $N^{i,1}_j(b)$ and $N^{i,2}_j(b)$ here naturally are always supposed to correspond to the same node.
This allows the definition of evaluation maps (as always, $kp > 2$)
\begin{align*}
\ev^{N^{i,1},N^{i,2}} : \mathcal{B}^{k,p}_{U_i}(\hat{X}, A, J, \mathcal{H}(\tilde{X})) &\to (X\oplus X)^{\oplus d} \\
\mathcal{B}^{k,p}_b(\hat{X}, A, J, H) \ni u &\mapsto ((\pr_2(u(N^{i,1}_1(b))), \pr_2(u(N^{i,1}_1(b)))), \dots, \\
&\qquad (\pr_2(u(N^{i,1}_d(b))), \pr_2(u(N^{i,1}_d(b)))))\text{.}
\end{align*}
In contrast, the marked points $\hat{R}_1, \dots, \hat{R}_n : B\to S$ allow the definition of a globally defined evaluation map
\begin{align*}
\ev^{\hat{R}} : \mathcal{B}^{k,p}(\hat{X}, A, J, \mathcal{H}(\tilde{X})) &\to \hat{R}_1^\ast \hat{X} \oplus \cdots \oplus \hat{R}_n^\ast \hat{X} \\
\mathcal{B}^{k,p}_b(\hat{X}, A, J, H) \ni u &\mapsto (u(\hat{R}_1(b)), \dots, u(\hat{R}_n(b)))
\end{align*}
with a well-defined restriction to $\mathcal{M}(\hat{X}, A, J, \mathcal{H}(\tilde{X}))$. The target space of the above map is the fibre bundle over $B$ which is the Whitney sum of the fibre bundles $\hat{R}_i^\ast \hat{X}$. Writing $u \in \mathcal{B}^{k,p}_b(\hat{X}, A, J, H)$ in the form $z \mapsto (z, \overline{u}(z)) \in S_b\times X = \hat{X}_b$, then $\ev^{\hat{R}}(u) = (b, \overline{u}(\hat{R}_1(b)), \dots, \overline{u}(\hat{R}_n(b)))$.
As before for the $N^{i,1}_j, N^{i,2}_j$, assume that the $\phi_i$ preserve the markings in the sense that $\phi_i(b, \hat{R}_j(a_i)) = \hat{R}_j(b)$ for all $b\in U_i$, $i\in I$.
The reason for this is the following:
If $f : M\to N$ is a map between manifolds, and $\gamma : \R \to M$ is a path in $M$, then $\frac{\d}{\d t}f(\gamma(t)) = df(\dot{\gamma}(t))$ depends on the first derivative of $f$, and correspondingly for the higher derivatives.
If $f$ is of some Sobolev class, then this is only well defined by the Sobolev embedding theorem as long as $f$ has enough weak derivatives.
This problem is circumvented here, because with the choices of $\phi_i$ above, the markings and nodal points under the $\phi_i$ correspond to constant points on the $S_{a_i}$.
Hence the restriction $\ev^{\hat{R}} : \mathcal{B}^{k,p}_{U_i,\phi_i}(\hat{X}, A, J, \mathcal{H}(\tilde{X})) \to \hat{R}_1^\ast \hat{X} \oplus \cdots \oplus \hat{R}_n^\ast \hat{X}|_{U_i}$ is actually smooth \wrt to the smooth structure defined via $\phi_i$ and hence by the previous corollary
\[
\ev^{\hat{R}} : \mathcal{M}(\hat{X}, A, J, \mathcal{H}(\tilde{X})) \to \hat{R}_1^\ast \hat{X} \oplus \cdots \oplus \hat{R}_n^\ast \hat{X}
\]
is a smooth map.
Analogously, all the restrictions
\[
\ev^{N^{i,1},N^{i,2}} : \mathcal{M}_{U_i}(\hat{X}, A, J, \mathcal{H}(\tilde{X})) \to (X\oplus X)^{\oplus d}
\]
are smooth maps. \\
Letting $\Delta \definedas \{(x,x)\in X\oplus X \;|\; x\in X\}$, the space of holomorphic curves, at least over one of the $U_i$, is the preimage $\left(\ev^{N^{i,1},N^{i,2}}\right)\inv(\Delta^d)$, which is the space of those curves mapping each pair of points in a desingularisation corresponding to a node to a single point.
Furthermore, $\left(\ev^{N^{i,1},N^{i,2}}\right)\inv(\Delta^d)$ is independent of the choice of the $N^{i,1}_j$ and $N^{i,2}_j$, since any compatible reordering (in $j$ or switching $N^{i,1}_j$ and $N^{i,2}_j$ for a fixed $j$) leaves the set $\Delta^d$ invariant.
Hence there are well-defined sets
\[
\mathcal{M}_{U_i}(\tilde{X}|_B, A, J, \mathcal{H}(\tilde{X})) \definedas \left(\ev^{N^{i,1},N^{i,2}}\right)\inv(\Delta^d)
\]
which patch together to a well-defined set
\[
\mathcal{M}(\tilde{X}|_B, A, J, \mathcal{H}(\tilde{X})) \definedas \bigcup_{i\in I} \mathcal{M}_{U_i}(\tilde{X}|_B, A, J, \mathcal{H}(\tilde{X}))
\]
in the sense that for any $i,j\in I$ and for any $b\in U_i\cap U_j$, the sets of those points in $\mathcal{M}_{U_i}(\tilde{X}|_B, A, J, \mathcal{H}(\tilde{X}))$ and $\mathcal{M}_{U_j}(\tilde{X}|_B, A, J, \mathcal{H}(\tilde{X}))$ lying over $b$ coincide and furthermore, $\mathcal{M}(\tilde{X}|_B, A, J, \mathcal{H}(\tilde{X}))$ is independent of choices.
Given $V\subseteq \tilde{X}$ and $H\in \mathcal{H}(\tilde{X})$, there are analogously defined sets
\[
\mathcal{M}^V_{U_i}(\tilde{X}|_B, A, J, H + \mathcal{H}^V(\tilde{X})) \text{ and } \mathcal{M}^V(\tilde{X}|_B, A, J, H + \mathcal{H}^V(\tilde{X}))\text{.}
\]
Also, one can restrict $\ev^{\hat{R}}$ to the above subsets.
At this point it also makes sense to introduce what is mainly a change in notation.
Namely remember that $\hat{X}$ was the pullback of $\tilde{X}$ under the desingularisation $\hat{\iota} : S \to \Sigma$ of the restriction of the nodal family $\Sigma$ to the subset $\iota : B \to M$, where $M$ was the base of the family $\Sigma$.
Also, the markings $\hat{R}$ of $S$ were the pullbacks of the markings $R$ of $\Sigma$.
So one can canonically identify $\hat{R}_1^\ast \hat{X} \oplus \cdots \oplus \hat{R}_n^\ast \hat{X} \cong R_1^\ast \tilde{X} \oplus \cdots \oplus R_n^\ast \tilde{X}|_B$.
Note that because $\tilde{X}$ was defined to be the pullback $\pi^\ast X$ and because the $R_i$ are sections, every $R_i^\ast \tilde{X}$ is canonically identified with $X$.
But to distinguish the factors, the above notation is kept.
Using this, write
\[
\ev^{R} \definedas \ev^{\hat{R}}|_{\mathcal{M}(\tilde{X}|_B, A, J, \mathcal{H}(\tilde{X}))} : \mathcal{M}(\tilde{X}|_B, A, J, \mathcal{H}(\tilde{X})) \to R_1^\ast \tilde{X} \oplus \cdots \oplus R_n^\ast \tilde{X}|_B
\]

\begin{lemma}\label{Lemma_The_universal_moduli_space}
For any choice of $U_i$, $\phi_i$ and $N^{i,1}, N^{i,2}$ as above, the maps
\[
\ev^{N^{i,1}, N^{i,2}}\times \ev^{\hat{R}} : \mathcal{M}_{U_i}(\hat{X}, A, J, \mathcal{H}(\tilde{X})) \to (X^2)^d\times \hat{R}_1^\ast \hat{X} \oplus \cdots \oplus \hat{R}_n^\ast \hat{X}
\]
are submersions. \\
The sets $\mathcal{M}_{U_i}(\tilde{X}|_B, A, J, \mathcal{H}(\tilde{X}))$ are split submanifolds of $\mathcal{M}_{U_i}(\hat{X}, A, J, \mathcal{H}(\tilde{X}))$ that define a cocycle that equips $\mathcal{M}(\tilde{X}|_B, A, J, \mathcal{H}(\tilde{X}))$ with the structure of a split Banach submanifold of $\mathcal{M}(\hat{X},A,J,\mathcal{H}(\tilde{X}))$ of codimension $\dim_\R(X)\, d = \dim_\C(X)\,2d$. \\
Furthermore,
\[
\ev^R : \mathcal{M}(\tilde{X}|_B, A, J, \mathcal{H}(\tilde{X})) \to R_1^\ast \tilde{X}\oplus \cdots \oplus R_n^\ast \tilde{X}|_B
\]
is a submersion and in particular so are
\[
\pi^\mathcal{M}_B : \mathcal{M}(\tilde{X}|_B, A, J, \mathcal{H}(\tilde{X})) \to B\text{,}
\]
the composition of $\ev^R$ with the projection $R_1^\ast \tilde{X}\oplus\cdots\oplus R_n^\ast \tilde{X}|_B \to B$, and every
\[
\ev^R_i : \mathcal{M}(\tilde{X}|_B, A, J, \mathcal{H}(\tilde{X})) \to R_i^\ast \tilde{X}|_B
\]
for $1 \leq i \leq n$, the composition of $\ev^R$ with the projection $R_1^\ast \tilde{X}\oplus \cdots\oplus R_n^\ast\tilde{X}|_B \to R_i^\ast \tilde{X}|_B$. \\
Finally,
\[
\pi^{\mathcal{M}}_{\mathcal{H}} : \mathcal{M}(\tilde{X}|_B,A,J,\mathcal{H}(\tilde{X})) \to \mathcal{H}(\tilde{X})
\]
is a Fredholm map of index
\[
\ind(\pi^{\mathcal{M}}_{\mathcal{H}}) = \dim_\C(X)\chi +
2c_1(A) + \dim_\R(B)\text{.}
\]
Given an open subset $V \subseteq \tilde{X}$ and any $H\in \mathcal{H}(\tilde{X})$, the same statements hold with $\mathcal{M}_{U_i}(\hat{X}, A, J, \mathcal{H}(\tilde{X}))$ replaced by $\mathcal{M}_{U_i}^V(\hat{X}, A, J, H + \mathcal{H}^V(\tilde{X}))$, $\mathcal{M}(\tilde{X}|_B, A, J, \mathcal{H}(\tilde{X}))$ replaced by $\mathcal{M}^V(\tilde{X}|_B, A, J, H + \mathcal{H}^V(\tilde{X}))$ and $\mathcal{H}(\tilde{X})$ replaced by $H + \mathcal{H}^V(\tilde{X})$.
\end{lemma}
\begin{proof}
Lemma \ref{Lemma_Main_transversality_result}, the implicit function theorem and an easy index calculation.
\end{proof}

\begin{corollary}
For generic $H \in \mathcal{H}(\tilde{X})$, $\mathcal{M}(\tilde{X}|_B, A, J, H)$ is a manifold of dimension
\[
\dim \mathcal{M}(\tilde{X}|_B, A, J, H) = \dim_\C(X)\chi + 2c_1(A) +
\dim_\R(B)\text{.}
\]
\end{corollary}
\begin{proof}
Sard-Smale and Lemma \ref{Lemma_The_universal_moduli_space}.
\end{proof}

\section{Bubbling and the Gromov compactification}\label{Section_Compactification}

So far, a topology has only been defined on $\mathcal{M}(\tilde{X}|_B, A, J, \mathcal{H}(\tilde{X}))$, where $B\subseteq M$ is a locally closed submanifold over which there exists a desingularisation of $\Sigma$.
But even if $M$ has a well-defined stratification by signature, this does not define a well-behaved topology on all of $\mathcal{M}(\tilde{X}, A, J, \mathcal{H}(\tilde{X}))$.
Well-behaved here is to mean at least that the maps $\pi^\mathcal{M}_M : \mathcal{M}(\tilde{X}, A, J, \mathcal{H}(\tilde{X})) \to M$ and $\pi^\mathcal{M}_\mathcal{H} : \mathcal{M}(\tilde{X}, A, J, \mathcal{H}(\tilde{X})) \to \mathcal{H}(\tilde{X})$ are to be continuous. \\
Furthermore, to be able to apply the compactness results from \cite{MR1451624} and \cite{MR2026549}, this topology has to be chosen to be compatible in a sense to that of Deligne-Mumford convergence.
The relevant construction here can be found in the proof of Theorem 13.6 in \cite{MR2262197} (the direction (ii) $\Rightarrow$ (i)). The implication of this theorem can be stated as saying that the map $M \to \overline{M}_{g,n}$ from the base space of a marked nodal family of Riemann surfaces of type $(g,n)$ to the Deligne-Mumford space equipped with the topology of Deligne-Mumford convergence (as defined \eg in \cite{MR1451624} or \cite{MR2026549}) is continuous.
The topology on $\mathcal{M}(\tilde{X}, A, J, \mathcal{H}(\tilde{X}))$ will be described in terms of convergence of sequences as in Section 5.6 of \cite{MR2045629}.
To do so the following result from \cite{MR2262197} will be used, where still $(\pi : \Sigma \to M, R_\ast)$ is an arbitrary family of marked nodal Riemann surfaces of type $(g,n)$:
\begin{construction}\label{Construction_Deligne_Mumford_convergence}
Let $b\in M$. Then there exists a neighbourhood $U\subseteq M$ of $b$ with the following properties: Let $n_1, \dots, n_d \in \Sigma_b$ be the nodal points on $\Sigma_b$. For $i = 1, \dots ,d$ there are pairwise disjoint neighbourhoods $N_i \subseteq \Sigma$ of the $n_i$ with $\pi(N_i) = U$ and \st $R_j\cap N_i = \emptyset \;\forall\, j=1, \dots n, i = 1, \dots, d$ and holomorphic maps
\[
(x_i, y_i) : N_i \to \D^2, \quad  z_i : U \to \D, \quad t_i : U \to \D^{\dim_\C(M) - 1}
\]
\st
\begin{align*}
(z_i,t_i) : U &\to \D^{\dim_\C(M)} \\
(x_i,y_i, t_i\circ\pi|_{N_i}) : N_i &\to \D^{\dim_\C(\Sigma)}
\end{align*}
are holomorphic coordinate systems with
\[
(x_i,y_i)(n_i) = (0,0) \quad\text{and}\quad x_iy_i = z_i\circ \pi|_{N_i}\text{.}
\]
Denote for $b'\in U$ and $i = 1, \dots, d$
\begin{align*}
\Gamma_i(b') &\definedas \{z \in \Sigma_{b'} \cap N_i \;|\; |x_i(z)| = |y_i(z)| = \sqrt{|z_i(b')|}\}
\intertext{and}
\Gamma(b') &\definedas \bigcup_{i=1}^d \Gamma_i(b'), \quad \Gamma \definedas \bigcup_{b'\in U} \Gamma(b')\text{.}
\end{align*}
Then each $\Gamma(b')$ is a disjoint union of nodal points (one for each $i$ with $z_i(b') = 0$) and pairwise disjoint embedded circles (one for each $i$ with $z_i(b') \neq 0$) disjoint from all the nodal and marked points.
Especially $\Gamma_i(b) = n_i$ and hence $\Gamma(b) = \{n_1, \dots, n_d\}$.
Also, for every $b' \in U$ there exists a continuous map
\[
\psi_{b'} : \Sigma_{b'} \to \Sigma_b
\]
with the following properties:
\begin{itemize}
  \item $\psi_{b'}(\Gamma_i(b')) = n_i$ for all $i = 1, \dots, d$.
  \item $\psi_{b'}|_{\Sigma_{b'}\setminus \Gamma(b')} : \Sigma_{b'}\setminus \Gamma(b') \to \Sigma_b\setminus\{n_1, \dots, n_d\}$ is a diffeomorphism.
  \item The map
\begin{align*}
\psi : \Sigma|_U\setminus \Gamma &\to U\times (\Sigma_b\setminus \{n_1 \dots, n_d\}) \\
z &\mapsto (\pi(z), \psi_{\pi(z)}(z))
\end{align*}
is a diffeomorphism.
\end{itemize}
These maps have the property that if $(b_i)_{i\in \N} \subseteq U$ is a sequence converging to $b$, then the sequence $(j_i)_{i\in\N}$ of complex structures on $\Sigma_b \setminus \{n_1, \dots, n_d\}$ defined by $j_i \definedas \psi_{b_i, \ast} j_{b_i}$, where $j_{b_i}$ denotes the complex structure on $\Sigma_{b_i}\setminus \Gamma(b')$, converges in the $C^\infty$-topology to the restriction of $j_b$ to $\Sigma_b \setminus \{n_1, \dots, n_d\}$.
\end{construction}

With this one can define sequential convergence in $\mathcal{M}(\tilde{X}, A, J, \mathcal{H}(\tilde{X}))$.
\begin{defn}
Let $(u_i)_{i\in \N} \subseteq \mathcal{M}(\tilde{X}, A, J, \mathcal{H}(\tilde{X}))$ be a sequence.
Then $u_i$ converges to $u\in \mathcal{M}(\tilde{X}, A, J, \mathcal{H}(\tilde{X}))$ \iff the following hold:
\begin{itemize}
  \item Let $b_i \definedas \pi^{\mathcal{M}}_M(u_i)$ and $b \definedas \pi^\mathcal{M}_M(u)$. Then $b_i \overset{i\to \infty}{\longrightarrow} b$ in $M$.
  \item Let $H_i \definedas \pi^\mathcal{M}_\mathcal{H}(u_i)$ and $H \definedas \pi^\mathcal{M}_\mathcal{H}(u)$. Then $H_i \overset{i\to\infty}{\longrightarrow} H$ in $\mathcal{H}(\tilde{X})$.
  \item In the notation of Construction \ref{Construction_Deligne_Mumford_convergence}, for $b'\in U$, let $\phi_{b'} \definedas (\psi_{b'}|_{\Sigma_{b'}\setminus \Gamma(b')})\inv : \Sigma_b\setminus \{n_1, \dots, n_d\} \to \Sigma_{b'}$. Let $N\in \N$ be \st $b_i \in U$ for all $i\geq N$. Then $u_i\circ \phi_{b_i} : \Sigma_b \setminus \{n_1, \dots, n_d\} \to \tilde{X}$, for $i \geq N$, converges uniformly to $u|_{\Sigma_b \setminus \{n_1, \dots, n_d\}}$.
\end{itemize}
\end{defn}

Due to bubbling, see \cite{MR1451624} Section V.3, even for $M$ compact, the moduli space $\mathcal{M}(\tilde{X}, A, J, \mathcal{H}(\tilde{X}))$ will not be compact.
To remedy this situation and still get a compact moduli space, the Gromov compactification of $\mathcal{M}(\tilde{X}, A, J, \mathcal{H}(\tilde{X}))$ has to be introduced.
To describe this space, first of all assume that the marked nodal family of Riemann surfaces $(\pi : \Sigma \to M, R_\ast)$ is regular.
For $\ell \geq 0$, let $(\pi^\ell : \Sigma^\ell \to M^\ell, R^\ell_\ast, T^\ell_\ast)$, $\Sigma^\ell = M^{\ell+1}$, and $\hat{\pi}^{\ell-1} : \Sigma^\ell \to \Sigma^{\ell-1}$, $\hat{\pi}^\ell_k : \Sigma^\ell \to \Sigma^k$, $\pi^\ell_k : M^\ell \to M^k$ be the marked nodal families and maps from Lemma \ref{Lemma_Forget_marked_point} and Proposition \ref{Proposition_Sequence_of_branched_coverings}.
Also, for all $\ell \geq 1$, let $\sigma^\ell$ and $\hat{\sigma}^\ell$ be the actions of $\mathcal{S}_\ell$, by reordering the last $\ell$ marked points, on $M^\ell$ and $\Sigma^\ell$, from Proposition \ref{Proposition_Reordering_marked_points}. \\
Finally, for $\ell \geq 0$, let $\tilde{X}^\ell \definedas (\hat{\pi}^\ell_0)^\ast \tilde{X}$.

Then the following is proved in the monograph \cite{MR1451624}, Chapter V (esp.~Theorem 1.2, Theorem 3.3, Proposition 1.1 and the proofs of these results) as well as, in a generalised version, in \cite{MR2026549}.

\begin{proposition}\label{Proposition_Gromov_compactness_detailled}
Let $(u_i)_{i\in \N}$ be a sequence in $\mathcal{M}(\tilde{X}, A, J,
\mathcal{H}(\tilde{X}))$, $b_i \definedas \pi^\mathcal{M}_M(u_i)$, $H_i \definedas \pi^\mathcal{M}_\mathcal{H}(u_i)$, \st $b_i \underset{i\to \infty}{\longrightarrow} b$ for some $b\in M$ and $H_i \underset{i\to \infty}{\longrightarrow} H$ for some $H \in \mathcal{H}(\tilde{X})$.
Then there exist the following:
\begin{itemize}
  \item an integer $\ell \in \N_0$,
  \item a subsequence $(u_{i_j})_{j\in \N}$ of
    $(u_i)_{i\in\N}$,
  \item $\hat{b}_{i_j} \in \overset{\circ}{M}{}^\ell$ with $\pi^\ell_0(\hat{b}_{i_j}) = b_{i_j}$
  \item and an element $\hat{u} \in
    \mathcal{M}_{\hat{b}}(\tilde{X}^\ell, A, J,(\hat{\pi}^\ell_0)^\ast H)$,
\end{itemize}
for some $\hat{b} \in M^\ell$ with $\pi^\ell_0(\hat{b}) = b$ and
\[
(\hat{\pi}^\ell_0)^\ast_{\hat{b}_{i_j}} u_{i_j} \underset{j\to \infty}{\longrightarrow} \hat{u}\text{.}
\]
Furthermore, the $\hat{b}_{i_j}$, $\hat{b}$ and $\hat{u}$ can be chosen \st the following holds: Let $\Sigma^\ell_{i,\hat{b}}$ be a component of
$\Sigma^\ell_{\hat{b}}$ on which $\hat{\pi}^\ell_0$ is not a
homeomorphism, \ie either $\hat{\pi}^\ell_0(\Sigma^\ell_{i,\hat{b}}) =
\{n_i\}$, for some node $n_i \in \Sigma_b$ or $\hat{\pi}^\ell_0(\Sigma^\ell_{i,\hat{b}}) =
\{R_j(b)\}$ for some $j = 1, \dots, n$. Then
$\hat{u}|_{\Sigma^\ell_{i,\hat{b}}}$ has nonvanishing vertical
homology class and hence defines a nonconstant $J$-holomorphic sphere
in $\tilde{X}_{n_i}$ or $\tilde{X}_{R_j(b)}$.
\end{proposition}

\begin{defn}\label{Definition_Equivalence_Relation_Gromov_Compactness}
Let 
\begin{align*}
(\hat{\pi}^\ell_0)^\ast \mathcal{H}(\tilde{X}) &\definedas
\{(\hat{\pi}^\ell_0)^\ast H \;|\; H\in \mathcal{H}(\tilde{X})\}
\\
&\subseteq \mathcal{H}(\tilde{X}^\ell)
\intertext{and}
\mathcal{M}(\tilde{X}^\ell, A, J, 
(\hat{\pi}^\ell_0)^\ast \mathcal{H}(\tilde{X})) &\definedas
(\pi^{\mathcal{M}}_{\mathcal{H}})\inv((\hat{\pi}^\ell_0)^\ast
\mathcal{H}(\tilde{X})) \\
&\subseteq
\mathcal{M}(\tilde{X}^\ell, A, J,
\mathcal{H}(\tilde{X}^\ell))\text{.}
\end{align*}
Then for any $\ell,\tilde{\ell} \in \N_0$ with $\ell \leq \tilde{\ell}$ there is a canonical map
\[
(\hat{\pi}^{\tilde{\ell}}_\ell)^\ast : (\pi^{\tilde{\ell}}_\ell)^\ast \mathcal{M}(\tilde{X}^\ell, A, J, (\hat{\pi}^\ell_0)^\ast \mathcal{H}(\tilde{X})) \to \mathcal{M}(\tilde{X}^{\tilde{\ell}}, A, J, (\hat{\pi}^{\tilde{\ell}}_0)^\ast \mathcal{H}(\tilde{X}))
\]
of topological spaces, where
\[
(\pi^{\tilde{\ell}}_\ell)^\ast \mathcal{M}(\tilde{X}^\ell, A, J, (\hat{\pi}^\ell_0)^\ast \mathcal{H}(\tilde{X})) = \mathcal{M}(\tilde{X}^\ell, A, J, (\hat{\pi}^\ell_0)^\ast \mathcal{H}(\tilde{X})) \times_{\pi^\mathcal{M}_M, M^\ell, \pi^{\tilde{\ell}}_\ell} M^{\tilde{\ell}}
\]
is the fibred product of topological spaces. \\
Furthermore, for all $\ell \geq 1$, the actions $\sigma^\ell$ and $\hat{\sigma}^\ell$ of $\mathcal{S}_\ell$ on $M^\ell$ and $\Sigma^\ell$, respectively, induce actions
\[
\tilde{\sigma}^\ell : \mathcal{S}_\ell \times \mathcal{M}(\tilde{X}^\ell, A, J, (\hat{\pi}^\ell_0)^\ast\mathcal{H}(\tilde{X})) \to \mathcal{M}(\tilde{X}^\ell, A, J, (\hat{\pi}^\ell_0)^\ast\mathcal{H}(\tilde{X}))\text{,}
\]
compatible via $\pi^\mathcal{M}_M$ with the actions $\sigma^\ell$ in the obvious way. \\
Together, the spaces $\mathcal{M}(\tilde{X}^\ell, A, J, 
(\hat{\pi}^\ell_0)^\ast \mathcal{H}(\tilde{X}))$, maps $(\hat{\pi}^{\tilde{\ell}}_\ell)^\ast$ and actions $\tilde{\sigma}^\ell$ form a system of topological spaces, whose colimit is called the \emph{Gromov compactification of $\mathcal{M}(\tilde{X}, A, J, \mathcal{H}(\tilde{X}))$} and denoted by
\[
\overline{\mathcal{M}}(\tilde{X}, A, J, \mathcal{H}(\tilde{X}))\text{.}
\]
It is equipped with canonical maps
\begin{align*}
\pi^{\overline{\mathcal{M}}}_M : \overline{\mathcal{M}}(\tilde{X}, A, J, \mathcal{H}(\tilde{X})) &\to M
\intertext{and}
\pi^{\overline{\mathcal{M}}}_{\mathcal{H}} : \overline{\mathcal{M}}(\tilde{X}, A, J, \mathcal{H}(\tilde{X})) &\to \mathcal{H}(\tilde{X})\text{.}
\end{align*}
\end{defn}
\begin{remark}
For $u\in \mathcal{M}(\tilde{X}^\ell, A, J,
(\hat{\pi}^\ell_0)^\ast\mathcal{H}(\tilde{X}))$ with $\pi^\mathcal{M}_M(u) = b$, $\pi^\mathcal{M}_\mathcal{H}(u) = H$ and $\hat{b}\in M^{\tilde{\ell}}$ \st $b = \pi^{\tilde{\ell}}_\ell(\hat{b})$,
\[
(\hat{\pi}^{\tilde{\ell}}_\ell)^\ast u \in \mathcal{M}_{\hat{b}}(\tilde{X}^{\tilde{\ell}}, A, J,
(\hat{\pi}^{\tilde{\ell}}_\ell)^\ast H) \subseteq 
\mathcal{M}(\tilde{X}^{\tilde{\ell}}, A, J,
(\hat{\pi}^{\tilde{\ell}}_0)^\ast \mathcal{H}(\tilde{X}))\text{.}
\]
For $\tilde{\ell} = \ell+1$ this is clear, for on every component of $\Sigma^{\tilde{\ell}}_{\hat{b}}$, $\hat{\pi}^{\tilde{\ell}}_\ell$ is either a diffeomorphism or a constant map onto a point.
On each component on which $\hat{\pi}^{\tilde{\ell}}_\ell$ is constant (which then is diffeomorphic to a sphere), $(\hat{\pi}^{\tilde{\ell}}_\ell)^\ast H$ vanishes, so the restriction of a section $u \in \mathcal{M}(\tilde{X}^{\tilde{\ell}}, A, J, (\hat{\pi}^{\tilde{\ell}}_\ell)^\ast H)$ to such a component is just given by a $J$-holomorphic map to $X_b$.
In particular, the constant map corresponding to the restriction of $(\hat{\pi}^{\tilde{\ell}}_\ell)^\ast u$ to such a component is holomorphic.
For $\ell \geq 1$, the claim follows by induction.
\end{remark}
\begin{remark}
Note that in the above definition, $(\hat{\pi}^\ell_0)^\ast :
\mathcal{H}(\tilde{X}) \to
\mathcal{H}(\tilde{X}^\ell)$ is an injection.
\end{remark}
\begin{remark}\label{Remark_Gromov_compactification_explicit}
The colimit over the above system of topological spaces is the quotient space
\[
\coprod_{\ell \geq 0} \mathcal{M}(\tilde{X}^\ell, A, J, (\hat{\pi}^\ell_0)^\ast \mathcal{H}(\tilde{X}))/_\sim\text{,}
\]
where
\[
\mathcal{M}(\tilde{X}^{\ell'}, A, J, 
(\hat{\pi}^{\ell'}_0)^\ast \mathcal{H}(\tilde{X}))\ni u' \sim u'' \in \mathcal{M}(\tilde{X}^{\ell''}, A, J, 
(\hat{\pi}^{\ell''}_0)^\ast \mathcal{H}(\tilde{X}))
\]
\iff there exists an $\tilde{\ell} \geq \ell',\ell''$, an $g\in \mathcal{S}_{\tilde{\ell}}$ and $b\in M^{\tilde{\ell}}$ \st $\pi^{\tilde{\ell}}_{\ell'}(b) = \pi^\mathcal{M}_M(u')$, $\pi^{\tilde{\ell}}_{\ell''}(\sigma^{\tilde{\ell}}_{g\inv}(b)) = \pi^\mathcal{M}_M(u'')$ and
\[
(\hat{\pi}^{\tilde{\ell}}_{\ell'})_b^\ast u' = \tilde{\sigma}^{\tilde{\ell}}_g\left((\hat{\pi}^{\tilde{\ell}}_{\ell''})_{\sigma^{\tilde{\ell}}_{g\inv}(b)}^\ast u''\right) \in \mathcal{M}((\tilde{X}^{\tilde{\ell}}, A, J, 
(\hat{\pi}^{\tilde{\ell}}_0)^\ast \mathcal{H}(\tilde{X}))\text{.}
\]
In particular, this equivalence relation has the following property: If $u \in \mathcal{M}(\tilde{X}^\ell, A, J, (\hat{\pi}^\ell_0)^\ast \mathcal{H}(\tilde{X}))$ is constant on a ghost component (as in Definition \ref{Definition_ghost_component}) of its underlying nodal Riemann surface, then there exists a $k < \ell$ and a $u' \in \mathcal{M}(\tilde{X}^k, A, J, (\hat{\pi}^k_0)^\ast \mathcal{H}(\tilde{X}))$, \st $u$ and $u'$ define the same point in $\overline{\mathcal{M}}(\tilde{X}, A, J, \mathcal{H}(\tilde{X}))$.
\end{remark}
\begin{remark}\label{Remark_M_included_in_Mbar}
Also, directly from the definition, there is a canonical injection
\[
\mathcal{M}(\tilde{X}, A, J, \mathcal{H}(\tilde{X})) \hookrightarrow \overline{\mathcal{M}}(\tilde{X}, A, J, \mathcal{H}(\tilde{X}))\text{.}
\]
\end{remark}

\begin{corollary}
\[
\pi^{\overline{\mathcal{M}}}_M \times \pi^{\overline{\mathcal{M}}}_{\mathcal{H}} : \overline{\mathcal{M}}(\tilde{X}, A, J, \mathcal{H}(\tilde{X})) \to M \times \mathcal{H}(\tilde{X})
\]
is a proper map and $\overline{\mathcal{M}}(\tilde{X}, A, J, \mathcal{H}(\tilde{X}))$ is a Hausdorff topological space. In particular, if $M$ is compact, then for any $H\in \mathcal{H}(\tilde{X})$,
\[
\overline{\mathcal{M}}(\tilde{X}, A, J, H) \definedas (\pi^{\overline{\mathcal{M}}}_M \times \pi^{\overline{\mathcal{M}}}_{\mathcal{H}})\inv (M\times \{H\})
\]
is a compact Hausdorff topological space.
\end{corollary}

\begin{lemma}
If $M$ is compact, then there exists an $\ell \in \N_0$ \st the canonical map $\mathcal{M}(\tilde{X}^\ell, A, J, (\hat{\pi}^\ell_0)^\ast \mathcal{H}(\tilde{X})) \to \overline{\mathcal{M}}(\tilde{X}, A, J, \mathcal{H}(\tilde{X}))$ is surjective.
\end{lemma}
\begin{proof}
By the definition of $\mathcal{M}(\tilde{X}^\ell, A, J, (\hat{\pi}^\ell_0)^\ast \mathcal{H}(\tilde{X}))$ and the definition of $\mathcal{H}(\tilde{X})$, there exists a universal bound on the vertical energy of every element of $\mathcal{M}(\tilde{X}^\ell, A, J, (\hat{\pi}^\ell_0)^\ast \mathcal{H}(\tilde{X}))$ independent of $\ell$, by Lemma 8.2.9 in \cite{MR2045629}, where the vertical energy is defined as in Section 8.2 in \cite{MR2045629}, p.~249.
By the usual Gromov-Schwarz and Monotonicity lemmas, this implies a universal bound on the number of components on which an element of $\mathcal{M}(\tilde{X}^\ell, A, J, (\hat{\pi}^\ell_0)^\ast \mathcal{H}(\tilde{X}))$ can be nonconstant, which by definition of the equivalence relation in the definition of $\overline{\mathcal{M}}(\tilde{X}, A, J, \mathcal{H}(\tilde{X}))$ implies the lemma.
\end{proof}

Furthermore, for $\ell \leq \tilde{\ell}$ denote by $M^{\tilde{\ell},\ell}$ the set
\[
M^{\tilde{\ell},\ell} \definedas \{ b\in M^{\tilde{\ell}} \;|\; \hat{\pi}^{\tilde{\ell}}_{\ell,b} : \Sigma^{\tilde{\ell}}_{b} \to \Sigma^\ell_{\pi^{\tilde{\ell}}_{0}(b)} \text{ is a homeomorphism}\}
\]
and by $\Sigma^{\tilde{\ell},\ell} \definedas \Sigma^{\tilde{\ell}}|_{M^{\tilde{\ell},\ell}}$.

In addition, let $\pi^{\tilde{\ell},\ell} \definedas \pi^{\tilde{\ell}}|_{M^{\tilde{\ell},\ell}} : M^{\tilde{\ell},\ell} \to M^\ell$ and $\hat{\pi}^{\tilde{\ell},\ell} \definedas \hat{\pi}^{\tilde{\ell}}_\ell|_{\Sigma^{\tilde{\ell},\ell}} : \Sigma^{\tilde{\ell},\ell} \to \Sigma^\ell$.

\begin{lemma}
$\pi^{\tilde{\ell},\ell} : M^{\tilde{\ell},\ell} \to M^\ell$ is a surjective submersion of complex fibre dimension $\tilde{\ell} - \ell$ and $\Sigma^{\tilde{\ell},\ell} \cong
(\pi^{\tilde{\ell},\ell})^\ast \Sigma^\ell$ via $\hat{\pi}^{\tilde{\ell},\ell}$. \\
Furthermore,
\[
\mathcal{M}((\hat{\pi}^{\tilde{\ell},\ell})^\ast \tilde{X}^\ell, A, J, (\hat{\pi}^{\tilde{\ell},\ell}_\ell)^\ast (\hat{\pi}^\ell_0)^\ast \mathcal{H}(\tilde{X})) \cong (\pi^{\tilde{\ell},\ell})^\ast \mathcal{M}(\tilde{X}^\ell, A, J, (\hat{\pi}^\ell_0)^\ast \mathcal{H}(\tilde{X}))
\]
via $(\hat{\pi}^{\tilde{\ell},\ell})^\ast$.
\end{lemma}
\begin{proof}
This again follows by induction from the case $\tilde{\ell} = \ell+1$.
But in this case $M^{\tilde{\ell}} = M^{\ell+1} = \Sigma^\ell$ and $M^{\tilde{\ell},\ell}$ by Lemma \ref{Lemma_Forget_marked_point} is the complement of the nodes and markings in $\Sigma^\ell$.
The restriction $\Sigma^{\tilde{\ell},\ell}$ of $\Sigma^{\tilde{\ell}}$ to this subset, from the proof of Lemma \ref{Lemma_Forget_marked_point}, is by definition the pullback of $\Sigma^\ell$ via $\pi^\ell$ and the restriction of $\hat{\pi}^{\tilde{\ell}}_\ell$ is by definition the canonical map covering $\pi^\ell$. \\
The second claim follows directly from the definitions.
\end{proof}

To make sense of the following remark, remember that by definition $\mathcal{M}((\hat{\pi}^\ell_0)\tilde{X}, A, J, \mathcal{K})$, for any subset $\mathcal{K} \subseteq \mathcal{H}((\hat{\pi}^\ell_0)\tilde{X})$, is a disjoint union of subsets that are mapped to the strata of $M^\ell$ in the stratification by signature under $\pi^\mathcal{M}_M$.
By abuse of language I will call these subsets strata, even though in general it is not claimed that they are (Banach) manifolds or form any kind of reasonable stratification.
Also, by the codimension of such a subset I will mean the codimension of the corresponding stratum in $M^\ell$.

The \emph{transversality problem} now can be formulated as follows: \\
Does there exist a (generic subset of) $H\in \mathcal{H}(\tilde{X})$ \st for every $\ell \geq 0$ (and for all generic $H$), $\mathcal{M}(\tilde{X}^\ell, A, J, (\hat{\pi}^{\ell}_0)^\ast H)$ is stratified by smooth manifolds as in the previous section, induced from the stratification by signature on $M^\ell$. And in such a way that $\overline{\mathcal{M}}(\tilde{X}, A, J, H)$ has a stratification by smooth manifolds, induced by the canonical maps $\mathcal{M}(\tilde{X}^\ell, A, J, (\hat{\pi}^{\ell}_0)^\ast H) \to \overline{\mathcal{M}}(\tilde{X}, A, J, H)$. So that the stratification in particular coincides with the one from before on $\mathcal{M}(\tilde{X}, A, J, H)$ under the inclusion from Remark \ref{Remark_M_included_in_Mbar}? Furthermore, there should be a top-dimensional stratum which coincides with the top-dimensional stratum in $\mathcal{M}(\tilde{X}, A, J, H)$, corresponding to the smooth curves, and the codimension of every other stratum should coincide with the codimension of the stratum in $\mathcal{M}(\tilde{X}^\ell, A, J, (\hat{\pi}^{\ell}_0)^\ast H)$ from which it arises.

In general, it is known that the answer to this question is no, for all the Hamiltonian perturbations of the form $(\hat{\pi}^\ell_0)^\ast H$ vanish on ghost components, so the Banach space $(\hat{\pi}^\ell_0)^\ast \mathcal{H}(\tilde{X})$ is ``too small'' to achieve the transversality results in Lemma \ref{Lemma_Main_transversality_result}. \\
One hence is faced with two conflicting aims: On the one hand one would like to enlarge the spaces of perturbations in the construction of the universal moduli spaces from $(\hat{\pi}^\ell_0)^\ast \mathcal{H}(\tilde{X}) \cong \mathcal{H}(\tilde{X})$ to $\mathcal{H}((\hat{\pi}^\ell_0)^\ast\tilde{X})$ to achieve transversality, on the other hand one needs to restrict to perturbations coming from $\mathcal{H}(\tilde{X})$ so that the equivalence relation is preserved and the conditions on the dimensions of the strata of the stratification one wants to construct have any chance of holding true.

The solution to this problem, first applied in the genus $0$ case in \cite{MR2399678} and which will be extended to the present situation in the rest of this text, can now roughly be described as follows (all these notions will be made precise later on): \\
For every $\ell \geq 0$ there exists a subset $\mathcal{K}^\ell \subseteq \mathcal{H}(\tilde{X}^\ell)$ \st $(\hat{\pi}^{\tilde{\ell}}_\ell)^\ast \mathcal{K}^\ell \subseteq \mathcal{K}^{\tilde{\ell}}$. \\
There also exists an $\ell \in \N_0$ and for every $\tilde{\ell}\geq \ell$ a subset $\mathcal{N}^{\tilde{\ell}}(\mathcal{K}^{\tilde{\ell}}) \subseteq \mathcal{M}(\tilde{X}^{\tilde{\ell}}, A, J, \mathcal{K}^{\tilde{\ell}})$ with $\pi^\mathcal{M}_M(\mathcal{N}^{\tilde{\ell}}(\mathcal{K}^{\tilde{\ell}})) \subseteq \overset{\circ}{M}{}^{\tilde{\ell}}$ (the part corresponding to smooth curves, as in Section \ref{Section_III.1}) \st the closure of $\mathcal{N}^{\tilde{\ell}}(\mathcal{K}^{\tilde{\ell}})$ in $\overline{\mathcal{M}}(\tilde{X}^{\tilde{\ell}}, A, J, \mathcal{K}^{\tilde{\ell}})$ lies in $\mathcal{M}(\tilde{X}^{\tilde{\ell}}, A, J, \mathcal{K}^{\tilde{\ell}})$. \\
Since $\overset{\circ}{M}{}^{\tilde{\ell}} \subseteq M^{\tilde{\ell},\ell}$ for all $\ell \leq \tilde{\ell}$, for every $H \in \mathcal{K}^0$ there is a well-defined map $(\hat{\pi}^{\tilde{\ell}}_0)_\ast : \mathcal{N}^{\tilde{\ell}}((\hat{\pi}^{\tilde{\ell}}_0)^\ast H) \to \mathcal{M}(\tilde{X}, A, J, H)$ (the left-hand side is defined in the obvious way) given by $u \mapsto ((\hat{\pi}^{\tilde{\ell}}_{0,b})^{-1})^\ast u$, where $\pi^\mathcal{M}_M(u) = b$. \\
Then for generic $H \in \mathcal{K}^0$ the above will be \st $\mathcal{N}^\ell((\hat{\pi}^\ell_0)^\ast H)$ is invariant under the $\mathcal{S}_\ell$-action and the map $(\hat{\pi}^\ell_0)_\ast$ is an $\ell!$-sheeted covering on the complement of a subset of codimension at least $2$ (see Lemma \ref{Lemma_ell_factorial_sheeted_covering}).

Roughly speaking, the $\mathcal{N}^{\tilde{\ell}}(\mathcal{K}^{\tilde{\ell}})$ will be defined as spaces of holomorphic sections that map the first $\ell$ additional marked points to a subbundle $\tilde{Y} \subseteq \tilde{X}$ with real codimension $2$ fibres and the sets $\mathcal{K}^{\tilde{\ell}}$ will be spaces of Hamiltonian perturbations satisfying a set of compatibility conditions with this subbundle.
Making these notions precise and showing the properties above will be pretty much the rest of this work.

\section{Hypersurfaces and tangency}\label{Section_Hypersurfaces}

Throughout this section, let $(\pi : \Sigma \to M, R)$ be a stable marked nodal family Riemann surfaces of type $(g,n)$ and denote their Euler characteristic by $\chi$.
Furthermore, let $(\kappa : X \to M, \omega)$ be a family of symplectic manifolds together with a family $(\kappa|_Y : Y \to M, \omega|_Y)$ of symplectic hypersurfaces in $X$.
Define $\tilde{\kappa} : \tilde{X} \to \Sigma$ as the pullback of $\kappa : X \to M$ to $\Sigma$ via $\pi$ and likewise for $\tilde{Y}$.
As before, $\mathcal{J}_\omega(X)$ is the set of $\omega$-compatible vertical almost complex structures on $X$, \ie bundle morphisms $J \in \End(VX)$ with $J^2 = -\id$ and \st $\omega(\cdot, J\cdot)$ defines a metric on $VX$.
In other words, for any $b\in M$, $J_b$ is a compatible almost complex structure on the symplectic manifold $(X_b, \omega_b)$.

To define the sets $\mathcal{K}^\ell$ from the previous subsection, almost complex structures and Hamiltonian perturbations compatible with the family of symplectic hypersurface $Y$ in the sense of \cite{MR1954264}, Definition 3.2 are needed.

\begin{defn}\label{Definition_J_ni}
The set of \emph{$Y$-compatible vertical almost complex structures} on $X$ is defined as
\[
  \mathcal{J}_\omega(X,Y) \definedas \{J\in \mathcal{J}_\omega(X) \;|\;
  J(VY) = VY\}\text{.}
\]
The set of \emph{normally integrable} $Y$-compatible almost complex structures on $X$ is defined as
\[
  \mathcal{J}_{\omega, \mathrm{ni}}(X,Y) \definedas
  \{J\in \mathcal{J}_\omega(X,Y) \;|\;
  \pi^{VX}_{VY^{\perp_\omega}}N_J(v,\xi) = 0 \;\forall v\in V_yY, \xi
  \in V_yY^{\perp_\omega}, y \in Y\}\text{,}
\]
where $N_J$ denotes the Nijenhuis tensor of $J$, $V_yY^{\perp_\omega}
\subseteq V_yX$ denotes the symplectic orthogonal complement and
$\pi^{VX}_{VY^{\perp_\omega}} : VX \to VY^{\perp_\omega}$ denotes the
projection along $VY$.
\end{defn}

One considers $\mathcal{J}_\omega(X,Y)$ and $\mathcal{J}_{\omega,\mathrm{ni}}(X,Y)$ as subsets of $\mathcal{J}_\omega(\tilde{X},\tilde{Y})$ and $\mathcal{J}_{\omega,\mathrm{ni}}(\tilde{X},\tilde{Y})$, respectively, via pullback.

The proof of Theorem A.2 in \cite{MR1954264} shows:
\begin{lemma}\label{Lemma_J_ni_nonempty_connected}
$\mathcal{J}_{\omega, \mathrm{ni}}(X,Y)$ is nonempty and path-connected.
\end{lemma}

Now remember that if for $b\in M$, $S_b$ is a smooth Riemann surface and $\iota_b : S_b \to \Sigma_b \subseteq \Sigma$ a desingularisation of the fibre of $\Sigma$ over $b$, then $\iota_b^\ast \tilde{X} = (\pi\circ \iota_b)^\ast X$ is a trivial bundle, for $\pi\circ \iota_b$ is the constant map to $b$.
Likewise for the subbundle $Y\subseteq X$.
Making the identification with the trivial bundle, $\hat{X}_b \definedas S_b\times X_b$ and $\hat{Y}_b \definedas S_b\times Y_b$, one can pull back any $H\in \mathcal{H}(\tilde{X})$ to $H_b \in \mathcal{H}(\hat{X}_b)$. Given such $H$ and any $J\in \mathcal{J}_\omega(X)$, which induces a vertical almost complex structure on every $\hat{X}_b$, one hence gets an almost complex structure $\hat{J}^H_b$ on $\hat{X}_b$ as in Definition \ref{Definition_J_H}.

\begin{defn}\label{Definition_Y_compatible_Hamiltonian_perturbation}
Let $H\in \mathcal{H}(\tilde{X})$. $H$ is called a \emph{$Y$-compatible Hamiltonian perturbation}, $H\in \mathcal{H}(\tilde{X},\tilde{Y})$, if for every $b\in M$ and every desingularisation $\iota_b : S_b \to \Sigma_b\subseteq \Sigma$, $\hat{Y}_b \subseteq \hat{X}_b$ is $H_b$-parallel, \ie $\im X_{H_b(\zeta)}|_{\hat{Y}_b} \subseteq V\hat{Y}_b\;\forall\, \zeta\in TS_b$. \\
Given $J\in \mathcal{J}_{\omega, \mathrm{ni}}(X,Y)$, if furthermore for every $b\in M$ and every desingularisation $\iota_b : S_b \to \Sigma_b\subseteq \Sigma$,
\[
\pi^{V\hat{X}_b}_{V\hat{Y}_b^{\perp_\omega}} N_{\hat{J}^H_b}(\hat{v},
\hat{\xi}) = 0 \;\forall\, \hat{v}\in V_{\hat{y}}\hat{Y}_b, \hat{\xi} \in V_{\hat{y}}\hat{Y}^{\perp_\omega}, \hat{y}\in \hat{Y}_b\text{,}
\]
where $V_{\hat{y}}\hat{Y}^{\perp_\omega} \definedas \{0\}\times T Y_b^{\perp_\omega}$, then $H$ is called a \emph{$J$-compatible normally integrable Hamiltonian perturbation}, $H\in \mathcal{H}_{\mathrm{ni}}(\tilde{X}, \tilde{Y}, J)$. \\
This space has the two subspaces
\begin{align*}
\mathcal{H}^0_{\mathrm{ni}}(\tilde{X}, \tilde{Y}, J) &\definedas \{ H\in \mathcal{H}_{\mathrm{ni}}(\tilde{X}, \tilde{Y}, J) \;|\; H|_{\tilde{Y}} = 0\} \\
\mathcal{H}^{00}(\tilde{X}, \tilde{Y}) &\definedas \closure \left(\{H\in \mathcal{H}(\tilde{X}) \;|\; \supp(H)\subseteq \tilde{X} \setminus\tilde{Y} \text{ compact}\}\right)\text{,}
\end{align*}
where $\closure$ denotes the closure in $\mathcal{H}(\tilde{X})$.
\end{defn}

In the course of the ensuing construction, Hamiltonian perturbations will be chosen with increasing specialisation in the form $H + H^0 + H^{00}$, starting with some $H \in \mathcal{H}(\tilde{X}, \tilde{Y}, J)$ (which actually lies in some other yet to be defined subspace of $\mathcal{H}(\tilde{X}, \tilde{Y}, J)$) and then modifying it to $H + H^0$ for $H^0 \in \mathcal{H}_{\mathrm{ni}}^0(\tilde{X}, \tilde{Y}, J)$ and subsequently to $H + H^0 + H^{00}$ for some $H^{00}\in \mathcal{H}^{00}(\tilde{X}, \tilde{Y})$.

\begin{remark}
If one endows, for $b\in M$, $\hat{X}_b$ with a symplectic form
$\hat{\omega}_b$ that is of the form $\pr_1^\ast\sigma +
\pr_2^\ast\omega_b$ for a symplectic form $\sigma_b$ on $S_b$, then
$\{0\}\times T_yY_b^{\perp_\omega} = T_{\hat{y}} \hat{Y}_b^{\perp_{\hat{\omega}_b}}$ for $\hat{y} = (z,y)\in \hat{Y}_b$, so
the definition of $\mathcal{H}_{\mathrm{ni}}(\tilde{X}, \tilde{Y}, J)$
is in complete analogy to that of $\mathcal{J}_{\omega, \mathrm{ni}}(X,Y)$.
\end{remark}

\begin{lemma}\label{Lemma_H_X_Y}
\[
\mathcal{H}(\tilde{X},\tilde{Y}) = \{ H\in\mathcal{H}(\tilde{X}) \;|\; \d(H(\zeta))(v) = 0
\;\forall\, \zeta\in T\Sigma, v\in TY^{\perp_\omega} \}\text{.}
\]
defines a closed linear subspace of $\mathcal{H}(\tilde{X})$. \\
Given any $J\in \mathcal{J}_\omega(\tilde{X})$, for any $b\in M$ and every desingularisation $\iota_b : S_b \to \Sigma_b\subseteq \Sigma$, $\hat{Y}_b \subseteq \hat{X}_b$ is a $\hat{J}^H_b$-complex hypersurface.
\end{lemma}

The following is the reason for the above definition, which recovers
Lemma 3.3 from \cite{MR1954264} in the present notation:

\begin{corollary}\label{Corollary_Ddbar_complex_linear_II}
Let $J\in \mathcal{J}_{\omega, \mathrm{ni}}(X, Y)$, let $H \in
\mathcal{H}_{\mathrm{ni}}(\tilde{X}, \tilde{Y}, J)$, let
$b\in M$ and $\iota_b : S_b \to \Sigma_b\subseteq \Sigma$ a desingularisation. Then for any $u \in \mathcal{M}(\hat{X}_b, A, J_b, H_b)$ with $\im(u) \subseteq \hat{Y}_b$,
\[
\pi^{V\hat{X}_b}_{V\hat{Y}_b^{\perp_\omega}}\circ \left(D\dbar^{J_b, H_b}_{S_b}\right)_u : L^{k,p}(u^\ast V\hat{Y}_b^{\perp_\omega}) \to L^{k-1,p}(\overline{\Hom}_{(j_b, J_b)}(TS_b, u^\ast V\hat{Y}_b^{\perp_\omega}))
\]
is complex linear (for any $k\in \N, p > 1$ with $kp > 2$).
\end{corollary}
\begin{proof}
This is a special case of Lemma \ref{Corollary_Ddbar_complex_linear}.
\end{proof}

The following lemma and remark recover formulas (3.3) (b) and (c) from
\cite{MR1954264}, which will be used in Lemma
\ref{Lemma_Enough_ni_ham_perturbations} below showing the
existence of ``enough'' normally integrable Hamiltonian perturbations:
\begin{lemma}\label{Lemma_Characterisation_H_ni}
Let $J\in \mathcal{J}_\omega(X)$, $H\in
\mathcal{H}(\tilde{X})$ and assume that $\tilde{X} = \Sigma \times X$ is a trivial bundle. Then \wrt the decomposition
$T\tilde{X} = T\Sigma\times TX$, for $(w,v), (0,\xi)
\in T\tilde{X}$,
\[
N_{\tilde{J}^H}((w,v),(0,\xi)) = (0, N_J(v,\xi)) - (0,
2(L_{JX_{H(w)}^{0,1}}J)\xi)\text{.}
\]
In particular, for $J\in \mathcal{J}_{\omega, \mathrm{ni}}(X,Y)$,
\begin{align*}
\mathcal{H}_{\mathrm{ni}}(\tilde{X}, \tilde{Y}, J) = \{ H\in
\mathcal{H}(\tilde{X}, \tilde{Y}) \;|\; &
\pi^{TX}_{TY^{\perp_\omega}}(L_{JX_{H(w)}^{0,1}}J)\xi = 0 \\
& \;\forall\, w\in T\Sigma, \xi\in TY^{\perp_\omega}\}\text{.}
\end{align*}
Also, for $J\in \mathcal{J}_{\omega, \mathrm{ni}}(X,Y)$,
\begin{align*}
\pi^{TX}_{TY^{\perp_\omega}}(L_{JX_{H(w)}^{0,1}}J)\xi =
\frac{1}{2}\pi^{TX}_{TY^{\perp_\omega}} \Bigl( & [X_{H(w)}, \xi] +
J[X_{H(w)}, J\xi] \;+ \\
& +\; J\bigl([X_{H(jw)}, \xi] + J[X_{H(jw)}, J\xi]\bigr)\Bigr)\text{.}
\end{align*}
\end{lemma}
\begin{proof}
By definition of the Nijenhuis tensor and Remark \ref{Remark_tilde_J_H}
\begin{align*}
N_{\tilde{J}^H}((w,v), (0,\xi)) &= [(w,v), (0,\xi)] +
\tilde{J}^H[\tilde{J}^H(w,v), (0,\xi)] \;+ \\
&\quad\; +\; \tilde{J}^H[(w,v), \tilde{J}^H(0,\xi)] -
[\tilde{J}^H(w,v), \tilde{J}^H(0,\xi)] \\
&= (0, [v, \xi]) + \tilde{J}^H[(jw, Jv + 2JX_{H(w)}^{0,1}), (0,\xi)]
\;+ \\
&\quad\; +\; \tilde{J}^H[(w,v), (0,J\xi)] - [(jw, Jv +
2JX_{H(w)}^{0,1}), (0,J\xi)] \\
&= (0, [v,\xi]) + \tilde{J}^H(0, [Jv,\xi] + 2[JX_{H(w)}^{0,1},\xi])
\;+ \\
&\quad\; +\; \tilde{J}^H(0,[v,J\xi]) - (0, [Jv,J\xi] +
2[JX_{H(w)}^{0,1},J\xi]) \\
&= (0, [v,\xi] + J[Jv,\xi] + J[v,J\xi] - [Jv,J\xi) \;+ \\
&\quad\; +\; 2(0, J[JX_{H(w)}^{0,1},\xi] - [JX_{H(w)}^{0,1}, J\xi]) \\
&= (0, N_J(v,\xi)) - 2(0,[JX_{H(w)}^{0,1}, J\xi] - J[JX_{H(w)}^{0,1},\xi]) \\
&= (0, N_J(v,\xi)) - (0,2(L_{JX_{H(w)}^{0,1}}J)\xi)\text{,}
\end{align*}
for $[JX_{H(w)}^{0,1}, J\xi] - J[JX_{H(w)}^{0,1},\xi] =
L_{JX_{H(w)}^{0,1}}(J\xi) - JL_{JX_{H(w)}^{0,1}}\xi =
(L_{JX_{H(w)}^{0,1}}J)\xi + JL_{JX_{H(w)}^{0,1}}\xi -
JL_{JX_{H(w)}^{0,1}}\xi = (L_{JX_{H(w)}^{0,1}}J)\xi$. \\
To show the last equation, one can explicitely write out
$X_{H(w)}^{0,1}$ to get
\begin{align*}
2([JX_{H(w)}^{0,1}, J\xi] - J[JX_{H(w)}^{0,1}, \xi]) &= [JX_{H(w)},
J\xi] - J[JX_{H(w)}, \xi] \;- \\
&\quad\; -\; ([X_{H(jw)}, J\xi] - J[X_{H(jw)}, \xi])\text{.}
\end{align*}
Now using that $J\in \mathcal{J}_{\omega, \mathrm{ni}}(X,Y)$, hence
$\pi^{TX}_{TY^{\perp_\omega}}N_J(v,\xi) =
\pi^{TX}_{TY^{\perp_\omega}}([v,\xi] + J[v,J\xi] - ([Jv,J\xi] -
J[Jv,\xi])) = 0$ for $v\in TY$, $\xi \in TY^{\perp_\omega}$, with $v =
X_{H(w)}$, shows the last equation in the statement of the lemma.
\end{proof}

\begin{remark}
If $\nabla$ denotes \emph{any} torsion-free connection on $X$, then
the second part in the above formula for the Nijenhuis tensor can also
be written as
\[
2(L_{JX_{H(w)}^{0,1}}J)\xi = 2J(\nabla_\xi(JX_{H(w)}^{0,1}) +
J\nabla_{J\xi}(JX_{H(w)}^{0,1}) - J(\nabla_{JX_{H(w)}^{0,1}}
J)\xi)\text{,}
\]
which recovers formula (3.3) (c) in Definition 3.2 from
\cite{MR1954264}, although it will not be used in this form in this
text. For starting with the second to last line in the string of
equalities in the above proof, because $\nabla$ is torsion-free,
\begin{align*}
[JX_{H(w)}^{0,1}, J\xi] - J[JX_{H(w)}^{0,1}, \xi] &=
\nabla_{JX_{H(w)}^{0,1}}(J\xi) - \nabla_{J\xi}(JX_{H(w)}^{0,1}) \;- \\
&\quad\; -\; J\nabla_{JX_{H(w)}^{0,1}}\xi -
\nabla_\xi(JX_{H(w)}^{0,1}) \\
&= (\nabla_{JX_{H(w)}^{0,1}}J)\xi + J\nabla_{JX_{H(w)}^{0,1}}\xi -
\nabla_{J\xi}(JX_{H(w)}^{0,1}) \;- \\
&\quad\; -\; J\nabla_{JX_{H(w)}^{0,1}}\xi +
J\nabla_\xi(JX_{H(w)}^{0,1}) \\
&= J(\nabla_\xi(JX_{H(w)}^{0,1}) +
J\nabla_{J\xi}(JX_{H(w)}^{0,1}) - J(\nabla_{JX_{H(w)}^{0,1}}
J)\xi)\text{.}
\end{align*}
\end{remark}

\begin{lemma}\label{Lemma_Enough_ni_ham_perturbations}
There exists a continuous linear right inverse $\iota :
\mathcal{H}(\tilde{Y}) \to \mathcal{H}(\tilde{X}, \tilde{Y})$ to
the restriction map $\mathcal{H}(\tilde{X},\tilde{Y}) \to
\mathcal{H}(\tilde{Y})$, $H \mapsto (\zeta\mapsto H(\zeta)|_Y)$, \ie
the restriction map 
$\mathcal{H}(\tilde{X},\tilde{Y}) \to \mathcal{H}(\tilde{Y})$ is a
split surjection. \\
Furthermore, $\iota$ can be chosen \st $\im \iota \subseteq
\mathcal{H}_{\mathrm{ni}}(\tilde{X}, \tilde{Y}, J)$ for any $J\in
\mathcal{J}_{\omega, \mathrm{ni}}(X,Y)$.
\end{lemma}
\begin{proof}
First, choosing a locally finite covering of $M$ over which $X$ and $Y$ are trivial and a subordinate partition of unity, one can reduce to the case that $X$ and $Y$ are trivial bundles, so assume that to be the case. \\
By the Weinstein symplectic neighbourhood theorem, Theorem 3.30,
p.~101, in \cite{MR1698616}, there exists a neighbourhood $N(Y)$ of $Y$
in $X$, symplectomorphic to an open neighbourhood $V$ of the zero section
in $TY^{\perp_\omega}$ and mapping the zero section to $Y$ via the
inclusion. $\omega$ turns $TY^{\perp_\omega}$ into a symplectic vector
bundle. Choose any $\omega$-compatible Riemannian metric $g$ on
$TY^{\perp_\omega}$ and let $\varepsilon >
0$ be so small that for all $y\in Y$, the ball of radius (\wrt
$g$) in $T_yY^{\perp_\omega}$ lies in $V$. Now choose a smooth
cutoff-function $\rho : [0,\infty) \to [0,1]$ \st $\rho(r) = 1$ for
all $0\leq r \leq \varepsilon/3$ and $\rho(r) = 0$ for all $r\geq
2\varepsilon/3$. Let $\tau : TY^{\perp_\omega} \to Y$ the
projection. Given $H_0 \in C^\infty(Y,\R)$, define $\hat{H}_0 :
TY^{\perp_\omega} \to \R$, $\hat{H}_0(v) \definedas
\rho(\|v\|)\tau^\ast H(v)$. Then $\hat{H}_0$ has compact support in $V$
and by identifying $V$ with $N(Y)$, $\hat{H}_0$ hence defines a
function $H\in C^\infty(X,\R)$. Furthermore, for $v\in TY^{\perp_\omega}$,
$\d H(v) = 0$, for again identifying $N(Y)$ with $V$, by construction
and since $\rho$ is constant in a neighbourhood of zero, $\d H(v) = \d
H_0 (\tau_\ast v) = \d H_0(0) = 0$, for $v\in TY^{\perp_\omega} = \ker
\tau_\ast$. Also, by definition, $H|_Y = H_0$. Denote the resulting
map $\eta : C^\infty(Y,\R) \to C^\infty(X,\R)$. One can now define $\iota :
\mathcal{H}(\tilde{Y}) \to \mathcal{H}(\tilde{X}, \tilde{Y})$ by
$H_0 \mapsto (\zeta \mapsto \eta(H_0(\zeta)))$. By Lemma
\ref{Lemma_H_X_Y}, this defines a right inverse to the restriction
map.
To show the second statement, let $J\in \mathcal{J}_{\omega, \mathrm{ni}}(X,
Y)$ be arbitrary. By Lemma \ref{Lemma_Characterisation_H_ni} it has to
be shown that for the $H$ just constructed
\[
\pi^{TX}_{TY^{\perp_\omega}} \Bigl( [X_{H(w)}, \xi] +
J[X_{H(w)}, J\xi] + J\bigl([X_{H(jw)}, \xi] + J[X_{H(jw)},
J\xi]\bigr)\Bigr) = 0
\]
for all $w\in T\Sigma$, $\xi\in TY^{\perp_\omega}$. I will show that
each of the four summands 
\[
[X_{H(w)}, \xi], [X_{H(w)}, J\xi], [X_{H(jw)}, \xi], [X_{H(jw)}, J\xi]
\]
vanishes separately for asuitably chosen extension of $\xi$ to a locally defined vector field.
Here and in the following it is used that $J$ leaves $TY$ and
$TY^{\perp_\omega}$ invariant and so in particular
$\pi^{TX}_{TY^{\perp_\omega}}\circ J = J\circ
\pi^{TX}_{TY^{\perp_\omega}}$. Let $\xi\in T_yY^{\perp_\omega}$, $y\in
Y$, and $w\in T\Sigma$. Choose local coordinates around $y$ in
$X$ of the form $(y^1,\dots, y^{2n-2}, x^1, x^2)$ by use of the
Weinstein symplectic neighbourhood theorem. By a smooth change
of trivialisation in the corresponding trivialisation of
$TY^{\perp_\omega}$ over this neighbourhood one can assume that $J$ is
the standard complex structure along $Y$ in the coordinates
$x^1$ and $x^2$, \ie $J\frac{\partial}{\partial x^1}|_{x^1=x^2=0} =
\frac{\partial}{\partial x^2}$ and $J\frac{\partial}{\partial x^2}|_{x^1=x^2=0} =
- \frac{\partial}{\partial x^1}$. Extend $\xi =
a^1\frac{\partial}{\partial x^1} + a^2\frac{\partial}{\partial x^2}$,
with $a^1,a^2\in\R$, locally by the same formula. Then $X_{H(w)}$ can
be written in these coordinates as $X_{H(w)} = \sum_j
b^j\frac{\partial}{\partial y^j}$ with $\frac{\partial}{\partial
  x^i}b^j = 0$ by construction of $H$. Then
\[
[X_{H(w)},\xi]|_{x^1=x^2=0} = \sum_i \sum_j
b^j\frac{\partial}{\partial y^j}|_{x^1=x^2=0} a^i \frac{\partial}{\partial x^i} -
\sum_j \sum_i a^i \frac{\partial}{\partial x^i}|_{x^1=x^2=0} b^j
\frac{\partial}{\partial y^j} = 0\text{.}
\]
Similarly,
\begin{align*}
[X_{H(w)},J\xi]|_{x^1=x^2=0} &= \sum_j
b^j\frac{\partial}{\partial y^j}|_{x^1=x^2=0} a^1
\frac{\partial}{\partial x^2} - \sum_j
b^j\frac{\partial}{\partial y^j}|_{x^1=x^2=0} a^2
\frac{\partial}{\partial x^1} \;- \\
&\quad\; \;-
\sum_j (a^1 \frac{\partial}{\partial x^2} - a^2
\frac{\partial}{\partial x^1}) |_{x^1=x^2=0} b^j
\frac{\partial}{\partial y^j} = 0\text{.}
\end{align*}
The other two cases are completely analogous.
\end{proof}

For $Y$-compatible almost complex structures and Hamiltonian perturbations one can now define the sets $\mathcal{N}^\ell$ from the previous subsection.
The main observation used in the definition is the following, which for convenience subsequently is summarised from Section 7, in \cite{MR2399678}.
\begin{defn}
Let $(S,j)$ be a Riemann surface, $f : S \to X$ a differentiable
map. An \emph{isolated intersection} of $f$ with $Y$ is a point $z\in
f\inv(Y)$ \st there exists a closed disk $D \subseteq S$ around $z$ and
a closed disk $B\subseteq Y$ around $f(z)$ with $f\inv(B)\cap D =
\{z\}$. \\
Given such an isolated intersection $z\in f\inv(Y)$, the \emph{local
  intersection number} $\iota(f, Y; z)$ of $f$ with $Y$ at $z$ is
defined as follows: Assume that $f$ intersects $Y$ in $z$
transversely. Then $\iota(f, Y; z) = 1$, if the orientation on $T_{f(z)}X$
agrees with the orientation induced (via $T_{f(z)}X \cong (f_\ast T_zS)
\oplus T_{f(z)}Y$) by the orientations on $T_zS$ and
$T_{f(z)}Y$, and $\iota(f, Y; z) = -1$, otherwise. In general, choose a
differentiable perturbation $f_t : S \to X$, $t\in [0,1]$, of $f$ with
compact support in the interior of $D$ and \st $f_1|_D$ is transverse
to $B$. Then
\[
\iota(f, Y; z) \definedas \sum_{z'\in f_1\inv(B)\cap D} \iota(f_1, Y;
z')\text{.}
\]
If $S$ is compact and all intersections of $f$ with $Y$ are isolated
(in particular by compactness there are only finitely many), then the
intersection number of $f$ with $Y$ is defined as
\[
\iota(f, Y) \definedas \sum_{z\in f\inv(Y)} \iota(f, Y; z)\text{.}
\]
\end{defn}

The adaptation of Proposition 7.1, in \cite{MR2399678} to the
present situation.
\begin{lemma}\label{Lemma_Order_of_tangency}
Let $\tilde{u} \in \mathcal{M}(\tilde{X}, A, J, \mathcal{H}(\tilde{X}, \tilde{Y}))$. Define $u\definedas \pr_2\circ \tilde{u} : \Sigma_b\to X$.
Then for every component (\ie connected component of a desingularisation) $\Sigma^i_b$ of $\Sigma_b$, either $u(\Sigma^i_b) \subseteq Y$ or $(u|_{\Sigma^i_b})\inv(Y)$ is finite.
In the latter case,
\[
\iota(u|_{\Sigma^i_b}, Y) = [u|_{\Sigma^i_b}]\cdot [Y]\text{,}
\]
\ie the intersection number of $u|_{\Sigma^i_b}$ with $Y$ coincides
with the topological intersection number of the homology classes in
$X$ defined by $u|_{\Sigma^i_b}$ and $Y$. Furthermore, at each
intersection point $z\in (u|_{\Sigma^i_b})\inv(Y)$, $u$ is tangent to
$Y$ of some finite order $s \geq 0$ with
\[
\iota(u|_{\Sigma^i_b}, Y; z) = s + 1\text{.}
\]
In particular, each local intersection number $\iota(u|_{\Sigma^i_b},
Y; z)$ is positive.
\end{lemma}
\begin{proof}
$(\tilde{X}|_{\Sigma^i_b}, \tilde{J}^H)$ is a complex manifold with
$\tilde{Y}|_{\Sigma^i_b}$ as a complex submanifold by definition of
$\mathcal{H}(\tilde{X}, \tilde{Y})$. Furthermore,
$\tilde{u}|_{\Sigma^i_b} : \Sigma^i_b \to \tilde{X}|_{\Sigma^i_b}$ is
a holomorphic map. Now observe that $\tilde{u}(z) \in \tilde{Y}$ \iff
$u(z) \in Y$, and the orders of tangency coincide. Now apply
Proposition 7.1, in \cite{MR2399678} to $\tilde{u}|_{\Sigma^i_b}$.
\end{proof}

This allows for the following definition:
\begin{defn}
Let $(\ell_1, \dots, \ell_n) \in (\Z_{\geq -1})^n$ and denote $\ell'_j \definedas \min\{0,\ell_j\}$. Given any open subset $V \subseteq \tilde{X} \setminus \tilde{Y}$ and $H \in \mathcal{H}(\tilde{X}, \tilde{Y})$, define
\begin{align*}
\tilde{Y}^{\ell'_j} &\definedas \begin{cases} \tilde{X} & \ell'_j = -1 \\ \tilde{Y} & \ell'_j = 0 \end{cases}
\intertext{and note that $\mathcal{H}^V(\tilde{X}) \subseteq \mathcal{H}(\tilde{X}, \tilde{Y})$. Then}
\mathcal{M}(\tilde{X}, \tilde{Y}^{(\ell'_1, \dots, \ell'_n)}, A, J, H + \mathcal{H}^V(\tilde{X})) &\definedas \left(\ev^R\right)^{-1}\left( R_1^\ast \tilde{Y}^{\ell'_1} \oplus \cdots \oplus R_n^\ast \tilde{Y}^{\ell'_n} \right)\text{,}
\intertext{for}
\ev^R : \mathcal{M}^V(\tilde{X}, A, J, H + \mathcal{H}^V(\tilde{X})) &\to R_1^\ast \tilde{X} \oplus\cdots\oplus R_n^\ast \tilde{X}\text{.}
\end{align*}
Furthermore for any subset $B\subseteq M$,
\begin{multline*}
\mathcal{M}^V(\tilde{X}|_B, \tilde{Y}^{(\ell_1, \dots, \ell_n)}, A, J, H + \mathcal{H}^V(\tilde{X})) \definedas \\
\{ u \in \mathcal{M}^V_b(\tilde{X}, \tilde{Y}^{(\ell'_1, \dots, \ell'_n)}, A, J, H + \mathcal{H}^V(\tilde{X})) \;|\;
b \in B, \iota(u, \tilde{Y}|_{\Sigma_b}; R_j(b)\} = \ell_j\}\text{.}
\end{multline*}
\end{defn}

By the previous lemma, if $u$ is a holomorphic curve in $\tilde{X}$ \st $u$ intersects $\tilde{Y}$ at each of $\ell$ different marked points, the last $\ell$, say, $u$ is not contained completely in $\tilde{Y}$ and $[u]\cdot [Y] = \ell$, then $u$ intersects $\tilde{Y}$ transversely.
Unfortunately one cannot expect this behaviour to persevere under limits of sequences of such maps.
For example even for a fixed complex structure on the underlying curve, two of the last $\ell$ marked points could converge on the domain forming a nodal curve, built up of the original curve together with a sphere component that gets mapped to $\tilde{Y}$.
Since the restriction of $\tilde{Y}$ to every fibre $\Sigma_b$ of $\Sigma$ is trivial by definition, it makes sense to say that the sphere component is constant. In this case this map actually factors through a map from the original surface, but with the two converging marked points replaced by the point at which the sphere component is attached and which gets mapped to $\tilde{Y}$.
At this new point, the curve no longer needs to be transverse to $\tilde{Y}$, but the previous lemma states that, if the curve does not lie completely in $\tilde{Y}$, the limit map can only have tangencies of second order.
So apart from moduli spaces of curves with marked points lying on a given submanifold, a case already dealt with in Lemma \ref{Lemma_The_universal_moduli_space}, one should also construct moduli spaces of curves with tangencies to a given complex hypersurface of (at least) a given order.
The tangency of order $1$-condition is easy enough to define, if $u\in \mathcal{M}_b(\tilde{X}, A, J, H)$ with $u(R_i(b)) \in \tilde{Y}$, then $u$ is tangent to $\tilde{Y}$ at $R_i(b)$ to first order simply if $\im \left(D^\mathrm{v}u\right)_{R_i(b)} \subseteq V\tilde{Y}$.
For $J\in \mathcal{J}_\omega(X,Y)$, $V\tilde{Y}^{\perp_\omega}$ is a $J_b$-complex subspace of complex dimension $1$.
If $H\in \mathcal{H}(\tilde{X}, \tilde{Y})$, then since $\dbar^{J,H}_b u = 0$, $\pi^{V\tilde{X}}_{V\tilde{Y}^{\perp_\omega}} (D^\mathrm{v}u)_{R_i(b)}$ is a $j_b$-$J_b$-complex linear map from $V_{R_i(b)}\Sigma$ to $V_{u(R_i(b))} \tilde{Y}^{\perp_\omega}$.
Hence over the subset of elements of $\mathcal{M}(\tilde{X}, A, J, \mathcal{H}(\tilde{X}, \tilde{Y}))$ that map the $i^\text{th}$ marked point to $\tilde{Y}$ (a submanifold by Lemma \ref{Lemma_The_universal_moduli_space}), one can consider the complex line bundle with line over $u$ given by $\Hom_{(j,J)}(V_{R_i}\Sigma, V_{u(R_i)} V\tilde{Y}^{\perp_\omega})$ and the section $u \mapsto \pi^{V\tilde{X}}_{V\tilde{Y}^{\perp_\omega}} (D^\mathrm{v}u)_{R_i}$. 
In case of transversality of this section to the zero section, the moduli space of curves tangent to $\tilde{Y}$ at the $i^\text{th}$ marked point then has complex codimension one in the submanifold of those curves that map the $i^\text{th}$ marked point to $\tilde{Y}$.
Unfortunately the higher order tangency conditions do not seem to admit such an easy description as global sections of a globally defined complex vector bundle (of the ``correct'' rank) over the universal moduli space.
 \cite{MR2399678}, which allows to use the transversality result (or rather a slight variation of its proof) from \cite{MR2399678}.

\begin{construction}
Let $(\rho : S \to B, \hat{R}, \iota, \hat{\iota})$ be a desingularisation of $\Sigma$ over $B \subseteq M$ and as before denote $\hat{X} \definedas \rho^\ast \iota^\ast X = \hat{\iota}^\ast \tilde{X}$ and $\hat{Y} \definedas \rho^\ast \iota^\ast Y = \hat{\iota}^\ast \tilde{Y}$.
For $a\in B$, let $U\subseteq B$ be an open neighbourhood of $a$ \st both $X|_U$ and $Y|_U$ are trivial, and hence so are $\hat{X}|_U$ and $\hat{Y}|_U$.
Also let $\phi_a : U\times S_a \to S|_U$ be a trivialisation that preserves the marked points and nodes.
Assume that there are pairwise disjoint open neighbourhoods $D_j \subseteq S_a$ of the marked points $\hat{R}_j(a) \in S_a$, biholomorphically equivalent to the unit disk $\mathbb{D} \subseteq \C$ and disjoint from all the nodal points. These are assumed to have the property that for all $b\in U$, $\phi_{ab}|_{D_j} : S_a \supseteq D_j \to S_b$ is a biholomorphic map from $D_j$ onto a neighbourhood of $\hat{R}_j(b) \in S_b$. Let $u_0\in \mathcal{M}_a(\tilde{X}|_B, A, J, H)$ for some $H\in \mathcal{H}(\tilde{X}, \tilde{Y})$.
Fix some $i \in \{1, \dots, n\}$ and assume that $\ev^{\hat{R}}_i(u_0) \in \hat{Y}$, but that the component of $\Sigma_a$ containing $\hat{R}_i(a)$ does not get mapped completely to $\hat{Y}$ by $u_0$.
Using triviality of $X$ and $Y$ over $U$, pick a neighbourhood $W \subseteq \hat{X}$ of $\ev^{\hat{R}}_i(u_0)$ diffeomorphic to $U\times S_a\times \C^r$, where $r\definedas \dim_\C(X)$, via a diffeomorphism $\Psi$ that maps $\hat{Y}\cap W$ to $U\times S_a\times \C^{r-1}\times \{0\}$.
Also assume that this diffeomorphism covers $\phi_a$. On the right hand side then for any $H\in\mathcal{H}(\tilde{X}, \tilde{Y})$ and $b\in U$, $\{b\}\times S_a\times \C^r$ is equipped with the pullback complex structure $\overline{J}^H_b$ of $\hat{J}^H$ which turns $\{b\}\times S_a\times \C^r$ into an almost complex manifold and $\{b\}\times S_a \times \C^{r-1}\times\{0\}$ into a complex submanifold.
Remember that the topology on $\mathcal{M}_U(\tilde{X}|_B, A, J, \mathcal{H}(\tilde{X}, \tilde{Y}))$ is finer than the topology induced by that on $U\times \mathcal{B}^{k,p}_a(\hat{X}|_B, A, J, H)\times \mathcal{H}(\tilde{X}, \tilde{Y})$ (for some $k,p$ with $kp>2$) by the chart defined via $\phi_a$ from Construction \ref{Construction_Universal_Dbar}. And that the topology on $\mathcal{B}^{k,p}_a(\hat{X}|_B, A, J, H)$ in turn is finer than the $C^0$-topology.
Also, the intersection of $u_0$ with $\hat{Y}$ at $R_i(a)$ is isolated by Lemma \ref{Lemma_Order_of_tangency}.
Hence there is a neighbourhood $\mathcal{V}$ of $u_0$ in $\mathcal{M}(\tilde{X}|_B, A, J, \mathcal{H}(\tilde{X}, \tilde{Y}))$ \st $u(\phi_{ab}(D_j)) \subseteq W$ for all $u\in \mathcal{V}$, $\pi^{\mathcal{M}}_{B}(u) = b$.
With the help of the above one can now assign, for every $j = 1, \dots, n$ and to every $u \in \mathcal{V}$ with $\pi^{\mathcal{M}}_B(u) = b$ and $\pi^{\mathcal{M}}_{\mathcal{H}}(u) = H$ an ($i$ here is the standard complex structure on $D_j\cong \mathbb{D}$) $i$-$\overline{J}^H_b$-holomorphic map $D_j \to \{b\}\times \D_j\times \C^r$.
Now one is pretty much exactly in (a parametrised version of) the situation of Section 6 of \cite{MR2399678} and can follow the discussion leading up to Proposition 6.9 almost to the letter, dropping the simplicity requirement and replacing the space of perturbations of the almost complex structures by the space of Hamiltonian perturbations used in this text, esp.~in Lemma 6.6, to show the following result:
\begin{lemma}\label{Lemma_Tangency_condition}
Let $V\subseteq \tilde{X}$ be an open subset \st $V\cap \tilde{Y} = \emptyset$, let $H \in \mathcal{H}(\tilde{X}, \tilde{Y})$ and let $B\subseteq M$ be a stratum over which $\Sigma$ has a desingularisation.
Then for any $n$-tuple $(\ell_1, \dots, \ell_n) \in (\Z_{\geq -1})^n$, $\mathcal{M}^V(\tilde{X}|_B, \tilde{Y}^{(\ell_1, \dots, \ell_n)}, A, J, H + \mathcal{H}^V(\tilde{X}))$ is a Banach submanifold of $\mathcal{M}^V(\tilde{X}|_B, A, J, H + \mathcal{H}^V(\tilde{X}))$ of real codimension $2\sum_{i=1}^n (\ell_i + 1)$.
\end{lemma}
\end{construction}

\section{Definition of the pseudocycle and outline of the main theorems}\label{Section_Definition_and_Outline}

In this final part, it is made precise in which sense the map (\ref{Equation_Fundamental_cycle}) from the introduction defines a homology class, after suitable modifications.
I.\,e.~Theorem \ref{Theorem_Main_Theorems_Summary} is given a precise formulation in the form of Theorems \ref{Theorem_Main_Theorem_1} and \ref{Theorem_Main_Theorem_2} and the proofs of these theorems are sketched. \\
To do so, the notion of a pseudocycle from \cite{MR2045629}, Section 6.5, and more generally that of rational pseudocycles and rational cobordism from \cite{MR2399678} will be used. \\
Given the following data:
\begin{enumerate}\label{Data_Pseudocycle}
  \item A closed symplectic manifold $(X, \omega)$ with integer symplectic form, $[\omega] \in H^2(X, \Z)$.
  \item $0 \neq A \in H_2(X;\Z)$ \st $\omega(A) > 0$, $E\definedas \omega(A) + 1$.
  \item\label{Data_Pseudocycle_3} A regular marked nodal family $(\pi : \Sigma \to M, R_\ast)$ of type $(g,n)$, hence Euler characteristic $\chi = 2(1-g)$, and consequently regular marked nodal families $(\pi^\ell : \Sigma^\ell \to M^\ell, R^\ell_\ast, T^\ell_\ast)$ for all $\ell \geq 0$ of Euler characteristic $\chi$ with $n + \ell$ marked points $R^\ell_1, \dots, R^\ell_n$, $T^\ell_1, \dots, T^\ell_\ell$, as in Lemma \ref{Lemma_Forget_marked_point} and Proposition \ref{Proposition_Sequence_of_branched_coverings}:
\[
\xymatrix{
\cdots\ar[r]^-{\hat{\pi}^{\ell+1}} & \Sigma^{\ell+1} \ar[r]^-{\hat{\pi}^{\ell}} \ar[d]^-{\pi^{\ell+1}} & \Sigma^\ell \ar[r]^-{\hat{\pi}^{\ell-1}} \ar[d]^-{\pi^{\ell}} & \Sigma^{\ell-1} \ar[r]^-{\hat{\pi}^{\ell-2}} \ar[d]^-{\pi^{\ell-1}} & \cdots\ar[r]^-{\hat{\pi}^1} & \Sigma^1\ar[r]^-{\hat{\pi}^{0}} \ar[d]^-{\pi^{1}} & \Sigma^0 \ar[d]^-{\pi^0} \ar@{=}[r] & \Sigma \ar[d]^-{\pi} \\
\cdots\ar[r]^-{\pi^{\ell+1}} & M^{\ell+1} \ar[r]^-{\pi^{\ell}} \ar@{=}[ru] & M^\ell \ar[r]^-{\pi^{\ell-1}} \ar@{=}[ru] & M^{\ell-1} \ar[r]^-{\pi^{\ell-2}} & \cdots \ar[r]^-{\pi^1} & M^1 \ar[r]^-{\pi^{0}} \ar@{=}[ru] & M^0 \ar@{=}[r] & M
}
\]
\begin{align*}
&\xymatrix{
\Sigma^\ell \ar[r]^-{\hat{\pi}^{\ell-1}} \ar[d]^-{\pi^\ell} & \Sigma^{\ell-1} \ar[d]_-{\pi^{\ell-1}} \\
M^\ell \ar[r]^-{\pi^{\ell-1}} \ar@/^1pc/[u]^-{R^\ell_j} & M^{\ell-1} \ar@/_1pc/[u]_-{R^{\ell-1}_j\quad{\displaystyle\forall\, j=1,\dots, n}}
}
\\
&\xymatrix{
\Sigma^\ell \ar[r]^-{\hat{\pi}^\ell} \ar[d]^-{\pi^\ell} & \Sigma^{\ell-1} \ar[d]_-{\pi^{\ell-1}} \\
M^\ell \ar[r]^-{\pi^{\ell-1}} \ar@/^1pc/[u]^-{T^\ell_j} & M^{\ell-1} \ar@/_1pc/[u]_-{T^{\ell-1}_j\quad{\displaystyle\forall\, j=1,\dots, \ell-1}}
}
\end{align*}
where $M$ is assumed to be closed, and hence so are the $M^\ell$ for all $\ell\geq 0$.
Furthermore, for every $b\in M^\ell$, putting $b' \definedas \pi^{\ell-1}(b) \in M^{\ell-1}$, the map
\begin{multline*}
\hat{\pi}^{\ell-1}_b : (\Sigma^\ell_b, R^\ell_1(b), \dots, R^\ell_n(b), T^\ell_1(b), \dots, T^\ell_{\ell-1}(b)) \to  \\
\to (\Sigma^{\ell-1}_{b'}, R^{\ell-1}_1(b'), \dots, R^{\ell-1}_n(b'), T^{\ell-1}_1(b'), \dots, T^{\ell-1}_{\ell-1}(b'))
\end{multline*}
is stabilising, \ie biholomorphic on every stable component of
\[
(\Sigma^\ell_b, R^\ell_1(b), \dots, R^\ell_n(b), T^\ell_1(b), \dots, T^\ell_{\ell-1}(b))
\]
and constant on every unstable component.
For $\ell > k$ denote the compositions
\begin{align*}
\hat{\pi}^\ell_k \definedas \hat{\pi}^{k}\circ \hat{\pi}^{k+1}\circ \cdots\circ \hat{\pi}^{\ell-1} : \Sigma^\ell &\to \Sigma^{k}
\intertext{and}
\pi^\ell_k \definedas \pi^k\circ \pi^{k+1}\circ \cdots\circ \pi^{\ell-1} : M^\ell &\to M^{k}\text{.}
\end{align*}
By the same argument as in Section \ref{Section_III.1}, assume that $M$ and hence all the $M^\ell$ are connected.
  \item Metrics $h^\ell$ on the $\Sigma^\ell$, restricting to a hermitian metric on every $\Sigma^\ell_b$, $b\in M^\ell$.
  \item Unless stated otherwise, $Y \subseteq X$ is a symplectic hypersurface, $J \in \mathcal{J}_\omega(X, Y)$ is an arbitrary $\omega$- and $Y$-compatible almost complex structure and $\ell \in \N_0$ is arbitrary as well.
\end{enumerate}

Let, as before, $\tilde{X} \definedas \Sigma\times X$, $\tilde{Y} \definedas \Sigma \times Y$, $\tilde{X}^\ell \definedas (\hat{\pi}^\ell_0)^\ast \tilde{X} \cong \Sigma^\ell\times X$, $\tilde{Y}^\ell \definedas (\hat{\pi}^\ell_0)^\ast \tilde{Y} \cong \Sigma^\ell\times Y$.
\begin{defn}\label{Definition_Fundamental_cycle}
Let $\overset{\circ}{M}{}^\ell$ be the top stratum of $M^\ell$ in the stratification by signature, corresponding to the smooth surfaces. \\
For $H\in \mathcal{H}(\tilde{X}^\ell)$, define
\begin{align*}
\overset{\circ}{\mathcal{M}}(\tilde{X}^\ell, \tilde{Y}^\ell, A, J, H) \definedas \{  u
\in \mathcal{M}(\tilde{X}^\ell|_{\overset{\circ}{M}{}^\ell}, A, J, H) \;|\; & \im(u\circ T^{\ell}_j) \subseteq \tilde{Y}^\ell,\; j=1,\dots, \ell, \\
& \im(u) \cap \tilde{X}^\ell \setminus \tilde{Y}^\ell \neq \emptyset\}
\end{align*}
and as before, for any (affine) subspace $\mathcal{K} \subseteq \mathcal{H}(\tilde{X}^\ell)$,
\[
\overset{\circ}{\mathcal{M}}(\tilde{X}^\ell, \tilde{Y}^\ell, A, J, \mathcal{K}) \definedas \bigcup_{H\in \mathcal{K}} \overset{\circ}{\mathcal{M}}(\tilde{X}^\ell, \tilde{Y}^\ell, A, J, H)\text{.}
\]
Also denote by
\[
\cl{\overset{\circ}{\mathcal{M}}(\tilde{X}^\ell, \tilde{Y}^\ell, A, J, H)}
\]
the closure in $\mathcal{M}(\tilde{X}^\ell, A, J, H)$ of $\overset{\circ}{\mathcal{M}}(\tilde{X}^\ell, \tilde{Y}^\ell, A, J, H)$. Finally, let
\[
\mathrm{gw}^\ell_{\Sigma}(X, Y, A, J, H) : \overset{\circ}{\mathcal{M}}(\tilde{X}^\ell, \tilde{Y}^\ell, A, J, H) \to M\times X^n
\]
be defined as the composition
\[
\overset{\circ}{\mathcal{M}}(\tilde{X}^\ell, \tilde{Y}^\ell, A, J, H) \xrightarrow{\ev^{R^\ell}} \bigoplus_{i=1}^n\tilde{X}^\ell \cong \Sigma^\ell\times X^n \xrightarrow{(\pi^\ell_0 \circ \pi^\ell)\times \id} M\times X^n\text{.}
\]
\end{defn}

\subsection{Statements of the main results}

For the formulation of the first main result of this article remember the spaces $\mathcal{J}_{\omega, \mathrm{ni}}(X,Y)$, $\mathcal{H}_{\mathrm{ni}}(\tilde{X}^\ell, \tilde{Y}^\ell, J)$, $\mathcal{H}^0_{\mathrm{ni}}(\tilde{X}^\ell, \tilde{Y}^\ell, J)$ and $\mathcal{H}^{00}(\tilde{X}^\ell, \tilde{Y}^\ell)$ from Definitions \ref{Definition_J_ni} and \ref{Definition_Y_compatible_Hamiltonian_perturbation} and that one can consider $\mathcal{H}(\tilde{Y})$ as a subset of $\mathcal{H}_{\mathrm{ni}}(\tilde{X}^\ell, \tilde{Y}^\ell, J)$ for any $J \in \mathcal{J}_{\omega, \mathrm{ni}}(X, Y)$ and $\ell \geq 0$ by Lemma \ref{Lemma_Enough_ni_ham_perturbations} and pullback via $\hat{\pi}^\ell_0$.

Also remember that a Donaldson pair $(Y, J_0)$ of degree $D$ is a pair consisting of an almost complex structure $J_0\in \mathcal{J}_\omega(X)$ and a hypersurface $Y\subseteq X$. $Y$ is assumed to be an, in the sense of \cite{MR2399678}, Section 8, approximately $J_0$-holomorphic, in particular symplectic, hypersurface with $\operatorname{PD}(Y) = D[\omega]$.
The necessary existence and uniqueness results for Donaldson pairs can be found in Theorem 8.1 from \cite{MR2399678} and the references quoted there.

The first main result deals with the existence of a well-defined Gromov-Witten pseudocycle depending on a choice of $\omega$-compatible almost complex structure and Donaldson hypersurface, after a generic choice of Hamiltonian perturbation, and independence of the choice of perturbation.
\clearpage
\begin{theorem}\label{Theorem_Main_Theorem_1}
Let $(X, \omega)$, $J_0$, $A$, $E$ be as above.
There exists an integer $D^\ast = D^\ast(X, \omega, J_0)$ \st for every $D\geq D^\ast$, there exists a symplectic hypersurface $Y \subseteq X$, making $(Y, J_0)$ a Donaldson pair of degree $D$ \st the following hold:
\begin{enumerate}[a)]
  \item\label{Theorem_Main_Theorem_1a} Let $\ell \definedas D\omega(A)$. Then there exist:
\begin{itemize}
  \item A nonempty subset $\mathcal{J}_{\omega, \mathrm{ni}}(X,Y, E) \subseteq \mathcal{J}_{\omega, \mathrm{ni}}(X,Y)$;
  \item given $J \in \mathcal{J}_{\omega, \mathrm{ni}}(X,Y, E)$, a generic subset $\mathcal{H}_{\mathrm{reg}}(\tilde{Y}, J) \subseteq \mathcal{H}(\tilde{Y})$;
  \item for every $J \in \mathcal{J}_{\omega, \mathrm{ni}}(X,Y, E)$ and $H^Y\in \mathcal{H}_{\mathrm{reg}}(\tilde{Y}, J)$, a generic subset \\
$\mathcal{H}^0_{\mathrm{reg}}(\tilde{X}^\ell, \tilde{Y}^\ell, J, H^Y) \subseteq \mathcal{H}^0_{\mathrm{ni}}(\tilde{X}^\ell, \tilde{Y}^\ell, J)$;
  \item for every $J \in \mathcal{J}_{\omega, \mathrm{ni}}(X,Y, E)$, $H^Y \in \mathcal{H}_{\mathrm{reg}}(\tilde{Y}, J)$ and $H^0 \in \mathcal{H}^0_{\mathrm{reg}}(\tilde{X}^\ell, \tilde{Y}^\ell, J, H^Y)$ a generic subset $\mathcal{H}^{00}_{\mathrm{reg}}(\tilde{X}^\ell, \tilde{Y}^\ell, J, H^Y + H^0) \subset \mathcal{H}^{00}(\tilde{X}^\ell, \tilde{Y}^\ell)$,
\end{itemize}
\st for $H^Y \in \mathcal{H}_{\mathrm{reg}}(\tilde{Y}, J)$, $H^0 \in \mathcal{H}^{0}_{\mathrm{reg}}(\tilde{X}^\ell, \tilde{Y}^\ell, J, H^Y)$ and $H^{00} \in \mathcal{H}^{00}_{\mathrm{reg}}(\tilde{X}^\ell, \tilde{Y}^\ell, J, H^Y + H^0)$, $H \definedas H^Y + H^0 + H^{00}$,
\[
u \in \mathcal{M}(\tilde{X}^{\ell}|_{\overset{\circ}{M}{}^{\ell}}, A, J, H) \quad \Rightarrow\quad \im(u) \cap \tilde{X}^{\ell} \setminus \tilde{Y}^{\ell} \neq \emptyset
\]
and
\[
\mathrm{gw}_{\Sigma}^\ell(X, Y, A, J, H)
\]
defines a pseudocycle of dimension
\[
\dim_\C(X)\chi + 2c_1(A) + \dim_\R(M)
\]
in $M\times X^n$ with image in $\overset{\circ}{M}\times X^n$. \\
Furthermore, if $\{B_i\}_{i\in\N}$ is a countable family of locally closed submanifolds of $\overset{\circ}{M}$, given $J, H^Y, H^0$ as above, then by replacing $\mathcal{H}^{00}_{\mathrm{reg}}(\tilde{X}^\ell, \tilde{Y}^\ell, J, H^Y + H^0)$ by another generic subset, one can assume that
\[
\mathrm{gw}^\ell_\Sigma(X, Y, A, J, H) : \overset{\circ}{\mathcal{M}}(\tilde{X}^\ell, \tilde{Y}^\ell, A, J, H) \to \overset{\circ}{M}\times X^n \subseteq M\times X^n
\]
is transverse to all the submanifolds $B_i\times X^n \subseteq \overset{\circ}{M}\times X^n$ in the sense that the codimension of the preimage of each $B_i\times X^n$ under this map in $\overset{\circ}{\mathcal{M}}(\tilde{X}^\ell, \tilde{Y}^\ell, A, J, H)$ is equal to the codimension of $B_i\times X^n$ in $\overset{\circ}{M}\times X^n$.
  \item\label{Theorem_Main_Theorem_1b} For $Y$ as above there exists a nonempty subset
\[
\mathcal{J}_{\omega, \mathrm{ni}}(X,Y, J_0, E) \subseteq \mathcal{J}_{\omega, \mathrm{ni}}(X,Y, E)\text{,}
\]
for $D$ large enough and by choice of $Y$ containing elements arbitrarily $C^0$-close to $J_0$, \st any two elements in $\mathcal{J}_{\omega, \mathrm{ni}}(X,Y, J_0, E)$ can be connected by a path in $\mathcal{J}_{\omega, \mathrm{ni}}(X,Y, E)$. \\
Furthermore, let $J_t \in \mathcal{J}_{\omega, \mathrm{ni}}(X, Y, E)$, $t\in \R$, be a family of almost complex structures \st $J_t = J_0$ for $t\leq 0$ as well as $J_t = J_1$ for $t\geq 1$.
Then, for $i =1,2$ and for any choice of $H^Y_i \in \mathcal{H}_{\mathrm{reg}}(\tilde{Y}, J_i)$, $H_i^0 \in \mathcal{H}_{\mathrm{reg}}^0(\tilde{X}^\ell, \tilde{Y}^\ell, J_i, H_i^Y)$ and $H_i^{00} \in \mathcal{H}_{\mathrm{reg}}^{00}(\tilde{X}^\ell, \tilde{Y}^\ell, J_i, H^Y_i + H^0_i)$, and setting $H_i \definedas H^Y_i + H^0_i + H^{00}_i$, the pseudocycles defined by $\mathrm{gw}_{\Sigma}^\ell(X, Y, A, J_i, H_i)$ are cobordant.
\end{enumerate}
In particular, given $(Y, J_0)$ as above, there is a well-defined cobordism class
\[
\mathrm{gw}^\ell_{\Sigma}(X, Y, A, J_0)
\] of pseudocycles in $M\times X^n$, independent of the choice of $J \in \mathcal{J}_{\omega, \mathrm{ni}}(X, Y, J_0, E)$ and Hamiltonian perturbation.
\end{theorem}

The second main result deals with independence of the Gromov-Witten pseudocycle of the choice of Donaldson pair $(Y, J_0)$.
\begin{theorem}\label{Theorem_Main_Theorem_2}
Let $(X, \omega)$, $A \in H_2(X; \Z)$ be as above. For $i = 0,1$ let $(Y_i, J_i)$ be Donaldson pairs of degrees $D_i \geq D^\ast$, where $D^\ast$ is as in the previous theorem. Let $\ell_i \definedas D_i\omega(A)$. Then the rational equivalence classes of pseudocycles defined by
\[
\frac{1}{\ell_i!} \mathrm{gw}^{\ell_i}_\Sigma(X, Y_i, A, J_i)
\]
coincide. Hence there is a well-defined rational cobordism class
\[
\mathrm{gw}_\Sigma(X, A)
\]
of rational pseudocycles in $M \times X^n$ with image in $\overset{\circ}{M}\times X^n$, independent of any choices.
\end{theorem}

Up to now, $(\pi : \Sigma \to M, R_\ast)$ always denoted an arbitrary regular nodal family of Riemann surfaces.
For a definition of a Gromov-Witten pseudocycle that can be attributed to the Deligne-Mumford space $\overline{M}_{g,n}$, one has to restrict to (connected, closed) orbifold branched coverings of $\overline{M}_{g,n}$ that branch over the Deligne-Mumford boundary (\cf Definition \ref{Definition_Orbifold_branched_covering}).
\begin{defn}
Let $(\pi_i : \Sigma_i \to M_i, R_{i,\ast})$, $i = 0,1,2$, be orbifold branched coverings of $\overline{M}_{g,n}$ that branch over the Deligne-Mumford boundary, with the $M_i$ closed, connected. \\
$(\pi_0 : \Sigma_0 \to M_0, R_{0, \ast})$ is said to be a \emph{refinement} of $(\pi_1 : \Sigma_1 \to M_1, R_{1,\ast})$, if there exists a morphism of marked nodal families of Riemann surfaces
\[
\xymatrix{
\Sigma_0 \ar[d]_-{\pi_0} \ar[r]^-{\Phi_1} & \Sigma_1 \ar[d]^-{\pi_1} \\
M_0 \ar[r]^-{\phi_1} & M_1\text{.}
}
\]
$(\pi_1 : \Sigma_1 \to M_1, R_{1,\ast})$ and $(\pi_2 : \Sigma_2 \to M_2, R_{2,\ast})$ are said to be \emph{equivalent} if they have a common refinement, \ie if there exists $(\pi_0 : \Sigma_0 \to M_0, R_{0,\ast})$ together with morphisms of marked nodal families of Riemann surfaces
\[
\xymatrix{
& \Sigma_0\ar[d]^-{\pi_0} \ar[rd]^-{\Phi_2}  \ar[ld]_-{\Phi_1} & \\
\Sigma_1 \ar[d]_-{\pi_1} & M_0 \ar[rd]^-{\phi_2} \ar[ld]_-{\phi_1} & \Sigma_2 \ar[d]^-{\pi_2} \\
M_1 & & M_2\text{.}
}
\]
\end{defn}
Now let, as in Section \ref{Section_III.1}, $\overset{\circ\circ}{M} \subseteq \overset{\circ}{M}$ be the top-dimensional stratum of the stratification by orbit type and and let $B_i \subseteq \overset{\circ}{M}$, $i \in \N$, be an at most countable collection of locally closed submanifolds of real codimension at least $2$, covering the complement of $\overset{\circ\circ}{M}$ in $\overset{\circ}{M}$. Also, let $|\mathcal{O}(\overset{\circ\circ}{M})|$ and $|\mathrm{Aut}(\overset{\circ\circ}{M})|$ be as at the end of Section \ref{Section_III.1}.
\clearpage
\begin{defn}
For $k \in \N_0$, denote by
\[
\mathrm{Cob}^k_\Q(M\times X^n, \overset{\circ\circ}{M}\times X^n)
\]
the $\Q$-vector space generated by rational cobordism classes of $k$-dimensional rational pseudocycles in $M\times X^n$ with image in $\overset{\circ\circ}{M}\times X^n$.
\end{defn}
\begin{defn}
Let $(X, \omega)$, $A$ be as before, let $(\pi : \Sigma \to M, R_\ast)$ be an orbifold branched covering of $\overline{M}_{g,n}$ that branches over the Deligne-Mumford boundary with $M$ closed and connected and let $Y, \ell, J, H$ satisfy all the regularity assumptions and also the transversality assumptions for the family $\{\overset{\circ\circ}{M}, B_i\}_{i\in \N}$ in \ref{Theorem_Main_Theorem_1}. Then by the previous two theorems,
\[
\frac{1}{|\mathcal{O}(\overset{\circ\circ}{M})||\mathrm{Aut}(\overset{\circ\circ}{M})|\ell!} \mathrm{gw}^\ell_\Sigma(X, Y, A, J, H)|_{\mathrm{gw}^\ell_\Sigma(\dots)\inv(\overset{\circ\circ}{M}\times X^n)}
\]
defines an element
\[
\mathrm{GW}_\Sigma(X, A) \in \mathrm{Cob}^k_\Q(M\times X^n, \overset{\circ\circ}{M}\times X^n)\text{,}
\]
for
\[
k = \dim_\C(X)\chi + 2c_1(A) - 3\chi + 2n\text{,}
\]
independent of the choice of data as above.
\end{defn}
\begin{theorem}\label{Theorem_Main_Theorem_3}
Let $(\pi_0 : \Sigma_0 \to M_0, R_{0, \ast})$, $(\pi_1 : \Sigma_1 \to M_1, R_{1, \ast})$ be orbifold branched coverings of $\overline{M}_{g,n}$ that branch over the Deligne-Mumford boundary with $M_0, M_1$ closed and connected.
\begin{enumerate}[a)]
  \item\label{Theorem_Main_Theorem_3a} If $(\pi_0 : \Sigma_0 \to M_0, R_{0, \ast})$ is a refinement of $(\pi_1 : \Sigma_1 \to M_1, R_{1,\ast})$ under a morphism
\[
\xymatrix{
\Sigma_0 \ar[d]_-{\pi_0} \ar[r]^-{\Phi} & \Sigma_1 \ar[d]^-{\pi_1} \\
M_0 \ar[r]^-{\phi} & M_1\text{,}
}
\]
then the restriction of $\phi$ to $\overset{\circ\circ}{M}_0$ is a finite covering of $\overset{\circ\circ}{M}_1$, of order $d$, say, and the map
\[
\frac{1}{d}\phi_\ast : \mathrm{Cob}^k_\Q(M_0\times X^n, \overset{\circ\circ}{M}_0\times X^n) \to \mathrm{Cob}^k_\Q(M_1\times X^n, \overset{\circ\circ}{M}_1\times X^n)
\]
defined by composition with $\phi\times \id_{X^n}$ and multiplication with $\frac{1}{d}$ is an isomorphism. Furthermore,
\[
\frac{1}{d}\phi_\ast(\mathrm{GW}_{\Sigma_0}(X, A)) = \mathrm{GW}_{\Sigma_1}(X, A)\text{.}
\]
  \item\label{Theorem_Main_Theorem_3b} If $(\pi_0 : \Sigma_0 \to M_0, R_{0, \ast})$ and $(\pi_1 : \Sigma_1 \to M_1, R_{1,\ast})$ are equivalent, then $\mathrm{Cob}^k_\Q(M_0\times X^n, \overset{\circ\circ}{M}_0\times X^n) \cong \mathrm{Cob}^k_\Q(M_1\times X^n, \overset{\circ\circ}{M}_1\times X^n)$ under an isomorphism that maps $\mathrm{GW}_{\Sigma_0}(X, A)$ to $\mathrm{GW}_{\Sigma_1}(X, A)$.
\end{enumerate}
\end{theorem}
\begin{proof}
\ref{Theorem_Main_Theorem_3b}) follows immediately from \ref{Theorem_Main_Theorem_3a}) and \ref{Theorem_Main_Theorem_3a}) follows in a straightforward way from the constructions in the proofs of the previous two theorems, because all the moduli spaces involved in the definitions of the pseudocycles clearly pull back under finite coverings (over the strata of the stratification by signature) \st regularity is preserved.
\end{proof}

\subsection{Outline of the proofs}\label{Subsection_Outline_of_the_proofs}

The proofs of Theorems \ref{Theorem_Main_Theorem_1} and \ref{Theorem_Main_Theorem_2} are the most technically demanding, so the main outline will be given here with some of the details deferred to the next section.

\subsubsection{The proof of Theorem \ref{Theorem_Main_Theorem_1}}\label{Subsubsection_Proof_Main_Theorem_1}
\begin{enumerate}[Step 1)]
  \item\label{Step_1} The first thing to be shown is the existence of the sets $\mathcal{J}_{\omega,\mathrm{ni}}(X, Y, J_0, E) \subseteq \mathcal{J}_{\omega, \mathrm{ni}}(X, Y, E)$ \st for every $J \in \mathcal{J}_{\omega, \mathrm{ni}}(X, Y, E)$, $\cl{\overset{\circ}{\mathcal{M}}(\tilde{X}^\ell, \tilde{Y}^\ell, A, J, H)}$ is compact, or rather that the closure of $\overset{\circ}{\mathcal{M}}(\tilde{X}^\ell, \tilde{Y}^\ell, A, J, H)$ in $\overline{\mathcal{M}}(\tilde{X}^\ell, A, J, H)$ actually lies in the subspace $\mathcal{M}(\tilde{X}^\ell, A, J, H) \subseteq \overline{\mathcal{M}}(\tilde{X}^\ell, A, J, H)$.
This is a rather straightforward extension of the corresponding compactness results in \cite{MR2399678}, which is carried out in Subsection \ref{Subsection_Compactness}.
Existence of the set $\mathcal{J}_{\omega,\mathrm{ni}}(X, Y, J_0, E)$ is shown in Lemma \ref{Lemma_Nice_J_exist}, the compactness result is given by Lemma \ref{Lemma_Main_compactness_result}.
  \item\label{Step_2} Next, let $H_0 \in \mathcal{H}(\tilde{X}^\ell, \tilde{Y}^\ell)$ be arbitrary and let $\{B_i\}_{i\in \N}$ be a family as in Theorem \ref{Theorem_Main_Theorem_1}.
Because $\pi^\ell_0|_{\overset{\circ}{M}{}^\ell} : \overset{\circ}{M}{}^\ell \to \overset{\circ}{M}$ is a submersion, for every $i\in\N$, $B^\ell_i \definedas (\pi^\ell_0|_{\overset{\circ}{M}{}^\ell})\inv(B_i) \subseteq \overset{\circ}{M}{}^\ell$ is a submanifold of codimension the codimension of $B_i$ in $\overset{\circ}{M}$.
Then by Lemma \ref{Lemma_The_universal_moduli_space} and Lemma \ref{Lemma_Tangency_condition},
\[
\pi^{\mathcal{M}}_{\mathcal{H}} : \overset{\circ}{\mathcal{M}}(\tilde{X}^\ell, \tilde{Y}^\ell, A, J, H_0 + \mathcal{H}^{00}(\tilde{X}^\ell, \tilde{Y}^\ell)) \to H_0 + \mathcal{H}^{00}(\tilde{X}^\ell, \tilde{Y}^\ell)
\]
is a Fredholm map of index $\dim_\C(X)\chi + 2c_1(A) + \dim_\R(M)$ and
\[
\pi^{\mathcal{M}}_M : \overset{\circ}{\mathcal{M}}(\tilde{X}^\ell, \tilde{Y}^\ell, A, J, H_0 + \mathcal{H}^{00}(\tilde{X}^\ell, \tilde{Y}^\ell)) \to \overset{\circ}{M}{}^\ell
\]
is a submersion, so for every $i\in\N$, $(\pi^{\mathcal{M}}_M)\inv(B^\ell_i)$ is a split submanifold of $\overset{\circ}{\mathcal{M}}(\tilde{X}^\ell, \tilde{Y}^\ell, A, J, H_0 + \mathcal{H}^{00}(\tilde{X}^\ell, \tilde{Y}^\ell))$ of codimension the codimension of $B_i$ in $\overset{\circ}{M}$.
So by a simple index calculation,
\[
\pi^\mathcal{M}_{\mathcal{H}} : (\pi^{\mathcal{M}}_M)\inv(B^\ell_i) \to H_0 + \mathcal{H}^{00}(\tilde{X}^\ell, \tilde{Y}^\ell)
\]
is a Fredholm map of index $\dim_\C(X)\chi + 2c_1(A) + \dim_\R(B_i)$.
Hence, by the Sard-Smale theorem, one gets a generic subset of $H \in \mathcal{H}^{00}(\tilde{X}^\ell, \tilde{Y}^\ell)$, \st $\overset{\circ}{\mathcal{M}}(\tilde{X}^\ell, \tilde{Y}^\ell, A, J, H_0 + H)$ is a manifold of dimension $\dim_\C(X)\chi + 2c_1(A) + \dim_\R(M)$ and for every $i\in \N$ one gets a generic subset of $H \in \mathcal{H}^{00}(\tilde{X}^\ell, \tilde{Y}^\ell)$, \st $(\pi^\mathcal{M}_M)\inv(B_i^\ell) \cap (\pi^\mathcal{M}_\mathcal{H})\inv(H_0 + H)$ is a manifold of dimension $\dim_\C(X)\chi + 2c_1(A) + \dim_\R(B_i)$.
Taking the intersection of these at most countably many generic subsets produces a generic subset of $H \in \mathcal{H}^{00}(\tilde{X}^\ell, \tilde{Y}^\ell)$ \st \\
$\overset{\circ}{\mathcal{M}}(\tilde{X}^\ell, \tilde{Y}^\ell, A, J, H_0 + H)$ is a manifold of dimension $\dim_\C(X)\chi + 2c_1(A) + \dim_\R(M)$ and the $(\pi^\mathcal{M}_M)\inv(B_i^\ell) \cap (\pi^\mathcal{M}_\mathcal{H})\inv(H_0 + H)$ for $i \in \N$ are submanifolds of dimensions $\dim_\C(X)\chi + 2c_1(A) + \dim_\R(B_i)$.
  \item\label{Step_3} Furthermore, for such generic $H$ as above, $\overset{\circ}{\mathcal{M}}(\tilde{X}^\ell, \tilde{Y}^\ell, A, J, H_0 + H)$ carries a natural orientation: First, note that $\overset{\circ}{\mathcal{M}}(\tilde{X}^\ell, \tilde{Y}^\ell, A, J, H_0 + \mathcal{H}(\tilde{X}^\ell, \tilde{Y}^\ell))$ carries a natural coorientation as split submanifold of $\overset{\circ}{\mathcal{M}}(\tilde{X}^\ell, A, J, H_0 + \mathcal{H}(\tilde{X}^\ell, \tilde{Y}^\ell))$, since it is the preimage under the evaluation map at the last $\ell$ marked points of $\tilde{Y}^\ell$, which is cooriented in $\tilde{X}^\ell$. \\
Second, for the Fredholm map
\[
\pi^\mathcal{M}_\mathcal{H} : \overset{\circ}{\mathcal{M}}(\tilde{X}^\ell, A, J, \mathcal{H}(\tilde{X}^\ell, \tilde{Y}^\ell)) \to \mathcal{H}(\tilde{X}^\ell, \tilde{Y}^\ell)\text{,}
\]
at a regular point the kernel of its differential is canonically oriented, since it is identified with the kernel of the corresponding Cauchy-Riemann operator, by Lemma A.3.6 in \cite{MR2045629}.
This in turn is oriented by the usual argument as in the proof of Theorem 3.1.5, p.~50, in \cite{MR2045629}.
Hence the kernel of the restriction
\[
\pi^\mathcal{M}_\mathcal{H} : \overset{\circ}{\mathcal{M}}(\tilde{X}^\ell, \tilde{Y}^\ell A, J, \mathcal{H}(\tilde{X}^\ell, \tilde{Y}^\ell)) \to \mathcal{H}(\tilde{X}^\ell, \tilde{Y}^\ell)\text{,}
\]
at a regular point carries an induced orientation as well.
  \item\label{Step_4} At this point, what is left to complete the proof of Theorem \ref{Theorem_Main_Theorem_1}, \ref{Theorem_Main_Theorem_1a}), is to show that the $\Omega$-limit set of $\mathrm{gw}^\ell_\Sigma(X, Y, A, J, H)$ can be covered by manifolds of real dimension at least $2$ less than that of $\overset{\circ}{\mathcal{M}}(\tilde{X}^\ell, \tilde{Y}^\ell, A, J, H)$.
Because of Step \ref{Step_1}) above, it actually suffices to show that this holds for
\[
\partial{\overset{\circ}{\mathcal{M}}(\tilde{X}^\ell, \tilde{Y}^\ell, A, J, H)} \definedas \cl{\overset{\circ}{\mathcal{M}}(\tilde{X}^\ell, \tilde{Y}^\ell, A, J, H)} \setminus \overset{\circ}{\mathcal{M}}(\tilde{X}^\ell, \tilde{Y}^\ell, A, J, H)\text{.}
\]
This is the lengthiest and most technical part of the proof, summed up in Theorem \ref{Theorem_Step_4}, so (a simplified version) will be described below, the actual proof is distributed over Subsections \ref{Subsection_Description_of_the_closure}--\ref{Subsection_Putting_it_all_together}.
  \item\label{Step_5} Before that, one can give the proof of Theorem \ref{Theorem_Main_Theorem_1}, \ref{Theorem_Main_Theorem_1b}), which is fairly straightforward, given the above has been established: \\
Consider the marked nodal families $(\pi'^\ell : \Sigma'^\ell \to M'^\ell, R'^\ell)$, where $\Sigma'^\ell \definedas \Sigma^\ell \times \R$, $M'^\ell \definedas M^\ell\times \R$, $\pi'^\ell \definedas \pi^\ell\times \id_\R$, $(R'^\ell)_j \definedas (R^\ell_j \times \id_\R)$.
These spaces are stratified by taking the product of a stratum of the original space with $\R$.
Correspondingly, define $\tilde{X}'^\ell \definedas \Sigma'^\ell\times X$ and $\tilde{Y}'^\ell \definedas \Sigma'^\ell\times Y$ so that $J_\cdot$ defines an $\omega$-compatible vertical almost complex structure on $\tilde{X}'^\ell$.
But, instead of the spaces $\mathcal{H}(\tilde{Y}')$, $\mathcal{H}^0_{\mathrm{ni}}(\tilde{X}'^\ell, \tilde{Y}'^\ell, J_\cdot)$ and $\mathcal{H}^{00}(\tilde{X}'^\ell, \tilde{Y}'^\ell)$, now consider the spaces
\begin{align*}
\mathcal{H}(\tilde{Y}', H^Y_i) &\definedas \bigl\{H^Y \in \mathcal{H}(\tilde{Y}') \;|\; H^Y|_{\tilde{Y}\times \{t\}} = \begin{cases} H^Y_0 & t \leq 0 \\ H^Y_1 & t \geq 1\end{cases}\bigr\} \\
\mathcal{H}^0_{\mathrm{ni}}(\tilde{X}'^\ell, \tilde{Y}'^\ell, J, H^0_i) & \definedas \bigl\{H^0 \in \mathcal{H}^0_{\mathrm{ni}}(\tilde{X}'^\ell, \tilde{Y}'^\ell, J) \;|\; H^0|_{\tilde{X}^\ell\times \{t\}} = \begin{cases} H^0_0 & t \leq 0 \\ H^0_1 & t \geq 1\end{cases}\bigr\} \\
\mathcal{H}^{00}(\tilde{X}'^\ell, \tilde{Y}'^\ell, H^{00}_i) & \definedas \bigl\{H^{00} \in \mathcal{H}^{00}(\tilde{X}'^\ell, \tilde{Y}'^\ell) \;|\; H^{00}|_{\tilde{X}^\ell\times \{t\}} = \begin{cases} H^{00}_0 & t \leq 0 \\ H^{00}_1 & t \geq 1\end{cases}\bigr\}\text{.}
\end{align*}
These spaces of Hamiltonian perturbations are large enough for all the transversality results to hold, because for $t \leq 0$ or $t\geq 1$, transversality holds by choice of $H^Y_i$, $H^0_i$ and $H^{00}_i$ and for $0 < t < 1$ one is free in the choice of perturbation.
For the analogue of Lemma \ref{Lemma_Reduction_to_vanishing_A} to hold, one possibly has to replace $D^\ast$ by $D^\ast + 1$.
So for generic choices of perturbations in these spaces, as above, one gets strata-wise cobordisms between the moduli spaces associated to $H^Y_0, H^0_0, H^{00}_0$ and $H^Y_1, H^0_1, H^{00}_1$.
\end{enumerate}
This finishes the proof of Theorem \ref{Theorem_Main_Theorem_1}, modulo the details from Steps \ref{Step_1}) and \ref{Step_4}) which can be found in Subsections \ref{Subsection_Compactness}--\ref{Subsection_Putting_it_all_together}, esp.~Lemmas \ref{Lemma_Nice_J_exist} and \ref{Lemma_Main_compactness_result} for Step \ref{Step_1}) as well as Theorem \ref{Theorem_Step_4} for Step \ref{Step_4}).

A short description of what happens in the proof of Step \ref{Step_4} is as follows: \\
What can immediately be said about $\cl\overset{\circ}{\mathcal{M}}(\tilde{X}^\ell, \tilde{Y}^\ell, A, J, H)$ is that it is contained in 
\[
\{ u\in \mathcal{M}(\tilde{X}^\ell, A, J, H) \;|\; \im(u\circ T^\ell_j)\subseteq \tilde{Y}^\ell, j = 1, \dots, \ell\}\text{.}
\]
The problem now is that if $u \in \mathcal{M}(\tilde{X}^\ell, A, J, H)$, with $\pi^{\mathcal{M}}_M(u) = b$, lies in $\cl\overset{\circ}{\mathcal{M}}(\tilde{X}^\ell, \tilde{Y}^\ell, A, J, H)$, then it cannot be excluded that some (or all, esp.~in the case that $\Sigma_b^\ell$ is smooth) of the components of $\Sigma_b$ are mapped by $u$ into $\tilde{Y}^\ell$.
The condition that all the Hamiltonian perturbations must lie in $\mathcal{H}(\tilde{X}^\ell, \tilde{Y}^\ell)$ then prevents one from achieving transversality for the universal Cauchy-Riemann operator as well as transversality to $\tilde{Y}^\ell$ of the evaluation maps at the $T^\ell_j$.
To present these problems and their solutions in a bit more detail, assume for simplicity that $\Sigma_b^\ell = \Sigma_b^{\ell, X} \cup \Sigma_b^{\ell,Y}$ has two smooth components of Euler characteristics $\chi^X$ and $\chi^Y$, respectively, with one node $N_b \in \Sigma_b^\ell$ and that $u(\Sigma_b^{\ell,Y}) \subseteq \tilde{Y}^\ell$ as well as $u(\Sigma_b^{\ell,X}) \cap (\tilde{X}^\ell \setminus \tilde{Y}^\ell) \neq \emptyset$.
Furthermore, assume that $\ell = \ell^X + \ell^Y$, $\ell^X,\ell^Y \geq 1$, and $T^\ell_1(b), \dots, T^\ell_{\ell^X}(b) \in \Sigma_b^{\ell,X}$, $T^\ell_{\ell^X+1}(b), \dots, T^\ell_{\ell}(b) \in \Sigma_b^{\ell,Y}$.
\def\svgwidth{\textwidth}
\ifpdf
\begingroup%
  \makeatletter%
  \providecommand\color[2][]{%
    \errmessage{(Inkscape) Color is used for the text in Inkscape, but the package 'color.sty' is not loaded}%
    \renewcommand\color[2][]{}%
  }%
  \providecommand\transparent[1]{%
    \errmessage{(Inkscape) Transparency is used (non-zero) for the text in Inkscape, but the package 'transparent.sty' is not loaded}%
    \renewcommand\transparent[1]{}%
  }%
  \providecommand\rotatebox[2]{#2}%
  \ifx\svgwidth\undefined%
    \setlength{\unitlength}{3865.85bp}%
    \ifx\svgscale\undefined%
      \relax%
    \else%
      \setlength{\unitlength}{\unitlength * \real{\svgscale}}%
    \fi%
  \else%
    \setlength{\unitlength}{\svgwidth}%
  \fi%
  \global\let\svgwidth\undefined%
  \global\let\svgscale\undefined%
  \makeatother%
  \begin{picture}(1,0.5575609)%
    \put(0,0){\includegraphics[width=\unitlength]{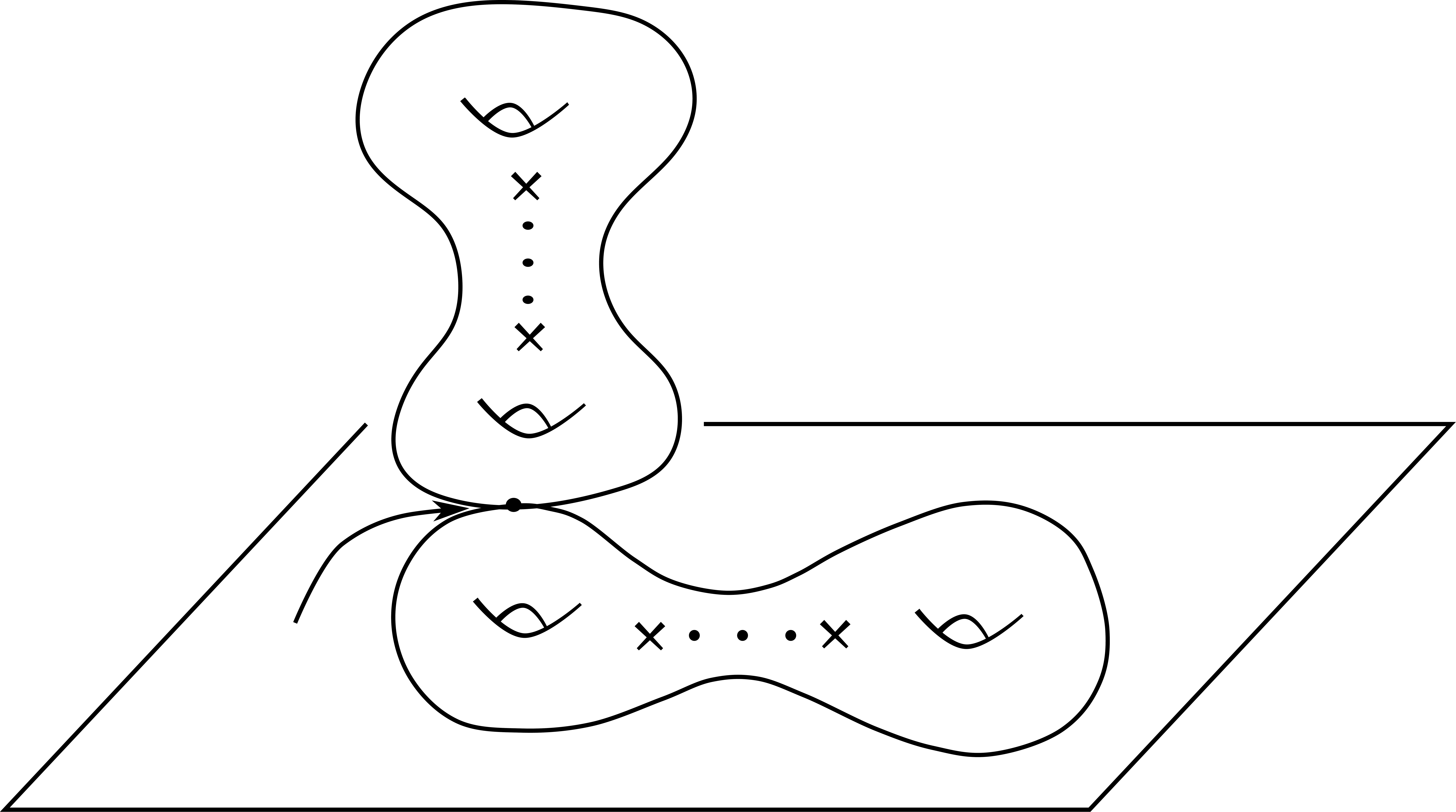}}%
    \put(0.27559037,0.49952746){\color[rgb]{0,0,0}\makebox(0,0)[lb]{\smash{$\Sigma_b^{\ell,X}$}}}%
    \put(0.65514386,0.17217616){\color[rgb]{0,0,0}\makebox(0,0)[lb]{\smash{$\Sigma_b^{\ell,Y}$}}}%
    \put(0.17553617,0.10148571){\color[rgb]{0,0,0}\makebox(0,0)[lb]{\smash{$N_b$}}}%
    \put(0.36911928,0.44515015){\color[rgb]{0,0,0}\makebox(0,0)[lb]{\smash{$T_1^\ell(b)$}}}%
    \put(0.36911928,0.29071863){\color[rgb]{0,0,0}\makebox(0,0)[lb]{\smash{$T^\ell_{\ell^X}(b)$}}}%
    \put(0.88135341,0.21024024){\color[rgb]{0,0,0}\makebox(0,0)[lb]{\smash{$Y$}}}%
    \put(0.64644353,0.41687398){\color[rgb]{0,0,0}\makebox(0,0)[lb]{\smash{$X\setminus Y$}}}%
    \put(0.38434495,0.09496037){\color[rgb]{0,0,0}\makebox(0,0)[lb]{\smash{$T^\ell_{\ell^X+1}$}}}%
    \put(0.58989114,0.09278532){\color[rgb]{0,0,0}\makebox(0,0)[lb]{\smash{$T^\ell_\ell(b)$}}}%
  \end{picture}%
\endgroup%

\else
\begingroup%
  \makeatletter%
  \providecommand\color[2][]{%
    \errmessage{(Inkscape) Color is used for the text in Inkscape, but the package 'color.sty' is not loaded}%
    \renewcommand\color[2][]{}%
  }%
  \providecommand\transparent[1]{%
    \errmessage{(Inkscape) Transparency is used (non-zero) for the text in Inkscape, but the package 'transparent.sty' is not loaded}%
    \renewcommand\transparent[1]{}%
  }%
  \providecommand\rotatebox[2]{#2}%
  \ifx\svgwidth\undefined%
    \setlength{\unitlength}{3865.85bp}%
    \ifx\svgscale\undefined%
      \relax%
    \else%
      \setlength{\unitlength}{\unitlength * \real{\svgscale}}%
    \fi%
  \else%
    \setlength{\unitlength}{\svgwidth}%
  \fi%
  \global\let\svgwidth\undefined%
  \global\let\svgscale\undefined%
  \makeatother%
  \begin{picture}(1,0.5575609)%
    \put(0,0){\includegraphics[width=\unitlength]{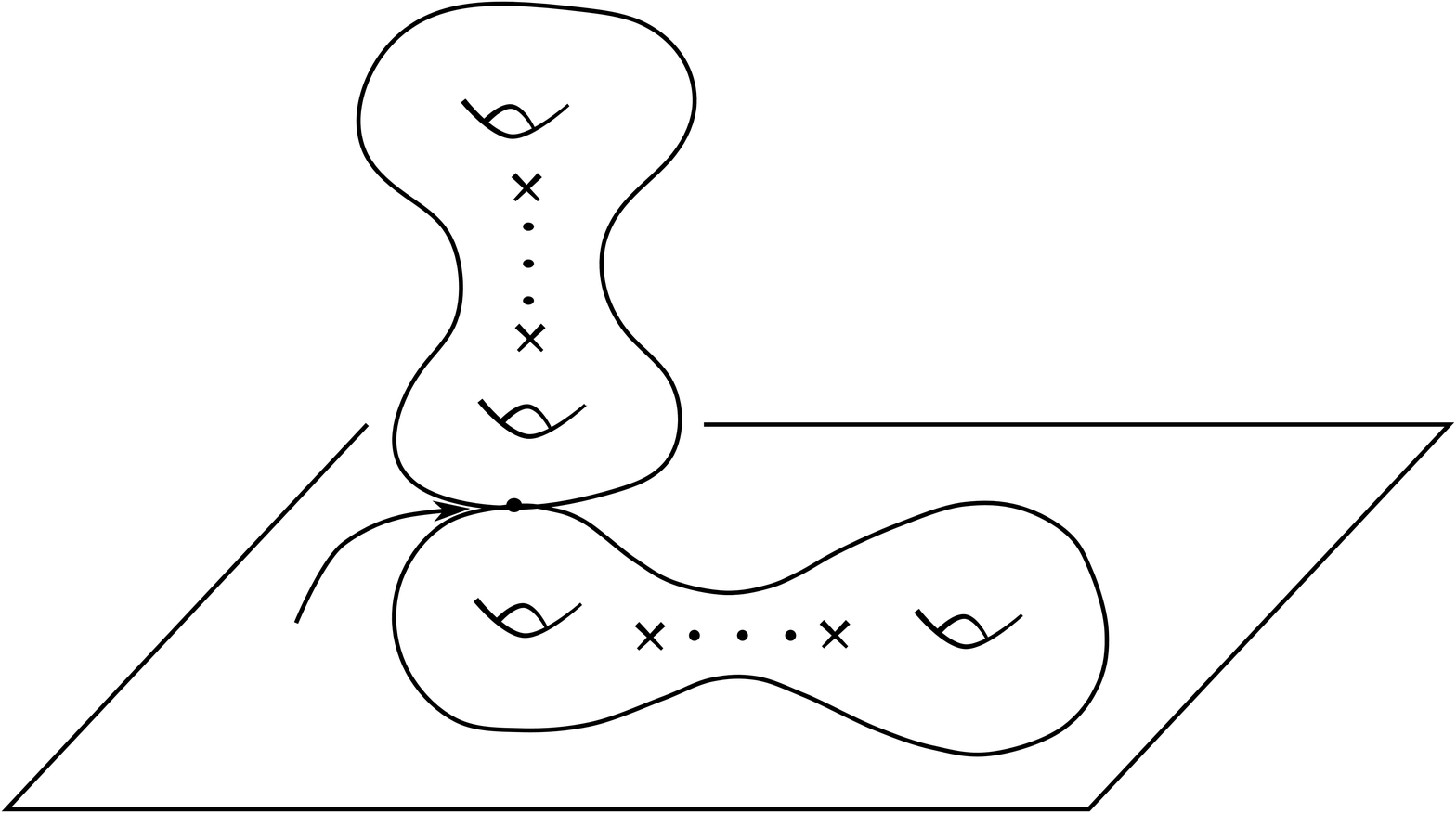}}%
    \put(0.27559037,0.49952746){\color[rgb]{0,0,0}\makebox(0,0)[lb]{\smash{$\Sigma_b^{\ell,X}$}}}%
    \put(0.65514386,0.17217616){\color[rgb]{0,0,0}\makebox(0,0)[lb]{\smash{$\Sigma_b^{\ell,Y}$}}}%
    \put(0.17553617,0.10148571){\color[rgb]{0,0,0}\makebox(0,0)[lb]{\smash{$N_b$}}}%
    \put(0.36911928,0.44515015){\color[rgb]{0,0,0}\makebox(0,0)[lb]{\smash{$T_1^\ell(b)$}}}%
    \put(0.36911928,0.29071863){\color[rgb]{0,0,0}\makebox(0,0)[lb]{\smash{$T^\ell_{\ell^X}(b)$}}}%
    \put(0.88135341,0.21024024){\color[rgb]{0,0,0}\makebox(0,0)[lb]{\smash{$Y$}}}%
    \put(0.64644353,0.41687398){\color[rgb]{0,0,0}\makebox(0,0)[lb]{\smash{$X\setminus Y$}}}%
    \put(0.38434495,0.09496037){\color[rgb]{0,0,0}\makebox(0,0)[lb]{\smash{$T^\ell_{\ell^X+1}$}}}%
    \put(0.58989114,0.09278532){\color[rgb]{0,0,0}\makebox(0,0)[lb]{\smash{$T^\ell_\ell(b)$}}}%
  \end{picture}%
\endgroup%

\fi
Let $U^\ell$ be a neighbourhood of $b$ in the stratum of the stratification by signature on $M^\ell$ that contains $b$ and \st $\Sigma^\ell_{U^\ell} \definedas \Sigma^\ell|_{U^\ell} = \Sigma^{\ell, X}_{U^\ell} \cup \Sigma^{\ell, Y}_{U^\ell}$ for a desingularisation of $\Sigma^\ell_{U^\ell}$ where the nodal points are given in the form of two sections $N^X : U^\ell \to \Sigma^{\ell,X}_{U^\ell}$ and $N^Y : U^\ell \to \Sigma^{\ell,Y}_{U^\ell}$.
Let furthermore $u^X \definedas u|_{\Sigma^{\ell,X}_b}$, $u^Y \definedas u|_{\Sigma^{\ell,Y}_b}$ and $A^X \definedas [u^X], A^Y \definedas [u^Y]$, \st $A = A^X + A^Y$.
The first problem now is that $u^Y \in \mathcal{M}(\tilde{X}^\ell|_{\Sigma^{\ell,Y}_{U^\ell}}, A^Y, J, \mathcal{H}(\tilde{X}^\ell, \tilde{Y}^\ell))$ need not be a regular point of the universal Cauchy-Riemann operator, because the space $\mathcal{H}(\tilde{X}^\ell, \tilde{Y}^\ell)$ of Hamiltonian perturbations is not large enough.
Instead, one has to consider $u^Y$ as lying in $\mathcal{M}(\tilde{Y}^\ell|_{\Sigma^{\ell,Y}_{U^\ell}}, A^Y, J, \mathcal{H}(\tilde{X}^\ell, \tilde{Y}^\ell))$. In Subsection \ref{Subsection_Vanishing_homology_class} it will be shown that, provided the degree of $Y$ is large enough, for generic $H^Y \in \mathcal{H}(\tilde{Y})$ (remember that this is considered as a subset of $\mathcal{H}(\tilde{X}^\ell, \tilde{Y}^\ell)$), $\mathcal{M}(\tilde{Y}^\ell|_{\Sigma^{\ell,Y}_{U^\ell}}, A^Y, J, H^Y)$ is empty for $A^Y \neq 0$ (by extending an argument from \cite{MR2399678}) and a manifold of dimension $\dim_\C(Y)\chi^Y + \dim_\R(U^\ell)$ for $A^Y = 0$. \\
$u^X$ on the other hand, after a generic perturbation $H^{00}\in \mathcal{H}^{00}(\tilde{X}^\ell, \tilde{Y}^\ell)$, lies in a submanifold of $\mathcal{M}(\tilde{X}^\ell|_{\Sigma^{\ell,X}_{U^\ell}}, A, J, H^Y + H^{00})$ of dimension $\dim_\C(X)\chi^X + 2c_1(A) + \dim_\R(U^\ell) - 2\ell^X$, cut out by the condition $\im(\ev^{T^\ell_{j}}) \subseteq \tilde{Y}^\ell$ for $j = 1, \dots, \ell^X$.
Taking into account the compatibility of $u^X$ and $u^Y$ at $N_b$, after another generic perturbation in $\mathcal{H}^{00}(\tilde{X}^\ell, \tilde{Y}^\ell)$, $u$ lies in a submanifold of $\mathcal{M}(\tilde{X}^\ell|_{\Sigma^{\ell}_{U^\ell}}, A, J, H^Y + H^{00})$ of dimension
\begin{align*}
& \dim_\C(X)\chi^X + 2c_1(A) + \dim_\R(U^\ell) - 2\ell^X \;+ \\
& +\; \dim_\C(Y)\chi^Y + \dim_\R(U^\ell) \underbrace{\;-\; (\dim_\C(X) + \dim_\R(U^\ell))}_{\text{from } \ev^{N^X}(u^X) = \ev^{N^Y}(u^Y)} \\
= & \dim_\C(X)\chi + 2c_1(A) + \dim_\R(U^\ell) - 2\ell^X - \chi^Y \\
= & \dim_\R(\overset{\circ}{M}(\tilde{X}^\ell, \tilde{Y}^\ell, A, J, H^Y + H^{00})) - \underbrace{\codim_\R^M(U^\ell)}_{=\; 2} + 2\ell^Y - \chi^Y\text{,}
\end{align*}
which is too large by $2\ell^Y - \chi^Y$ ($\chi^Y$ will be negative in higher genus) to serve as a covering of part of the $\Omega$-limit set that contains the image of $u$. \\
The solution here is a more precise compactness result, from \cite{MR1954264} or \cite{MR2026549}, describing the elements of $\cl\overset{\circ}{\mathcal{M}}(\tilde{X}^\ell, \tilde{Y}^\ell, A, J, H)$ in more detail.
This will be presented in the detail needed for the present purposes in Subsection \ref{Subsection_Refined_compactness_result}.
The upshot of this is the following: Restricting the Cauchy-Riemann operator to sections of $\tilde{X}^\ell|_{\Sigma^{\ell,Y}_b}$, its linearisation at $u^Y$ is a Cauchy-Riemann operator on the vector bundle $(u^Y)^\ast V\tilde{X}^\ell$, where $V\tilde{X}^\ell$ is the vertical tangent bundle of $\tilde{X}^\ell$.
Restricting that Cauchy-Riemann operator to sections of $(u^Y)^\ast (V\tilde{Y}^\ell)^{\perp_\omega}$ produces a Cauchy-Riemann operator $L_{u^Y}$ on $(u^Y)^\ast (V\tilde{Y}^\ell)^{\perp_\omega}$.
The compactness result cited above then implies that $u$ actually comes with a meromorphic section $\xi_u$ of $(u^Y)^\ast (V\tilde{Y}^\ell)^{\perp_\omega}$.
This meromorphic section has simple zeroes precisely at the $\ell^Y$ points $T^\ell_{\ell^X + 1}(b), \dots, T^\ell_{\ell}(b)$ (coming from the transverse intersections with $\tilde{Y}^\ell$ of the elements of the approximating sequence) and one pole at $N^Y_b$.
Since $[u^Y] = 0$, the first Chern class of $(u^Y)^\ast (V\tilde{Y}^\ell)^{\perp_\omega}$ vanishes and hence the orders of zeroes and poles of $\xi_u$ need to coincide, hence $\xi_u$ has a pole of order $\ell^Y$ at $N^Y_b$. \\
Another consequence of the compactness result is that $u^X$ is tangent to $\tilde{Y}^\ell$ at $N^X_b$ of the same order as the order of the pole of the meromorphic section at $N^Y_b$.
This tangency condition allows one to cut down the dimension of the moduli space containing $u$ by another term $2(\ell^Y - 1)$ (the $-1$ is due to the fact that the condition $u(N^X_b) \in \tilde{Y}^\ell$ is no longer transversely cut out and was already accounted for by the matching condition for $u^X$ and $u^Y$ at $N_b$). \\
Also, one can consider $B \definedas \mathcal{M}(\tilde{Y}^\ell|_{\Sigma^{\ell,Y}_{U^\ell}}, 0, J, H^Y)$ as the base of the family of Riemann surfaces $S \definedas (\pi^\mathcal{M}_M)^\ast \Sigma^{\ell,Y}_{U^\ell}$.
Over $S$ one has the complex line bundle $Z$ that is given, on the fibre $S_{u^Y}$ over $u^Y\in B$, by $(u^Y)^\ast (V\tilde{Y}^\ell)^{\perp_\omega}$.
For every $u^Y \in B$, $L_{u^Y}$ then defines a Cauchy-Riemann operator for sections of $Z|_{S_{u^Y}} \to S_{u^Y}$.
The meromorphic sections from above with the given behaviour for the zeroes and poles then can be described as elements of a moduli space for a Fredholm problem associated to $B, S, Z, L_\cdot$, \st the Fredholm index of each $L_{u^Y}$, restricted to sections with the given behaviour for poles and zeroes, is equal to $\chi^Y$.
Associated to this Fredholm problem is furthermore a (transversely cut out) universal moduli space, where as spaces of perturbations the spaces $H^Y + \mathcal{H}^0_{\mathrm{ni}}(\tilde{X}^\ell, \tilde{Y}^\ell, J)$ are used.
Finally, the condition of normal integrability ensures that each $L_{u^Y}$ is complex linear \wrt the standard action of $\C$ on $(V\tilde{Y}^\ell)^{\perp_\omega}$.
Hence these moduli spaces of meromorphic sections come with a free $\C^\ast$-action under which the projection to $B$ is invariant. \\
In conclusion, for generic $H^0 \in \mathcal{H}^0_{\mathrm{ni}}(\tilde{X}^\ell, \tilde{Y}^\ell, J)$, the subset of $\mathcal{M}(\tilde{Y}^\ell|_{\Sigma^{\ell,Y}_{U^\ell}}, 0, J, H^Y + H^0) = \mathcal{M}(\tilde{Y}^\ell|_{\Sigma^{\ell,Y}_{U^\ell}}, 0, J, H^Y)$ that contains those $u^Y$ that come from $\cl\overset{\circ}{\mathcal{M}}(\tilde{X}^\ell, \tilde{Y}^\ell, A, J, H^Y + H^0)$ is covered by a moduli space of meromorphic sections as above which has real dimension $\dim_\R(\mathcal{M}(\tilde{Y}^\ell|_{\Sigma^{\ell,Y}_{U^\ell}}, 0, J, H^Y + H^0)) + \chi^Y - 2$, where the $-2$ comes from dividing out the free $\C^\ast$-action. \\
Note that all of the above is completely oblivious to changes in the Hamiltonian perturbation that come from $\mathcal{H}^{00}(\tilde{X}^\ell, \tilde{Y}^\ell)$, \ie under replacing $H^Y + H^0$ by $H^Y + H^0 + H^{00}$ for $H^{00} \in \mathcal{H}^{00}(\tilde{X}^\ell, \tilde{Y}^\ell)$. \\
Taken together, these facts show that for appropriate choices of $J, H^Y, H^0$ and $H^{00}$, $u$ is actually contained in a moduli space of dimension
\begin{align*}
& \dim_\R(\overset{\circ}{M}(\tilde{X}^\ell, \tilde{Y}^\ell, A, J, H^Y + H^0 + H^{00})) - 2 + 2\ell^Y - \chi^Y  \;- \\
& -\; 2(\ell^Y - 1) \;+ \\
& +\; \chi^Y - 2 \\
= & \dim_\R(\overset{\circ}{M}(\tilde{X}^\ell, \tilde{Y}^\ell, A, J, H^Y + H^0 + H^{00})) - 2\text{,}
\end{align*}
which finally is good enough to serve as a covering of a part of the $\Omega$-limit set containing $u$.

\subsubsection{The proof of Theorem \ref{Theorem_Main_Theorem_2}}\label{Subsubsection_Proof_Main_Theorem_2}

The proof of Theorem \ref{Theorem_Main_Theorem_2} relies on the following observation, which also explains in which sense the moduli spaces $\overset{\circ}{\mathcal{M}}(\tilde{X}^\ell, \tilde{Y}^\ell, A, J, H)$ used in the definition of the rational pseudocycles $\mathrm{gw}_\Sigma(X, A)$ and $\mathrm{GW}_\Sigma(X, A)$ can be regarded as perturbations of the moduli spaces $\overset{\circ}{\mathcal{M}}(\tilde{X}, A, J, H)$. These are the moduli spaces one is originally interested in, since $\overset{\circ}{\mathcal{M}}(\tilde{X}, A, J, 0)/_\sim$ is just (canonically identified with) the moduli space $\mathcal{M}_{g,n}(X, A, J)$ of smooth holomorphic curves of genus $g$ in $X$ with $n$ marked points from the introduction, where $\sim$ denotes the equivalence relation generated by biholomorphic maps between the fibres of $\Sigma|_{\overset{\circ}{M}} \to \overset{\circ}{M}$ that respect the markings.
\begin{lemma}\label{Lemma_ell_factorial_sheeted_covering}
For $A\neq 0$ and $D$ large enough, for generic $H^Y \in \mathcal{H}(\tilde{Y}) \subseteq \mathcal{H}(\tilde{X}, \tilde{Y})$ there exists a generic subset (depending on $H^Y$) of $\mathcal{H}^{00}(\tilde{X}, \tilde{Y})$ \st for each $H^{00}$ in this subset, outside a subset of $\overset{\circ}{\mathcal{M}}(\tilde{X}, A, J, H^Y + H^{00})$ with complement of codimension at least $2$,
\begin{align*}
\overset{\circ}{\mathcal{M}}(\tilde{X}^\ell, \tilde{Y}^\ell, A, J, (\hat{\pi}^\ell_0)^\ast (H^Y + H^{00})) &\to \overset{\circ}{\mathcal{M}}(\tilde{X}, A, J, H^Y + H^{00}) \\
u &\mapsto  u\circ ((\hat{\pi}^\ell_0)^\ast_{\pi^\mathcal{M}_M(u)})\inv
\end{align*}
is an $\ell!$-sheeted covering.
\end{lemma}
\begin{proof}
As will be shown in Subsection \ref{Subsection_Vanishing_homology_class}, for $A\neq 0$ and $D$ large enough, for generic $H^Y \in \mathcal{H}(\tilde{Y})$ and for any $H^{00} \in \mathcal{H}^{00}(\tilde{X}, \tilde{Y})$ one can assume that no section in $\overset{\circ}{\mathcal{M}}(\tilde{X}, A, J, H^Y + H^{00})$ lies completely in $\tilde{Y}$.
Note that over $\overset{\circ}{M}{}^\ell$, $\hat{\pi}^\ell_0$ is a fibrewise isomorphism, hence the space $(\hat{\pi}^\ell_0)^\ast(H^Y + \mathcal{H}^{00}(\tilde{X}, \tilde{Y}))$ is ``large enough'' for all the transversality results in the following to hold for generic $H \definedas H^Y + H^{00}$, in particular all the strata of $\overset{\circ}{\mathcal{M}}(\tilde{X}, A, J, H)$ can be assumed to be manifolds of the correct dimensions.
Now every element of $\overset{\circ}{\mathcal{M}}(\tilde{X}, A, J, H)$ has intersection number $\ell$ with $\tilde{Y}$, so $\overset{\circ}{\mathcal{M}}(\tilde{X}, A, J, H)$ is covered by a union of spaces
\[
\mathcal{M}(\tilde{X}^{\ell'}|_{\overset{\circ}{M}{}^{\ell'}}, (\tilde{Y}^{\ell'})^{(-1, \dots, -1, t_1, \dots, t_{\ell'})}, A, J, (\hat{\pi}^{\ell'}_0)^\ast H)
\]
for $1 \leq \ell' \leq \ell$ and $\sum_{i=1}^{\ell'}(t_i + 1) = \ell$.
For $\ell' = \ell$ this is just the space $\overset{\circ}{\mathcal{M}}(\tilde{X}^\ell, \tilde{Y}^\ell, A, J, (\hat{\pi}^\ell_0)^\ast H)$, where the fibre over a point is given by the $\ell!$ choices to label the $\ell$ intersection points with $\tilde{Y}$.
And for $\ell' < \ell$, by Lemma \ref{Lemma_Tangency_condition}, this space has dimension at least two less than $\overset{\circ}{\mathcal{M}}(\tilde{X}^\ell, \tilde{Y}^\ell, A, J, (\hat{\pi}^\ell_0)^\ast H)$.
\end{proof}

The proof of Theorem \ref{Theorem_Main_Theorem_2}, which extends the corresponding results from \cite{MR2399678}, Section 10, then goes as follows, with some of the technical results deferred to Subsection \ref{Subsection_Proof_Main_Theorem_2}:
\begin{enumerate}[Step 1)]
  \item\label{Theorem_Main_Theorem_2_Step1} Given two Donaldson pairs $(Y_i, J_i)$ of degrees $D_i$, for $i = 0,1$, by Lemma \ref{Lemma_Existence_Transverse_Hypersurface} in Subsection \ref{Subsection_Proof_Main_Theorem_2}, there exists $\overline{D} \in \N$ and a hypersurface $\overline{Y} \subseteq X$ of degree $\overline{D}$ that intersects $Y_0$ and $\phi_1(Y_1)$ $\varepsilon$-transversely, where $\phi_t$, $t\in [0,1]$, is an isotopy of $X$ through symplectomorphisms.
Furthermore, there exist $\omega$-compatible almost complex structures $\overline{J}_0$ and $\overline{J}_1$ \st $(\overline{Y}, \overline{J}_i)$ are Donaldson pairs and the two pseudocycles $\mathrm{gw}^{\overline{\ell}}_\Sigma(X, \overline{Y}, A, \overline{J}_i)$ with $\overline{\ell} \definedas \overline{D}\omega(A)$ and $i = 0,1$, are cobordant.
To see that the latter holds, one chooses, using Lemma \ref{Lemma_Existence_Transverse_Hypersurface} \ref{Donaldson_quadruples_a}), a subdivision $0 = t_0 < t_1 < \cdots < t_k = 1$ of $[0,1]$ \st there exist $J'_i \in \mathcal{J}_{\omega, \mathrm{ni}}(X, \overline{Y}, \overline{J}_{t_i}, E) \cap \mathcal{J}_{\omega, \mathrm{ni}}(X, \overline{Y}, \overline{J}_{t_{i+1}}, E)$, $i = 0, \dots, k-1$, and connects, by Lemma \ref{Lemma_Nice_J_exist}, $J'_i$ and $J'_{i+1}$ by a path in $\mathcal{J}_{\omega, \mathrm{ni}}(X, \overline{Y}, E)$.
Theorem \ref{Theorem_Main_Theorem_1} then shows equivalence of the pseudocycles associated to $(\overline{Y}, \overline{J}_0)$ and $(\overline{Y}, \overline{J}_1)$. \\
Now note that the pseudocycles associated to $(Y_1,J_1)$ and $(\phi_1(Y_1), (\phi_1)_\ast J_1)$ are equivalent. For $(\phi_1)_\ast A = A$, since $\phi_1$ is isotopic to the identity, and $\phi_1$ induces a well defined map between the corresponding moduli spaces.
Theorem \ref{Theorem_Main_Theorem_2} then follows once it has been shown that the rational pseudocycles $\frac{1}{\ell_0!}\mathrm{gw}^\ell_\Sigma(X, Y_0, A, J_0)$ and $\frac{1}{\overline{\ell}^\ast}\mathrm{gw}^{\overline{\ell}}_\Sigma(X, \overline{Y}, A, \overline{J}_0)$ are rationally cobordant as well as the rational pseudocycles $\frac{1}{\ell_1!}\mathrm{gw}^\ell_\Sigma(X, \phi_1(Y_1), A, (\phi_1)_\ast(J_1))$ and $\frac{1}{\overline{\ell}!}\mathrm{gw}^{\overline{\ell}}_\Sigma(X, \overline{Y}, A, \overline{J}_1)$.
This is clearly symmetric in $Y_0$ and $Y_1$, so it suffices to consider $Y_0$ and one can drop the subscript ${}_0$.
  \item\label{Theorem_Main_Theorem_2_Step2} The desired rational cobordism is defined in the following way: First of all define $\hat{\ell} \definedas \ell + \overline{\ell}$ and consider again the bundles $\tilde{X}^k\definedas \Sigma^k\times X$, $\tilde{Y}^k \definedas \Sigma^k\times Y$ and $\tilde{\overline{Y}}{}^k \definedas \Sigma^k\times \overline{Y}$, for $k = \ell, \overline{\ell}, \hat{\ell}$. Then for $J' \in \mathcal{J}_{\omega}(X, Y, J, E) \cap \mathcal{J}_{\omega}(X, \overline{Y}, \overline{J}, E)$ as in Lemma \ref{Lemma_Existence_Transverse_Hypersurface} \ref{Donaldson_quadruples_b}) and after restricting to Hamiltonian perturbations in $\mathcal{H}(\tilde{X}^k, \tilde{Y}^k, \tilde{\overline{Y}}{}^k) \definedas \mathcal{H}(\tilde{X}^k, \tilde{Y}^k) \cap \mathcal{H}(\tilde{X}^k, \tilde{\overline{Y}}{}^k)$ one considers moduli spaces of the form
\newlength{\StepX}
\settowidth{\StepX}{Step 1)\;}
\begin{align*}
\hspace{-\StepX}\overset{\circ}{\mathcal{M}}(\tilde{X}^{\hat{\ell}}, \tilde{Y}^{\hat{\ell}}, \tilde{\overline{Y}}{}^{\hat{\ell}}, A, J', H) \definedas \{ u \in \mathcal{M}(\tilde{X}^{\hat{\ell}}|_{\overset{\circ}{M}{}^{\hat{\ell}}}, A, J', H) \;|\; & \im(u\circ T^{\hat{\ell}}_j) \subseteq \tilde{Y}^{\hat{\ell}},\; j=1,\dots, \ell, \\
& \im(u\circ T^{\hat{\ell}}_j) \subseteq \tilde{\overline{Y}}{}^{\hat{\ell}},\; j=\ell + 1,\dots, \hat{\ell}, \\
& \im(u) \cap \tilde{X}^{\hat{\ell}} \setminus (\tilde{Y}^{\hat{\ell}} \cup \tilde{\overline{Y}}{}^{\hat{\ell}}) \neq \emptyset\}
\end{align*}
and defines
\[
\mathrm{gw}^{\hat{\ell}}_\Sigma(X, Y, \overline{Y}, A, J', H) : \overset{\circ}{\mathcal{M}}(\tilde{X}^{\hat{\ell}}, \tilde{Y}^{\hat{\ell}}, \tilde{\overline{Y}}{}^{\hat{\ell}}, A, J', H) \to M \times X^n
\]
via evaluation at the first $n$ marked points, as before.
  \item\label{Theorem_Main_Theorem_2_Step3} The following extension of Theorem \ref{Theorem_Main_Theorem_1} holds, see Theorem \ref{Theorem_Main_Theorem_3_Extension_of_1} in Subsection \ref{Subsection_Proof_Main_Theorem_2}: There are well-defined subsets of $H \in \mathcal{H}(\tilde{X}^\ell, \tilde{Y}^\ell, \tilde{\overline{Y}}{}^\ell)$, $\overline{H} \in \mathcal{H}(\tilde{X}^{\overline{\ell}}, \tilde{Y}^{\overline{\ell}}, \tilde{\overline{Y}}{}^{\overline{\ell}})$ and $\hat{H} \in \mathcal{H}(\tilde{X}^{\hat{\ell}}, \tilde{Y}^{\hat{\ell}}, \tilde{\overline{Y}}{}^{\hat{\ell}})$ \st no elements in $\mathcal{M}(\tilde{X}^\ell|_{\overset{\circ}{M}{}^\ell}, A, J', H)$, $\mathcal{M}(\tilde{X}^{\overline{\ell}}|_{\overset{\circ}{M}{}^{\overline{\ell}}}, A, J', \overline{H})$ and $\mathcal{M}(\tilde{X}^{\hat{\ell}}|_{\overset{\circ}{M}{}^{\hat{\ell}}}, A, J', \hat{H})$ have image in $\tilde{Y}^\ell\cup \tilde{\overline{Y}}{}^\ell$, $\tilde{Y}^{\overline{\ell}}\cup \tilde{\overline{Y}}{}^{\overline{\ell}}$ and $\tilde{Y}^{\hat{\ell}} \cup \tilde{\overline{Y}}{}^{\hat{\ell}}$, respectively.
Furthermore, $\mathrm{gw}^\ell(X, Y, A, J', H)$, $\mathrm{gw}^{\overline{\ell}}(X, Y, A, J', \overline{H})$ and $\mathrm{gw}^{\hat{\ell}}_\Sigma(X, Y, \overline{Y}, A, J', \hat{H})$ are well-defined pseudocycles of dimension $\dim_\C(X)\chi + 2c_1(A) + \dim_\R(M)$, independent (up to cobordism) of the choice of $H$, $\overline{H}$ and $\hat{H}$.
Hence there is a well-defined cobordism class
\[
\mathrm{gw}_\Sigma(X, Y, \overline{Y}, A, J')
\]
of rational pseudocycles given by
\[
\frac{1}{\overline{\ell}!\ell!}\mathrm{gw}^{\hat{\ell}}_\Sigma(X, Y, \overline{Y}, A, J', \hat{H})\text{.}
\]
This is the most technical part of the proof, deferred to Subsection \ref{Subsection_Proof_Main_Theorem_2}. It follows the same line of argument as the proof of Theorem \ref{Theorem_Main_Theorem_1}, just with slightly more complex notation.
  \item By the argument from the previous lemma, for appropriate $H \in \mathcal{H}(\tilde{X}^\ell, \tilde{Y}^\ell, \tilde{\overline{Y}}{}^\ell)$ and $\overline{H} \in \mathcal{H}(\tilde{X}^{\overline{\ell}}, \tilde{Y}^{\overline{\ell}}, \tilde{\overline{Y}}{}^{\overline{\ell}})$,
\begin{align*}
\overset{\circ}{\mathcal{M}}(\tilde{X}^{\hat{\ell}}, \tilde{Y}^{\hat{\ell}}, \tilde{\overline{Y}}{}^{\hat{\ell}}, A, J', (\hat{\pi}^{\hat{\ell}}_\ell)^\ast H) &\to \overset{\circ}{\mathcal{M}}(\tilde{X}^\ell, \tilde{Y}^\ell, A, J', H)
\intertext{and}
\overset{\circ}{\mathcal{M}}(\tilde{X}^{\hat{\ell}}, \tilde{Y}^{\hat{\ell}}, \tilde{\overline{Y}}{}^{\hat{\ell}}, A, J', (\hat{\pi}^{\hat{\ell}}_{\overline{\ell}})^\ast \overline{H}) &\to \overset{\circ}{\mathcal{M}}(\tilde{X}^{\overline{\ell}}, \tilde{Y}^{\overline{\ell}}, A, J', \overline{H})
\end{align*}
define $\overline{\ell}!$- and $\ell!$-sheeted coverings, respectively (outside a subset of real codimension at least $2$).
Here, one has to observe that the Hamiltonian perturbation $\hat{H} = (\hat{\pi}^{\hat{\ell}}_{\ell})^\ast H$ does not satisfy the general requirements for the definition of the pseudocycle $\mathrm{gw}^{\hat{\ell}}_\Sigma(X, Y, \overline{Y}, A, J', \hat{H})$, but only for the stratum over $\overset{\circ}{M}{}^{\hat{\ell}}$.
Instead, in this case the $\Omega$-limit set of $\mathrm{gw}^{\hat{\ell}}_\Sigma(X, Y, \overline{Y}, A, J', \hat{H})$ is contained in the $\Omega$-limit set of $\mathrm{gw}^{\ell}_\Sigma(X, Y, A, J', H)$, which one knows a priori to have real codimension at least $2$. \\
And likewise for $\hat{H} = (\hat{\pi}^{\hat{\ell}}_{\overline{\ell}})^\ast \overline{H}$. \\
Now connect $(\hat{\pi}^{\hat{\ell}}_\ell)^\ast H$ and $(\hat{\pi}^{\hat{\ell}}_{\overline{\ell}})^\ast \overline{H}$ by a path $(H_t)_{t\in [0,1]}$ in $ \mathcal{H}(\tilde{X}^{\hat{\ell}}, \tilde{Y}^{\hat{\ell}}, \tilde{\overline{Y}}{}^{\hat{\ell}})$ with $H_0 = (\hat{\pi}^{\hat{\ell}}_\ell)^\ast H$ and $H_1 = (\hat{\pi}^{\hat{\ell}}_{\overline{\ell}})^\ast \overline{H}$ that induces a cobordism \\
$\coprod_{t\in [0,1]} \overset{\circ}{\mathcal{M}}(\tilde{X}^{\hat{\ell}}, \tilde{Y}^{\hat{\ell}}, \tilde{\overline{Y}}{}^{\hat{\ell}}, A, J, H_t)$ between $\overset{\circ}{\mathcal{M}}(\tilde{X}^{\hat{\ell}}, \tilde{Y}^{\hat{\ell}}, \tilde{\overline{Y}}{}^{\hat{\ell}}, A, J, (\hat{\pi}^{\hat{\ell}}_\ell)^\ast H)$ and \\
$\overset{\circ}{\mathcal{M}}(\tilde{X}^{\hat{\ell}}, \tilde{Y}^{\hat{\ell}}, \tilde{\overline{Y}}{}^{\hat{\ell}}, A, J, (\hat{\pi}^{\hat{\ell}}_{\overline{\ell}})^\ast \overline{H})$ and is generic for $t\in (0,1)$. \\
I.\,e.~the boundary of $\coprod_{t\in [0,1]} \overset{\circ}{\mathcal{M}}(\tilde{X}^{\hat{\ell}}, \tilde{Y}^{\hat{\ell}}, \tilde{\overline{Y}}{}^{\hat{\ell}}, A, J, H_t)$ is the union of three parts.
Namely those parts for $t = 0,1$, where the evaluation maps factor through the boundary of $\overset{\circ}{\mathcal{M}}(\tilde{X}^{\hat{\ell}}, \tilde{Y}^{\hat{\ell}}, \tilde{\overline{Y}}{}^{\hat{\ell}}, A, J, (\hat{\pi}^{\hat{\ell}}_\ell)^\ast H)$ and $\overset{\circ}{\mathcal{M}}(\tilde{X}^{\hat{\ell}}, \tilde{Y}^{\hat{\ell}}, \tilde{\overline{Y}}{}^{\hat{\ell}}, A, J, (\hat{\pi}^{\hat{\ell}}_{\overline{\ell}})^\ast \overline{H})$, which can be covered by manifolds of codimension at least $2$.
And the part of the boundary for $t\in (0,1)$, which can be covered by manifolds of codimension at least $2$ by a generic choice of $(H_t)_{t\in [0,1]}$. \\
This implies that $\mathrm{gw}_\Sigma(X, Y, \overline{Y}, A, J')$, $\mathrm{gw}_\Sigma(X, Y, A, J')$ and $\mathrm{gw}_\Sigma(X, \overline{Y}, A, J')$ are rationally cobordant and finishes the proof of Theorem \ref{Theorem_Main_Theorem_2}.
\end{enumerate}
\section{Details of the proofs of the main results}

\subsection{Compactness of the closure of the moduli space}\label{Subsection_Compactness}

In this subsection, the details for Step \ref{Step_1} in the outline of the proof of Theorem \ref{Theorem_Main_Theorem_1} from Subsection \ref{Subsubsection_Proof_Main_Theorem_1} are given.

The appropriate conditions for compactness of $\cl\overset{\circ}{\mathcal{M}}(\tilde{X}^\ell, \tilde{Y}^\ell, A, J, H)$ to hold have been formulated in \cite{MR2399678}, Section 8:
\begin{defn}\label{Definition_Nice_J}
Let, for $E > 0$, $\mathcal{J}_\omega(X, Y, E) \subseteq
\mathcal{J}_\omega(X, Y)$ be the subset of almost complex structures
$J \in \mathcal{J}_\omega(X, Y)$ \st
\begin{enumerate}
  \item All $J$-holomorphic spheres of energy $\leq\, E$ contained in
    $Y$ are constant.
  \item Every nonconstant $J$-holomorphic sphere of energy $\leq E$ in
    $X$ intersects $Y$ in at least $3$ distinct points in the domain.
\end{enumerate}
Also, define
\[
\mathcal{J}_{\omega, \mathrm{ni}}(X, Y, E) \definedas \mathcal{J}_\omega(X, Y, E) \cap \mathcal{J}_{\omega, \mathrm{ni}}(X, Y)\text{.}
\]
\end{defn}

In the same reference, in Corollary 8.14, it has been shown that this condition is non-void, which needs to be adapted to include the condition of normal integrability:
\begin{lemma}\label{Lemma_Nice_J_exist}
There exists a constant $D^\ast = D^\ast(X,\omega,J_0)$ and a nonempty $C^0$-neighbourhood $\mathcal{J}_\omega(X, J_0) \subseteq \mathcal{J}_\omega(X)$ of $J_0$ \st if $D\geq D^\ast$, then
\begin{align*}
\mathcal{J}_{\omega}(X, Y, J_0, E) &\definedas \mathcal{J}_{\omega}(X, Y, E) \cap \mathcal{J}_\omega(X, J_0)
\intertext{and}
\mathcal{J}_{\omega, \mathrm{ni}}(X, Y, J_0, E) &\definedas \mathcal{J}_{\omega, \mathrm{ni}}(X, Y, E) \cap \mathcal{J}_\omega(X, J_0)
\end{align*}
are nonempty for every $E > 0$. \\
Moreover, any two elements in $\mathcal{J}_{\omega, \mathrm{ni}}(X, Y, J_0, E)$ can be connected by a path in $\mathcal{J}_{\omega, \mathrm{ni}}(X, Y, E)$.
\end{lemma}
\begin{proof}
Let $\mathcal{J}_\omega(X, J_0)$ be the $C^0$-ball around $J_0$ in $\mathcal{J}_\omega(X)$ of radius $\theta_2$, where $\theta_2$ is as in Corollary 8.14 of \cite{MR2399678}. Then by that reference, there exists a $J' \in \mathcal{J}_\omega(X,Y, E) \cap \mathcal{J}_\omega(X, J_0)$.
Applying the procedure in the proof of Theorem A.2 in \cite{MR1954264} yields an arbitrarily $C^0$-close (to $J'$, the endomorphism $K$ in equation (A.2) in said proof can be chosen arbitrarily small in $C^0$, but not in $C^1$) $J'' \in \mathcal{J}_{\omega, \mathrm{ni}}(X,Y)$, in particular $J'' \in \mathcal{J}_\omega(X, J_0)$, with $J''|_Y = J'$.
Hence $J''$ still satisfies condition 1.~in Definition \ref{Definition_Nice_J}.
Now observe that condition (ii) of Proposition 8.11 in \cite{MR2399678} can be achieved by a perturbation $J$ of $J''$ \st $J - J''$ lies in the closure of those endomorphisms of $TX$ that have compact support in the complement of $Y$.
But such perturbations still lie in $\mathcal{J}_{\omega, \mathrm{ni}}(X, Y)$. \\
Now if $J_0, J_1\in \mathcal{J}_{\omega,\mathrm{ni}}(X,Y, J_0, E)$, then by Corollary 8.14 in op.~cit.~they can be connected by a path $(J'_\tau)_{\tau \in [0,1]}$ in $\mathcal{J}_{\omega}(X,Y, E)$.
Again applying the procedure from Theorem A.2 in \cite{MR1954264} produces a path $(J''_\tau)_{\tau \in [0,1]}$, arbitrarily close to $(J'_\tau)_{\tau \in [0,1]}$ in $C^0$-topology, that coincides with $(J'_\tau)_{\tau \in [0,1]}$ along $Y$ and satisfies $J''_0 = J'_0 = J_0$ as well as $J''_1 = J'_1 = J_1$.
So in particular $(J''_\tau)_{\tau \in [0,1]}$ still satisfies condition 1.~in Definition \ref{Definition_Nice_J}.
Now proceed as before: Condition (ii) of Proposition 8.12 in \cite{MR2399678} can be achieved by a perturbation $(J_\tau)_{\tau \in [0,1]}$ of $(J''_\tau)_{\tau \in [0,1]}$ \st $J''_\tau - J_\tau$ vanishes for $\tau = 0,1$ and for $\tau \in (0,1)$ lies in the closure of those endomorphisms of $TX$ that have compact support in the complement of $Y$.
\end{proof}

Finally, again in \cite{MR2399678}, Proposition 9.5, the necessary compactness result is given, which can easily be adapted to the present situation to show:
\begin{lemma}\label{Lemma_Main_compactness_result}
Let $J \in \mathcal{J}_\omega(X, Y, E)$ and let $\ell = [Y]\cdot A$. Then
\[
\pi^\mathcal{M}_M\times \pi^{\mathcal{M}}_{\mathcal{H}} :
\cl{\overset{\circ}{\mathcal{M}}(\tilde{X}^\ell, \tilde{Y}^\ell, A, J,
\mathcal{H}(\tilde{X}^\ell, \tilde{Y}^\ell))} \to
M^\ell \times \mathcal{H}(\tilde{X}^\ell, \tilde{Y}^\ell)
\]
is a proper map.
\end{lemma}
\begin{proof}
Let $(u_i)_{i\in \N} \subseteq \mathcal{M}(\tilde{X}^\ell,
\tilde{Y}^\ell, A, J, \mathcal{H}(\tilde{X}^\ell, \tilde{Y}^\ell))$
be a sequence \st $b_i \definedas \pi^\mathcal{M}_M(u_i) \to b\in M^\ell$ and $H_i \definedas \pi^\mathcal{M}_\mathcal{H}(u_i) \to H \in \mathcal{H}(\tilde{X}^\ell, \tilde{Y}^\ell)$. By Proposition
\ref{Proposition_Gromov_compactness_detailled}, there then exists an
$\ell'\in \N_0$ and a subsequence $(u_{i_j})_{j\in \N}$ together with a sequence $\hat{b}_{i_j}\in (M^\ell)^{\circ\ell'}$ and an element $\hat{u} \in \mathcal{M}((\hat{\pi}^{\ell'}_\ell)^\ast \tilde{X}^\ell, A, J, \mathcal{H}((\hat{\pi}^{\ell'}_\ell)^\ast \tilde{X}^\ell))$, $\hat{b} \definedas \pi^\mathcal{M}_M(\hat{u})$ \st $\pi^\mathcal{M}_\mathcal{H}(\hat{u}) = H$, $\pi^{\ell'}_\ell(\hat{b}) = b$ and
\[
(\hat{\pi}^{\ell'}_\ell)^\ast u_{i_j} \underset{j \to \infty}{\longrightarrow} \hat{u}\text{.}
\]
Furthermore, either $\hat{b} \in (M^\ell)^{\circ\ell'}$ and hence $\hat{u}$ defines an element
\[
(u, b, H)\in \cl\overset{\circ}{\mathcal{M}}(\tilde{X}^\ell, \tilde{Y}^\ell, A, J, \mathcal{H}(\tilde{X}^\ell, \tilde{Y}^\ell))\text{,}
\]
or there exists a component $(\Sigma^\ell)^{\circ\ell'}_{i,\hat{b}}$ with $\hat{\pi}^{\ell'}_\ell((\Sigma^\ell)^{\circ\ell'}_{i,\hat{b}}) = z$ for some $z\in \Sigma^\ell_b$ either a node or a marked point.
Furthermore, $\hat{u}$ defines a nonconstant $J$-holomorphic sphere in $\tilde{X}^\ell_z \cong X$.
Because $J\in \mathcal{J}_\omega(X, Y, E)$, the image of this sphere is not contained in $Y$ and intersects $Y$ in at least $3$ distinct points, at least one of which is not one of the last $\ell$ marked points $T^\ell_j(b)$.
Now proceed literally as in the proof of Proposition 9.5 in \cite{MR2399678}.
\end{proof}

\subsection{A description of the closure of the moduli space}\label{Subsection_Description_of_the_closure}

In this subsection a description of the elements of $\cl\overset{\circ}{\mathcal{M}}(\tilde{X}^\ell, \tilde{Y}^\ell, A, J, H)$ is given together with the compactness result from Step \ref{Step_4} of the outline of the proof of Theorem \ref{Theorem_Main_Theorem_1} in Subsection \ref{Subsubsection_Proof_Main_Theorem_1}.
Namely to each $u \in \cl\overset{\circ}{\mathcal{M}}(\tilde{X}^\ell, \tilde{Y}^\ell, A, J, H)$ one can associate data from the following at most countable set of choices:
\begin{enumerate}
  \item\label{Subsection_Description_of_the_closure_1} Letting $b \in M^\ell$, there is a connected open neighbourhood $U^\ell$ of $b$ in the stratum of the stratification by signature on $M^\ell$ containing $b$ \st $\Sigma^\ell_{U^\ell} \definedas \Sigma^\ell|_{U^\ell}$ has a desingularisation
\[
(\rho^\ell : \hat{S}^\ell \to U^\ell, \hat{R}^\ell, \hat{T}^\ell, \iota^\ell : U^\ell \to M^\ell, \hat{\iota}^\ell : \hat{S}^\ell \to \Sigma^\ell)
\]
\st $\rho^\ell : \hat{S}^\ell \to U^\ell$ is topologically trivial and \st there are sections $N^{\ell,1}_r, N^{\ell,2}_r : U^\ell \to \hat{S}^\ell$ parametrising the nodal points.
$M^\ell$ can be covered by at most countably many subsets $U^\ell$ of this form.
  \item\label{Subsection_Description_of_the_closure_2} For each of the above, there is a finite index set $I$ for the connected components of $\hat{S}^\ell$, \st $\hat{S}^\ell = \coprod_{i\in I} \hat{S}^\ell_i$, and also a finite number of decompositions $I = I^X \amalg I^Y$ into two disjoint subsets.
\end{enumerate}
For $i\in I$ and $b\in U^\ell$, $\hat{S}^\ell_{i,b}$ will usually be treated as a subset of $\Sigma^\ell_b$.
To $u \in \cl\overset{\circ}{\mathcal{M}}(\tilde{X}^\ell, \tilde{Y}^\ell, A, J, H)$ one then associates $U^\ell$ \st $\pi^{\mathcal{M}}_M(u) \defines b \in U^\ell$ and $I = I^X \amalg I^Y$ \st $u(\hat{S}^\ell_{i,b}) \subseteq \tilde{Y}^\ell \;\Leftrightarrow\; i \in I^Y$. \\
Furthermore, to each of the above, one can associate the following:
\begin{enumerate}\setcounter{enumi}{2}
  \item\label{Subsection_Description_of_the_closure_3} $\hat{S}^\ell = \hat{S}^{\ell,X} \amalg \hat{S}^{\ell,Y} \definedas \left(\coprod_{i\in I^X} \hat{S}^\ell_i\right) \amalg \left(\coprod_{i\in I^Y} \hat{S}^\ell_i\right)$.
  \item\label{Subsection_Description_of_the_closure_4} Denote by $\Sigma^{\ell,X}_{U^\ell}$ and $\Sigma^{\ell,Y}_{U^\ell}$ the image of $\hat{S}^{\ell,X}$ and $\hat{S}^{\ell,Y}$ under $\hat{\iota}^\ell$, respectively, so that $\Sigma^\ell_{U^\ell} = \Sigma^{\ell,X}_{U^\ell} \cup \Sigma^{\ell,Y}_{U^\ell}$.
  \item\label{Subsection_Description_of_the_closure_5} Denote by $\chi^X$ and $\chi^Y$ the Euler characteristics of the fibres of $\Sigma^{\ell,X}_{U^\ell}$ and $\Sigma^{\ell,Y}_{U^\ell}$, respectively.
  \item\label{Subsection_Description_of_the_closure_6} Let $\{1, \dots, \ell\} = K^X \amalg K^Y$ be the decomposition \st $T^\ell_j(b) \in \Sigma^{\ell,X}_{U^\ell}$ for all $j\in K^X$ and $b\in U^\ell$ and $T^\ell_j(b) \in \Sigma^{\ell,Y}_{U^\ell}$ for all $j\in K^Y$ and $b\in U^\ell$.
  \item\label{Subsection_Description_of_the_closure_7} Among the nodal points on $\hat{S}^\ell$, there is a subset of those pairs, where one of the two points corresponding to a node lies on $\hat{S}^{\ell,X}$ and the other lies on $\hat{S}^{\ell,Y}$. Denote these by $N^{\ell,XY,X}_r, N^{\ell,XY,Y}_r$, $r = 1, \dots, d$, the first one lying on $\hat{S}^{\ell, X}$, the second one on $\hat{S}^{\ell,Y}$.
  \item\label{Subsection_Description_of_the_closure_8} Denote by $N^{\ell,Y,1}_r, N^{\ell,Y,2}_r$, $r=1,\dots,d'$, the nodal points where both lie on $\hat{S}^{\ell,Y}$.
  \item\label{Subsection_Description_of_the_closure_9} Regard both $\Sigma^{\ell,X}_{U^\ell}$ and $\Sigma^{\ell,Y}_{U^\ell}$ as families of nodal Riemann surfaces with marked points $((T^\ell_j)_{j\in K^X}, (N^{\ell,XY,X}_r)_{r = 1, \dots, d})$ and $((T^\ell_j)_{j\in K^Y}, (N^{\ell,XY,Y}_r)_{r = 1, \dots, d})$, respectively.
\end{enumerate}
From now on unless specified otherwise, a set of data as in \ref{Subsection_Description_of_the_closure_1}.--\ref{Subsection_Description_of_the_closure_9}.~above is always supposed to be given. \\
Also, write $u_i \definedas u|_{\hat{S}^\ell_i}$, $u^X \definedas u|_{\Sigma^{\ell,X}_b}$ and $u^Y \definedas u|_{\Sigma^{\ell,Y}_b}$.

\clearpage
\begin{landscape}
\enlargethispage{2cm}
\def\svgwidth{24cm}
\ifpdf
\begingroup%
  \makeatletter%
  \providecommand\color[2][]{%
    \errmessage{(Inkscape) Color is used for the text in Inkscape, but the package 'color.sty' is not loaded}%
    \renewcommand\color[2][]{}%
  }%
  \providecommand\transparent[1]{%
    \errmessage{(Inkscape) Transparency is used (non-zero) for the text in Inkscape, but the package 'transparent.sty' is not loaded}%
    \renewcommand\transparent[1]{}%
  }%
  \providecommand\rotatebox[2]{#2}%
  \ifx\svgwidth\undefined%
    \setlength{\unitlength}{8146.825bp}%
    \ifx\svgscale\undefined%
      \relax%
    \else%
      \setlength{\unitlength}{\unitlength * \real{\svgscale}}%
    \fi%
  \else%
    \setlength{\unitlength}{\svgwidth}%
  \fi%
  \global\let\svgwidth\undefined%
  \global\let\svgscale\undefined%
  \makeatother%
  \begin{picture}(1,0.57631654)%
    \put(0,0){\includegraphics[width=\unitlength]{Boundary_1.pdf}}%
    \put(0.35445959,0.45980707){\color[rgb]{0,0,0}\makebox(0,0)[lb]{\smash{$N^{\ell,XY,X}_1$}}}%
    \put(0.59804215,0.46056354){\color[rgb]{0,0,0}\makebox(0,0)[lb]{\smash{$N^{\ell,XY,X}_2$}}}%
    \put(0.79018489,0.46056355){\color[rgb]{0,0,0}\makebox(0,0)[lb]{\smash{$N^{\ell,XY,X}_3$}}}%
    \put(0.30453274,0.38869913){\color[rgb]{0,0,0}\makebox(0,0)[lb]{\smash{$N^{\ell XY,Y}_1$}}}%
    \put(0.57761754,0.39399441){\color[rgb]{0,0,0}\makebox(0,0)[lb]{\smash{$N^{\ell,XY,Y}_2$}}}%
    \put(0.78035083,0.39172501){\color[rgb]{0,0,0}\makebox(0,0)[lb]{\smash{$N^{\ell,XY,Y}_3$}}}%
    \put(0.29999392,0.34255463){\color[rgb]{0,0,0}\makebox(0,0)[lb]{\smash{$N^{\ell,Y,1}_1$}}}%
    \put(0.32949616,0.33272054){\color[rgb]{0,0,0}\makebox(0,0)[lb]{\smash{$N^{\ell,Y,1}_2$}}}%
    \put(0.37034541,0.29565362){\color[rgb]{0,0,0}\makebox(0,0)[lb]{\smash{$N^{\ell,Y,2}_2$}}}%
    \put(0.23191184,0.31002648){\color[rgb]{0,0,0}\makebox(0,0)[lb]{\smash{$N^{\ell,Y,2}_1$}}}%
    \put(0.57459165,0.35087574){\color[rgb]{0,0,0}\makebox(0,0)[lb]{\smash{$N^{\ell,Y,1}_3$}}}%
    \put(0.59955506,0.3357464){\color[rgb]{0,0,0}\makebox(0,0)[lb]{\smash{$N^{\ell,Y,1}_4$}}}%
    \put(0.51104838,0.31153946){\color[rgb]{0,0,0}\makebox(0,0)[lb]{\smash{$N^{\ell,Y,2}_3$}}}%
    \put(0.61846673,0.2797678){\color[rgb]{0,0,0}\makebox(0,0)[lb]{\smash{$N^{\ell,Y,2}_4$}}}%
    \put(0.78791551,0.33423347){\color[rgb]{0,0,0}\makebox(0,0)[lb]{\smash{$N^{\ell,Y,1}_8$}}}%
    \put(0.79169785,0.28279367){\color[rgb]{0,0,0}\makebox(0,0)[lb]{\smash{$N^{\ell,Y,2}_8$}}}%
    \put(0.78186378,0.23437975){\color[rgb]{0,0,0}\makebox(0,0)[lb]{\smash{$N^{\ell,Y,1}_7$}}}%
    \put(0.72664164,0.19201755){\color[rgb]{0,0,0}\makebox(0,0)[lb]{\smash{$N^{\ell,Y,2}_7$}}}%
    \put(0.51180483,0.2358927){\color[rgb]{0,0,0}\makebox(0,0)[lb]{\smash{$N^{\ell,Y,1}_5$}}}%
    \put(0.46112152,0.19277403){\color[rgb]{0,0,0}\makebox(0,0)[lb]{\smash{$N^{\ell,Y,2}_5$}}}%
    \put(0.44826155,0.14209073){\color[rgb]{0,0,0}\makebox(0,0)[lb]{\smash{$N^{\ell,Y,1}_6$}}}%
    \put(0.39757825,0.09367678){\color[rgb]{0,0,0}\makebox(0,0)[lb]{\smash{$N^{\ell,Y,2}_6$}}}%
    \put(0.11541584,0.49233519){\color[rgb]{0,0,0}\makebox(0,0)[lb]{\smash{{\LARGE $\hat{S}^{\ell,X},\; \Sigma^{\ell,X}_{U^\ell}$}}}}%
    \put(0.26013508,0.16571793){\color[rgb]{0,0,0}\makebox(0,0)[lb]{\smash{{\LARGE $\hat{S}^{\ell,Y},\; \Sigma^{\ell,Y}_{U^\ell}$}}}}%
    \put(0.07910541,0.04601932){\color[rgb]{0,0,0}\makebox(0,0)[lb]{\smash{{\LARGE $Y$}}}}%
    \put(0.8964566,0.49566242){\color[rgb]{0,0,0}\makebox(0,0)[lb]{\smash{{\LARGE $X\setminus Y$}}}}%
    \put(0.00874881,0.36211346){\color[rgb]{0,0,0}\makebox(0,0)[lb]{\smash{{\LARGE $\circ : T^\ell_i, \; i\in K^X$}}}}%
    \put(0.00874881,0.32283436){\color[rgb]{0,0,0}\makebox(0,0)[lb]{\smash{{\LARGE $\bullet : T^\ell_j ,\; j\in K^Y$}}}}%
  \end{picture}%
\endgroup%

\else
\begingroup%
  \makeatletter%
  \providecommand\color[2][]{%
    \errmessage{(Inkscape) Color is used for the text in Inkscape, but the package 'color.sty' is not loaded}%
    \renewcommand\color[2][]{}%
  }%
  \providecommand\transparent[1]{%
    \errmessage{(Inkscape) Transparency is used (non-zero) for the text in Inkscape, but the package 'transparent.sty' is not loaded}%
    \renewcommand\transparent[1]{}%
  }%
  \providecommand\rotatebox[2]{#2}%
  \ifx\svgwidth\undefined%
    \setlength{\unitlength}{8146.825bp}%
    \ifx\svgscale\undefined%
      \relax%
    \else%
      \setlength{\unitlength}{\unitlength * \real{\svgscale}}%
    \fi%
  \else%
    \setlength{\unitlength}{\svgwidth}%
  \fi%
  \global\let\svgwidth\undefined%
  \global\let\svgscale\undefined%
  \makeatother%
  \begin{picture}(1,0.57631654)%
    \put(0,0){\includegraphics[width=\unitlength]{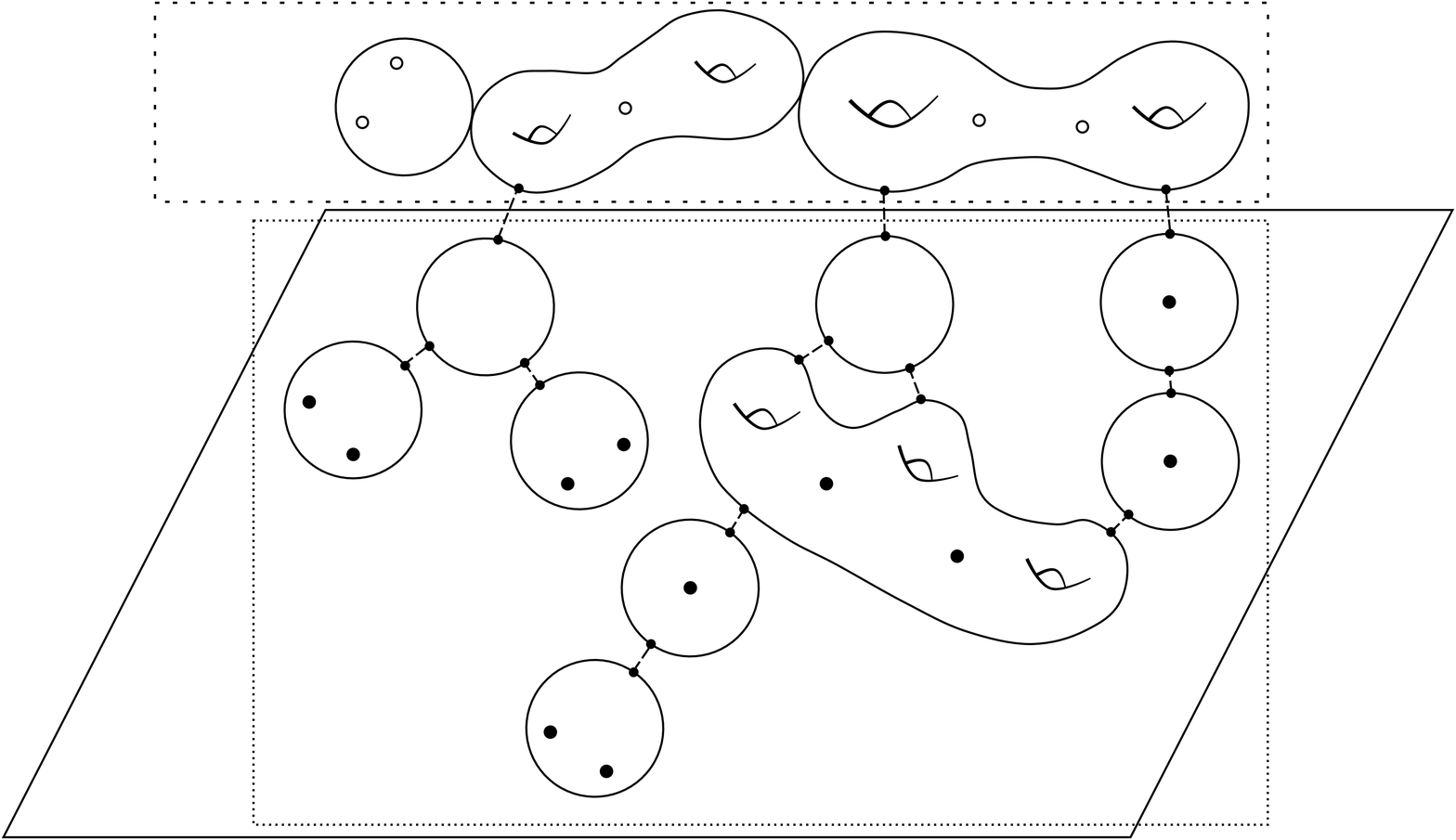}}%
    \put(0.35445959,0.45980707){\color[rgb]{0,0,0}\makebox(0,0)[lb]{\smash{$N^{\ell,XY,X}_1$}}}%
    \put(0.59804215,0.46056354){\color[rgb]{0,0,0}\makebox(0,0)[lb]{\smash{$N^{\ell,XY,X}_2$}}}%
    \put(0.79018489,0.46056355){\color[rgb]{0,0,0}\makebox(0,0)[lb]{\smash{$N^{\ell,XY,X}_3$}}}%
    \put(0.30453274,0.38869913){\color[rgb]{0,0,0}\makebox(0,0)[lb]{\smash{$N^{\ell XY,Y}_1$}}}%
    \put(0.57761754,0.39399441){\color[rgb]{0,0,0}\makebox(0,0)[lb]{\smash{$N^{\ell,XY,Y}_2$}}}%
    \put(0.78035083,0.39172501){\color[rgb]{0,0,0}\makebox(0,0)[lb]{\smash{$N^{\ell,XY,Y}_3$}}}%
    \put(0.29999392,0.34255463){\color[rgb]{0,0,0}\makebox(0,0)[lb]{\smash{$N^{\ell,Y,1}_1$}}}%
    \put(0.32949616,0.33272054){\color[rgb]{0,0,0}\makebox(0,0)[lb]{\smash{$N^{\ell,Y,1}_2$}}}%
    \put(0.37034541,0.29565362){\color[rgb]{0,0,0}\makebox(0,0)[lb]{\smash{$N^{\ell,Y,2}_2$}}}%
    \put(0.23191184,0.31002648){\color[rgb]{0,0,0}\makebox(0,0)[lb]{\smash{$N^{\ell,Y,2}_1$}}}%
    \put(0.57459165,0.35087574){\color[rgb]{0,0,0}\makebox(0,0)[lb]{\smash{$N^{\ell,Y,1}_3$}}}%
    \put(0.59955506,0.3357464){\color[rgb]{0,0,0}\makebox(0,0)[lb]{\smash{$N^{\ell,Y,1}_4$}}}%
    \put(0.51104838,0.31153946){\color[rgb]{0,0,0}\makebox(0,0)[lb]{\smash{$N^{\ell,Y,2}_3$}}}%
    \put(0.61846673,0.2797678){\color[rgb]{0,0,0}\makebox(0,0)[lb]{\smash{$N^{\ell,Y,2}_4$}}}%
    \put(0.78791551,0.33423347){\color[rgb]{0,0,0}\makebox(0,0)[lb]{\smash{$N^{\ell,Y,1}_8$}}}%
    \put(0.79169785,0.28279367){\color[rgb]{0,0,0}\makebox(0,0)[lb]{\smash{$N^{\ell,Y,2}_8$}}}%
    \put(0.78186378,0.23437975){\color[rgb]{0,0,0}\makebox(0,0)[lb]{\smash{$N^{\ell,Y,1}_7$}}}%
    \put(0.72664164,0.19201755){\color[rgb]{0,0,0}\makebox(0,0)[lb]{\smash{$N^{\ell,Y,2}_7$}}}%
    \put(0.51180483,0.2358927){\color[rgb]{0,0,0}\makebox(0,0)[lb]{\smash{$N^{\ell,Y,1}_5$}}}%
    \put(0.46112152,0.19277403){\color[rgb]{0,0,0}\makebox(0,0)[lb]{\smash{$N^{\ell,Y,2}_5$}}}%
    \put(0.44826155,0.14209073){\color[rgb]{0,0,0}\makebox(0,0)[lb]{\smash{$N^{\ell,Y,1}_6$}}}%
    \put(0.39757825,0.09367678){\color[rgb]{0,0,0}\makebox(0,0)[lb]{\smash{$N^{\ell,Y,2}_6$}}}%
    \put(0.11541584,0.49233519){\color[rgb]{0,0,0}\makebox(0,0)[lb]{\smash{{\LARGE $\hat{S}^{\ell,X},\; \Sigma^{\ell,X}_{U^\ell}$}}}}%
    \put(0.26013508,0.16571793){\color[rgb]{0,0,0}\makebox(0,0)[lb]{\smash{{\LARGE $\hat{S}^{\ell,Y},\; \Sigma^{\ell,Y}_{U^\ell}$}}}}%
    \put(0.07910541,0.04601932){\color[rgb]{0,0,0}\makebox(0,0)[lb]{\smash{{\LARGE $Y$}}}}%
    \put(0.8964566,0.49566242){\color[rgb]{0,0,0}\makebox(0,0)[lb]{\smash{{\LARGE $X\setminus Y$}}}}%
    \put(0.00874881,0.36211346){\color[rgb]{0,0,0}\makebox(0,0)[lb]{\smash{{\LARGE $\circ : T^\ell_i, \; i\in K^X$}}}}%
    \put(0.00874881,0.32283436){\color[rgb]{0,0,0}\makebox(0,0)[lb]{\smash{{\LARGE $\bullet : T^\ell_j ,\; j\in K^Y$}}}}%
  \end{picture}%
\endgroup%

\fi
\end{landscape}

\subsection{Reduction to the case of vanishing homology classes}\label{Subsection_Vanishing_homology_class}

The first goal is to show that one can choose $H$ \st every component of a section in $\cl\overset{\circ}{\mathcal{M}}(\tilde{X}^\ell, \tilde{Y}^\ell, A, J, H)$ with image contained in $\tilde{Y}^\ell$ needs to represent vanishing homology class. Assuming $A\neq 0$, this in particular implies that no section over a smooth curve in $\cl\overset{\circ}{\mathcal{M}}(\tilde{X}^\ell, \tilde{Y}^\ell, A, J, H)$ has image contained in $\tilde{Y}^\ell$.
The way this will be proved is by following the line of argument in \cite{MR2399678} leading up to Proposition 8.11.
To formulate the main result, first of all the analogue of Lemma 8.9 in \cite{MR2399678} is needed:
\begin{lemma}\label{Lemma_Dstar}
Let $\Sigma$ be a fixed smooth Riemann surface equipped with a compatible volume
form $\dvol_\Sigma$ \st $\vol_\Sigma(\Sigma) = 1$, let $(X,\omega)$ be a
closed symplectic manifold, $\hat{X} \definedas \Sigma \times X$, and let $J_0$ be an $\omega$-compatible almost complex structure.
Then there exists a constant $D_\ast = D_\ast(X,\omega,J_0)$ \st for any $J \in \mathcal{J}_\omega(X)$ and $H\in \mathcal{H}(\hat{X})$ with $\|J - J_0\|_{C^0}, \|H\|_{C^1}, R_H < \frac{1}{4}$, where $R_H : \hat{X} \to \R$ is \st $R_H\dvol_\Sigma$ is the curvature of the connection defined by $H$, the following holds: \\
If $A \in H_2(X; \Z)$ is a homology class \st there exists a $\hat{J}^H$-holomorphic section $u = (\id,u_2) : \Sigma \to \hat{X}$ with $[u_2] = A$, then
\[
\langle c_1(TW), A\rangle < D_\ast (\omega(A) + 1)\text{.}
\]
\end{lemma}
\begin{proof}
Let $J, H$ be as in the statement of the lemma and let $\kappa \definedas \frac{1}{4}$. Then by Exercise 8.1.3, p.~260, in \cite{MR2045629},
$\hat{\omega}_\kappa \definedas \pr_2^\ast\omega + \pr_1^\ast(\kappa\dvol_\Sigma)$ is a symplectic form on $\hat{X}$ \st $\hat{J}^H$ is $\hat{\omega}_\kappa$-compatible.
Now proceed as in the proof of Lemma 8.9 in \cite{MR2399678}:
Let $\alpha\in\Omega^2(X)$ be a closed $2$-form that represents $c_1(TX)$. Then
\begin{align*}
\langle c_1(TX), A\rangle &= \int_\Sigma u_2^\ast\alpha \\
&= \int_\Sigma u^\ast \pr_2^\ast\alpha \\
&\leq \|\pr_2^\ast\alpha\|_{\hat{\omega}_\kappa, \hat{J}^H} \int_\Sigma u^\ast \hat{\omega}_\kappa
\intertext{as in op.~cit., because $u$ is $\hat{J}^H$-holomorphic and $\hat{J}^H$ is $\hat{\omega}_\kappa$-compatible}
&= \|\pr_2^\ast\alpha\|_{\hat{\omega}_\kappa, \hat{J}^H} \left(\int_\Sigma
u_2^\ast\omega + \kappa\right) \\
&= \|\pr_2^\ast\alpha\|_{\tilde{\omega}_\kappa, \hat{J}^H} (\omega(A) + \kappa)\text{.}
\end{align*}
$\|\pr_2^\ast\alpha\|_{\hat{\omega}_\kappa, \hat{J}^H}$ here denotes the norm \wrt the metric defined by $\hat{\omega}_\kappa$ and $\hat{J}^H$.
The claim now follows, because $\|\pr_2^\ast\alpha \|_{\hat{\omega}_\kappa, \hat{J}^H}$ depends continuously on $\hat{J}^H$, which in turns depends continuously in $C^0$-norm on $J$ and in $C^1$-norm on $H$, and coincides with $\|\alpha\|_{\omega,J_0}$ for $J = J_0$ and $H = 0$.
\end{proof}

Using this, one can formulate the main result of this section:
\begin{lemma}\label{Lemma_Reduction_to_vanishing_A}
Given $J \in \mathcal{J}_{\omega}(X,Y,E)$, there exists an integer $D^\ast$ depending only on $g$, $n$ and $D_\ast$ \st for $D\geq D^\ast$ there exists a generic subset $\mathcal{H}_{\mathrm{reg}}(\tilde{Y}, J) \subseteq \mathcal{H}(\tilde{Y})$ with the property that for every $H \in \mathcal{H}_{\mathrm{reg}}(\tilde{Y}, J)$ and for any choice of data $U^\ell, I^X, I^Y$ as above, for $0\neq B \in H_2(Y)$,
\[
\mathcal{M}(\tilde{Y}^\ell|_{\Sigma^{\ell,Y}_{U^\ell}}, B, J, (\hat{\pi}^\ell_0)^\ast H) = \emptyset
\]
and $\mathcal{M}(\tilde{Y}^\ell|_{\Sigma^{\ell,Y}_{U^\ell}}, 0, J, (\hat{\pi}^\ell_0)^\ast H)$ is a smooth manifold diffeomorphic to
\[
(\pi^\ell_0)^\ast \left( \mathcal{M}(\tilde{Y}|_{\Sigma^Y_{U}}, 0, J, H) \amalg \coprod_{C\in \mathcal{C}_1} Y|_U\right)
\]
and hence of dimension
\[
\dim_\R\left( \mathcal{M}(\tilde{Y}^\ell|_{\Sigma^{\ell,Y}_{U^\ell}}, 0, J, (\hat{\pi}^\ell_0)^\ast H) \right) = \dim_\C(Y)\chi^Y + \dim_\R(U^\ell)\text{.}
\]
Furthermore this manifold comes with the smooth evaluation map
\[
\ev^{N^{\ell, XY, Y}} : \mathcal{M}(\tilde{Y}^\ell|_{\Sigma^{\ell,Y}_{U^\ell}}, 0, J, (\hat{\pi}^\ell_0)^\ast H) \to \bigoplus_{r=1}^d \left(N^{\ell, XY, Y}_r\right)^\ast \tilde{Y}^\ell\text{.}
\]
\end{lemma}
\begin{proof}
Let $u \in \cl\overset{\circ}{\mathcal{M}}(\tilde{X}^\ell, \tilde{Y}^\ell, A, J, H)$ be as in the previous subsection.
For $i \in I$ denote $\tilde{X}^\ell_i \definedas \tilde{X}^\ell|_{\hat{S}^\ell_i})$ and analogously for $\tilde{X}^\ell_i$.
Then $u$ pulls back to $u_i \in \mathcal{M}(\tilde{X}^\ell_i, A_i, J, \mathcal{H}(\tilde{X}^\ell, \tilde{Y}^\ell))$ for $i \in I^X$ and to $u_i \in \mathcal{M}(\tilde{Y}^\ell_i, A_i, J, \mathcal{H}(\tilde{X}^\ell, \tilde{Y}^\ell))$ for $i \in I^Y$, where $A_i \definedas [\pr_2(u_i)] \in H_2(X; \Z)$.

One would now like to reproduce the argument in \cite{MR2399678}, Proposition 8.11 (a), to show that for generic $H$, $\mathcal{M}(\tilde{Y}^\ell_i, A_i, J, H) = \emptyset$ for $D$ large enough.
To do so, first observe the following:
\begin{claim}
In the notation and under the assumptions on $H\in \mathcal{H}(\tilde{X}^\ell, \tilde{Y}^\ell)$ of Lemma \ref{Lemma_Dstar}, if $A_i \neq 0$, then $\omega(A_i) > 0$.
\end{claim}
To see this, let $\hat{\omega}_\kappa$ be as in the proof of Lemma \ref{Lemma_Dstar}.
Then because $u_i$ defines a $\hat{J}^H$-holomorphic map,
\begin{align*}
0 &< \int_{\hat{S}^\ell_{i,b}} u_i^\ast \hat{\omega}_\kappa \\
&= \int_{\hat{S}^\ell_{i,b}} (\pr_2\circ u_i)^\ast \omega + \kappa \\
&= \omega(A_i) + \kappa\text{.}
\end{align*}
Because $\kappa = \frac{1}{4}$ and $\omega(A_i) \in \Z$, the claim follows.

By Lemma \ref{Lemma_The_universal_moduli_space}, $\pi^\mathcal{M}_\mathcal{H} : \mathcal{M}(\tilde{Y}^\ell_i, A_i, J, \mathcal{H}(\tilde{X}^\ell, \tilde{Y}^\ell)) \to \mathcal{H}(\tilde{X}^\ell, \tilde{Y}^\ell)$ has Fredholm index given by
\begin{align*}
\ind(\pi^\mathcal{M}_\mathcal{H}) &= \dim_\C(Y)\chi(S^\ell_i) + 2c_1^{TY}(A_i) + \dim_\R(U^\ell) \\
&\leq 2\dim_\C(Y) + 2(c_1^{TX}(A_i) - D\omega(A_i)) + \dim_\R(M^\ell) \\
&\leq 2\dim_\C(Y) + 2D_\ast\kappa + 2(D_\ast - D)\omega(A_i) + \dim_\R(M^\ell)\text{,}
\end{align*}
where $D_\ast$ and $\kappa$ are as in Lemma \ref{Lemma_Dstar}.
But $\dim_\R(M^\ell) = \dim(M) + 2\ell = \dim(M) + 2[Y]\cdot A = \dim(M) + 2D\omega(A)$, choosing $\ell = [Y]\cdot A$ to satisfy Lemma \ref{Lemma_Main_compactness_result}.
So while the middle term in the above index formula decreases with increasing $D$, the last term increases just as quickly, at least for $A_i=A$. This is a case that one definitely would like to deal with in this way.
But observe that if $S^\ell_i$ is a component of genus zero (a sphere) and $H\in \mathcal{H}(\tilde{X}^\ell, \tilde{Y}^\ell)$ satisfies $(\hat{\iota}^\ell)^\ast H|_{\hat{S}^\ell_i} \equiv 0$, then any $u_i \in \mathcal{M}(\tilde{Y}^\ell_i, A_i, J, H)$ defines a $J$-holomorphic sphere in $Y$. But for $J\in \mathcal{J}(X, Y, E)$, the only such spheres are the constant ones, implying $A_i = 0$.
This allows for the following construction, which first of all requires the introduction of quite a bit of notation to signify the two parts of a curve in the family $\Sigma^\ell|_{U^\ell}$ that lie in $\tilde{Y}^\ell$ and those that intersect $\tilde{X}^\ell \setminus \tilde{Y}^\ell$:
Now fix some $b\in U^\ell$. Under $\hat{\pi}^\ell_0\circ \hat{\iota}^\ell|_{\hat{S}^{\ell,Y}_b} : \hat{S}^{\ell,Y}_b \to \Sigma_b$, a certain number of genus zero components of $\hat{S}^{\ell,Y}_b$ are mapped to points.
This happens if and only if a component contains fewer than three special points apart from the $\hat{T}^{\ell}_j$, \ie fewer than three nodal points or marked points among the $\hat{R}^{\ell}_j(b)$.
These can be grouped together into ``collapsed subtrees'' as in Section 2 in \cite{MR2399678} in the following way: Call two components of $\hat{S}^{\ell,Y}_b$ connected if there exists an $r$ \st $N^{\ell,Y,1}_r(b)$ lies on one of them, $N^{\ell,Y,2}_r(b)$ on the other.
Now take the equivalence relation this generates on the set of components of $\hat{S}^{\ell,Y}_b$ on which $\hat{\pi}^\ell_0$ is constant.
Since $U^\ell$ was assumed to be connected, this is independent of $b\in U^\ell$.
An equivalence class of this equivalence relation then corresponds to a collapsed subtree.
\begin{enumerate}
  \item Denote the set of equivalence classes from above by $\mathcal{C}$.
    This gives a decomposition $I^Y = I^Y_0 \amalg \coprod_{C\in \mathcal{C}} I^Y_C$, \st $\hat{\pi}^\ell_0\left(\coprod_{i\in I^Y_C} \hat{S}^{\ell,Y}_i\right) = \mathrm{const}$, for every $C\in \mathcal{C}$, and $\hat{\pi}^\ell_0|_{\hat{S}^\ell_{i,b}}$ is a biholomorphic map onto its image for every $i\in I^Y_0$, $b\in U^\ell$.
  \item $\mathcal{C}$ can be further decomposed into subsets $\mathcal{C}_0$ and $\mathcal{C}_1$, where every $C\in\mathcal{C}_0$ has the property that there exists at least one (and at most two) $i\in I^Y_C$ \st for every $b\in U$, $\hat{S}^{\ell,Y}_{i,b}$ is connected to $\hat{S}^{\ell,Y}_{j,b}$ for some $j\in I^Y_0$ and $\mathcal{C}_1 \definedas \mathcal{C} \setminus \mathcal{C}_0$. Let $\hat{S}^{\ell,Y,0}$ be the subfamily of $\hat{S}^{\ell,Y}$ consisting of the components in $I^Y_0\cup \coprod_{C\in\mathcal{C}_0} I^Y_C$.
  \item Denote by $\Sigma^{\ell,Y,0}_{U^\ell}$ the image of $\hat{S}^{\ell,Y,0}$ in $\Sigma^\ell_{U^\ell}$ under $\hat{\iota}^\ell$, by $\chi^Y_0$ the Euler characteristic of the fibres of $\Sigma^{\ell,Y,0}_{U^\ell}$ and denote by $U$ the open subset of the stratum of $M$ to which $U^\ell$ gets mapped under $\pi^{\ell - 1}_{-1}$.
  \item Then $\hat{\pi}^\ell_0$ is a well-defined map from $\Sigma^{\ell,Y,0}_{U^\ell}$ to a subfamily of $\Sigma_U$ (the restriction of $\Sigma$ to $U$), which will be denoted by $\Sigma_U^Y$ and has fibres of Euler characteristic $\chi^Y_0$ as well.
\end{enumerate}

\clearpage
\def\svgwidth{18cm}
\ifpdf
\hspace{-2.5cm}
\begingroup%
  \makeatletter%
  \providecommand\color[2][]{%
    \errmessage{(Inkscape) Color is used for the text in Inkscape, but the package 'color.sty' is not loaded}%
    \renewcommand\color[2][]{}%
  }%
  \providecommand\transparent[1]{%
    \errmessage{(Inkscape) Transparency is used (non-zero) for the text in Inkscape, but the package 'transparent.sty' is not loaded}%
    \renewcommand\transparent[1]{}%
  }%
  \providecommand\rotatebox[2]{#2}%
  \ifx\svgwidth\undefined%
    \setlength{\unitlength}{6704.5180598bp}%
    \ifx\svgscale\undefined%
      \relax%
    \else%
      \setlength{\unitlength}{\unitlength * \real{\svgscale}}%
    \fi%
  \else%
    \setlength{\unitlength}{\svgwidth}%
  \fi%
  \global\let\svgwidth\undefined%
  \global\let\svgscale\undefined%
  \makeatother%
  \begin{picture}(1,1.18855504)%
    \put(0,0){\includegraphics[width=\unitlength]{Boundary_2.pdf}}%
    \put(0.52281526,0.94703386){\color[rgb]{0,0,0}\makebox(0,0)[lb]{\smash{$k_1 \in I^Y_0$}}}%
    \put(0.55888821,0.79904229){\color[rgb]{0,0,0}\makebox(0,0)[lb]{\smash{$k_2 \in I^Y_0$}}}%
    \put(0.16671053,0.94333408){\color[rgb]{0,0,0}\makebox(0,0)[lb]{\smash{$i_1 \in I^Y_C$}}}%
    \put(0.10389678,0.8495137){\color[rgb]{0,0,0}\makebox(0,0)[lb]{\smash{$i_2 \in I^Y_C$}}}%
    \put(0.23246445,0.85691328){\color[rgb]{0,0,0}\makebox(0,0)[lb]{\smash{$i_3 \in I^Y_C$}}}%
    \put(0.15098641,0.77036892){\color[rgb]{0,0,0}\makebox(0,0)[lb]{\smash{$C \in \mathcal{C}_1$}}}%
    \put(0.88724453,0.89431186){\color[rgb]{0,0,0}\makebox(0,0)[lb]{\smash{$C'' \in \mathcal{C}_0$}}}%
    \put(0.79382484,0.97015755){\color[rgb]{0,0,0}\makebox(0,0)[lb]{\smash{$m_1 \in I^Y_{C''}$}}}%
    \put(0.80122441,0.85546407){\color[rgb]{0,0,0}\makebox(0,0)[lb]{\smash{$m_2 \in I^Y_{C''}$}}}%
    \put(0.43864505,0.69267335){\color[rgb]{0,0,0}\makebox(0,0)[lb]{\smash{$j_1 \in I^Y_{C'}$}}}%
    \put(0.42847063,0.61127797){\color[rgb]{0,0,0}\makebox(0,0)[lb]{\smash{$j_2 \in I^Y_{C'}$}}}%
    \put(0.42569579,0.52155808){\color[rgb]{0,0,0}\makebox(0,0)[lb]{\smash{$C' \in \mathcal{C}_0$}}}%
    \put(0.14543675,0.67139955){\color[rgb]{0,0,0}\makebox(0,0)[lb]{\smash{{\LARGE $\hat{\pi}^\ell_0$}}}}%
    \put(0.62733431,0.59185408){\color[rgb]{0,0,0}\makebox(0,0)[lb]{\smash{{\LARGE $\hat{S}^{\ell,Y,0},\; \Sigma^{\ell,Y,0}_{U^\ell}$}}}}%
    \put(0.62919109,0.37805285){\color[rgb]{0,0,0}\makebox(0,0)[lb]{\smash{{\LARGE $\hat{\pi}^\ell_0$}}}}%
    \put(0.78180053,0.21632545){\color[rgb]{0,0,0}\makebox(0,0)[lb]{\smash{{\LARGE $\Sigma^Y_U$}}}}%
  \end{picture}%
\endgroup%

\else
\hspace{-2.5cm}
\begingroup%
  \makeatletter%
  \providecommand\color[2][]{%
    \errmessage{(Inkscape) Color is used for the text in Inkscape, but the package 'color.sty' is not loaded}%
    \renewcommand\color[2][]{}%
  }%
  \providecommand\transparent[1]{%
    \errmessage{(Inkscape) Transparency is used (non-zero) for the text in Inkscape, but the package 'transparent.sty' is not loaded}%
    \renewcommand\transparent[1]{}%
  }%
  \providecommand\rotatebox[2]{#2}%
  \ifx\svgwidth\undefined%
    \setlength{\unitlength}{6704.5180598bp}%
    \ifx\svgscale\undefined%
      \relax%
    \else%
      \setlength{\unitlength}{\unitlength * \real{\svgscale}}%
    \fi%
  \else%
    \setlength{\unitlength}{\svgwidth}%
  \fi%
  \global\let\svgwidth\undefined%
  \global\let\svgscale\undefined%
  \makeatother%
  \begin{picture}(1,1.18855504)%
    \put(0,0){\includegraphics[width=\unitlength]{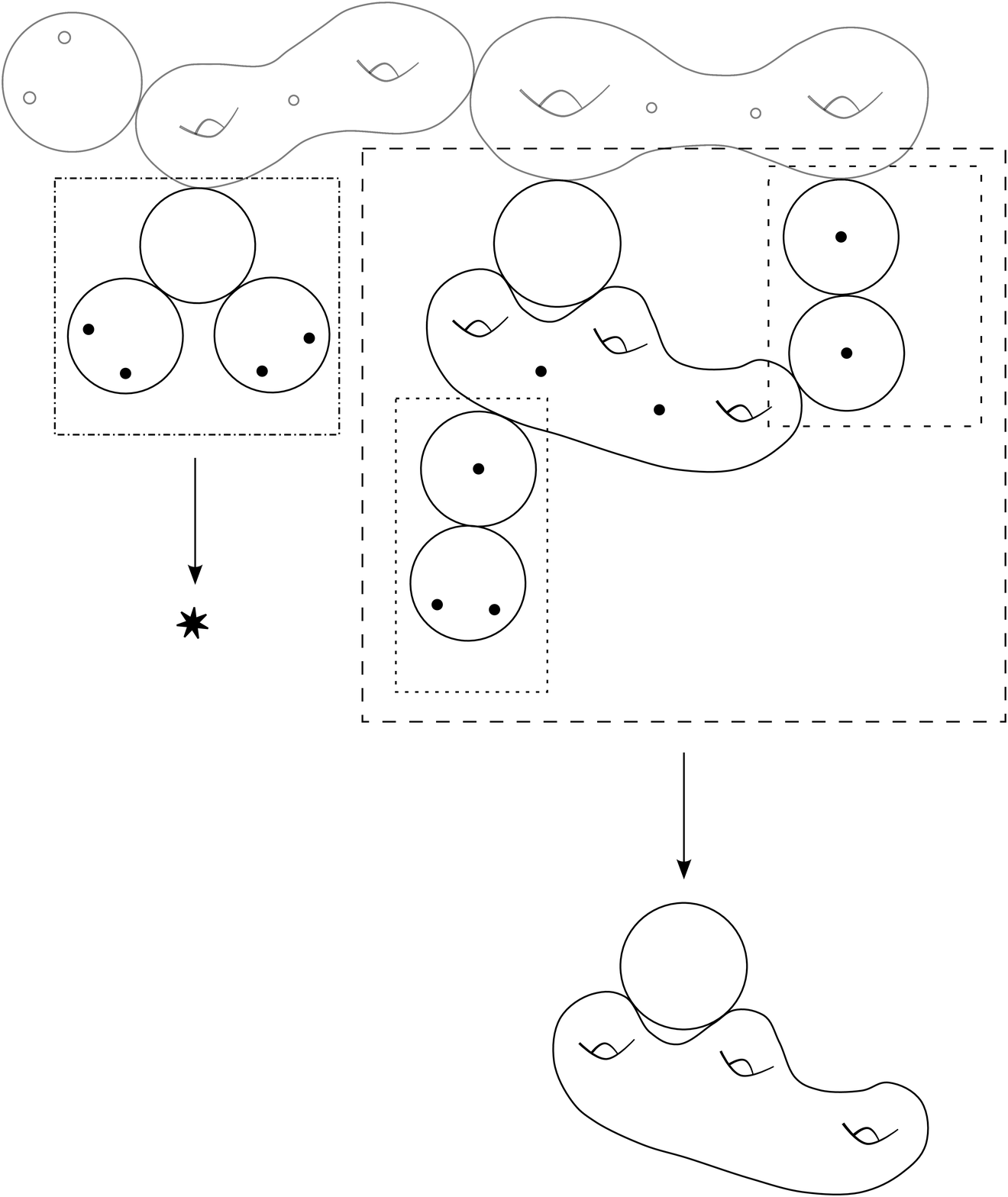}}%
    \put(0.52281526,0.94703386){\color[rgb]{0,0,0}\makebox(0,0)[lb]{\smash{$k_1 \in I^Y_0$}}}%
    \put(0.55888821,0.79904229){\color[rgb]{0,0,0}\makebox(0,0)[lb]{\smash{$k_2 \in I^Y_0$}}}%
    \put(0.16671053,0.94333408){\color[rgb]{0,0,0}\makebox(0,0)[lb]{\smash{$i_1 \in I^Y_C$}}}%
    \put(0.10389678,0.8495137){\color[rgb]{0,0,0}\makebox(0,0)[lb]{\smash{$i_2 \in I^Y_C$}}}%
    \put(0.23246445,0.85691328){\color[rgb]{0,0,0}\makebox(0,0)[lb]{\smash{$i_3 \in I^Y_C$}}}%
    \put(0.15098641,0.77036892){\color[rgb]{0,0,0}\makebox(0,0)[lb]{\smash{$C \in \mathcal{C}_1$}}}%
    \put(0.88724453,0.89431186){\color[rgb]{0,0,0}\makebox(0,0)[lb]{\smash{$C'' \in \mathcal{C}_0$}}}%
    \put(0.79382484,0.97015755){\color[rgb]{0,0,0}\makebox(0,0)[lb]{\smash{$m_1 \in I^Y_{C''}$}}}%
    \put(0.80122441,0.85546407){\color[rgb]{0,0,0}\makebox(0,0)[lb]{\smash{$m_2 \in I^Y_{C''}$}}}%
    \put(0.43864505,0.69267335){\color[rgb]{0,0,0}\makebox(0,0)[lb]{\smash{$j_1 \in I^Y_{C'}$}}}%
    \put(0.42847063,0.61127797){\color[rgb]{0,0,0}\makebox(0,0)[lb]{\smash{$j_2 \in I^Y_{C'}$}}}%
    \put(0.42569579,0.52155808){\color[rgb]{0,0,0}\makebox(0,0)[lb]{\smash{$C' \in \mathcal{C}_0$}}}%
    \put(0.14543675,0.67139955){\color[rgb]{0,0,0}\makebox(0,0)[lb]{\smash{{\LARGE $\hat{\pi}^\ell_0$}}}}%
    \put(0.62733431,0.59185408){\color[rgb]{0,0,0}\makebox(0,0)[lb]{\smash{{\LARGE $\hat{S}^{\ell,Y,0},\; \Sigma^{\ell,Y,0}_{U^\ell}$}}}}%
    \put(0.62919109,0.37805285){\color[rgb]{0,0,0}\makebox(0,0)[lb]{\smash{{\LARGE $\hat{\pi}^\ell_0$}}}}%
    \put(0.78180053,0.21632545){\color[rgb]{0,0,0}\makebox(0,0)[lb]{\smash{{\LARGE $\Sigma^Y_U$}}}}%
  \end{picture}%
\endgroup%

\fi
\clearpage

One can now for any $B\in H_2(Y)$ look at the moduli spaces
$\mathcal{M}(\tilde{Y}|_{\Sigma_U^Y}, B, J, \mathcal{H}(\tilde{Y}))$, which are equipped with the smooth structure from Lemma \ref{Lemma_The_universal_moduli_space}. The calculation from before then shows that the Fredholm index of the projection $\pi^\mathcal{M}_\mathcal{H} : \mathcal{M}(\tilde{Y}|_{\Sigma_U^Y}, B, J, \mathcal{H}(\tilde{Y})) \to \mathcal{H}(\tilde{Y})$ can be bounded from above by
\[
\dim_\C(Y)\chi^Y_0 + 2D_\ast \kappa + 2(D_\ast - D)\omega(B) + \dim_\R(M)\text{.}
\]
In particular, taking a bound for $\chi^Y_0$ depending only on $g$ and $n$, there is a constant $D_0$ only depending on $g$, $n$ and $D_\ast$ but not depending on $\ell$ \st for $D \geq D_0$ this is negative, provided that $B \neq 0$, due to integrality of $\omega$.
So from now on assume that $D \geq D_0$.
Also, due to the choices made, one has an isomorphism
\[
\mathcal{M}(\tilde{Y}^\ell|_{\Sigma^{\ell,Y,0}_{U^\ell}}, B, J, (\hat{\pi}^\ell_0)^\ast\mathcal{H}(\tilde{Y})) \cong (\pi^\ell_0)^\ast \mathcal{M}(\tilde{Y}|_{\Sigma_U^Y}, B, J, \mathcal{H}(\tilde{Y}))\text{.}
\]
This means that by the Sard-Smale theorem there is a generic subset of $\mathcal{H}(\tilde{Y})$ \st for every $H$ in this subset, if $B\neq 0$, then
\[
\mathcal{M}(\tilde{Y}|_{\Sigma^{Y}_{U}}, B, J, H) = \mathcal{M}(\tilde{Y}^\ell|_{\Sigma^{\ell,Y,0}_{U^\ell}}, B, J, (\hat{\pi}^\ell_0)^\ast H) = \emptyset
\]
and if $B = 0$, then $\mathcal{M}(\tilde{Y}^\ell|_{\Sigma^{\ell,Y,0}_{U^\ell}}, 0, J, (\hat{\pi}^\ell_0)^\ast H)$ is a smooth manifold of dimension $\dim_\C(Y)\chi_0^Y + \dim_\R(U^\ell)$ that comes with a canonical map to the manifold $\mathcal{M}(\tilde{Y}|_{\Sigma^{Y}_{U}}, 0, J, H)$ of dimension $\dim_\C(Y)\chi_0^Y + \dim_\R(U)$.
Analogously, for $C\in\mathcal{C}_1$, let $\hat{S}^{\ell,Y,C} \definedas \coprod_{i\in I^Y_C} \hat{S}^{\ell,Y}_i$ be the subfamily of $\hat{S}^{\ell,Y}$ consisting of the components in $I^Y_C$ and $\Sigma^{\ell,Y,C}_{U^\ell}$ its image in $\Sigma^\ell$ under $\hat{\iota}$. Then for any $H \in \mathcal{H}(\tilde{Y})$, for $B\neq 0$ again $\mathcal{M}(\tilde{Y}^\ell|_{\Sigma^{\ell,Y,C}_{U^\ell}}, B, J, (\hat{\pi}^\ell_0)^\ast H) = \emptyset$ and for $B = 0$,
\[
\mathcal{M}(\tilde{Y}^\ell|_{\Sigma^{\ell,Y,C}_{U^\ell}}, 0, J, (\hat{\pi}^\ell_0)^\ast H) \cong \tilde{Y}^\ell|_{U^\ell} \cong (\hat{\pi}^\ell_0)^\ast (\tilde{Y}|_U)\text{,}
\]
the isomorphism given by evaluation at any special point on $\Sigma^{\ell,Y,C}_{U^\ell}$.
Note that the Euler characteristic $\chi^Y_C$ of any fibre of $\hat{S}^{\ell,Y,C}$ is $2$.
So
\[
\dim_\R(\mathcal{M}(\tilde{Y}^\ell|_{\Sigma^{\ell,Y,C}_{U^\ell}}, 0, J, (\hat{\pi}^\ell_0)^\ast H)) = \dim_\C(Y)\chi^Y_C + \dim_\R(U^\ell)\text{.}
\]
Finally, one can take the intersection of all the generic subsets one gets via the construction above, for all the countably many choices of data as above (\ie $U^\ell$, $I^X$ and $I^Y$, $B\in H_2(Y)$, and so on), to get the generic subset $\mathcal{H}_{\mathrm{reg}}(\tilde{Y}, J) \subseteq \mathcal{H}(\tilde{Y})$. \\
This concludes the proof of Lemma \ref{Lemma_Reduction_to_vanishing_A}.
\end{proof}

\subsection{The refined compactness result and reduction of the boundary strata}\label{Subsection_Refined_compactness_result}

In the following, still the notation from Subsection \ref{Subsection_Description_of_the_closure} is in effect. Also, some $J \in \mathcal{J}_{\omega, \mathrm{ni}}(X, Y, E)$ and $H^Y \in \mathcal{H}_{\mathrm{reg}}(\tilde{Y}, J)$ usually are assumed to be given, where $H^Y$ is identified with an element of $\mathcal{H}_{\mathrm{ni}}(\tilde{X}^\ell, \tilde{Y}^\ell, J)$ by pullback via $\hat{\pi}^\ell_0$ and Lemma \ref{Lemma_Enough_ni_ham_perturbations}. \\
The refined compactness result mentioned in the previous section is of the type that, among others, has been studied in \cite{MR2026549} and in \cite{MR1954264} (which, as is stated in the introduction of \cite{MR2026549}, is a special case of the ``stretching of the neck'' construction in \cite{MR2026549}).
But in the following I will use a different transversality result from \cite{MR1954264}.
The setup of the formulation of the compactness results above is actually quite involved and will never be used in full generality in this text.
So instead of reciting the whole story, I will only describe a fairly straightforward corollary of this, which sums up the results as needed in the following.
To do so, first observe that for any $H\in \mathcal{H}_{\mathrm{ni}}(\tilde{X}^\ell, \tilde{Y}^\ell, J)$ and $u_i\in \mathcal{M}(\tilde{Y}^\ell|_{\hat{S}^{\ell,Y}_i}, 0, J, H)$, one can form the complex line bundle $u_i^\ast (V\tilde{Y}^\ell)^{\perp_\omega}$ over the Riemann surface $\hat{S}^{\ell,Y}_{i,b}$. By Corollary \ref{Corollary_Ddbar_complex_linear_II}, the operator (as usual for some $kp > 2$)
\begin{multline*}
\overline{D}^{H_b}_{u_i} \definedas \pi^{V\tilde{X}^\ell}_{(V\tilde{Y}^\ell)^{\perp_\omega}}\circ \left(D\dbar^{J_b, H_b}_{\hat{S}^{\ell,Y}_{i,b}}\right)_{u_i} : \\
L^{k,p}(u_i^\ast (V\tilde{Y}^\ell)^{\perp_\omega}) \to L^{k-1,p}(\overline{\Hom}_{(j_b, J_b)}(T\hat{S}^{\ell,Y}_{i,b}, u_i^\ast (V\tilde{Y}^{\ell})^{\perp_\omega}))
\end{multline*}
is complex linear.
By the Koszul-Malgrange integrability theorem, this means that $u_i^\ast (V\tilde{Y}^\ell)^{\perp_\omega}$ is actually (can be identified with) a holomorphic line bundle over $\hat{S}^{\ell,Y}_{i,b}$ with $\overline{D}^{H_b}_{u_i}$ as Cauchy-Riemann operator.
Since $[\pr_2(u_i)] = 0 \in H_2(Y; \Z)$, the bundle $u_i^\ast (V\tilde{Y}^\ell)^{\perp_\omega}$ has vanishing first Chern class and it follows that every meromorphic section of this bundle has the same order of poles as of zeroes.
The compactness result from \cite{MR2026549} or \cite{MR1954264} then implies the following:
\begin{lemma}\label{Lemma_SFT_Compactness}
Let $u \in \cl\overset{\circ}{\mathcal{M}}(\tilde{X}^\ell, \tilde{Y}^\ell, A, J, H)$ and given the data associated to $u$ as in Subsection \ref{Subsection_Description_of_the_closure}, assume that $I^Y \neq \emptyset$. Then
\begin{enumerate}[a)]
  \item If $u^X \definedas u|_{\Sigma^{\ell,X}_{b}}$ is tangent to $\tilde{Y}^\ell$ at $N^{\ell,XY,X}_r$ of order $t_r$ ($t_r = 0$ meaning a transverse intersection), then $\sum_{r=1}^d (t_r+1) = |K^Y|$.
  \item For each $i \in I^Y$, denoting $K^Y_i \definedas \{j \in K^Y \;|\; T^\ell_j(U^\ell) \subseteq \hat{S}^\ell_i\}$, there exists a nonvanishing meromorphic section $\xi_i$ of $u_i^\ast (V\tilde{Y}^\ell_b)^{\perp_\omega}$ with zeroes of order $1$ at the $T^\ell_j(b)$ for $j \in K^Y_i$ and with other zeroes and poles only at the nodal points on $\hat{S}^\ell_{i,b}$.
\end{enumerate}
\end{lemma}
\begin{proof}
Let $(u_j)_{j\in \N} \subseteq \overset{\circ}{\mathcal{M}}(\tilde{X}^\ell, \tilde{Y}^\ell, A, J, H)$ be a sequence that converges to \\ $u \in \cl\overset{\circ}{\mathcal{M}}(\tilde{X}^\ell, \tilde{Y}^\ell, A, J, H)$.
Assume that $I^Y \neq \emptyset$.
Applying the stretching of the neck construction from \cite{MR2026549}, in the limit there exists a holomorphic building in $(V\tilde{Y}^\ell)^{\perp_\omega}\setminus \tilde{Y}^\ell$ ($\tilde{Y}^\ell$ is identified with the zero section in $(V\tilde{Y}^\ell)^{\perp_\omega}$) that projects to $u$ under the projection $(V\tilde{Y}^\ell)^{\perp_\omega} \to \tilde{Y}^\ell$.
In particular, if any of the bad compactness phenomena occur, \ie bubbling off of holomorphic spheres or planes and breaking of holomorphic cylinders, the limit curve needs to be a holomorphic sphere, plane or cylinder in a fibre $(V_z\tilde{Y}^\ell)^{\perp_\omega} \setminus \{0\}$ of $(V\tilde{Y}^\ell)^{\perp_\omega} \setminus \tilde{Y}^\ell$ for some $z \in \Sigma^\ell_b$.
This excludes the possibility of bubbling and the only holomorphic cylinders that can occur are the trivial ones that under an identification $(V_z\tilde{Y}^\ell)^{\perp_\omega} \setminus \{0\} \cong \C^\ast$ are of the form $z \mapsto z^k$ for some $k\in \Z$.
So the underlying nodal Riemannian surface of the holomorphic building is of the following form:
Given $k \geq 0$, let $(Z(k), \nu^k, r^k_\pm)$ be the marked nodal surface with $Z(k) \definedas \coprod_{i=1}^k \overline{\C}\times \{i\}$, where $\overline{\C} \definedas \C\cup \{\infty\} \cong S^2$, nodal points $\nu^k \definedas \{\{(\infty,i), (0,i+1)\} \;|\; i = 1, \dots, k-1\}$ and marked points $r^k_+ \definedas (\infty,k)$, $r^k_- \definedas (0,0)$. \\
\bigskip
\\
\def\svgwidth{\textwidth}
\ifpdf
\begingroup%
  \makeatletter%
  \providecommand\color[2][]{%
    \errmessage{(Inkscape) Color is used for the text in Inkscape, but the package 'color.sty' is not loaded}%
    \renewcommand\color[2][]{}%
  }%
  \providecommand\transparent[1]{%
    \errmessage{(Inkscape) Transparency is used (non-zero) for the text in Inkscape, but the package 'transparent.sty' is not loaded}%
    \renewcommand\transparent[1]{}%
  }%
  \providecommand\rotatebox[2]{#2}%
  \ifx\svgwidth\undefined%
    \setlength{\unitlength}{482.49188538bp}%
    \ifx\svgscale\undefined%
      \relax%
    \else%
      \setlength{\unitlength}{\unitlength * \real{\svgscale}}%
    \fi%
  \else%
    \setlength{\unitlength}{\svgwidth}%
  \fi%
  \global\let\svgwidth\undefined%
  \global\let\svgscale\undefined%
  \makeatother%
  \begin{picture}(1,0.15952865)%
    \put(0,0){\includegraphics[width=\unitlength]{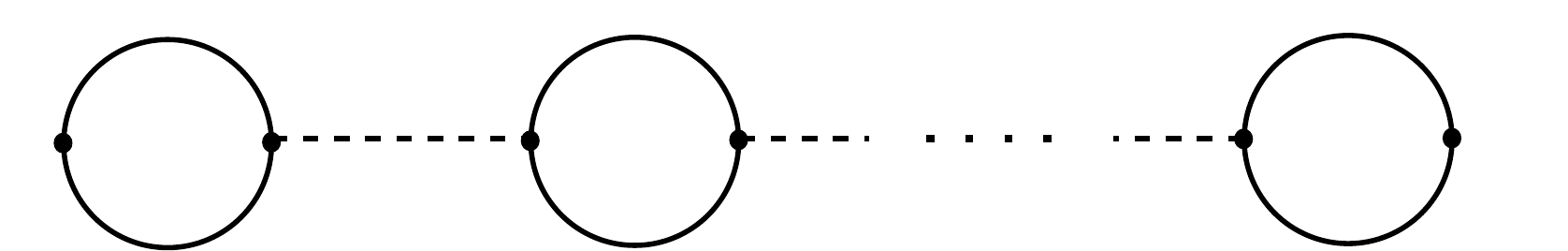}}%
    \put(0.06365653,0.14945076){\color[rgb]{0,0,0}\makebox(0,0)[lb]{\smash{$\overline{\C}\times \{1\}$}}}%
    \put(0.36210715,0.14945076){\color[rgb]{0,0,0}\makebox(0,0)[lb]{\smash{$\overline{\C}\times \{2\}$}}}%
    \put(0.81641531,0.14945076){\color[rgb]{0,0,0}\makebox(0,0)[lb]{\smash{$\overline{\C}\times \{k\}$}}}%
    \put(-0.00110105,0.06402679){\color[rgb]{0,0,0}\makebox(0,0)[lb]{\smash{$r^k_-$}}}%
    \put(0.93862227,0.06519864){\color[rgb]{0,0,0}\makebox(0,0)[lb]{\smash{$r^k_+$}}}%
    \put(0.17753517,0.04033354){\color[rgb]{0,0,0}\makebox(0,0)[lb]{\smash{$(\infty, 1)$}}}%
    \put(0.27446898,0.04033357){\color[rgb]{0,0,0}\makebox(0,0)[lb]{\smash{$(0,2)$}}}%
    \put(0.476997,0.03917961){\color[rgb]{0,0,0}\makebox(0,0)[lb]{\smash{$(\infty,2)$}}}%
    \put(0.72859941,0.03925248){\color[rgb]{0,0,0}\makebox(0,0)[lb]{\smash{$(0,k)$}}}%
  \end{picture}%
\endgroup%

\else
\begingroup%
  \makeatletter%
  \providecommand\color[2][]{%
    \errmessage{(Inkscape) Color is used for the text in Inkscape, but the package 'color.sty' is not loaded}%
    \renewcommand\color[2][]{}%
  }%
  \providecommand\transparent[1]{%
    \errmessage{(Inkscape) Transparency is used (non-zero) for the text in Inkscape, but the package 'transparent.sty' is not loaded}%
    \renewcommand\transparent[1]{}%
  }%
  \providecommand\rotatebox[2]{#2}%
  \ifx\svgwidth\undefined%
    \setlength{\unitlength}{504.2410141bp}%
    \ifx\svgscale\undefined%
      \relax%
    \else%
      \setlength{\unitlength}{\unitlength * \real{\svgscale}}%
    \fi%
  \else%
    \setlength{\unitlength}{\svgwidth}%
  \fi%
  \global\let\svgwidth\undefined%
  \global\let\svgscale\undefined%
  \makeatother%
  \begin{picture}(1,0.16534013)%
    \put(0,0){\includegraphics[width=\unitlength]{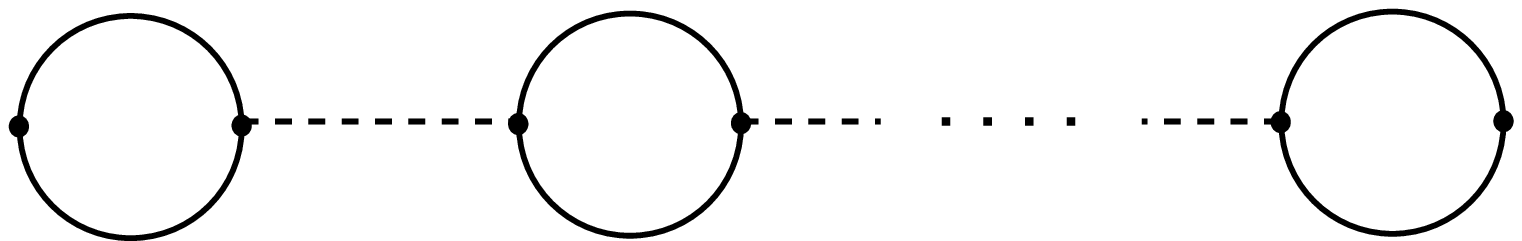}}%
    \put(0.10111538,0.15569693){\color[rgb]{0,0,0}\makebox(0,0)[lb]{\smash{$\overline{\C}\times \{1\}$}}}%
    \put(0.37082767,0.15569693){\color[rgb]{0,0,0}\makebox(0,0)[lb]{\smash{$\overline{\C}\times \{2\}$}}}%
    \put(0.79919426,0.15569693){\color[rgb]{0,0,0}\makebox(0,0)[lb]{\smash{$\overline{\C}\times \{k\}$}}}%
    \put(-0.00105356,0.06443825){\color[rgb]{0,0,0}\makebox(0,0)[lb]{\smash{$r^k_-$}}}%
    \put(0.94468791,0.06238646){\color[rgb]{0,0,0}\makebox(0,0)[lb]{\smash{$r^k_+$}}}%
    \put(0.21642833,0.03859386){\color[rgb]{0,0,0}\makebox(0,0)[lb]{\smash{$(\infty, 1)$}}}%
    \put(0.32504658,0.03859389){\color[rgb]{0,0,0}\makebox(0,0)[lb]{\smash{$(0,2)$}}}%
    \put(0.50297365,0.04066279){\color[rgb]{0,0,0}\makebox(0,0)[lb]{\smash{$(\infty,2)$}}}%
    \put(0.75641618,0.03755942){\color[rgb]{0,0,0}\makebox(0,0)[lb]{\smash{$(0,k)$}}}%
  \end{picture}%
\endgroup%

\fi

Starting with $\hat{S}^{\ell,Y}_b$,
\begin{itemize}
  \item for every $j \in K^Y$, either add components $Z(k)$ for some $k \geq 1$, add new nodal pairs $\{T^\ell_j(b), r^k_+\}$, $\nu^k$ and replace $T^\ell_j(b)$ by the new marked point $Z_j = r^k_-$ or just rename $T^\ell_j(b)$ to $Z_j$ (corresponding to $k=0$).
  \item For every $r = 1, \dots, d'$, either add components $Z(k)$ for some $k\geq 1$ and new nodal pairs $\{N^{\ell,Y,1}_r(b), r^k_-\}$, $\{N^{\ell,Y,2}_r(b), r^k_+\}$, $\nu^k$ or do nothing.
  \item For every $r = 1, \dots, d$, either add components $Z(k)$ for some $k \geq 1$, add new nodal pairs $\{N^{\ell,XY,Y}_r(b), r^k_-\}$, $\nu^k$ and replace $N^{\ell,XY,Y}_r(b)$ by the new marked point $P_r = r^k_+$ or just rename $N^{\ell,XY,Y}_r(b)$ to $P_r$ (corresponding to $k=0$).
\end{itemize}
The result of the above is a marked nodal Riemann surface $\overline{S}$ with a set of components $C$, nodal pairs $\{\{N^+_1, N^-_1\}, \dots, \{N^+_s, N^-_s\}\}$ (after renaming) and marked points $Z_1, \dots, Z_{|K^Y|}$ and $P_1, \dots, P_d$.
Let $\overline{u} : \overline{S} \to \tilde{Y}^\ell$ defined as follows: For every $i \in C$, let $\overline{u}_i \definedas \overline{u}|_{\overline{S}_i}$ be defined as $u_i$ if $\overline{S}_i$ is one of the original components $\hat{S}^{\ell,Y}_{i, b}$ and on every new component as the constant map to $u(z)$, where $z = T^\ell_j(b), N^{\ell,Y,1}_r(b), N^{\ell,XY,Y}(b)$ is the point on $\hat{S}^{\ell,Y}_b$ at which the new component is attached.

Then there exist the following:
\begin{enumerate}
  \item An integer $k\in\N$ and partitions $C = C_{-k} \amalg C_{-k+1} \amalg\cdots \amalg C_{-1}$.
  \item Up to reordering of the nodes $N^+_r, N^-_r$, \ie reordering of the index set $\{1,\dots,s\}$ and exchanging $N^+_r$ and $N^-_r$ for a fixed $r$, a partition $\{1, \dots, s\} = F_{-k} \amalg \cdots \amalg F_{-2} \amalg G{-k} \amalg \cdots \amalg G_{-1}$, $F_{-1} \definedas \{1, \dots, d\}$.
  \item For every $j = -k, \dots, -1$, $r \in F_j$, an integer $p^j_r\in \N$.
\end{enumerate}
For these, the following hold:
\begin{enumerate}[a)]
  \item For all $j = 1, \dots, |K^Y|$, $Z_j$ lies on $\overline{S}_i$ for some $i\in C_{-k}$.
  \item For all $j = 1, \dots, d$, $P_j$ lies on $\overline{S}_i$ for some $i \in C_{-1}$.
  \item For all $j = -k, \dots, -1$, $r\in G_j$, there are $i,i' \in C_j$ \st $N^+_r$ lies on $\overline{S}_i$ and $N^-_r$ lies on $\overline{S}_{i'}$.
  \item For all $j = -k, \dots, -2$, $r\in F_j$, $N^-_r$ lies on $\overline{S}_i$ for some $i\in C_j$ and $N^+_r$ lies on $\overline{S}_{i'}$ for some $i'\in C_{j+1}$.
  \item For each $i \in C$, there exists a meromorphic section $\xi_i$ of $\overline{u}_i^\ast (V\tilde{Y}^\ell)^{\perp_\omega}$ with the following properties:
      \begin{enumerate}[i)]
        \item For all $i\in C$, $\xi_i$ has simple zeroes at the points $Z_r$ for $r = 1, \dots, |K^Y|$ with $Z_r \in \overline{S}_i$.
        \item For all $j = -k, \dots, -2$, $r\in F_j$, let $i\in C_j, i' \in C_{j+1}$ be \st $N^-_r \in \overline{S}_i$, $N^+_r \in \overline{S}_{i'}$. Then $\xi_i$ has a pole of order $p_r^j$ at the point $N^-_r$ and $\xi_{i'}$ has a zero of order $p^j_r$ at the point $N^+_r$.
        \item For every $r\in D^{-1}$, let $i \in C_{-1}$ be \st $P_r \in \overline{S}_i$. Then $\xi_i$ has a pole of order $p^{-1}_r$ at the point $P_r$.
        \item Other than the above, the $\xi^i$ have no zeroes or poles.
        \item For every $r\in D^{-1}$, $u^X$ is tangent to $\tilde{Y}^\ell$ to order $p^{-1}_r-1$ at $N^{\ell,XY,X}_r$.
      \end{enumerate}
\end{enumerate}
Note that the above gives a countable number of choices: For the number and mode of attachment of the additional components, the integer $k$, the partitions $C_\ast$ and $F_\ast \amalg G_\ast$ and the orders of the zeros and poles $p^\ast_\ast$ of the $\xi_i$.
Also note that as remarked above, every $\xi_i$ has the same (total) order of zeroes as of poles, \ie at each level $j$ the total order of zeroes of the $\xi_i$, $i \in C_j$, is given by the total order of poles of $\xi_{i'}$, $i' \in C_{j-1}$, for $j
\geq -k+1$ or the number of marked points $Z_r$ for $j = -k$ and this is the same as the total order of the poles of the $\xi_i$ for $i \in C_j$. Hence the total order of the poles of the $\xi_i$ for $i \in C_{-1}$ is given by $|K^Y|$.
\end{proof}

\clearpage
\begin{landscape}
\ifpdf
\enlargethispage{8cm}
\def\svgwidth{30cm}
\addtolength{\oddsidemargin}{-2cm}
\addtolength{\evensidemargin}{-2cm}
\begingroup%
  \makeatletter%
  \providecommand\color[2][]{%
    \errmessage{(Inkscape) Color is used for the text in Inkscape, but the package 'color.sty' is not loaded}%
    \renewcommand\color[2][]{}%
  }%
  \providecommand\transparent[1]{%
    \errmessage{(Inkscape) Transparency is used (non-zero) for the text in Inkscape, but the package 'transparent.sty' is not loaded}%
    \renewcommand\transparent[1]{}%
  }%
  \providecommand\rotatebox[2]{#2}%
  \ifx\svgwidth\undefined%
    \setlength{\unitlength}{977.30082086bp}%
    \ifx\svgscale\undefined%
      \relax%
    \else%
      \setlength{\unitlength}{\unitlength * \real{\svgscale}}%
    \fi%
  \else%
    \setlength{\unitlength}{\svgwidth}%
  \fi%
  \global\let\svgwidth\undefined%
  \global\let\svgscale\undefined%
  \makeatother%
  \begin{picture}(1,0.56785587)%
    \put(0,0){\includegraphics[width=\unitlength]{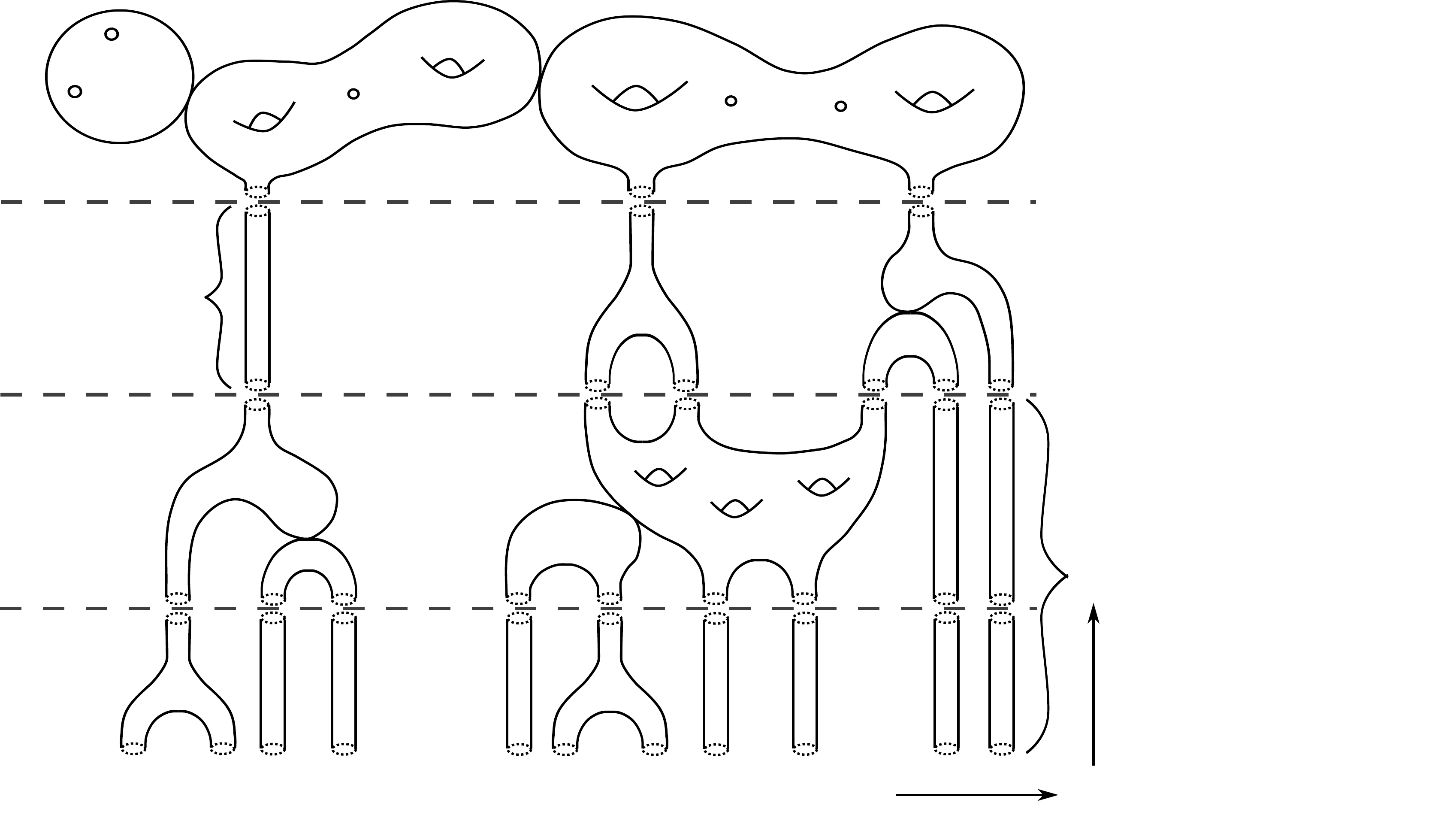}}%
    \put(0.66086681,0.00212558){\color[rgb]{0,0,0}\makebox(0,0)[lb]{\smash{$\tilde{Y}^\ell$}}}%
    \put(0.75948445,0.08860001){\color[rgb]{0,0,0}\makebox(0,0)[lb]{\smash{$(V\tilde{Y}^\ell)^{\perp_\omega}$}}}%
    \put(0.11076341,0.36081875){\color[rgb]{0,0,0}\makebox(0,0)[lb]{\smash{$Z(1)$}}}%
    \put(0.73952209,0.16872763){\color[rgb]{0,0,0}\makebox(0,0)[lb]{\smash{$Z(2)$}}}%
    \put(0.18942501,0.30650256){\color[rgb]{0,0,0}\makebox(0,0)[lb]{\smash{$N^+_r$}}}%
    \put(0.19011924,0.2815255){\color[rgb]{0,0,0}\makebox(0,0)[lb]{\smash{$N^-_r$}}}%
    \put(0.22708739,0.28071991){\color[rgb]{0,0,0}\makebox(0,0)[lb]{\smash{$r\in F_{-2}$}}}%
    \put(0.19886195,0.20974976){\color[rgb]{0,0,0}\makebox(0,0)[lb]{\smash{$N^+_{r'}$}}}%
    \put(0.19875059,0.18291285){\color[rgb]{0,0,0}\makebox(0,0)[lb]{\smash{$N^-_{r'}$}}}%
    \put(0.23466438,0.19714989){\color[rgb]{0,0,0}\makebox(0,0)[lb]{\smash{$r'\in G_{-2}$}}}%
    \put(0.08370678,0.03522956){\color[rgb]{0,0,0}\makebox(0,0)[lb]{\smash{$Z_1$}}}%
    \put(0.14629427,0.035701){\color[rgb]{0,0,0}\makebox(0,0)[lb]{\smash{$Z_2$}}}%
    \put(0.68112589,0.035701){\color[rgb]{0,0,0}\makebox(0,0)[lb]{\smash{$Z_{|K^Y|}$}}}%
    \put(0.19186779,0.41244352){\color[rgb]{0,0,0}\makebox(0,0)[lb]{\smash{$P_1$}}}%
    \put(0.45450947,0.41208344){\color[rgb]{0,0,0}\makebox(0,0)[lb]{\smash{$P_2$}}}%
    \put(0.64929201,0.41210937){\color[rgb]{0,0,0}\makebox(0,0)[lb]{\smash{$P_3$}}}%
    \put(0.00560516,0.35859875){\color[rgb]{0,0,0}\makebox(0,0)[lb]{\smash{{\large $j = -1$}}}}%
    \put(0.006188,0.22035245){\color[rgb]{0,0,0}\makebox(0,0)[lb]{\smash{{\large $j = -2$}}}}%
    \put(0.00560516,0.09468001){\color[rgb]{0,0,0}\makebox(0,0)[lb]{\smash{{\large $j = -k$}}}}%
  \end{picture}%
\endgroup%

\else
\enlargethispage{4cm}
\def\svgwidth{24cm}
\begingroup%
  \makeatletter%
  \providecommand\color[2][]{%
    \errmessage{(Inkscape) Color is used for the text in Inkscape, but the package 'color.sty' is not loaded}%
    \renewcommand\color[2][]{}%
  }%
  \providecommand\transparent[1]{%
    \errmessage{(Inkscape) Transparency is used (non-zero) for the text in Inkscape, but the package 'transparent.sty' is not loaded}%
    \renewcommand\transparent[1]{}%
  }%
  \providecommand\rotatebox[2]{#2}%
  \ifx\svgwidth\undefined%
    \setlength{\unitlength}{977.30082086bp}%
    \ifx\svgscale\undefined%
      \relax%
    \else%
      \setlength{\unitlength}{\unitlength * \real{\svgscale}}%
    \fi%
  \else%
    \setlength{\unitlength}{\svgwidth}%
  \fi%
  \global\let\svgwidth\undefined%
  \global\let\svgscale\undefined%
  \makeatother%
  \begin{picture}(1,0.56785587)%
    \put(0,0){\includegraphics[width=\unitlength]{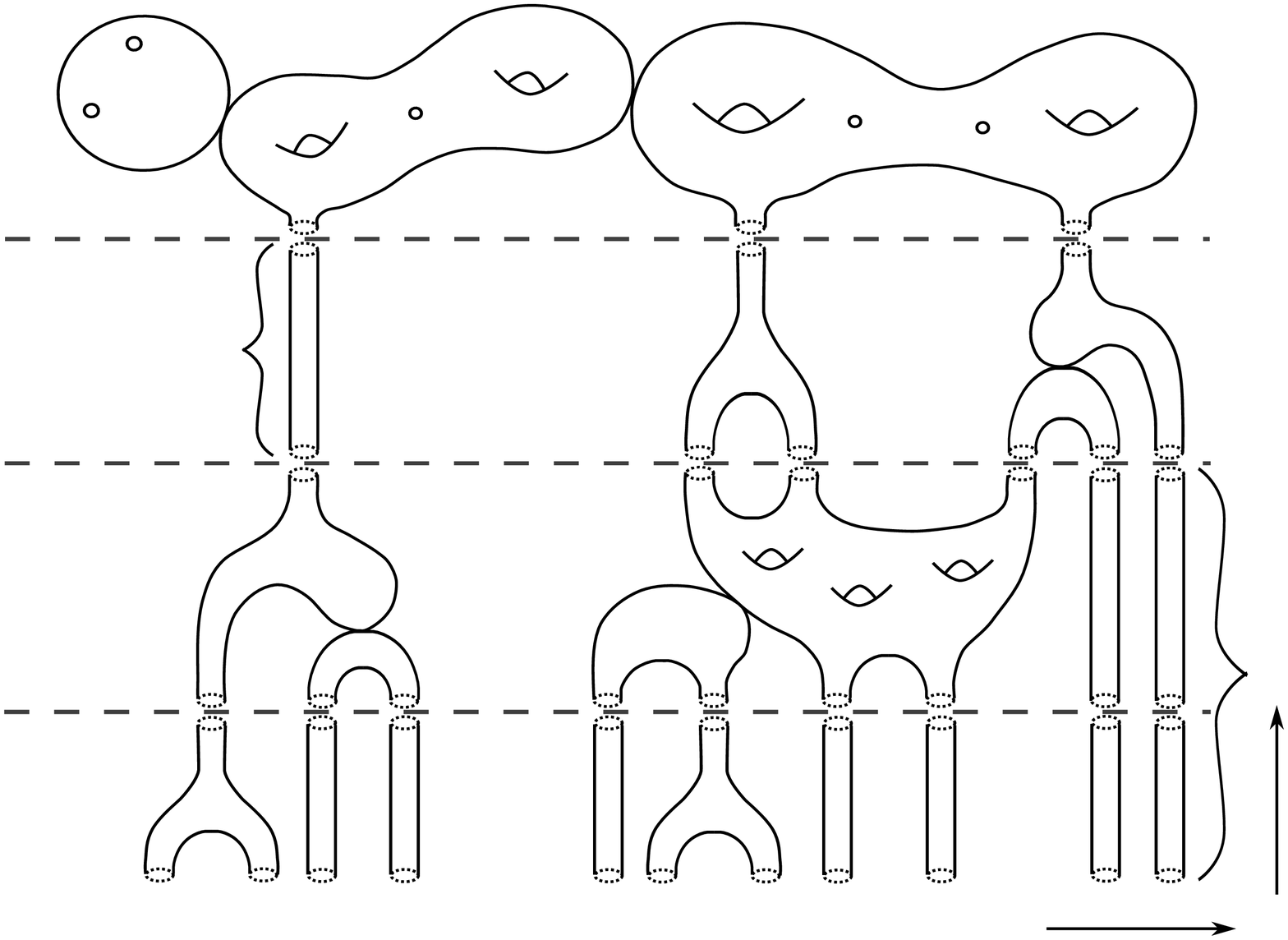}}%
    \put(0.66086681,0.00212558){\color[rgb]{0,0,0}\makebox(0,0)[lb]{\smash{$\tilde{Y}^\ell$}}}%
    \put(0.75948445,0.08860001){\color[rgb]{0,0,0}\makebox(0,0)[lb]{\smash{$(V\tilde{Y}^\ell)^{\perp_\omega}$}}}%
    \put(0.11076341,0.36081875){\color[rgb]{0,0,0}\makebox(0,0)[lb]{\smash{$Z(1)$}}}%
    \put(0.73952209,0.16872763){\color[rgb]{0,0,0}\makebox(0,0)[lb]{\smash{$Z(2)$}}}%
    \put(0.18942501,0.30650256){\color[rgb]{0,0,0}\makebox(0,0)[lb]{\smash{$N^+_r$}}}%
    \put(0.19011924,0.2815255){\color[rgb]{0,0,0}\makebox(0,0)[lb]{\smash{$N^-_r$}}}%
    \put(0.22708739,0.28071991){\color[rgb]{0,0,0}\makebox(0,0)[lb]{\smash{$r\in F_{-2}$}}}%
    \put(0.19886195,0.20974976){\color[rgb]{0,0,0}\makebox(0,0)[lb]{\smash{$N^+_{r'}$}}}%
    \put(0.19875059,0.18291285){\color[rgb]{0,0,0}\makebox(0,0)[lb]{\smash{$N^-_{r'}$}}}%
    \put(0.23466438,0.19714989){\color[rgb]{0,0,0}\makebox(0,0)[lb]{\smash{$r'\in G_{-2}$}}}%
    \put(0.08370678,0.03522956){\color[rgb]{0,0,0}\makebox(0,0)[lb]{\smash{$Z_1$}}}%
    \put(0.14629427,0.035701){\color[rgb]{0,0,0}\makebox(0,0)[lb]{\smash{$Z_2$}}}%
    \put(0.68112589,0.035701){\color[rgb]{0,0,0}\makebox(0,0)[lb]{\smash{$Z_{|K^Y|}$}}}%
    \put(0.19186779,0.41244352){\color[rgb]{0,0,0}\makebox(0,0)[lb]{\smash{$P_1$}}}%
    \put(0.45450947,0.41208344){\color[rgb]{0,0,0}\makebox(0,0)[lb]{\smash{$P_2$}}}%
    \put(0.64929201,0.41210937){\color[rgb]{0,0,0}\makebox(0,0)[lb]{\smash{$P_3$}}}%
    \put(0.00560516,0.35859875){\color[rgb]{0,0,0}\makebox(0,0)[lb]{\smash{{\large $j = -1$}}}}%
    \put(0.006188,0.22035245){\color[rgb]{0,0,0}\makebox(0,0)[lb]{\smash{{\large $j = -2$}}}}%
    \put(0.00560516,0.09468001){\color[rgb]{0,0,0}\makebox(0,0)[lb]{\smash{{\large $j = -k$}}}}%
  \end{picture}%
\endgroup%

\fi
\end{landscape}
\clearpage

For every $U^\ell$ and desingularisation $\hat{S}^\ell$, \etc, as above, $I^X$ and $I^Y$, for each $i \in I^Y$, rename the special points, \ie the $T^\ell_j$ for $j\in K^Y$, $N^{\ell,Y,s}_r$ for $r = 1, \dots, d', s = 1,2$, and $N^{\ell,XY,Y}_r$ for $r = 1, \dots, d$ that lie on $\hat{S}^{\ell,Y}_i$ to $N^i_r$ for $r = 1, \dots, k_i$ and some $k_i \in \N_0$.
Furthermore, also for each $i \in I^Y$ and $r = 1, \dots, k_i$, let $p^i_r \in \Z$ denote given orders of zeroes or poles at the $N^i_r$, $p^i_r > 0$ denoting a zero of order $p^i_r$ at $N^i_r$, $p^i_r < 0$ denoting a pole of order $-p^i_r$ at $N^i_r$ and $p^i_r = 0$ denoting that neither a zero nor a pole occurs at $N^i_r$, subject to the condition that they can appear in one of the configurations from Lemma \ref{Lemma_SFT_Compactness}.
Abbreviate such a choice of data as above by $\mathcal{D}$ (for which there are countably many choices).
\begin{defn}\label{Definition_M_D_Y}
Given $\mathcal{D}$ as above let $J \in \mathcal{J}_{\omega, \mathrm{ni}}(X, Y, E)$ and $H^Y \in \mathcal{H}_{\mathrm{reg}}(\tilde{Y}, J)$.
Define for $u\in \mathcal{M}(\tilde{Y}^\ell|_{\Sigma^{\ell,Y}_{U^\ell}}, 0, J, H^Y)$ and $H^0 \in \mathcal{H}^0_{\mathrm{ni}}(\tilde{X}^\ell, \tilde{Y}^\ell, J)$,
\begin{align*}
\mathcal{M}^{\mathcal{D}}_{Y,u}(\tilde{X}^\ell, \tilde{Y}^\ell, H^Y + H^0) \definedas \{(\xi_i)_{i\in I^Y} \;|\; & \xi_i \text{ a meromorphic section of } \\
& u_i^\ast (V\tilde{Y}^\ell)^{\perp_\omega} \text{ with zeroes/nodes} \\
& \text{at the $N^\ast_\ast$ given by the $p^\ast_\ast$}\}/_{(\C^\ast)^{I^Y}}
\end{align*}
and as usual
\begin{align*}
\mathcal{M}^{\mathcal{D}}_Y(\tilde{X}^\ell, \tilde{Y}^\ell, H^Y + H^0) &\definedas 
\coprod_{u\in \mathcal{M}(\tilde{Y}^\ell|_{\Sigma^{\ell,Y}_{U^\ell}}, 0, J, H^Y)} \mathcal{M}^{\mathcal{D}}_Y(\tilde{X}^\ell, \tilde{Y}^\ell, H^Y + H^0)
\intertext{and}
\mathcal{M}^{\mathcal{D}}_Y(\tilde{X}^\ell, \tilde{Y}^\ell, H^Y + \mathcal{H}^0_{\mathrm{ni}}(\tilde{X}^\ell, \tilde{Y}^\ell, J)) &\definedas
\coprod_{H^0\in \mathcal{H}_{\mathrm{ni}}(\tilde{X}^\ell, \tilde{Y}^\ell, J)}\mathcal{M}^{\mathcal{D}}_Y(\tilde{X}^\ell, \tilde{Y}^\ell, H^Y + H^0)\text{.}
\end{align*}
Furthermore, denote by
\begin{align*}
\pi^{\mathcal{M}^{\mathcal{D}}_Y}_{\mathcal{M}(\tilde{Y}^\ell)} : \mathcal{M}^{\mathcal{D}}_Y(\tilde{X}^\ell, \tilde{Y}^\ell, H^Y + \mathcal{H}^0_{\mathrm{ni}}(\tilde{X}^\ell, \tilde{Y}^\ell, J)) &\to \mathcal{M}(\tilde{Y}^\ell|_{\Sigma^{\ell,Y}_{U^\ell}}, 0, J, H^Y)
\intertext{and}
\pi^{\mathcal{M}^{\mathcal{D}}_Y}_{\mathcal{H}} : \mathcal{M}^{\mathcal{D}}_Y(\tilde{X}^\ell, \tilde{Y}^\ell, H^Y + \mathcal{H}^0_{\mathrm{ni}}(\tilde{X}^\ell, \tilde{Y}^\ell, J)) &\to \mathcal{H}^0_{\mathrm{ni}}(\tilde{X}^\ell, \tilde{Y}^\ell, J)
\intertext{the projections and by abuse of notation denote by}
\ev^{N^{\ell,XY,Y}} : \mathcal{M}^{\mathcal{D}}_Y(\tilde{X}^\ell, \tilde{Y}^\ell, H^Y + \mathcal{H}^0_{\mathrm{ni}}(\tilde{X}^\ell, \tilde{Y}^\ell, J)) &\to \bigoplus_{r = 1}^d (\hat{\iota}^\ell\circ N^{\ell,XY,Y}_r)^\ast \tilde{Y}^\ell
\end{align*}
the composition of $\pi^{\mathcal{M}^{\mathcal{D}}_Y}_{\mathcal{M}(\tilde{Y}^\ell)}$ with the evaluation map at the marked points $N^{\ell,XY,Y}_r$, on $\mathcal{M}(\tilde{Y}^\ell|_{\Sigma^{\ell,Y}_{U^\ell}}, 0, J, H^Y)$.
\end{defn}
The above is well-defined, because the Cauchy-Riemann operators $\overline{D}^{H_b}_{u_i}$ are complex linear for $H \in \mathcal{H}_{\mathrm{ni}}(\tilde{X}^\ell, \tilde{Y}^\ell)$, hence the spaces of meromorphic sections are invariant under fibrewise complex multiplication in $(V\tilde{Y}^\ell)^{\perp_\omega}$.
Also, note that for $H^0 \in \mathcal{H}^0(\tilde{X}^\ell, \tilde{Y}^\ell)$, $\mathcal{M}(\tilde{Y}^\ell|_{\Sigma^{\ell,Y}_{U^\ell}}, 0, J, H^Y + H^0) = \mathcal{M}(\tilde{Y}^\ell|_{\Sigma^{\ell,Y}_{U^\ell}}, 0, J, H^Y)$.

\begin{lemma}\label{Lemma_SFT_Compactness_II}
Let $u\in \cl\overset{\circ}{\mathcal{M}}(\tilde{X}^\ell, \tilde{Y}^\ell, A, J, H^Y + H^0)$ for some $J \in \mathcal{J}_{\omega, \mathrm{ni}}(X, Y, E)$, $H^Y \in \mathcal{H}(\tilde{Y},J)$ and $H^0 \in \mathcal{H}^0_{\mathrm{ni}}(\tilde{X}^\ell, \tilde{Y}^\ell, J)$. \\
Then there exists $\mathcal{D}$ as above \st the restriction $u^Y$ lies in the image of $\pi^{\mathcal{M}^{\mathcal{D}}_Y}_{\mathcal{M}(\tilde{Y}^\ell)}$ in $\mathcal{M}(\tilde{Y}^\ell|_{\Sigma^{\ell,Y}_{U^\ell}}, 0, J, H^Y + H^0)$. \\
Furthermore, $u^X$ has total order of tangency at the $N^{\ell,XY,X}_r$, $r = 1, \dots, d$, given by $|K^Y| - d$.
\end{lemma}
\begin{proof}
This is just a reformulation of Lemma \ref{Lemma_SFT_Compactness}.
\end{proof}

\begin{lemma}\label{Lemma_H_0_reg}
Given any $H^Y\in \mathcal{H}_{\mathrm{reg}}(\tilde{Y}, J)$, there exists a generic subset $\mathcal{H}_{\mathrm{reg}}^0(\tilde{X}^\ell, \tilde{Y}^\ell, J, H^Y)$ of $\mathcal{H}^0_{\mathrm{ni}}(\tilde{X}^\ell, \tilde{Y}^\ell, J)$ \st for every $H^0 \in \mathcal{H}_{\mathrm{reg}}^0(\tilde{X}^\ell, \tilde{Y}^\ell, J, H^Y)$ and every choice of $\mathcal{D}$,
\[
\mathcal{M}^{\mathcal{D}}_Y(\tilde{X}^\ell, \tilde{Y}^\ell, J, H^Y + H^0)
\]
is a smooth manifold of dimension
\begin{align*}
\dim_\R\left(\mathcal{M}^{\mathcal{D}}_Y(\tilde{X}^\ell, \tilde{Y}^\ell, J, H^Y + H^0)\right) = \dim_\C(X)\chi^Y + \dim_\R(U^\ell) + 2d' - 2|I^Y|\text{.}
\end{align*}
\end{lemma}
\begin{proof}
Consider the family of Riemann surfaces $\rho : P \to B$, where the base $B$ is given by $\mathcal{M}(\tilde{Y}^\ell|_{\Sigma^{\ell,Y}_{U^\ell}}, 0, J, H^Y)$ and the family $P$ of (disconnected smooth) Riemann surfaces over $B$ is given by $(\pi^\mathcal{M}_{U^{\ell}})^\ast \hat{S}^{\ell,Y}$.
Fibrewise deleting the nodal points $N^{i}_r$, for $i\in I^Y$ and $r = 1, \dots, k_i$, where applicable, gives a family of punctured Riemann surfaces $\dot{\rho} : \dot{P} \to B$.
Over $P$ and by restriction over $\dot{P}$, there is a complex line bundle $Z \to P$, where for $u\in B$, $Z|_{P_u} = (u)^\ast (V\tilde{Y}^\ell)^{\perp_\omega}$.
The complex structure is given by the restriction of $J$ to $(V\tilde{Y}^\ell)^{\perp_\omega}$ and is compatible with the restriction of $\omega$.
By abuse of notation, both these structures will be denoted by $J$ and $\omega$ again.
This complex line bundle can hence be regarded as a symplectic fibre bundle with real $2$-dimensional fibres and deleting the zero-section also gives a symplectic fibre bundle $\dot{Z} \to \dot{P}$.
An important property of this bundle is that it comes with a free action of $(\C^\ast)^{I^Y}$ ($\C^\ast \definedas \C\setminus \{0\}$) on the fibres of $\dot{Z}$. For $u\in B$, $\pi^{\mathcal{M}}_{U^\ell}(u) = b$, the $i$-th component of $(\C^\ast)^{I^Y}$ acts fibrewise on $(u_i)^\ast(V\tilde{Y}^\ell)^{\perp_\omega}$, $u_i \definedas u|_{\hat{S}^{\ell, Y}_{i,b}}$.
The restriction of the $\xi_u^j$ as above to the components of $P_u$ then defines a section of $\dot{Z}$ over $\dot{P}_u$.
Finally, this bundle also comes with a connection, induced by the Levi-Civita connection on $V\tilde{X}^\ell$.
Next, observe that the operator $\overline{D}_{i,u}^{H_b}$ above is a complex linear Cauchy-Riemann operator in the sense of \cite{MR2045629}, Appendix C.1: Let $u\in \mathcal{M}_b(\tilde{Y}^\ell|_{\hat{S}^{\ell,Y}_i}, 0, J, (\pi_0^\ell)^\ast H)$ for any component given by $i\in I^Y$, where $H \in \mathcal{H}_{\mathrm{ni}}(\tilde{X}^\ell, \tilde{Y}^\ell, J)$ restricts to $H^Y$ along $\tilde{Y}^\ell$.
Then because $J \in \mathcal{J}_{\omega, \mathrm{ni}}(\tilde{X}^\ell, \tilde{Y}^\ell)$, $H \in \mathcal{H}_{\mathrm{ni}}(\tilde{X}^\ell, \tilde{Y}^\ell)$, by Lemma \ref{Lemma_Ddbar}, for a section $\xi$ of $u^\ast (V\tilde{Y}^\ell)^{\perp_\omega}$, $Z \in T\hat{S}^{\ell,Y}_{i,b}$,
\begin{equation}\label{Equation_CR_on_normal_bundle}
(\overline{D}_{i,u}^{H_b}\xi)(Z) = \pi^{V\tilde{X}^\ell_b}_{(V\tilde{Y}^\ell_b)^{\perp_\omega}}\left(\nabla^{0,1}_Z\xi - K_{\hat{J}^{H}}(\xi, Du(Z))\right)\text{.}
\end{equation}
These observations allow one to define a Fredholm problem whose solutions are the meromorphic sections from above.
Basically, one now goes through the steps of the previous chapter. What was the bundle $\pi : \Sigma \to B$ there now is the bundle $\rho : \dot{P} \to B$, and what was the bundle $\tilde{X} \to \Sigma$ there now is the bundle $\dot{Z} \to \dot{P}$.
One just has to be more careful in the definition of the (linear) Sobolev spaces, see \eg \cite{MR610188} and \cite{MR879560}, but as long as the definition is such that the usual embedding theorems, elliptic estimates for the Cauchy-Riemann operator, etc., hold, the details are not that important.
Also remember that $\hat{S}^\ell \to U^\ell$ was a (topologically but not holomorphically) trivial bundle and that one can choose tubular neighbourhoods of all the marked points and nodal points on which the trivialisation preserves the complex structure in the fibres of $\hat{S}^\ell$.
This allows one to use the SFT Fredholm theory from \cite{Cieliebak_Lecture_Notes_SFT_II}.
Because $\hat{S}^\ell \to U^\ell$ is holomorphically trivial in a neighbourhood of all the nodes, and hence so is $P \to B$, one can pick holomorphic coordinates defined on $[0,\infty) \times S^1$ to punctured disk neighbourhoods $D^i_r$ of the $N^i_r$ in $\dot{P}$ that are preserved under the (smooth) identification of the fibres of $\dot{P}$.
Denote the resulting maps $\sigma^i_r : B\times [0,\infty)\times S^1 \to \dot{P}$.
Also note that as was remarked above, by the Koszul-Malgrange integrability theorem (and because everything extends from a punctured disk to a disk), for every $b\in B$, over $D^i_{r,b}$ one can find a holomorphic trivialisation of $\dot{Z}|_{D^i_{r,b}}$ with fibre $E^i_{r,b} \cong \C\setminus \{0\}$, \ie $\dot{Z}|_{D^i_{r,b}} \cong D^i_{r,b}\times E^i_{r,b}$.
This gives maps $\overline{\sigma}^i_r : B\times ([0,\infty)\times S^1)\times (\C\setminus\{0\}) \to \dot{Z}$ covering the maps $\sigma^i_r$.
If $[0,\infty)\times S^1 \to \mathbb{D}\setminus\{0\}$, $(s,\theta) \mapsto e^{-(s+i\theta)}$ is the standard identification, then a zero or pole of order $p$ that is given in standard coordinates on $\mathbb{D}$ by $z \mapsto cz^p$, for some $c = e^{-(a + i \vartheta)}\in \C$, under this identification is the map $(s, \theta) \mapsto ce^{-p(s + i\theta)} = e^{-(ps + a + i(p\theta + \vartheta))}$.
Fix some $b\in B$, $l\in\N$, $q > 1$ with $lq > 2$, and a weight $\delta > 0$. Furthermore, fix a smooth function $s : \dot{P} \to (0,\infty)$ that in all the coordinates $\sigma^i_r$ from above is given by the projection onto the factor $[0,\infty)$.
Then for any metric vector bundle with connection $E \to \dot{P}_b$, one can define the weighted Sobolev space
\[
L^{l,q,\delta}(E) \definedas \{ \eta \in L^{l,q}_{\mathrm{loc}}(E) \;|\; e^{\delta s}\eta \in L^{l,q}(E) \}\text{.}
\]
With these choices, let (\cf the first definition in Section 3 of \cite{Cieliebak_Lecture_Notes_SFT_II})
\begin{align*}
\mathcal{B}_b \definedas \{\xi : \dot{P}_b \to \dot{Z}_b \;|\; & \xi \text{ a section of } \dot{Z}_b \to \dot{P}_b \text{ of class } L^{l,q}_{\mathrm{loc}} \text{ \st } \forall\, r = 1, \dots, k_io\text{,} \\
& (\pr_2\circ (\overline{\sigma}^i_{r,b})\inv \circ \xi\circ\sigma^i_{r,b})(s,\theta) : \\
& (s,\theta) \mapsto e^{-((t(s,\theta) - (p^i_r s + a^i_r)) + i(\varphi(s,\theta) - (p^i_r \theta + \vartheta^i_r)))} \\
& \in L^{l,q,\delta}([0,\infty)\times S^1, \C) \text{ for some} \\
& (a^i_r, \vartheta^i_r) \in [0,\infty) \times S^1 \text{ and for all $r$} \}\text{.}
\end{align*}
This is a Banach manifold, that around a smooth $\xi \in \mathcal{B}_b$ is modelled on the Banach space $L^{l,q,\delta}(\xi^\ast V\dot{Z}_b)$, just with the Sobolev spaces used in the previous chapter replaced by the weighted Sobolev spaces.
Analogously to the situation in the previous chapter there is then also a Banach space bundle $\mathcal{E}_b$ over $\mathcal{B}_b$, with fibre
\[
(\mathcal{E}_b)_\xi \definedas L^{l-1,q,\delta}(\overline{\Hom}_{(j_b,J_b)}(T\dot{P}_b, \xi^\ast V\dot{Z}_b))
\]
over $\xi\in \mathcal{B}_b$.
$\mathcal{E}_b \to \mathcal{B}_b$ comes with a section $\overline{\nabla}_b$, defined by the Cauchy-Riemann operator from Equation \ref{Equation_CR_on_normal_bundle} and for an appropriate choice of $\delta > 0$, this is a Fredholm operator.
By definition of $\mathcal{B}_b$, the zero set of $\overline{\nabla}_b$ is given by the meromorphic sections of $Z_b \to P_b$ that have zeros and poles at the $N^i_r$, of orders given by the numbers $p^i_r$, respectively.
Because $\overline{\nabla}_b$ is a linear Cauchy-Riemann operator on $Z_b$, its linearisation $D\overline{\nabla}_b$ in the sense of Section 3 in \cite{Cieliebak_Lecture_Notes_SFT_II} is (modulo canonical identifications) given by $\overline{\nabla}_b$ itself.
In particular the operators $S_i(t)$ in op.~cit.~vanish identically, and the paths of symplectic matrices $\Phi_i(t)$, as in the same reference, are the constant paths at the identity.
By Corollary 3.6 in \cite{Cieliebak_Lecture_Notes_SFT_II}, again for $\delta > 0$ sufficiently small, $\overline{\nabla}_b$ is a Fredholm operator of index $\chi(P_b) = \sum_{i \in I^Y}2(1-g_i)$ (not $\chi(\dot{P}_b)$), as was expected from the classical Riemann-Roch theorem from the start.
Now as in the previous chapter, remembering that $P$ and hence $\dot{P}$ were trivial, one can take the union over all $b\in B$ to get a Banach manifold $\mathcal{B}$, together with a projection to $B$, and a Banach space bundle $\mathcal{E}$ over $\mathcal{B}$, together with a Fredholm section $\overline{\nabla} : \mathcal{B} \to \mathcal{E}$ of index $\sum_{i \in I^Y}2(1-g_i) + \dim_\R(B)$.
Remembering that $\overline{\nabla}$ depends on the choice of $H \in \mathcal{H}_{\mathrm{ni}}(\tilde{X}^\ell, \tilde{Y}^\ell, J)$, whereas $B = \mathcal{M}(\tilde{Y}^\ell|_{\Sigma^{\ell,Y}_{U^\ell}}, 0, J, H^Y)$ and hence $P$ and $Z$ only depend on the restriction of $H$ to $\tilde{Y}^\ell$.
So one can look at the affine subspace $H^Y + \mathcal{H}_{\mathrm{ni}}^0(\tilde{X}^\ell, \tilde{Y}^\ell, J)$ of $\mathcal{H}_{\mathrm{ni}}(\tilde{X}^\ell, \tilde{Y}^\ell, J)$.
Making the dependence on $H$ explicit and writing $\mathcal{B}^{H}$, $\mathcal{E}^{H}$ and $\overline{\nabla}^{H}$, one can then as in the previous chapter look at the spaces
\begin{align*}
\mathcal{B}(H^Y + \mathcal{H}_{\mathrm{ni}}^0(\tilde{X}^\ell, \tilde{Y}^\ell, J)) &\definedas \coprod_{H \in H^Y + \mathcal{H}_{\mathrm{ni}}^0(\tilde{X}^\ell, \tilde{Y}^\ell, J)} \mathcal{B}^{H} \\
\mathcal{E}(H^Y + \mathcal{H}_{\mathrm{ni}}^0(\tilde{X}^\ell, \tilde{Y}^\ell, J)) &\definedas \coprod_{H \in H^Y + \mathcal{H}_{\mathrm{ni}}^0(\tilde{X}^\ell, \tilde{Y}^\ell, J)} \mathcal{E}^{H}
\intertext{and the map}
\overline{\nabla}^\mathcal{H} &\definedas \coprod_{H \in H^Y + \mathcal{H}_{\mathrm{ni}}^0(\tilde{X}^\ell, \tilde{Y}^\ell, J)} \overline{\nabla}^{H}
\end{align*}
Finally, note that because $H\in \mathcal{H}_{\mathrm{ni}}(\tilde{X}^\ell, \tilde{Y}^\ell, J)$, $\overline{\nabla}^{\mathcal{H}}$ is equivariant \wrt the free $(\C^\ast)^{I^Y}$-actions on $\mathcal{B}(H^Y + \mathcal{H}^0_{\mathrm{ni}}(\tilde{X}^\ell, \tilde{Y}^\ell, J))$ and $\mathcal{E}(H^Y + \mathcal{H}^0_{\mathrm{ni}}(\tilde{X}^\ell, \tilde{Y}^\ell, J))$ induced by the one on $\dot{P}$.
Furthermore, the projections to $B$ and $H^Y + \mathcal{H}^0_{\mathrm{ni}}(\tilde{X}^\ell, \tilde{Y}^\ell, J)$ are invariant under this action.
The proof (not the statement) of Proposition 6.4 in \cite{MR1954264} shows that $\overline{\nabla}^\mathcal{H}$ is transverse to the zero section.
One should note here that for any $H \in \mathcal{H}^0_{\mathrm{ni}}(\tilde{X}^\ell, \tilde{Y}^\ell, J)$, $\d H|_{V\tilde{X}^\ell}$ vanishes along $\tilde{Y}^\ell$, because the condition that $H$ lies in $\mathcal{H}(\tilde{X}^\ell, \tilde{Y}^\ell)$ implies that $\d H$ vanishes on $(V\tilde{Y}^\ell)^{\perp_\omega}$ and the condition that $H$ vanishes along $\tilde{Y}^\ell$ implies that $\d H$ vanishes on $V\tilde{Y}^\ell$.
From this it follows that in Formula \ref{Equation_CR_on_normal_bundle} for $\overline{\nabla}^{H}$, the term involving $\nabla^{0,1}$ is independent of $H \in H^Y + \mathcal{H}^0_{\mathrm{ni}}(\tilde{X}^\ell, \tilde{Y}^\ell, J)$ since it only depends on the restriction of $\d H$ to $V\tilde{X}^\ell|_{\tilde{Y}^\ell}$.
For transversality, the crucial term is the second one, involving $K_{\hat{J}^{H}}$, \ie the symmetric part of the morphism $\frac{1}{2}\hat{J}^{H}(\hat{\nabla}^{H} \hat{J}^{H})$. Together with the vanishing of certain components of its antisymmetric part, which is given by the Nijenhuis tensor, to satisfy normal integrability, this gives a number of conditions on the Hessian of $H$ along $\tilde{Y}^\ell$.
By the usual line of argument using Lemma A.3.6 in \cite{MR2045629}, the universal moduli space $(\overline{\nabla}^\mathcal{H})\inv(0)$ hence is a smooth Banach manifold and the projection onto $H^Y + \mathcal{H}_{\mathrm{ni}}^0(\tilde{X}^\ell, \tilde{Y}^\ell, J)$ is a Fredholm map of index $\sum_{i \in I^Y}2(1-g_i) + \dim_\R(B)$.
So by the Sard-Smale theorem, for generic $H \in H^Y + \mathcal{H}_{\mathrm{ni}}^0(\tilde{X}^\ell, \tilde{Y}^\ell, J)$, $(\overline{\nabla}^{H})\inv(0)$ is a smooth manifold of dimension $\sum_{i \in I^Y}2(1-g_i) + \dim_\R(B)$.
Also, it comes with a free $(\C^\ast)^{I^Y}$-action and projection/forgetful map (smooth for the generic $H$ above) to $\mathcal{M}(\tilde{Y}^\ell|_{\Sigma^{\ell,Y}_{U^\ell}}, 0, J, H^Y)$ invariant under this action. \\
To arrive at the dimension formula given in the lemma one now finally observes that $\chi^Y = \sum_{i \in I^Y}2(1-g_i) - 2d'$.
\end{proof}

\subsection{Finishing the proof of Theorem \ref{Theorem_Main_Theorem_1}}\label{Subsection_Putting_it_all_together}

Here, the remaining part of the proof of Theorem \ref{Theorem_Main_Theorem_1}, Step \ref{Step_4} in Subsection \ref{Subsubsection_Proof_Main_Theorem_1}, is finally proved.

Let again $\mathcal{D}$ signify a set of data as in the previous subsections. More specifically, the following:
\begin{itemize}
  \item An open subset $U^\ell$ of a stratum of the stratification by signature on $M^\ell$ and a choice of desingularisation $\hat{S}^\ell$ of $\Sigma^\ell_{U^\ell}$ as well as a decomposition $I = I^X \amalg I^Y$ of the index set for the connected components of $\hat{S}^\ell$, together with all the notation this implies, as in Subsection \ref{Subsection_Description_of_the_closure}.
  \item Numbers $p^i_r \in \Z$, $r = 1, \dots, k_i$, $i \in I^Y$ as described in the paragraph before Definition \ref{Definition_M_D_Y}, in particular numbers $t_r \in \N_0$, $r = 1, \dots, d$, as in Lemma \ref{Lemma_SFT_Compactness}.
\end{itemize}

\begin{defn}
Let $J \in \mathcal{J}_{\omega,\mathrm{ni}}(X,Y,E)$, $H^Y \in \mathcal{H}(\tilde{Y})$, $H^0 \in \mathcal{H}^0_{\mathrm{ni}}(\tilde{X}^\ell, \tilde{Y}^\ell, J)$ and $H^{00} \in \mathcal{H}^{00}(\tilde{X}^\ell, \tilde{Y}^\ell)$. Set $H \definedas H^Y + H^0 + H^{00}$. \\
Given $\mathcal{D}$ as above, for $I^Y = \emptyset$ assume that $U^\ell = M^{\ell,i}$ is one of the strata in the stratification by signature on $M^\ell$, other than the top-stratum and define
\begin{align*}
\mathcal{M}^{\mathcal{D}}(\tilde{X}^\ell, \tilde{Y}^\ell, A, J, H) \definedas \{  u
\in \mathcal{M}_b(\tilde{X}^\ell, A, J, H) \;|\; & b \in M^{\ell,i} \\
& u(T^{\ell}_j(b)) \in \tilde{Y}^\ell,\; j=1,\dots, \ell, \\
& \im(u|_{\Sigma_{b,s}}) \cap \tilde{X}^\ell \setminus \tilde{Y}^\ell \neq \emptyset \text{ for} \\
& \text{every component } \Sigma^\ell_{b,s} \text{ of } \Sigma^\ell_b\}\text{.}
\end{align*}
For $I^Y \neq \emptyset$ define
\begin{align*}
\mathcal{M}^{\mathcal{D}}(\tilde{X}^\ell, \tilde{Y}^\ell, A, J, H) \definedas \{ (u^X, \xi) \;|\; & u^X \in \mathcal{M}(\tilde{X}^\ell|_{\Sigma^{\ell,X}_{U^\ell}}, A, J, H) \\
& \xi = (\xi_i)_{i\in I^Y} \in \mathcal{M}^{\mathcal{D}}_Y(\tilde{X}^\ell, \tilde{Y}^\ell,J,H) \\
& \pi^{\mathcal{M}}_M(u^X) = \pi^{\mathcal{M}}_M\circ \pi^{\mathcal{M}^{\mathcal{D}}_Y}_{\mathcal{M}(\tilde{Y}^\ell)}(\xi) \defines b \in U^\ell \\
& u^X(T^{\ell}_j(b)) \in \tilde{Y}^\ell \;\forall\, j \in K^X \\
& u(\Sigma^{\ell,X}_{i,b}) \cap \tilde{X}^\ell \setminus \tilde{Y}^\ell \neq \emptyset \;\forall\, i\in I^X \\
& \ev^{N^{\ell,XY,X}_r}(u^X) = \ev^{N^{\ell,XY,Y}_r}(\xi) \;\forall\, r=1, \dots, d \\
& \iota(u^X, \tilde{Y}^\ell|_{\Sigma_b}; N^{\ell,XY,X}_r(b)) = t_r \;\forall\, r = 1, \dots, d\}\text{.}
\end{align*}
\end{defn}
In the above, note that for all $H^{00} \in \mathcal{H}^{00}(\tilde{X}^\ell, \tilde{Y}^\ell)$,
\[
\mathcal{M}^{\mathcal{D}}_Y(\tilde{X}^\ell, \tilde{Y}^\ell,J,H^Y + H^0 + H^{00}) = \mathcal{M}^{\mathcal{D}}_Y(\tilde{X}^\ell, \tilde{Y}^\ell,J,H^Y + H^0)\text{.}
\]
Also note that for any $\mathcal{D}$ as above there is a canonical map
\[
\mathcal{M}^{\mathcal{D}}(\tilde{X}^\ell, \tilde{Y}^\ell, A, J, H) \to \mathcal{M}(\tilde{X}^\ell, A, J, H)\text{.}
\]
The last part of the proof of Theorem \ref{Theorem_Main_Theorem_1}, Step \ref{Step_4} in Section \ref{Section_Definition_and_Outline}, then follows from the following:
\begin{theorem}\label{Theorem_Step_4}
For every $J\in \mathcal{J}_{\omega, \mathrm{ni}}(X, Y, E)$, $H^Y \in \mathcal{H}_{\mathrm{reg}}(\tilde{Y}, J)$ and $H^{0}_{\mathrm{reg}}(\tilde{X}^\ell, \tilde{Y}^\ell, J, H^Y)$ there exists a generic subset $\mathcal{H}^{00}_{\mathrm{reg}}(\tilde{X}^\ell, \tilde{Y}^\ell, J, H^Y + H^0) \subseteq \mathcal{H}^{00}(\tilde{X}^\ell, \tilde{Y}^\ell)$ \st for every $H^{00} \in \mathcal{H}^{00}_{\mathrm{reg}}(\tilde{X}^\ell, \tilde{Y}^\ell, J, H^Y + H^0)$, setting $H \definedas H^Y + H^0 + H^{00}$, the following hold:
\begin{enumerate}
  \item $\partial \overset{\circ}{\mathcal{M}}(\tilde{X}^\ell, \tilde{Y}^\ell, A, J, H)$ is covered by the images of the $\mathcal{M}^{\mathcal{D}}(\tilde{X}^\ell, \tilde{Y}^\ell, A, J, H)$ in $\mathcal{M}(\tilde{X}^\ell, A, J, H)$, where $U^\ell$ is an open subset of a stratum in the stratification by signature on $M^\ell$ other than the top-stratum.
  \item For every $\mathcal{D}$ as above, $\mathcal{M}^{\mathcal{D}}(\tilde{X}^\ell, \tilde{Y}^\ell, A, J, H)$ is a manifold of real dimension at most
\[
\dim_\C(X)\chi + 2c_1(A) + \dim_\R(M) - 2\text{,}
\]
\ie real dimension at least $2$ less than the expected dimension of $\overset{\circ}{\mathcal{M}}(\tilde{X}^\ell, \tilde{Y}^\ell, A, J, H)$.
\end{enumerate}
\end{theorem}
\begin{proof}
\begin{enumerate}
  \item This follows immediately from Lemmas \ref{Lemma_SFT_Compactness_II} and \ref{Lemma_Reduction_to_vanishing_A}, which implies that no smooth curve ends up in $\tilde{Y}^\ell$.
  \item Let $H_0 \definedas H^Y + H^0$ and $V \definedas \tilde{X}^\ell \setminus \tilde{Y}^\ell$. As usual, define
\[
\mathcal{M}^{\mathcal{D}}(\tilde{X}^\ell, \tilde{Y}^\ell, A, J, H_0 + \mathcal{H}^{00}(\tilde{X}^\ell, \tilde{Y}^\ell)) \definedas \coprod_{H^{00} \in \mathcal{H}^{00}(\tilde{X}^\ell, \tilde{Y}^\ell)} \mathcal{M}^{\mathcal{D}}(\tilde{X}^\ell, \tilde{Y}^\ell, A, J, H_0 + H^{00})\text{.}
\]
First, consider the case $I^Y = \emptyset$.
By Lemma \ref{Lemma_Main_transversality_result} and the lemmas up to and including Lemma \ref{Lemma_The_universal_moduli_space}, $\mathcal{M}^V(\tilde{X}^\ell|_{M^{\ell,i}}, A, J, H_0 + \mathcal{H}^{00}(\tilde{X}^\ell, \tilde{Y}^\ell))$ is a Banach manifold and the map
\[
\ev^{T^\ell} : \mathcal{M}^V(\tilde{X}^\ell|_{M^{\ell,i}}, A, J, H_0 + \mathcal{H}^{00}(\tilde{X}^\ell, \tilde{Y}^\ell) \to (T^\ell_1)^\ast \tilde{X}^\ell \oplus \cdots \oplus (T^\ell_\ell)^\ast \tilde{X}^\ell
\]
is a submersion. So
\begin{multline*}
\mathcal{M}^{\mathcal{D}}(\tilde{X}^\ell, \tilde{Y}^\ell, A, J, H_0 + \mathcal{H}^{00}(\tilde{X}^\ell, \tilde{Y}^\ell)) = \\ \left(\ev^{T^\ell}\right)\inv\left((T^\ell_1)^\ast \tilde{Y}^\ell \oplus \cdots \oplus (T^\ell_\ell)^\ast \tilde{Y}^\ell\right)
\end{multline*}
is a Banach manifold and
\[
\pi^{\mathcal{M}}_{\mathcal{H}} : \mathcal{M}^{\mathcal{D}}(\tilde{X}^\ell, \tilde{Y}^\ell, A, J, H_0 + \mathcal{H}^{00}(\tilde{X}^\ell, \tilde{Y}^\ell)) \to H_0 + \mathcal{H}^{00}(\tilde{X}^\ell, \tilde{Y}^\ell)
\]
a Fredholm map of index
\begin{multline*}
\dim_\C(X)\chi + 2c_1(A) + \dim_\R(M^{\ell,i}) - 2\ell \; = \\
= \; \dim_\C(X)\chi + 2c_1(A) + \dim_\R(M) - \underbrace{\codim_\R^{M^\ell}(M^{\ell,i})}_{\geq 2}\text{.}
\end{multline*}
Now apply the Sard-Smale theorem. \\
In case $I^Y \neq \emptyset$, for $r = 1, \dots, d$, let $N^{\ell,XY}_r = \hat{\iota}^\ell\circ N^{\ell,XY,X}_r = \hat{\iota}^\ell\circ N^{\ell,XY,Y}_r : U^\ell \to \Sigma^\ell_{U^\ell}$ be the nodal points at which $\Sigma^{\ell,X}_{U^\ell}$ and $\Sigma^{\ell,Y}_{U^\ell}$ meet.
Also, let $\ev^{T^\ell}_{K^X}$ be the evaluation at the marked point $T^\ell_j$ for $j \in K^X$.
Then
\begin{multline}
\ev^{T^\ell}_{K^X} \times \ev^{N^{\ell, XY, X}} \times \ev^{N^{\ell, XY, Y}} :  \\
\mathcal{M}^V(\tilde{X}^\ell|_{\Sigma^{\ell,X}_{U^\ell}}, A, J, H_0 + \mathcal{H}^{00}(\tilde{X}^\ell, \tilde{Y}^\ell)) \times \mathcal{M}^{\mathcal{D}}_Y(\tilde{X}^\ell, \tilde{Y}^\ell, J, H_0) \to \\
\bigoplus_{j\in K^X} \left(T^\ell_j\right)^\ast \tilde{X}^\ell \oplus \bigoplus_{r=1}^d \left(\left(N^{\ell, XY}_r\right)^\ast \tilde{X}^\ell \oplus \left(N^{\ell, XY}_r\right)^\ast \tilde{Y}^\ell\right)
\end{multline}
is transverse to $\bigoplus_{j\in K^X} \left(T^\ell_j\right)^\ast \tilde{Y}^\ell \oplus \bigoplus_{r=1}^d \Delta_r$, where $\Delta$ denotes the diagonal in $\left(N^{\ell, XY}_r\right)^\ast \tilde{Y}^\ell \oplus \left(N^{\ell, XY}_r\right)^\ast \tilde{Y}^\ell \subseteq \left(N^{\ell, XY}_r\right)^\ast \tilde{X}^\ell \oplus \left(N^{\ell, XY}_r\right)^\ast \tilde{Y}^\ell$.
So 
\begin{multline}
\tilde{\mathcal{M}}^{\mathcal{D}}(\tilde{X}^\ell, \tilde{Y}^\ell, A, J, H_0 + \mathcal{H}^{00}(\tilde{X}^\ell, \tilde{Y}^\ell)) \definedas \\
\left(\ev^{T^\ell}_{K^X} \times \ev^{N^{\ell, XY, X}} \times \ev^{N^{\ell, XY, Y}}\right)\inv \left(\bigoplus_{j\in K^X} \left(T^\ell_j\right)^\ast \tilde{Y}^\ell \oplus \bigoplus_{r=1}^d \Delta_r\right)
\end{multline}
is a split submanifold of codimension $\dim_\R(U^\ell) + 2|K^X| + 2d\dim_\C(X)$ and using Lemma \ref{Lemma_H_0_reg},
\[
\pi^\mathcal{M}_\mathcal{H} : \tilde{\mathcal{M}}^{\mathcal{D}}(\tilde{X}^\ell, \tilde{Y}^\ell, A, J, H_0 + \mathcal{H}^{00}(\tilde{X}^\ell, \tilde{Y}^\ell)) \to H_0 + \mathcal{H}^{00}(\tilde{X}^\ell, \tilde{Y}^\ell)
\]
is a Fredholm map of index
\begin{align*}
\ind\pi^\mathcal{M}_\mathcal{H} &= \dim_\C(X)\chi^X + 2c_1(A) + \dim_\R(U^\ell) \;+ \\
&\quad\; +\; \dim_\C(X)\chi^Y + \dim_\R(U^\ell) + 2d' - 2|I^Y| \;- \\
&\quad\; -\; (\dim_\R(U^\ell) + 2|K^X| + 2d\dim_\C(X)) \\
&= \dim_\C(X)\chi + 2c_1(A) + \dim_\R(U^\ell) + 2d' - 2|K^X| - 2|I^Y|\text{.}
\end{align*}
By the same reasoning leading to Lemma \ref{Lemma_Tangency_condition},
\begin{multline*}
\mathcal{M}^{\mathcal{D}}(\tilde{X}^\ell, \tilde{Y}^\ell, A, J, H_0 + \mathcal{H}^{00}(\tilde{X}^\ell, \tilde{Y}^\ell)) = \\
\{
(u^X,\xi)\in \tilde{\mathcal{M}}^{\mathcal{D}}(\tilde{X}^\ell, \tilde{Y}^\ell, A, J, H_0 + \mathcal{H}^{00}(\tilde{X}^\ell, \tilde{Y}^\ell)) \;|\; \\
b \definedas \pi^{\mathcal{M}}_M(u^X)\in U^\ell, \iota(u^X, \tilde{Y}^\ell|_{\Sigma_b}; N^{\ell,XY,X}_r(b)) = t_r \;\forall\, r = 1, \dots, d
\}
\end{multline*}
is a smooth submanifold of real codimension $2\sum_{r=1}^{d} t_r = 2(|K^Y| - d)$ and the projection
\[
\pi^{\mathcal{M}}_{\mathcal{H}} : \mathcal{M}^{\mathcal{D}}(\tilde{X}^\ell, \tilde{Y}^\ell, A, J, H_0 + \mathcal{H}^{00}(\tilde{X}^\ell, \tilde{Y}^\ell)) \to H_0 + \mathcal{H}^{00}(\tilde{X}^\ell, \tilde{Y}^\ell)
\]
is Fredholm of index ($|K^X| + |K^Y| = \ell$)
\[
\ind\pi^{\mathcal{M}}_{\mathcal{H}} = \dim_\C(X)\chi + 2c_1(A) + \dim_\R(U^\ell) + 2d' + 2d - 2\ell - 2|I^Y|\text{.}
\]
Because $\dim_\R(U^\ell)$ is $\dim_\R(M^\ell) = \dim_\R(M) + 2\ell$ minus $2$ times the total number of nodes, which is at least $2(d' + d)$,
\[
\ind\pi^{\mathcal{M}}_{\mathcal{H}} \leq \dim_\C(X)\chi + 2c_1(A) + \dim_\R(M) - 2|I^Y|\text{.} 
\]
So by the Sard-Smale theorem, there exists a generic subset of $\mathcal{H}^{00}(\tilde{X}^\ell, \tilde{Y}^\ell)$ \st for every $H^{00}$ in this subset, $\mathcal{M}^{\mathcal{D}}(\tilde{X}^\ell, \tilde{Y}^\ell, A, J, H_0 + H^{00})$ is a smooth manifold of dimension at most $\dim_\C(X)\chi + 2c_1(A) + \dim_\R(M) - 2|I^Y|$. \\
Taking the intersection of all the generic subsets from above for the at most countably many choices of $\mathcal{D}$, one gets a generic subset
\[
\mathcal{H}^{00}_{\mathrm{reg}}(\tilde{X}^\ell, \tilde{Y}^\ell, H^Y + H^0) \subseteq \mathcal{H}^{00}(\tilde{X}^\ell, \tilde{Y}^\ell)\text{.}
\]
\end{enumerate}
\end{proof}

\subsection{Finishing the proof of Theorem \ref{Theorem_Main_Theorem_2}}\label{Subsection_Proof_Main_Theorem_2}

In this subsection, the precise statements of the results used in Subsection \ref{Subsubsection_Proof_Main_Theorem_2} are stated and the necessary changes to the proofs above are outlined.
The first one is the following Lemma used in Steps \ref{Theorem_Main_Theorem_2_Step1}) and \ref{Theorem_Main_Theorem_2_Step2}) in Subsection \ref{Subsubsection_Proof_Main_Theorem_2}, which describes in which sense the choice of Donaldson hypersurface and adapted $\omega$-compatible almost complex structure is unique.
\begin{lemma}\label{Lemma_Existence_Transverse_Hypersurface}
Let $(Y_i,J_i)$ be Donaldson pairs of degrees $D_i \geq D_i^\ast$, $i = 0,1$, where $D_i^\ast = D^\ast(X, \omega, J_i)$ is as in Lemma \ref{Lemma_Nice_J_exist}.
Then there exist
\begin{itemize}
  \item an isotopy $\phi_\cdot : [0,1]\times X \to X$, $\phi_0 = \id$, through symplectomorphisms,
  \item an integer $\overline{D} \in \N$,
  \item a hypersurface $\overline{Y} \subseteq X$ of degree $\overline{D}$,
  \item a path $(\overline{J}_t)_{t\in [0,1]} \subseteq \mathcal{J}_{\omega}(X)$ \st $\overline{Y}$ is approximately $\overline{J}_t$-holomorphic for all $t\in [0,1]$,
  \item a constant $\varepsilon > 0$,
\end{itemize}
\st the following hold:
\begin{enumerate}[a)]
  \item\label{Donaldson_quadruples_a} $\mathcal{J}_{\omega, \mathrm{ni}}(X, \overline{Y}, \overline{J}_t, E) \neq \emptyset$ for all $t\in [0,1]$.
  \item\label{Donaldson_quadruples_b} There exists $J \in \mathcal{J}_{\omega}(X, Y_0, J_0, E) \cap \mathcal{J}_{\omega}(X, \overline{Y}, \overline{J}_0, E)$ and $Y_0$ and $\overline{Y}$ intersect $\varepsilon$-transversely. Setting $Z_0 \definedas Y_0 \cap \overline{Y}$, $J$ can be chosen \st $J|_{\overline{Y}} \in \mathcal{J}_{\omega, \mathrm{ni}}(\overline{Y}, Z_0, E)$ and $J|_{Y_0} \in \mathcal{J}_{\omega, \mathrm{ni}}(Y_0, Z_0, E)$.
  \item\label{Donaldson_quadruples_c} There exists $J \in \mathcal{J}_{\omega}(X, \phi_1(Y_1), (\phi_1)_\ast J_1, E) \cap \mathcal{J}_{\omega}(X, \overline{Y}, \overline{J}_1, E)$ and $\phi_1(Y_1)$ and $\overline{Y}$ intersect $\varepsilon$-transversely. Setting $Z_1 \definedas \phi_1(Y_1) \cap \overline{Y}$, $J$ can be chosen \st $J|_{\overline{Y}} \in \mathcal{J}_{\omega, \mathrm{ni}}(\overline{Y}, Z_1, E)$ and $J|_{\phi_1(Y_1)} \in \mathcal{J}_{\omega, \mathrm{ni}}(\phi_1(Y_1), Z_0, E)$.
\end{enumerate}
\end{lemma}
\begin{proof}
In the above, $\mathcal{J}_{\omega, \mathrm{ni}}(X, Y_0, J_0, E)$, etc., are as in Lemma \ref{Lemma_Nice_J_exist}. \\
It suffices to construct $\overline{Y}$ and show \ref{Donaldson_quadruples_b}), for \ref{Donaldson_quadruples_a}) follows the same way as in Lemma \ref{Lemma_Nice_J_exist} and \ref{Donaldson_quadruples_c}) follows by doing the construction below in the proof of \ref{Donaldson_quadruples_b}) and by then applying the uniqueness property in Theorem 8.1 in \cite{MR2399678}.

To explore the Nijenhuis condition, first consider how to achieve the Nijenhuis condition in a linear case: \\
Let $(Y,\omega_Y)$ be a symplectic manifold with $\omega_Y$-compatible almost complex structure $J_Y$ and consequently Riemannian metric $g_Y$.
Given a complex vector bundle $E \to Y$ of fibre dimension $d$ over $Y$ with fibrewise symplectic form $\omega_0 \in \Lambda^2 E^\ast$ and fibrewise $\omega_0$-compatible complex structure $J_0 \in \End(E)$, hence fibre-metric $g_0$.
Then every complex linear connection $TE = HE \oplus VE$ on $E$ defines a nondegenerate $2$-form $\omega'$ and an $\omega'$-compatible almost complex structure $J'$ given by the lifts of $\omega_Y$ and $J_Y$ on the horizontal tangent spaces and by the linear extension of $\omega_0$ and $J_0$ on the vertical tangent spaces.
$\omega'$ is not a closed form, for one can compute that $\d\omega'(X_0, X_1, X_2) = -\omega'(\Omega(X_0,X_1), X_2)$, if $X_0, X_1$ are horizontal and $X_2$ is vertical.
$\Omega$ here denotes the curvature of the connection.
Now one follows the method from the second proof of Theorem 6.21 ``(ii) $\Rightarrow$ (i)'' in \cite{MR1698616}, p.~224, to modify $\omega'$ to a closed form $\tilde{\omega}$ that coincides with $\omega'$ along the zero section:
For $y \in Y$ and $X_1, X_2 \in T_yY$, let $H_{X_1,X_2} : E_y \to \R$ be the unique Hamiltonian function \st $H_{X_1,X_2}(0) = 0$ and $\d H_{X_1, X_2} = \omega'(\Omega(X_1, X_2), \cdot)$.
Now set $\tilde{\omega}_y(X_1, X_2) \definedas \omega'_y(X_1, X_2) + H_{\pi_\ast X_1, \pi_\ast X_2}(y)$.
This is closed by op.~cit.~and in a neighbourhood of the zero section also nondegenerate, hence a symplectic form.
$J'$ is no longer $\tilde{\omega}$-compatible on the horizontal tangent spaces, but applying the usual procedure provided by Proposition 2.50 in \cite{MR1698616} along the horizontal tangent spaces, one finds a $\tilde{J}$ that is $\tilde{\omega}$-compatible. \\
With these choices, in a neighbourhood of the zero section, $\tilde{\omega}$ defines a symplectic form that coincides with $\omega'$ along $Y$ and $\omega_0$ along the fibres and an $\tilde{\omega}$-compatible almost complex structure that coincides with $J'$ along $Y$ and $J_0$ on the fibres.
Also note that the vertical and horizontal subspaces are still symplectic as well as complex subspaces and orthogonal \wrt the metric defined by $\tilde{\omega}$ and $\tilde{J}$.

To achieve the Nijenhuis condition for $\tilde{J}$ \wrt $Y \subseteq E$, one applies the procedure from \cite{MR1954264}, Theorem A.2, to $\tilde{J}$: Applying the construction given in op.~cit.~to the present simplified situation, define a section $K \in \End(TE)$ by $K_e \definedas -\frac{1}{2}(L_e + L_e^t)$, where on the right hand side one considers $e$ as lying in $V_0E$.
$L_e(v)$ is given as in op.~cit.~by $-\tilde{J}\circ \pr^{TE}_{VE}\circ N_{\tilde{J}}(e, v)$ for $v \in TY$ and extended by zero to $TE$.
Using that $K$ is self-adjoint and satisfies $K\tilde{J} = -\tilde{J}K$, $\overline{J} \definedas e^{\tilde{J}K}\tilde{J}$ defines an almost complex structure that satisfies the Nijenhuis condition \wrt $Y$.
Note that now outside of $Y$ the fibres of $E$ need no longer be complex subspaces.

In the nonlinear case, apply the above to $E \definedas TY^{\perp_\omega}$, the $\omega$-symplectic complement of $TY$ in $TX|_Y$ and the restrictions $\omega_0$ and $J_0$ of $\omega$ and $J$ to $TY^{\perp_\omega}$.
By the symplectic neighbourhood theorem there are neighbourhoods of $Y$ in $E$ and $Y$ in $X$ that are symplectomorphic under a symplectomorphism that is the identity on $Y$.
Pushing forward $J'$ to the neighbourhood of $Y$ in $X$ and extending via $J$ outside gives a perturbation of $J$ that satisfies the Nijenhuis condition.

To construct $\overline{Y}$, one uses the above to modify the proof of the (Stabilization) condition in Theorem 8.1, p.~43, in \cite{MR2399678}.
In the notation of that proof, first apply the construction to the line bundles $L^k|_Y$ to get sections $s_k^Y$ of $L^k|_Y \to Y$ that are $C$-bounded and $\eta$-transverse.
Let $Z \definedas (s_k^Y)\inv(0) \subseteq Y$ for $k$ large enough and modify $J|_Y$ in the way described above (with the modification that now $J|_{Z}$ and $J_0$ are chosen as well) to a $J_Y$ \st $Z$ is a $J_Y$-holomorphic submanifold of $Y$ that satisfies the Nijenhuis condition \wrt $Y$.
With this new complex structure $J_Y$ on $Y$ apply the above construction again, but without the last step (to achieve the Nijenhuis condition) to extend $J_Y$ to $\tilde{J}$ on $X$ \st a neighbourhood of $Y$ in $X$ is isomorphic to a neighbourhood $N$ of the zero section in $TY^{\perp_\omega} \to Y$.
Now proceed as in the proof of Theorem 8.1 in op.~cit.: If $\pi : N \to Y$ is the projection, then $\pi^\ast (L|_Y)$ carries the connection given by $A|_Y - \frac{i}{2}\iota_{\nabla r^2}\omega_0$ ($A$ is the connection of $L$).
Here $r$ is the radial distance function from the zero section in the fibres of $N$, $\nabla r^2$ is the gradient and $\omega_0$ is the symplectic form in the fibres.
With this identification, define a section $s_k^N$ of $L^k|_N$ by transfering $e^{-kr^2/4}\pi^\ast s_k^Y$ to $X$ as in op.~cit.
The zero set of $s_k^N$ is then identified with $N|_{Z} \subseteq TY^{\perp_\omega}|_{Z}\oplus \{0\}$ which is holomorphic.
Now extend the section $s_k^N$ to a section $s_k$ of $L^k$, only modifying outside a neighbourhood of $Y$, in the way described in op.~cit, so that the zero section $\overline{Y}$ of $s_k$ is symplectic.
With this, a neighbourhood $U$ of $Z = \overline{Y} \cap Y$ in $\overline{Y}$ looks like a neighbourhood of $Z$ in $TY^{\perp_\omega}|_{Z}$.
Applying the last step in the above construction (to achieve the Nijenhuis condition) produces a section $K$ of $\End(T\overline{Y})$ on $U$ that vanishes on $Z$.
Extend $K$ to a self-adjoint section of $\End(TX)$ with $K\tilde{J} = -\tilde{J}K$ on a small neighbourhood of $Z$ in $X$ that vanishes along $Y$ by first extending to a section of $\End(TX)|_U$ by setting $K_e|_{T_e\overline{Y}^{\perp_\omega}} = 0$.
Then extend to a neighbourhood $V$ of $U$ in $X$ by choosing $V$ \st $U$ is a strong deformation retract of $V$ via a retraction that sends $Y \cap V$ to $Z$.
By multiplying with a cutoff function extend $K$ to a section of $\End(TX)$.
Now follow the procedure above to use $K$ to make $\tilde{J}|_{\overline{Y}}$ satisfy the Nijenhuis condition \wrt $Z$ and do a further perturbation away from $Y$ to achieve compatibility with $\overline{Y}$ everywhere.
\end{proof}

What remains to finish the proof of Theorem \ref{Theorem_Main_Theorem_2} is the following generalisation of Theorem \ref{Theorem_Main_Theorem_1} used in Step \ref{Theorem_Main_Theorem_2_Step3}) of Subsection \ref{Subsubsection_Proof_Main_Theorem_2}.
But first, for convenience the relevant definition from Subsection \ref{Subsubsection_Proof_Main_Theorem_2} is repeated: \\
Let $(Y, J)$ and $(\overline{Y}, \overline{J})$ be $\varepsilon$-transversely intersecting Donaldson pairs of degrees $D$ and $\overline{D}$, respectively, and let $A \in H_2(X;\Z)$.
Define $Z \definedas Y\cap \overline{Y}$ as well as $\ell \definedas D\omega(A)$, $\overline{\ell} \definedas \overline{D}\omega(A)$ and $\hat{\ell} \definedas \ell + \overline{\ell}$.
Analogously to before, for $k \geq 0$, put $\tilde{X}^k \definedas \Sigma^k\times X$, $\tilde{Y}^k \definedas \Sigma^k \times Y$, $\tilde{\overline{Y}}{}^k \definedas \Sigma^k \times \overline{Y}$ and $\tilde{Z}^k \definedas \Sigma^k \times Z$.
Also, just for brevity, define $W \definedas Y \cup \overline{Y} \subseteq X$ and $\tilde{W}^k \definedas \Sigma^k \times W$.
Let $J' \in \mathcal{J}_{\omega}(X, Y, J, E) \cap \mathcal{J}_{\omega}(X, \overline{Y}, \overline{J}, E)$ satisfy the conditions in Lemma \ref{Lemma_Existence_Transverse_Hypersurface}, \ref{Donaldson_quadruples_b}).
Given an arbitrary neighbourhood $U \subseteq X$ of $Z$, one can also assume that $J' \in \mathcal{J}_{\omega, \mathrm{ni}}(X, Y\setminus \overline{U}) \cap \mathcal{J}_{\omega, \mathrm{ni}}(X, \overline{Y} \setminus \overline{U})$ by a small perturbation of  $J'$ outside of $U$ that leaves $J'|_{W}$ fixed.
Denoting $\tilde{U}^k \definedas \Sigma^k \times U$ and $\tilde{\overline{U}}{}^k \definedas \Sigma^k \times \overline{U}$, define
\begin{align*}
\mathcal{H}^{00}(\tilde{W}^k, \tilde{Z}^k) &\definedas \mathcal{H}^{00}(\tilde{Y}^k, \tilde{Z}^k) \times \mathcal{H}^{00}(\tilde{\overline{Y}}{}^k, \tilde{Z}^k) \\
\mathcal{H}^0_{\mathrm{ni}}(\tilde{W}^k, \tilde{Z}^k, J') &\definedas \mathcal{H}^0_{\mathrm{ni}}(\tilde{Y}^k, \tilde{Z}^k, J') \times \mathcal{H}^0_{\mathrm{ni}}(\tilde{\overline{Y}}{}^k, \tilde{Z}^k, J') \\
\mathcal{H}^{0}_{\mathrm{ni}}(\tilde{X}^k, \tilde{W}^k, \tilde{U}^k, J') &\definedas \{ H \in  \mathcal{H}^0_{\mathrm{ni}}(\tilde{X}^k, \tilde{Y}^k, J') \cap \mathcal{H}^0_{\mathrm{ni}}(\tilde{X}^k, \tilde{\overline{Y}}{}^k, J') \;|\; \supp(H) \subseteq \tilde{X}^k \setminus \tilde{U}^k\} \\
\mathcal{H}^{00}(\tilde{X}^k, \tilde{W}^k) &\definedas \mathcal{H}^{00}(\tilde{X}^k, \tilde{Y}^k) \cap \mathcal{H}^{00}(\tilde{X}^k, \tilde{\overline{Y}}{}^k) \\
\mathcal{H}(\tilde{X}^k, \tilde{W}^k) &\definedas \mathcal{H}(\tilde{X}^k, \tilde{Y}^k) \cap \mathcal{H}(\tilde{X}^k, \tilde{\overline{Y}}{}^k)\text{.}
\end{align*}
Note that by Lemma \ref{Lemma_Enough_ni_ham_perturbations} (applied twice) one can consider $\mathcal{H}^{00}(\tilde{W}^k, \tilde{Z}^k)$ as a subset of $\mathcal{H}^{0}_{\mathrm{ni}}(\tilde{W}^k, \tilde{Z}^k, J')$.
By extending via the construction in Lemma \ref{Lemma_Enough_ni_ham_perturbations} outside of $\tilde{U}^k$ and via the implicit function theorem, using that $Y$ and $\overline{Y}$ intersect transversely, in $\tilde{U}^k$, every element $H \in \mathcal{H}^{0}_{\mathrm{ni}}(\tilde{W}^k, \tilde{Z}^k, J')$ in turn can be considered as an element of $\mathcal{H}(\tilde{X}^k,\tilde{Y}^k) \cap \mathcal{H}(\tilde{X}^k, \tilde{\overline{Y}}{}^k)$ with the property that $H \in \mathcal{H}_{\mathrm{ni}}(\tilde{X}^k, \tilde{Y}^k\setminus \tilde{\overline{U}}{}^k) \cap \mathcal{H}_{\mathrm{ni}}(\tilde{X}^k, \tilde{\overline{Y}}{}^k\setminus \tilde{\overline{U}}{}^k)$.
By pullback via $\hat{\pi}^{k'}_k$ one can consider all of the above as subsets of $\mathcal{H}(\tilde{X}^{k'}, \tilde{W}^{k'})$ for any $k' \geq k$, which will be done from now on.
For $H \in \mathcal{H}(\tilde{X}^{\hat{\ell}}, \tilde{W}^{\hat{\ell}})$, define
\begin{align*}
\overset{\circ}{\mathcal{M}}(\tilde{X}^{\hat{\ell}}, \tilde{Y}^{\hat{\ell}}, \tilde{\overline{Y}}{}^{\hat{\ell}}, A, J', H) \definedas \{ u \in \mathcal{M}(\tilde{X}^{\hat{\ell}}|_{\overset{\circ}{M}{}^{\hat{\ell}}}, A, J', H) \;|\; & \im(u\circ T^{\hat{\ell}}_j) \subseteq \tilde{Y}^{\hat{\ell}},\; j=1,\dots, \ell, \\
& \im(u\circ T^{\hat{\ell}}_j) \subseteq \tilde{\overline{Y}}{}^{\hat{\ell}},\; j=\ell + 1,\dots, \hat{\ell}, \\
& \im(u) \cap \tilde{X}^{\hat{\ell}} \setminus (\tilde{Y}^{\hat{\ell}} \cup \tilde{\overline{Y}}{}^{\hat{\ell}}) \neq \emptyset\}
\end{align*}
and let
\[
\mathrm{gw}^{\hat{\ell}}_\Sigma(X, Y, \overline{Y}, A, J', H) : \overset{\circ}{\mathcal{M}}(\tilde{X}^{\hat{\ell}}, \tilde{Y}^{\hat{\ell}}, \tilde{\overline{Y}}{}^{\hat{\ell}}, A, J', H) \to M \times X^n
\]
be given by evaluation at the first $n$ marked points, as before.
\begin{theorem}\label{Theorem_Main_Theorem_3_Extension_of_1}
In the notation from above, let $J' \in \mathcal{J}_{\omega}(X, Y, J, E) \cap \mathcal{J}_{\omega}(X, \overline{Y}, \overline{J}, E)$ satisfy the conditions in Lemma \ref{Lemma_Existence_Transverse_Hypersurface}, \ref{Donaldson_quadruples_b}).
Then for $D,\overline{D}$ large enough, there exists a generic subset $\mathcal{H}_{\mathrm{reg}}(\tilde{Z}, J') \subseteq \mathcal{H}(\tilde{Z})$ and for any $H^Z \in \mathcal{H}_{\mathrm{reg}}(\tilde{Z}, J')$ there exists a generic subset $\mathcal{H}^{00}_{\mathrm{reg}}(\tilde{W}, \tilde{Z}, J', H^Z) \subseteq \mathcal{H}^{00}(\tilde{W}, \tilde{Z})$ \st for any $H^{W,Z} \in \mathcal{H}^{00}_{\mathrm{reg}}(\tilde{W}, \tilde{Z}, J', H^Z)$, setting $H^W \definedas H^Z + H^{W,Z}$ and for $k = \ell, \overline{\ell}, \hat{\ell}$, the following hold: \\
There exists
\begin{itemize}
  \item a generic subset $\mathcal{H}^0_{\mathrm{reg}}(\tilde{W}^k, \tilde{Z}^k, J', H^W) \subseteq \mathcal{H}^0_{\mathrm{ni}}(\tilde{W}^k, \tilde{Z}^k, J')$ of a neighbourhood of $0$ (depending on $H^W$) and
  \item for every $H^{Z,0} \in \mathcal{H}^0_{\mathrm{reg}}(\tilde{W}^k, \tilde{Z}^k, J', H^W)$ a neighbourhood $U \subseteq X$ of $Z$ (depending on choices so far) and a small perturbation $J''$ of $J'$ outside of $W\cup U$ \st $J'' \in \mathcal{J}_{\omega, \mathrm{ni}}(X, Y\setminus \overline{U}) \cap \mathcal{J}_{\omega, \mathrm{ni}}(X, \overline{Y}\setminus \overline{U})$ as well as a generic subset $\mathcal{H}^{0}_{\mathrm{reg}}(\tilde{X}^k, \tilde{W}^k, \tilde{U}^k, J'', H^W + H^{Z,0}) \subseteq \mathcal{H}^{0}_{\mathrm{ni}}(\tilde{X}^k, \tilde{W}^k, \tilde{U}^k, J'')$ and
  \item for every $H^{Z,0} \in \mathcal{H}^0_{\mathrm{reg}}(\tilde{W}^k, \tilde{Z}^k, J', H^W) = \mathcal{H}^0_{\mathrm{reg}}(\tilde{W}^k, \tilde{Z}^k, J'', H^W)$ and $H^{W,0}\in \mathcal{H}^{0,U}_{\mathrm{reg}}(\tilde{X}^k, \tilde{W}^k, \tilde{Z}^k, J'', H^W + H^{Z,0})$ a generic subset $\mathcal{H}^{00}_{\mathrm{reg}}(\tilde{X}^k, \tilde{W}^k, J'', H^W + H^{Z,0} + H^{W,0}) \subseteq \mathcal{H}^{00}(\tilde{X}^k, \tilde{W}^k)$,
\end{itemize}
\st for every $H^{Z,0} \in \mathcal{H}^0_{\mathrm{reg}}(\tilde{W}^k, \tilde{Z}^k, J'', H^W)$, $H^{W,0}\in \mathcal{H}^{0,U}_{\mathrm{reg}}(\tilde{X}^k, \tilde{W}^k, \tilde{Z}^k, J'', H^W + H^{Z,0})$ and $H^{X,00} \in \mathcal{H}^{00}_{\mathrm{reg}}(\tilde{X}^k, \tilde{W}^k, J'', H^W + H^{Z,0} + H^{W,0})$, setting $H \definedas H^W + H^{Z,0} + H^{W,0} + H^{X,00}$, 
\[
u \in \mathcal{M}(\tilde{X}^{k}|_{\overset{\circ}{M}{}^{k}}, A, J'', H) \quad \Rightarrow\quad \im(u) \cap \tilde{X}^{k} \setminus (\tilde{Y}^{k} \cup \tilde{\overline{Y}}{}^{k}) \neq \emptyset
\]
and
\begin{enumerate}[i)]
  \item\label{Case_i} for $k = \ell$
\[
\mathrm{gw}_{\Sigma}^\ell(X, Y, A, J'', H)\text{,}
\]
  \item\label{Case_ii} for $k = \overline{\ell}$
\[
\mathrm{gw}_{\Sigma}^{\overline{\ell}}(X, \overline{Y}, A, J'', H)\text{,}
\]
  \item\label{Case_iii} for $k = \hat{\ell}$
\[
\mathrm{gw}_{\Sigma}^{\hat{\ell}}(X, Y, \overline{Y}, A, J'', H)\text{,}
\]
\end{enumerate}
define pseudocycles of dimension
\[
\dim_\C(X)\chi + 2c_1(A) + \dim_\R(M)
\]
in $M\times X^n$ with image in $\overset{\circ}{M}\times X^n$, independent of the choice of $H$ up to cobordism.
\end{theorem}
\begin{proof}\emph{(Rough sketch only)}\quad
The proof runs along the lines of the proof of Theorem \ref{Theorem_Main_Theorem_1} as outlined in Subsection \ref{Subsubsection_Proof_Main_Theorem_1}. \\
In cases \ref{Case_i}) and \ref{Case_ii}) above, the compactness result in Step \ref{Step_1}) in Subsection \ref{Subsubsection_Proof_Main_Theorem_1} applies directly and in \ref{Case_iii}), the compactness of $\cl\overset{\circ}{\mathcal{M}}(\tilde{X}^{\hat{\ell}}, \tilde{Y}^{\hat{\ell}}, \tilde{\overline{Y}}{}^{\hat{\ell}}, A, J, H)$ is a simple extension of Lemma \ref{Lemma_Main_compactness_result}. \\
While Step \ref{Step_2}) is not relevant in the present situation, Steps \ref{Step_3}) and \ref{Step_5}) in Subsection \ref{Subsubsection_Proof_Main_Theorem_1} carry over pretty much ad verbatim. \\
Again, the most difficult part of the proof is the analogue of Step \ref{Step_4}) in Subsection \ref{Subsubsection_Proof_Main_Theorem_1}.
I will describe the necessary changes to the results from Subsections \ref{Subsection_Description_of_the_closure}--\ref{Subsection_Putting_it_all_together} above, which constituted the proof of Step \ref{Step_4}) in Subsection \ref{Subsubsection_Proof_Main_Theorem_1}, only for Case \ref{Case_iii}) above, Cases \ref{Case_i}) and \ref{Case_ii}) are done very similarly. \\
The major difference to the proof of Theorem \ref{Theorem_Main_Theorem_1} is that in $\cl\overset{\circ}{\mathcal{M}}(\tilde{X}^{\hat{\ell}}, \tilde{Y}^{\hat{\ell}}, A, J, H)$ instead of a curve having two parts, one that is mapped into $\tilde{Y}^{\hat{\ell}}$ and one that intersects $\tilde{Y}^{\hat{\ell}}$ only in a finite number of points, because $H$ needs to be compatible with both $\tilde{Y}^{\hat{\ell}}$ and $\tilde{\overline{Y}}{}^{\hat{\ell}}$, one now has to consider parts of a curve that are mapped into $\tilde{Z}^{\hat{\ell}}$, $\tilde{Y}^{\hat{\ell}}$ and $\tilde{\overline{Y}}{}^{\hat{\ell}}$ while intersecting $\tilde{Z}^{\hat{\ell}}$ only in finitely many points and finally the part that intersects both $\tilde{Y}^{\hat{\ell}}$ and $\tilde{\overline{Y}}{}^{\hat{\ell}}$ only in finitely many points, separately.
Correspondingly, in Subsection \ref{Subsection_Description_of_the_closure}, \ref{Subsection_Description_of_the_closure_2}., the index set $I$ for the components splits as $I = I^X \amalg I^Y \amalg I^{\overline{Y}} \amalg I^Z$.
Consequently, in \ref{Subsection_Description_of_the_closure_4}.~and \ref{Subsection_Description_of_the_closure_5}., one has a splitting $\Sigma^{\hat{\ell}}_{U^{\hat{\ell}}} = \Sigma_{U^{\hat{\ell}}}^{\hat{\ell},X} \cup \Sigma_{U^{\hat{\ell}}}^{\hat{\ell},Y} \cup \Sigma_{U^{\hat{\ell}}}^{\hat{\ell},\overline{Y}} \cup \Sigma_{U^{\hat{\ell}}}^{\hat{\ell},Z}$ into families of nodal Riemann surfaces of Euler characteristics $\chi^X, \chi^Y, \chi^{\overline{Y}}, \chi^Z$, respectively.
In \ref{Subsection_Description_of_the_closure_6}.~there are now two sets of additional marked points, with index sets $\{1, \dots, \ell\}$ and $\{1, \dots, \overline{\ell}\}$, which split as $\{1, \dots, \ell\} = K^X \amalg K^{Y} \amalg K^Z$ and $\{1, \dots, \overline{\ell}\} = L^X \amalg L^{\overline{Y}} \amalg L^Z$.
The nodal points $(N^{\hat{\ell},i}_r)_{i=1,2}$, split into disjoint sets $(N^{\hat{\ell},X,i}_r)_{i=1,2}$, $(N^{\hat{\ell},Y,i}_r)_{i=1,2}$, $(N^{\hat{\ell},\overline{Y},i}_r)_{i=1,2}$, $(N^{\hat{\ell},Z,i}_r)_{i=1,2}$, $r = 1, \dots, d^X,d^Y,d^{\overline{Y}},d^Z$, where both lie on $\Sigma_{U^{\hat{\ell}}}^{\hat{\ell},X}$, $\Sigma_{U^{\hat{\ell}}}^{\hat{\ell},Y}$, $\Sigma_{U^{\hat{\ell}}}^{\hat{\ell},\overline{Y}}$, $\Sigma_{U^{\hat{\ell}}}^{\hat{\ell},Z}$, respectively, and sets $(N^{\hat{\ell},XY,X}_r, N^{\hat{\ell},XY,Y}_r)$, $(N^{\hat{\ell},X\overline{Y},X}_r, N^{\hat{\ell},X\overline{Y},\overline{Y}}_r)$, $(N^{\hat{\ell},XZ,X}_r, N^{\hat{\ell},XZ,Z}_r)$, $(N^{\hat{\ell},YZ,Y}_r, N^{\hat{\ell},YZ,Z}_r)$, $(N^{\hat{\ell},\overline{Y}Z,\overline{Y}}_r, N^{\hat{\ell},\overline{Y}Z,Z}_r)$, $r = 1, \dots d^{XY},d^{X\overline{Y}},d^{XZ},d^{YZ},d^{\overline{Y}Z}$, where the first one lies on $\Sigma^{\hat{\ell},X}_{U^{\hat{\ell}}}$, $\Sigma^{\hat{\ell},X}_{U^{\hat{\ell}}}$, $\Sigma^{\hat{\ell},X}_{U^{\hat{\ell}}}$, $\Sigma^{\hat{\ell},Y}_{U^{\hat{\ell}}}$, $\Sigma^{\hat{\ell},\overline{Y}}_{U^{\hat{\ell}}}$, the second one on $\Sigma^{\hat{\ell},Y}_{U^{\hat{\ell}}}$, $\Sigma^{\hat{\ell},\overline{Y}}_{U^{\hat{\ell}}}$, $\Sigma^{\hat{\ell},Z}_{U^{\hat{\ell}}}$, $\Sigma^{\hat{\ell},Z}_{U^{\hat{\ell}}}$, $\Sigma^{\hat{\ell},Z}_{U^{\hat{\ell}}}$, respectively. \\
Consider $u \in \cl\overset{\circ}{\mathcal{M}}(\tilde{X}^{\hat{\ell}}, \tilde{Y}^{\hat{\ell}}, A, J, H)$ with $\pi^{\mathcal{M}}_M(u) = b$ that has a fixed set of data as above associated.
Denote by $u^X$, $u^Y$, $u^{\overline{Y}}$ and $u^Z$ the corresponding parts of $u$.

The results from Subsection \ref{Subsection_Vanishing_homology_class} carry over to show the following:
First, observe that $Z$ as a complex hypersurface of $Y$ and $\overline{Y}$ has degrees $\overline{D}$ and $D$, respectively, so for generic $H^Z \in \mathcal{H}(\tilde{Z})$, for any choice of data as above, $\mathcal{M}(\tilde{Z}^{\hat{\ell}}|_{\Sigma^{\hat{\ell},Z}_{U^{\hat{\ell}}}}, B, J, H^Z)$ is either empty, for $B \neq 0$, or a smooth manifold of dimension $\dim_\C(Z)\chi^Z + \dim_\R(U^{\hat{\ell}})$.
Then, for a generic choice of $H^{W,Z} \in \mathcal{H}^{00}(\tilde{W}, \tilde{Z})$, $\mathcal{M}^{\tilde{Y}^{\hat{\ell}} \setminus \tilde{Z}^{\hat{\ell}}}(\tilde{Y}^{\hat{\ell}}|_{\Sigma^{\hat{\ell},Y}_{U^{\hat{\ell}}}}, B, J, H^W)$, $H^W \definedas H^Z + H^{W,Z}$, is either empty, for $B \neq 0$, or a smooth manifold of dimension $\dim_\C(Y)\chi^Y + \dim_\R(U^{\hat{\ell}})$, and likewise for $\overline{Y}$.
Note that using Gromov compactness, the condition $\mathcal{M}^{\tilde{Y}^{\hat{\ell}} \setminus \tilde{Z}^{\hat{\ell}}}(\tilde{Y}^{\hat{\ell}}|_{\Sigma^{\hat{\ell},Y}_{U^{\hat{\ell}}}}, B, J, H^W) = \emptyset$ for $B \neq 0$, and likewise for $\overline{Y}$, is an open condition, so is satisfied for arbitrary perturbations of $H^W$, provided they are small enough.
Also note that if $u^Y$ and $u^{\overline{Y}}$ represent vanishing homology class, then positivity of intersections implies that they do not intersect $\tilde{Z}^{\hat{\ell}}$.
In particular, one can assume that $d^{YZ} = d^{\overline{Y}Z} = 0$.

Next, the results from Subsection \ref{Subsection_Refined_compactness_result} can be strengthened in the following way: Let $E \to S$ be any holomorphic line bundle over a Riemann surface and let $\xi$ be a nonvanishing meromorphic section. If $b \in S$ is a zero or pole of $\xi$, then $\xi$ induces an orthogonal map $|\xi|_b : S^1 \cong (T_bS \setminus \{0\})/_{\R^+} \to (E_b\setminus \{0\})/_{\R^+}\cong S^1$.
This can be applied directly to the meromorphic sections $\xi_i$ of $(u_i)^\ast (V\tilde{Y}^{\ell})^{\perp_\omega}$ from Lemma \ref{Lemma_SFT_Compactness} to get well-defined orthogonal maps (denote by $\xi^Y$ the collection of the $\xi_i$)
\[
|\xi^Y|_{N^{\ell,Y,i}_r} : (T_{N^{\ell,Y,i}_r} \hat{S}^{\ell}_b \setminus \{0\})/_{\R^+} \to ((V\tilde{Y}^{\ell})^{\perp_\omega}_{u^Y(N^{\ell,Y,i}_r)} \setminus \{0\})/_{\R^+}\text{,}
\]
for $i=1,2$ and those values of $r$ for which the corresponding meromorphic sections on the components $\hat{S}^{\ell}_{j,b}$ containing $N^{\ell,Y,i}_r(b)$ have a zero/pole.
Similarly, one gets
\[
|\xi^Y|_{N^{\ell,XY,Y}_r} : (T_{N^{\ell,XY,Y}_r} \hat{S}^{\ell}_b \setminus \{0\})/_{\R^+} \to ((V\tilde{Y}^{\ell})^{\perp_\omega}_{u^Y(N^{\ell,XY,Y}_r)} \setminus \{0\})/_{\R^+}\text{,}
\]
Using local coordinates in a neighbourhood of any of the $N^{\ell,XY,X}_r$ and a neighbourhood of $u^X(N^{\ell,XY,X}_r)$, one can also apply it to the normal component of $u^X$ to also get a well-defined orthogonal map
\[
|u^X|_{N^{\ell,XY,X}_r} : (T_{N^{\ell,XY,X}_r} \hat{S}^{\ell}_b \setminus \{0\})/_{\R^+} \to ((V\tilde{Y}^{\ell})^{\perp_\omega}_{u^X(N^{\ell,XY,X}_r)} \setminus \{0\})/_{\R^+}\text{.}
\]
The compactness results from \cite{MR2026549} then allow one to put the following additional restrictions on the $\xi_i$ (notation is from Lemma \ref{Lemma_SFT_Compactness}): \\
There exist smooth orthogonal identifications of $(T_{N^{\ell,Y,i}_r} \hat{S}^{\ell}_b \setminus \{0\})/_{\R^+}$, $i=1,2$, $(T_{N^{\ell,XY,X}_r} \hat{S}^{\ell}_b \setminus \{0\})/_{\R^+}$ and $(T_{N^{\ell,XY,Y}_r} \hat{S}^{\ell}_b \setminus \{0\})/_{\R^+}$ with $S^1$, for $b \in U^\ell$, \st the following hold:
\begin{itemize}
  \item If $i,j$ are such that $N^{\ell,Y,1}_r(b) \in \hat{S}^\ell_{i,b}$ and $N^{\ell,Y,2}_r(b) \in \hat{S}^\ell_{j,b}$ and \st $\xi_i(N^{\ell,Y,1}_r(b)) \neq 0,\infty$ (and consequently $\xi_j(N^{\ell,Y,2}_r(b)) \neq 0,\infty$), then $\xi_i(N^{\ell,Y,1}_r(b)) = \xi_j(N^{\ell,Y,2}_r(b))$;
  \item if $i,j$ are such that $N^{\ell,Y,1}_r(b) \in \hat{S}^\ell_{i,b}$ and $N^{\ell,Y,2}_r(b) \in \hat{S}^\ell_{j,b}$ and \st $\xi_i(N^{\ell,Y,1}_r(b)) = \infty$ (and consequently $\xi_j(N^{\ell,Y,2}_r(b)) = 0$), then $|\xi^Y|_{N^{\ell,Y,1}_r(b)} = |\xi^Y|_{N^{\ell,Y,2}_r(b)}$;
  \item $|\xi^Y|_{N^{\ell,XY,Y}_r(b)} = |u^X|_{N^{\ell,XY,X}_r(b)}$.
\end{itemize}
The above conditions are no longer invariant under the action of $(\C^\ast)^{I^Y}$ by multiplication in the fibres of the $(u_i)^\ast (V\tilde{Y}^{\ell})^{\perp_\omega}$, for $i \in I^Y$.
But if $k^Y\in\N$ is the number of levels of the holomorphic building formed by the meromorphic sections, as in the proof of Lemma \ref{Lemma_SFT_Compactness} (where this number was just called $k$), then there still is a free $(\R^+)^{k^Y}$-action on the space of meromorphic sections under which the above conditions are invariant.
One can hence modify Definition \ref{Definition_M_D_Y} to say
\begin{align*}
\mathcal{M}^{\mathcal{D}}_{Y,u}(\tilde{X}^{\ell}, \tilde{Y}^{\ell}, H) \definedas \{(\xi_i)_{i\in I^Y} \;|\; & \xi_i \text{ a meromorphic section of } \\
& u_i^\ast (V\tilde{Y}^{\ell})^{\perp_\omega} \text{ with zeroes/nodes} \\
& \text{at the $N^\ast_\ast$ given by the $p^\ast_\ast$ and} \\
& \forall\, i,j\in I^Y, r,s \text{ with } u(N^i_r) = u(N^j_s) : \\
& p^i_r = p^j_s = 0 \;\Rightarrow\; \xi_i(N^i_r) = \xi_j(N^j_s) \\
& p^i_r = -p^j_s \neq 0 \;\Rightarrow\; |\xi_i|_{N^i_r} = |\xi_j|_{N^j_s}\}/_{(\R^+)^{k^Y}}\text{.}
\end{align*}
Lemma \ref{Lemma_SFT_Compactness_II} then still holds true whereas in Lemma \ref{Lemma_H_0_reg}, the dimension can instead be estimated by
\begin{align*}
\dim_\R\left(\mathcal{M}^{\mathcal{D}}_Y(\tilde{X}^{\ell}, \tilde{Y}^{\ell}, J, H)\right) \leq \dim_\C(X)\chi^Y + \dim_\R(U^{\ell}) + d^Y - k^Y\text{,}
\end{align*}
($d'$ has been replaced by $d^Y$ in the present notation) because any condition of the form $\xi_i(N^i_r) = \xi_j(N^j_s)$ cuts the (real) expected dimension down by $2$ and any condition of the form $|\xi_i|_{N^i_r} = |\xi_j|_{N^j_s}$ cuts the (real) expected dimension down by $1$.

Now assume that $I^Y, I^{\overline{Y}}, I^Z \neq 0$, the case $I^Y = I^{\overline{Y}} = I^Z = \emptyset$ is treated exactly the same way as the case $I^Y = \emptyset$ was treated in Subsection \ref{Subsection_Putting_it_all_together}, and the other cases are treated similarly.
Also assume that $K^X, K^Y, K^Z, L^X, L^{\overline{Y}}, L^Z$ are all nonempty, the other cases being similar. \\
Still consider $u \in \cl\overset{\circ}{\mathcal{M}}(\tilde{X}^{\hat{\ell}}, \tilde{Y}^{\hat{\ell}}, A, J, H)$ that has a fixed set of data as above associated and still denote by $u^X$, $u^Y$, $u^{\overline{Y}}$ and $u^Z$ the corresponding parts of $u$.
First, apply the modified compactness result above twice for the component $u^Z$, projecting onto $\tilde{Z}^k$.
Once rescaling in $(V\tilde{Y}^{\hat{\ell}})^{\perp_\omega}|_{\tilde{Z}^{\hat{\ell}}} = (V\tilde{Z}^{\hat{\ell}})^{\perp_\omega}$, where on the right hand side the complement is taken in $V\tilde{\overline{Y}}{}^{\hat{\ell}}$.
And a second time rescaling in $(V\tilde{\overline{Y}}{}^{\hat{\ell}})^{\perp_\omega})|_{\tilde{Z}^{\hat{\ell}}} = (V\tilde{Z}^\ell)^{\perp_\omega}$, where on the right hand side the complement is taken in $V\tilde{Y}^{\hat{\ell}}$.
For a generic perturbation $H^{Z,0} \in \mathcal{H}^0_{\mathrm{ni}}(\tilde{W}^{\hat{\ell}}, \tilde{Z}^{\hat{\ell}}, J')$, achieve transversality for all the resulting pairs of meromorphic buildings $(\xi^{ZY}, \xi^{Z\overline{Y}})$.
Because the components $u^Y$ and $u^{\overline{Y}}$ then do not intersect $\tilde{Z}^{\hat{\ell}}$, by compactness one can find a neighbourhood $U \subseteq X$ of $Z$ \st the images of the components $u^Y$ and $u^{\overline{Y}}$ do not intersect $\tilde{U}^{\hat{\ell}}$. \\
Next, perturb $J'$ outside of $W \cup U$ to a $J''$ as in the statement of the theorem and extend all Hamiltonian perturbations so far to all of $X$ in such a way that they satisfy normal integrability outside of $\tilde{U}^{\hat{\ell}}$ as described in the paragraph before the statement of the theorem. \\
Then apply the modified compactness result described above twice more for the components $u^Y$ and $u^{\overline{Y}}$ (alternatively one can also apply the compactness result from \cite{1103.3977}), so $u^Y$ and $u^{\overline{Y}}$ come with meromorphic buildings $\xi^Y$ and $\xi^{\overline{Y}}$ as above with numbers of levels $k^Y$ and $k^{\overline{Y}}$, respectively, while $u^Z$ comes with a pair of meromorphic buildings $(\xi^{ZY}, \xi^{Z\overline{Y}})$ with numbers of levels $(k^{ZY}, k^{Z\overline{Y}})$.
It is fairly easy to see that the expected dimension of the moduli space of such pairs can be estimated by
\[
\dim_\C(X)\chi^Z + \dim_\R(U^{\hat{\ell}}) + 2d^Z - k^{ZY} - k^{Z\overline{Y}}\text{.}
\]
Denoting by $s^Y, s^Z$ and $t^{\overline{Y}}, t^Z$ the total orders of tangency of $u^X$ to $\tilde{Y}^{\hat{\ell}}$ at the $N^{\hat{\ell},XY,X}_r, N^{\hat{\ell},XZ,X}_r$ and to $\tilde{\overline{Y}}{}^{\hat{\ell}}$ at the $N^{\hat{\ell},X\overline{Y},X}_r, N^{\hat{\ell},XZ,X}_r$, the following equations hold:
\begin{align*}
d^{XY} + s^Y &= |K^Y| & d^{XZ} + s^Z &= |K^Z| \\
d^{X\overline{Y}} + t^{\overline{Y}} &= |L^{\overline{Y}}| & d^{XZ} + t^Z &= |L^Z|\text{.}
\end{align*}
Making further successive choices of perturbations as described in the statement of the theorem and also factoring in the compatibility conditions (described below), the tuple $(u^X, \xi^Y, \xi^{\overline{Y}}, (\xi^{ZY}, \xi^{Z\overline{Y}}))$ lies in a moduli space of real dimension at most
\begin{align*}
& \dim_\C(X)\chi^X + \dim_\R(U^{\hat{\ell}}) + 2c_1(A) \;+ \\
+\; & \dim_\C(X)\chi^Y + \dim_\R(U^{\hat{\ell}}) + d^Y - k^{Y} \;+ \\
+\; & \dim_\C(X)\chi^{\overline{Y}} + \dim_\R(U^{\hat{\ell}}) + d^{\overline{Y}} - k^{\overline{Y}} \;+ \\
+\; & \dim_\C(X)\chi^Z + \dim_\R(U^{\hat{\ell}}) + 2d^Z - k^{ZY} - k^{Z\overline{Y}} \;- \\
-\; & (\dim_\R(U^{\hat{\ell}}) + \dim_\C(X)2d^{XY} + d^{XY}) \;- \\
-\; & (\dim_\R(U^{\hat{\ell}}) + \dim_\C(X)2d^{X\overline{Y}} + d^{X\overline{Y}}) \;- \\
-\; & (\dim_\R(U^{\hat{\ell}}) + \dim_\C(X)2d^{XZ} + 2d^{XZ}) \;- \\
-\; & (2|K^X| + 2s^Y + 2s^Z + 2|L^X| + 2t^{\overline{Y}} + 2t^Z) \\
=\; & \dim_\C(X)\chi + 2c_1(A) \;+ \\
+\; & \dim_\R(U^{\hat{\ell}}) - 2(|K^X| + |K^Y| + |K^Z| + |L^X| + |L^{\overline{Y}}| + |L^Z|) \;+ \\
+\; & d^Y + d^{XY} - k^Y + d^{\overline{Y}} + d^{X\overline{Y}} - k^{\overline{Y}} + 2d^Z + 2d^{XZ} - k^{ZY} - k^{Z\overline{Y}} \\
=\; & \dim_\C(X)\chi + 2c_1(A) + \dim_\R(M) - 2d^X \;- \\
-\; & (d^Y + d^{XY} + k^Y + d^{\overline{Y}} + d^{X\overline{Y}} + k^{\overline{Y}} + k^{ZY} + k^{Z\overline{Y}}) \\
\leq\; & \dim_\R\Bigl(\overset{\circ}{\mathcal{M}}(\tilde{X}^{\hat{\ell}}, \tilde{Y}^{\hat{\ell}}, \tilde{\overline{Y}}{}^{\hat{\ell}}, A, J', H)\Bigr) - 2\text{.}
\end{align*}
The first four lines above are the maximal expected dimensions of the moduli spaces containing $u^X$, $\xi^Y$, $\xi^{\overline{Y}}$ and $(\xi^{ZY}, \xi^{Z\overline{Y}})$, respectively.
The $5^\text{th}$ line is given by the compatibility conditions $u^X(N^{\hat{\ell},XY,X}_r) = u^Y(N^{\hat{\ell},XY,Y}_r)$ and the conditions $|\xi^Y|_{N^{\hat{\ell},XY,Y}_r} = |u^X|_{N^{\hat{\ell},XY,X}_r}$, and similarly for the $6^\text{th}$ line.
The $7^\text{th}$ line is given by the compatibility conditions $u^X(N^{\hat{\ell},XZ,X}_r) = u^Z(N^{\hat{\ell},XZ,Z}_r)$ and the conditions $|\xi^{ZY}|_{N^{\hat{\ell},XZ,Z}_r} = |u^X|_{N^{\hat{\ell},XZ,X}_r}$ as well as $|\xi^{Z\overline{Y}}|_{N^{\hat{\ell},XZ,Z}_r} = |u^X|_{N^{\hat{\ell},XZ,X}_r}$.
The $8^\text{th}$ line is given by the tangency conditions of $u^X$ to $\tilde{Y}^{\hat{\ell}}$ and $\tilde{\overline{Y}}{}^{\hat{\ell}}$ at the additional marked points on $\Sigma^{\hat{\ell},X}_{U^{\hat{\ell}}}$ and at the $N^{\hat{\ell},XY,X}_r$, $N^{\hat{\ell},X\overline{Y},X}_r$ and $N^{\hat{\ell},XZ,X}_r$. \\
In the second to last line, if $I^Y \neq \emptyset$, then $d^{XY}, k^Y \geq 1$, otherwise the first three summands vanish.
And similarly, if $I^{\overline{Y}} \neq \emptyset$, then $d^{X\overline{Y}}, k^{\overline{Y}} \geq 1$, otherwise summands $4$--$6$ vanish.
Finally, if $I^Z \neq \emptyset$, then $k^{ZY}, k^{Z\overline{Y}} \geq 1$, otherwise the last two summands vanish.
Hence, if at least one of $I^Y, I^{\overline{Y}}, I^Z$ is not empty, then the estimate in the last line above holds.
\end{proof}

\clearpage

\bibliographystyle{alpha}
\addcontentsline{toc}{section}{Bibliography}
\bibliography{Bibliography}

\newcommand{\etalchar}[1]{$^{#1}$}
\begin{thebibliography}{FFGW12}

\bibitem[BEH{\etalchar{+}}03]{MR2026549}
F.~Bourgeois, Y.~Eliashberg, H.~Hofer, K.~Wysocki, and E.~Zehnder.
\newblock Compactness results in symplectic field theory.
\newblock {\em Geom. Topol.}, 7:799--888, 2003.

\bibitem[BP00]{MR1739177}
M.~Boggi and M.~Pikaart.
\newblock Galois covers of moduli of curves.
\newblock {\em Compositio Math.}, 120(2):171--191, 2000.

\bibitem[Cie06]{Cieliebak_Lecture_Notes_SFT_II}
Kai Cieliebak.
\newblock Lectures on {S}ymplectic {F}ield {T}heory {II}.
\newblock 2006.

\bibitem[CM07]{MR2399678}
Kai Cieliebak and Klaus Mohnke.
\newblock Symplectic hypersurfaces and transversality in {G}romov-{W}itten
  theory.
\newblock {\em J. Symplectic Geom.}, 5(3):281--356, 2007.

\bibitem[Fab10]{MR2563686}
Oliver Fabert.
\newblock Contact homology of {H}amiltonian mapping tori.
\newblock {\em Comment. Math. Helv.}, 85(1):203--241, 2010.

\bibitem[FFGW12]{1210.6670}
Oliver Fabert, Joel~W. Fish, Roman Golovko, and Katrin Wehrheim.
\newblock Polyfolds: A first and second look, 2012.

\bibitem[Flo88]{MR948771}
Andreas Floer.
\newblock The unregularized gradient flow of the symplectic action.
\newblock {\em Comm. Pure Appl. Math.}, 41(6):775--813, 1988.

\bibitem[Ger13]{ediss15878}
Andreas Gerstenberger.
\newblock {\em Universal moduli spaces in {Gromov}-{Witten} theory}.
\newblock {D}issertation, Ludwig-Maximilians-Universit{\"a}t M{\"u}nchen, 2013.
\newblock \url{http://nbn-resolving.de/urn:nbn:de:bvb:19-158780}.

\bibitem[Gro85]{MR809718}
M.~Gromov.
\newblock Pseudoholomorphic curves in symplectic manifolds.
\newblock {\em Invent. Math.}, 82(2):307--347, 1985.

\bibitem[Hum97]{MR1451624}
Christoph Hummel.
\newblock {\em Gromov's compactness theorem for pseudo-holomorphic curves},
  volume 151 of {\em Progress in Mathematics}.
\newblock Birkh\"auser Verlag, Basel, 1997.

\bibitem[HWZ10]{MR2644764}
Helmut Hofer, Kris Wysocki, and Eduard Zehnder.
\newblock sc-smoothness, retractions and new models for smooth spaces.
\newblock {\em Discrete Contin. Dyn. Syst.}, 28(2):665--788, 2010.

\bibitem[Ion11]{1103.3977}
Eleny-Nicoleta Ionel.
\newblock {GW} invariants relative normal crossings divisors, 2011.

\bibitem[IP03]{MR1954264}
Eleny-Nicoleta Ionel and Thomas~H. Parker.
\newblock Relative {G}romov-{W}itten invariants.
\newblock {\em Ann. of Math. (2)}, 157(1):45--96, 2003.

\bibitem[IP13]{1302.3472}
Eleny-Nicoleta Ionel and Thomas~H. Parker.
\newblock A natural {Gromov}-{Witten} virtual fundamental class, 2013.

\bibitem[Loc81]{MR610188}
Robert~B. Lockhart.
\newblock Fredholm properties of a class of elliptic operators on noncompact
  manifolds.
\newblock {\em Duke Math. J.}, 48(1):289--312, 1981.

\bibitem[Loc87]{MR879560}
Robert Lockhart.
\newblock Fredholm, {H}odge and {L}iouville theorems on noncompact manifolds.
\newblock {\em Trans. Amer. Math. Soc.}, 301(1):1--35, 1987.

\bibitem[Loo94]{MR1257324}
Eduard Looijenga.
\newblock Smooth {D}eligne-{M}umford compactifications by means of {P}rym level
  structures.
\newblock {\em J. Algebraic Geom.}, 3(2):283--293, 1994.

\bibitem[Mac90]{MacPherson_1990}
Robert MacPherson.
\newblock Intersection homology and perverse sheaves.
\newblock 1990.

\bibitem[MS98]{MR1698616}
Dusa McDuff and Dietmar Salamon.
\newblock {\em Introduction to symplectic topology}.
\newblock Oxford Mathematical Monographs. The Clarendon Press Oxford University
  Press, New York, second edition, 1998.

\bibitem[MS04]{MR2045629}
Dusa McDuff and Dietmar Salamon.
\newblock {\em {$J$}-holomorphic curves and symplectic topology}, volume~52 of
  {\em American Mathematical Society Colloquium Publications}.
\newblock American Mathematical Society, Providence, RI, 2004.

\bibitem[Mum83]{MR717614}
David Mumford.
\newblock Towards an enumerative geometry of the moduli space of curves.
\newblock In {\em Arithmetic and geometry, {V}ol. {II}}, volume~36 of {\em
  Progr. Math.}, pages 271--328. Birkh\"auser Boston, Boston, MA, 1983.

\bibitem[MW12]{1208.1340}
Dusa McDuff and Katrin Wehrheim.
\newblock Smooth kuranishi atlases with trivial isotropy, 2012.

\bibitem[PPT10]{1101.0180}
M.~J. Pflaum, H.~Posthuma, and X.~Tang.
\newblock Geometry of orbit spaces of proper lie groupoids, 2010.

\bibitem[RS06]{MR2262197}
Joel~W. Robbin and Dietmar~A. Salamon.
\newblock A construction of the {D}eligne-{M}umford orbifold.
\newblock {\em J. Eur. Math. Soc. (JEMS)}, 8(4):611--699, 2006.

\bibitem[RT95]{MR1366548}
Yongbin Ruan and Gang Tian.
\newblock A mathematical theory of quantum cohomology.
\newblock {\em J. Differential Geom.}, 42(2):259--367, 1995.

\bibitem[RT97]{MR1483992}
Yongbin Ruan and Gang Tian.
\newblock Higher genus symplectic invariants and sigma models coupled with
  gravity.
\newblock {\em Invent. Math.}, 130(3):455--516, 1997.

\bibitem[Weh09]{Wehrheim_Poly_Slides_1}
Katrin Wehrheim.
\newblock Polyfolds and holomorphic disks, 2009.

\bibitem[Weh12]{Wehrheim_Poly_Slides_2}
Katrin Wehrheim.
\newblock Introduction to polyfolds, 2012.

\bibitem[Wen12]{1202.4685}
Chris Wendl.
\newblock Contact hypersurfaces in uniruled symplectic manifolds always
  separate, 2012.

\end{thebibliography}

\end{document}